\documentclass[11pt,reqno]{amsart}

\usepackage[margin=0.94in]{geometry}  
 \usepackage{amsmath,amssymb}
 \usepackage{algpseudocode}
 \usepackage{algorithm}
 \usepackage{algorithmicx}
 \usepackage{caption}
 \usepackage{subcaption}
 \usepackage{amsthm}
 \numberwithin{equation}{section}
 \usepackage{amsfonts}
 \usepackage{graphicx}
  \usepackage{mathrsfs}
 \usepackage{hyperref}
 \usepackage{makecell}
 \usepackage{longtable}
 \usepackage{comment}
 \usepackage{enumitem}
\usepackage{tikz}

\usepackage{amsmath}
\usepackage{amssymb}
\usepackage{amsthm}
\usepackage{amsmath}
\usepackage{setspace}
\usepackage{xcolor}
\usepackage{fancyhdr}
\usepackage{amssymb}
\usepackage{amsthm}
\usepackage{listings}
    \usepackage{cite}
    \usepackage{tabularx}
    \usepackage{graphicx}
    \usepackage{epstopdf}    
    \usepackage{epsfig}
    \usepackage{float}
    \usepackage{ listings}
    \usepackage{appendix}
 \theoremstyle{plain}

 \newtheorem{thm}{Theorem}[section]
 
\newtheorem*{conjecture*}{Conjecture}

 \newtheorem{prop}[thm]{Proposition}
 \newtheorem{lem}[thm]{Lemma}
 \newtheorem{cor}[thm]{Corollary}
 \theoremstyle{definition}
 \newtheorem{definition}[thm]{Definition}

 \newtheorem{remark}[thm]{Remark}
  \newtheorem{ques}{Question}

\newcommand\ang[1]{ { \la {#1}\ra } }

\newcommand \nlinf[1]{ {\| #1 \|}_{\linf} }

\newcommand \nchi[1]{ {\| #1 \|}_{\cX} }
\newcommand \nchib[1]{ {\| #1 \|}_{\bcX} }

\def \beps {\bar \e}

\let \hyr = \hyperref
\let \its = \itshape

\def \cXb {\cX_2}
\def \cXc {\cX_3}

\def \cLa {\cL_{\al}}
\def \cLab {\cL_{\wa}}
\def \cEab {\cE_{\al, \b}}
\def \cNab {\cN_{\al, \b}}

\def \cRa { \td \cR_{\al}}

\def \wwd {\om_{\mw{1D}}}

\def \vpa {\vp_\sst}
\def \vpb {\vp_b}

\def \vpcc {\vp_c}

\def \vpk {\vp_{\mw{one}}}

\def \hk { {  \f{1+\kp}{2} } }
\def \mhk { {  \f{1-\kp}{2} } }

\def \sst { \mf{ne}}

\def \muc { \mu_{\mf{ctr}} }

\def \alb{\bar \al}
\def \vmu{ \vec \mu }

\def \nmu { { n_{\mu}} }
\def \nst { n_{\sst}}

\def \he {\hat \e_{\b}}

\def \hal {\hat \al_l}
\def \hau {\hat \al_u}

\def \cpsidxo {\mCC_{\psio/x,0} }
\def \cpsixo {\mCC_{\psio_x, 0} }

\def \cpsidxJ {\mCC_{\psio/x, J}}
\def \cpsixJ {\mCC_{\psio_x, J}}

\def \Cpsidxo {\MCC_{\psio/x,0} }

\def \Cpsidx {\MCC_{\psio/x} }

\def \CpsidxJ {\MCC_{\psio/x, J}}
\def \CpsixJ {\MCC_{\psio_x, J}}

\def \dda{\mathscr{E}_0}

\def \ddc{\mathscr{E}_2}

\def \ddd {\nu}

 \def\bw{ \bar W}

 \def \bwf{ \bar{\mathsf{ W}}_{\mf{F}} }
 \def \bwfa{ \bar{\mathsf{ W}}_{\mf{F, A}} }
 \def \bwp{ \bar{\mf{ W}}_{\mf{P}} }

 \def \cwf {\wh W_{F, \chi_2}}

\def \bD{\mathbb{D}}
\def \bDD { \mathbb{D}_{\d_*} }

\def \linf {L^{\infty}}

 \def\bv{ \bar V}

  \def \dfix { \d_{\cF}}
 \def \dfixa { \d_{\cF,1}}

\def \bbb { \kp_\mf{O}}

\def \wa{ W_{\b} }

\def \waa{ \bar{W}_{\al} }
\def \vaa {\bar V_{\al}}

\def \clone { \bar c_{l, \al } } 
\def \cwone { \bar c_{\om, \al } }

\def \wwwa {\overline W_{\alb}}
\def \vvva {\bar V_{\alb}}

\def \ww {\om}
\def \psio {\mathring{\psi}}
\def \psiox {\mathring{\psi}_x}
\def \psioy {\mathring{\psi}_y}
\def \psioa {\mathring{\psi}_{\al} }

\def \Wmix { \bw_{ \mf{mix}}}
\def \Vmix { \bv_{ \mf{mix}}}

\def \psia { \psi_{\al} }

\def \psioab {\mathring{\psi}_{\alb} }
\def \Vo { \mathring{V} }

\def \cls {\bar c_{l, \al}}
\def \cws  {\bar c_{\om, \al}}

\def \clss {c_{l,  \al, *}}
\def \cwss  { c_{\om, \al, *}}

\def \wwd {\ww_{\mathsf{1D}}}

\def \cw {c_{\om}}

 \def \va{ V_{\al ,\b}}
 \def \cJa{ \cJ_{\al} }
 
  \def \cJab{ \cJ_{\al, \b} }
 \def \cJaa{ \la \cJ_{\al, \b} \ra }

 \def \cJe{ \la \cJ_{\e} \ra }

\def \cJM { \cJ_{\mf{ M} }}

\def \mmf {\mf{m}}

 \def \ga {g_{\al, \b}}

\def \voa { \Vo_{\al, \b}}

\def \crab{ \cR_{\al,\b}}

\def \tcr {\td \cR}
\def \crbb {\td \cR_{\alb}}

\def \lamst {\lam_{\mathsf{st}}}

\def \lamalb  {\lam_{\alb}}
\def \lamcL {\lam_{\cL}}

\def \cff {C_{\mathsf{far}}}

\def \cFR {\cF_{\al, \R}}

\def \Rmua { R_{\mu, 1} }
\def \Rmub { R_{\mu, 2} }

\def \lgp { \log_+ }
\def \lgx { |\log_+ x| }
\def \lgy { |\log_+ y| }

\def \cWg { \cW_{\g} }

\def \com {c_{\om}}

\def \cca {\mf{C}_0}
\def \ccb {\mf{C}_1}
\def \ccc {\mf{C}_2}

\def \bcX {{\mathcal X_{\alb}}}

\def \linff {L^{\infty}}

\def \mBs  {\mathsf{Bs}}
\def \nnx { N_{\mmx}}
\def \nnX { N_X}
\def \xed {X_{\mf{max}}}

\def \mmx {\mw{x}}
\def \GQF { \mf{QG}_5 }
\def \AAF {\mf{A}}

\def \ss {\mathbf{s}}

\let \tts = \textstyle

 \let \mw = \mathrm
 \let \we = \wedge
\let \mf = \mathsf
\let  \rsa =\rightsquigarrow

\let \tf = \tfrac

\let \nolim = \nolimits

 \let\pa=\partial
 \let\al=\alpha
 \let\b=\beta
 \let\d=\delta
 \let\g=\gamma
 \let \gam = \gamma 
 \let\e=\varepsilon

 \let \kp = \kappa
 \let\lam=\lambda
 
 \let\s=\sigma
 \let\f=\frac

 \let \les = \lesssim
  \let \gtr = \gtrsim
 \let\om=\omega
 
 \let \th = \theta
  
 \let \pr = \prime
 \let \vp = \varphi
 
 \let \Gam = \Gamma
\let\B = \Big
 \let\D=\Delta

 \let\td = \tilde
 
 \let\wh=\widehat
 \let \mr = \mathring

 \let\teq \triangleq
 
 \let\pa=\partial
 
 \let \vs = \vspace

 \def\cE{{\mathcal E}}
 \def\cF{{\mathcal F}}

 \def\cI{{\mathcal I}}
 \def\cJ{{\mathcal J}}
 \def\cK{{\mathcal K}}
 \def\cL{{\mathcal L}}
 \def\cM{{\mathcal M}}
 \def\cN{{\mathcal N}}
 
 \def\cP{{\mathcal P}}
 
 \def\cR{{\mathcal R}}
 
 \def\cT{{\mathcal T}}

 \def\cW{{\mathcal W}}
 \def\cX{{\mathcal X}}

 \def\cM{{\mathcal M}}
 \def\cN{{\mathcal N}}

 \def\la{\langle}
 \def\ra{\rangle}

\def\one{\mathbf{1}}

\def \uds{ \underset }

 \newcommand{\bit}{\begin{itemize}}
 \newcommand{\eit}{\end{itemize}}

 \newcommand{\bseq}{\begin{subequations}}
 \newcommand{\eseq}{\end{subequations}}

 \newcommand{\beq}{\begin{equation}}
 \newcommand{\eeq}{\end{equation}}
  \newcommand{\bal}{\begin{aligned} }
  \newcommand{\eal}{\end{aligned}}
    \newcommand{\bga}{ \begin{gathered} }
  \newcommand{\ega}{ \end{gathered} }
 \newcommand{\ben}{\begin{eqnarray}}
 \newcommand{\een}{\end{eqnarray}}
 \newcommand{\beno}{\begin{eqnarray*}}
 \newcommand{\eeno}{\end{eqnarray*}}

  \newcommand{\bb}{\mathbf{b}}

 \newcommand{\uu}{\mathbf{u}}

 \newcommand{\xx}{\mathbf{x}}
 \newcommand{\yy}{\mathbf{y}}

\newcommand{\mCC}{\mathsf{c}}
\newcommand{\MCC}{ \mathscr{C} }

 \newcommand{\R}{\mathbb{R}}

\newcommand{\sgn}{\mathrm{sgn}}
 \newcommand{\supp}{\mathrm{supp}}

\def \alb{\bar \al}
\def \vmu{ \vec \mu }

\def \bD{\mathbb{D}}
\def \bDD { \mathbb{D}_{\dfix} }

 \def \wa{ W_{\b}}
 \def \va{ V_{\al ,\b}}

\def \crab{ \cR_{\al,\b}}

 \let \mw = \mathrm
 \let \we = \wedge
\let \mf = \mathsf

 \author{Jiajie Chen}

 \address{Department of Mathematics, University of Chicago, Chicago, IL 60637.}
\email{\href{jiajiechen@uchicago.edu}{jiajiechen@uchicago.edu}}

 \date{ \today}

\title[$C^{\infty}$ limiting 1D Profiles]{Asymptotically Self-Similar Blowup for 3D Incompressible Euler with $C^{1, 1/3-}$ Velocity I: $C^{\infty}$ 1D Limiting  Profiles}

\begin{document}

\begin{abstract}

We consider a one-parameter family of 1D models for the 3D axisymmetric incompressible Euler equation 
with $C^{\alpha}$ vorticity and without swirl near the symmetry axis.  For $\alpha = \frac13$, 
we impose a crucial normalization and construct a $C^{\infty}$ self-similar blowup profile with 
unbounded 1D stream function and infinite spatial blowup rate, using a fixed-point argument around a numerically constructed approximate profile. 
For $\alpha < \frac13$ sufficiently  close to $\frac13$, we perturb the $\frac13$-profile 
and analytically construct exact smooth 1D profiles with bounded stream function and finite spatial blowup rate. In the companion work~\cite{chen2026eulerII}, for any \(\alpha \in (0,\frac13)\), 
we lift these 1D blowup profiles to construct exact \(C^{1,\alpha}\) self-similar blowup 
profiles for 3D Euler, and build on them to prove sharp asymptotically self-similar blowup 
for 3D axisymmetric Euler without swirl from \(C_c^\alpha\) initial vorticity and 
\(C^{1,\alpha} \cap L^2\) initial velocity.

\end{abstract}

 \setcounter{tocdepth}{1}
 \hypersetup{bookmarksdepth=2}

 \maketitle
 \tableofcontents

\section{Introduction}

We consider a one-parameter family of 1D models for the 3D incompressible Euler equations
\beq\label{eq:1D}
\pa_t \ww +  2  \psia  \pa_x \ww =  - (1-\al) \pa_x \psi_{\al} \ww , \quad 
\psia( \ww )(x) \teq  - \int_\R |x-y|^{\al}  \ww(y)  d y ,
\eeq
for $x \in \R$ and $\al \in (0, 1)$. The model \eqref{eq:1D}
was derived in \cite{chen2026eulerII,hou2024potential} to describe the behavior of the 3D  axisymmetric Euler without swirl near the symmetry axis, with $C^{\al}$ vorticity. Here $\psia$ models the angular stream function, $\ww$ models the angular vorticity, and $\al$ is the H\"older regularity exponent.

The question of finite-time singularity formation for the 3D incompressible Euler equations from smooth finite-energy initial data is a major open problem in mathematical fluid mechanics \cite{majda2002vorticity}. In recent years, substantial progress on Euler singularity formation has been achieved. In the remarkable work \cite{elgindi2019finite}, Elgindi established singularity formation for 3D axisymmetric Euler  without swirl, with \(C^{1,\al}\) velocity for small \(\al\).
The result was later extended to allow small swirl \cite{elgindi2019stability}.  With Hou \cite{chen2019finite2}, we established finite-time blowup for 2D Boussinesq and 3D axisymmetric Euler with \(C^{1,\al}\) velocity, large swirl, boundary, and small $\al$.  In joint works with Hou \cite{ChenHou2023a,ChenHou2023b}, we established stable nearly self-similar blowup for 2D Boussinesq and 3D axisymmetric Euler with smooth boundary from smooth initial data. This provides the first construction of a 3D Euler singularity from smooth data in a smooth domain. See also the review paper \cite{chen2025singularity}.  In \cite{cordoba2023finite},  Cordoba, Martinez-Zoroa, and Zheng used iterative layer constructions to establish blowup for 3D axisymmetric Euler with $C^{1,\alpha} \cap C^{\infty}(\R^3 \backslash \{0 \})$ velocity and small $\al$. By introducing a solution-dependent bounded force $f \in C^{1,1/2-}$, Cordoba and Martinez-Zoroa \cite{cordoba2023blow} established blowup for the forced 3D Euler with a smooth velocity. Inspired by \cite{cordoba2023finite},  in \cite{chen2024remarks}, we refined the construction in \cite{elgindi2019finite,chen2019finite2} to obtain a self-similar blowup solution with $C^{\infty}$ regularity except the blowup point. In \cite{elgindi2023instability},  Elgindi-Pasqualotto constructed finite-time blowup for Boussinesq  
equations in $\R^2$ and for 3D axisymmetric Euler with $C^{1,\al}$ velocity and small $\al$.

Besides the rigorous blowup results for 3D Euler, considerable effort has also been devoted to studying potential singularity formation mechanisms through simplified 1D models.
The study of 1D fluid models has a long history, dating back to the CLM model \cite{CLM85} for 3D Euler. Since then, many 1D models have been proposed, including the De Gregorio \cite{DG90}, CCF \cite{cordoba2005formation}, gCLM \cite{OSW08}, CKY \cite{choi2015finite}, and HL \cite{choi2014on} models. 
In recent years, self-similar methods have played an important role in constructing 
 blowup from smooth data for several such models, including the gCLM model 
with small $|a|$ by Elgindi--Jeong \cite{Elg17}, the De Gregorio and HL models by Chen--Hou--Huang \cite{chen2019finite}, the gCLM model with \(a=\f12\) and sharp blowup results near the endpoints by Chen \cite{chen2020singularity,chen2021regularity,chen2020slightly}, and the CLM model with multi-scale structure by Huang--Qin--Wang \cite{huang2024multi}. See also the self-similar blowup for  gCLM model with $C^{\al}$ data \cite{Elg17,chen2019finite,zheng2023exactly} and the De Gregorio model with $H^1$ self-similar profiles \cite{huang2023self}.

\subsection{Motivation}

Although many 1D nonlocal fluid models develop singularities from \emph{smooth data}, including 
all the above-mentioned models on the real line, these singularities have not led to singularities for incompressible fluid equations in a smooth domain with \(C^{1,\al}\) velocity without requiring small \(\al\). This naturally leads to the following questions:

\begin{ques}\label{ques:1D_model}
Can one lift a singularity from a 1D nonlocal fluid model to an incompressible fluid equation in 
$\R^2$ or $\R^3$ with \(C^{1,\al}\cap L^2\) velocity---for example, Euler, Boussinesq, SQG, or IPM---without requiring \(\al\) to be small? 
What properties of the 1D singularity are needed for such a construction?
\end{ques}

There are two fundamental obstructions to such a lifting. First, many connections between 1D models and fluid equations produce solutions with \emph{infinite energy}. 
Second, incompressible fluid equations such as 3D Euler are \emph{nonlocal}, whereas a 1D model captures only \emph{lower-dimensional} information.

One of the main motivations of this work and the companion work \cite{chen2026implosion} is to provide an affirmative answer to Question \ref{ques:1D_model}. 
In this work, we construct \(C^\infty\) self-similar blowup profiles for \eqref{eq:1D} with \(\al\le \f13\) and \(\al\) close to \(\f13\). 
The central difficulty is to construct these 1D profiles with strong stability properties 
and \emph{arbitrary large} spatial blowup exponents, which induce a strongly anisotropic self-similar flow in $\R^3$. In \cite{chen2026eulerII}, we lift these 1D blowup profiles to construct exact 3D self-similar profiles, and establish sharp asymptotically self-similar blowup 
for 3D axisymmetric Euler without swirl  from \(C_c^\al\) initial vorticity and \(C^{1,\al}\cap L^2\) initial velocity for any \(\al\in(0,\tfrac13)\). Moreover, we precisely characterize the limiting blowup behavior as \(\al\to(\tfrac13)^-\) using the key $\f13$-profile \(\wwwa\) constructed in Theorem \ref{thm:main_1D}. To the best of our knowledge, the results in this work and \cite{chen2026eulerII} provide the first example in which a singularity from a 1D nonlocal fluid model is lifted to construct blowup for incompressible fluid equations in \(\R^2\) or \(\R^3\) with \(C^{1,\al}\cap L^2\) velocity, without requiring \(\al\) to be small.

In a very recent and remarkable paper, Shkoller \cite{shkoller2026incompressible} 
developed a new Lagrangian framework to prove finite-time blowup for 3D axisymmetric Euler without swirl \eqref{eq:Euler}, with initial data $\uu_0 \in C^{1,\alpha} \cap L^2$ for the \emph{entire range} $\alpha \in (0, \f13)$. 
The companion work~\cite{chen2026eulerII} and~\cite{shkoller2026incompressible} were carried out independently and use different methods to construct \(C^{1,\alpha-} \cap L^2\) blowup solutions to 3D axisymmetric Euler without swirl for the entire range \(\alpha\in(0,\f13)\).

\subsection{Self-similar profile equation}

Consider 3D axisymmetric Euler without swirl
\beq\label{eq:Euler}
\bga
 \pa_t \om^{\th} + ( u^r \pa_r + u^z \pa_z ) \om^{\th} =  r^{-1 } u^r \om^{\th} ,  
 \qquad
  -\D_{\R^5} \psi_{\mw{3D}} = r^{-1} \om^{\th},
 \quad (u^r, u^z) = \tf{1}{r} ( -\pa_z, \pa_r )(r^2 \psi_{\mw{3D}}), 
\ega
\eeq
where
$\om^{\th}(r, z), \psi^{\th}(r, z) = r \psi_{\mw{3D}}(r, z) $ are the angular vorticity and angular stream function,  
and $(r, \th, z)$ is the cylindrical coordinate in $\R^3$. 
For $ \al \in (0, 1)$. 
Consider $\om_{\mw{3D}}(r, z) = r^{-\al} \om^{\th}$. By assuming $\om_{\mw{3D}}$ are constant in $r$: $\om_{\mw{3D}}(r, z) = \ww(z)$, 
and restricting \eqref{eq:Euler}
on the symmetry axis $r=0$, one derives the 
1D model \eqref{eq:1D}. We refer to \cite[Section 8]{hou2024potential} and \cite[Section 2]{chen2026eulerII} for the derivations. \footnote{
The 1D model \eqref{eq:1D} was essentially derived in the work of Hou-Zhang \cite[Section 8]{hou2024potential} with $n=3$. 
Our main observation in \cite{chen2026eulerII} is that the nonlocal operator can be significantly simplified as in \eqref{eq:1D}. See \cite[Lemma 3.9]{chen2026eulerII}.
}

For some self-similar exponents $c_{\om,*}, c_{l, *}$ to be determined, plugging the self-similar blowup ansatz 
\beq\label{eq:blowup_ansatz}
 \om(x, t) = (1-t)^{c_{\om,*}} W_*\big( \f{x}{ (1-t)^{c_{l,*} }} \big)  ,
\eeq
to \eqref{eq:1D} and matching the power of $1-t$, we obtain the equation for the self-similar profile
\beq\label{eq:1D_dyn00}
\bal
 (c_{l, *} x + 2 \psi_{\al} ) \pa_x W_*  =   (c_{\om, *} - (1 - \al) \pa_x \psia) W_*, 
 \quad  c_{\om, *} + \al c_{l, *} = -1 ,
\eal 
\eeq
where $\psia(W_*)$ is given by \eqref{eq:1D}.
We consider profile $W_*$ with odd symmetry. To impose normalization conditions on the solution, it is convenient to introduce the rescaling equation 
\beq\label{eq:1D_dyn0}
 (c_{l} x + 2 \psi_{\al} ) \pa_x W  =   (c_{\om} - (1 - \al) \pa_x \psia) W, 
\quad \psia(W)(x)  \teq  \int_0^{\infty} ( |x+y|^{\al} - |x-y|^{\al}) W(y) d y ,
\eeq
where we have symmetrized the kernel for $\psia$ in \eqref{eq:1D} using the odd symmetry of $W$. 

Given a solution to \eqref{eq:1D_dyn0} with $\cw + \al c_l \neq 0$, using the relation 
\beq\label{eq:1D_rescale}
( c_{l, * },  \  c_{\om, *} , \ W_* ) \teq  - ( c_{\om} +\al c_l )^{-1} (c_l , \ \com, \ W), 
\eeq
and multiplying \eqref{eq:1D_dyn0} by $- (c_{\om} +\al c_l )^{-2}$, we obtain a solution 
to \eqref{eq:1D_dyn00}.

To distinguish between solutions of \eqref{eq:1D_dyn00} and \eqref{eq:1D_dyn0}, we use the star notation $W_*$ to denote a solution of \eqref{eq:1D_dyn00}. 
While equation \eqref{eq:1D_dyn0} and \eqref{eq:1D_dyn00} are equivalent, 
formally, we have two modulation parameters $c_l, c_{\om}$ in \eqref{eq:1D_dyn0}, which make it convenient to impose vanishing conditions.

\vs{0.1in}
\paragraph{\bf Normalization}
To analyze a profile $W$ with a slow decay, we introduce the modified stream function $\psio$ 
and the self-similar velocity $V$
\bseq\label{eq:1D_normal}
\beq
 \psioa \teq \psia - \psi_{\al, x}(0) x , \quad V \teq \tf12 c_l x  + \psia  ,
\eeq
for any $\al \in (0, \f13]$. We impose the following normalizations on $V$ at $0$ and choose $\com, c_l$ as
\footnote{
\label{foot:wx0}
We can impose the two normalizations in \eqref{eq:1D_normal:b} since \eqref{eq:1D_dyn0} admits two trivial symmetries.
If $(c_l, \com, W)$ solves \eqref{eq:1D_dyn0}, for any $a, b > 0$,  we obtain  that $( a c_l, a \com,  a W_b)$ with $W_b(y) \teq b^{\al} W(b y)$ also solves \eqref{eq:1D_dyn0}.
In addition to $\pa_x V(0)=1$, we will impose a normalization condition on $ W_x(0)$ close to $-1$ when we construct the approximate profile.
}
\beq\label{eq:1D_normal:b}
V_x(0) = 1, \quad c_l + 2 \psi_{\al, x}(0) = 2 V_x(0) = c_{\om} - (1-\al) \psi_{\al, x}(0) .
\eeq

With the above normalizations, we rewrite $V$ as 
\beq\label{eq:1D_normal:c}
V = (\tf{1}{2} c_l + \psi_{\al, x}(0) ) x  + (\psia - \psi_{\al, x}(0) x) 
= V_x(0) x + \psioa = x + \psioa ,
\eeq
\eseq
and rewrite the profile equation \eqref{eq:1D_dyn0} for odd function $W$ with general $\al$
\beq\label{eq:1D_dyn}
\bga 
 2 V \pa_x W = (3- \al - (1-\al) V_x) W.  
\ega 
\eeq

We derive $\psioa$ using its associated kernel \eqref{eq:ker} and then construct $V$ via \eqref{eq:1D_normal:c}.

\begin{definition}[\bf $\alpha$-profile]
We call a nontrivial, odd, locally Lipschitz solution \(W\) of \eqref{eq:1D_dyn} 
satisfying $|W(x)| \les \min(|x|, 1 )$ for any $x$ an \(\alpha\)-profile.
\end{definition}

\subsection{Main results: existence of $C^{\infty}$ $\al$-profiles }

The following theorem is an abbreviated form of Theorem \ref{thm:reg_alb} 
for the $\f13$-profile, and Theorem \ref{thm:1D_profile_prop} 
for the $\al$-profile with $\al < \f13$ and $\al$ close to $\f13$.

\begin{thm}[\bf $C^{\infty}$ self-similar profiles--abbreviated version]\label{thm:main_1D}
There exists a small $\beps_4 \in (0,\f13)$ such that for any 
$\al \in [\f13-\beps_4,\f13]$, there exists a $C^\infty$ 
$\al$-profile $\waa$ solving \eqref{eq:1D_dyn}. Moreover, $\waa$ is odd in $x$,
satisfies $\waa(x)<0$ for $x>0$, and obeys the normalization $\pa_x \waa(0)=\pa_x \wwwa(0)<0$.
Let $\vaa=x+\psioa(\waa)$ be the associated velocity. Then, for any $k\geq 0$,
\beq\label{eq:thm_smooth}
  |\pa_x^k \waa | \les_k \ang x^{-k -\hal},  \quad 
 |\pa_x^{k+2} \vaa |   \les_k  \ang x^{-k-1
  + \al - \hal}  ,
\eeq
where $\hal -\al \in [ \f{8.9}{8} \e, \f{9.1}{8}\e]$ and  $\e = \f13 - \al $. 
In particular, $\hal = \f13$ for $\al =\f13$.  Moreover, for any $\al \in [\f13 -\beps_4, \f13)$, the 1D model \eqref{eq:1D}
admits a self-similar blowup profile and blowup solution 
\[
\bal
 \om(x, t) & = (1-t)^{ \cwss } W_*\big( \f{x}{ (1-t)^{ \clss }} \big)  , \\
  (W_*, \clss, \cwss) & = - ( \cws + \al \cls )^{-1} (\waa, \cls, \cws) ,
  \eal
\]
satisfying $ \cwss + \al \clss = -1$ and 
\beq\label{eq:blowup_scaling}
 - ( \cws + \al \cls )   = 8 + O( \e^{1-\kp}) , \quad 
\clss = \f{8}{9} \e^{-1} + O( \e^{-\kp }) , \quad 
\cwss = - \f{8}{27} \e^{-1} + O( \e^{-\kp}) .
\quad  
\eeq
 All the implicit constants in the above estimates are independent of $\al, \e$.

\end{thm}

We introduce a linear operator 
around the profile $(\waa, \vaa)$ constructed in Theorem \ref{thm:main_1D} 
\beq\label{def:cLa_intro}
\bga
  \cLa (w) \teq \waa  \int_0^{x}  \f{  \cRa(w)}{2 \vaa }  ,  
  \quad  \cRa( w )   = - (1-\al) \psiox (w) - 2 \psio  \f{\pa_x \waa}{\waa}  , 
 \ega
\eeq

The second main result is the contraction estimate of $\cLa$ in Theorem \ref{thm:contra_lin_exact};
 an abbreviated version is stated below.

\begin{thm}[\bf Uniform contraction estimates--abbreviated version]\label{thm:main_contract}
There exist $\beps_5 \in (0, \beps_4]$ and $\lamcL \in (0, 1)$ such that, for any $ \al \in (\f13 - \beps_5, \f13)$, the following holds. There exists a weight $\vp_{\al}  > 0$ such that the operator $\cL_{\al}$ defined in \eqref{def:cLa_intro} satisfies the contraction estimate
\[
    \nlinf{ \vp_{\al} \cLa(w)} \leq  \lamcL \nlinf{ w \vp_{\al}} ,
\] 
for any $w \vp_{\al} \in L^{\infty}$.  In particular, the contraction parameter $\lamcL$ is uniform in $\al$. 

\end{thm}

We highlight two key features of the profiles and discuss other properties 
for the lifting construction.

\begin{remark}[\bf Infinite blowup rates and normalization]\label{rem:infty}
The blowup rates of the \emph{smooth} limiting $\f13$-profile $\wwwa$ are \emph{infinite}: $c_l = \infty, \com = -\infty$, and $\wwwa$ has a slow decay rate $\wwwa \sim x^{-1/3}$. 
As a result,  the integrand in \eqref{eq:1D_dyn0} only has a non-integrable decay rate $y^{-1}$ and its stream function is not well-defined by the formula \eqref{eq:1D_dyn0}; moreover, we obtain  $\psi_{\alb, x}(0) = -\infty$. To capture these unbounded quantities, we impose  the normalization in \eqref{eq:1D_normal:c} \emph{a priori} so that the $+\infty$ from $c_l$ and  $-\infty$ 
from $\pa_x \psi_{\alb} (0)$ are \emph{cancelled}, leading to a normalized value $2$. 
The same cancellation applies to $\com,  \psi_{\alb, x} (0)$ in \eqref{eq:1D_dyn0}. 
To overcome the slow decay of the profile, we introduce  the modified stream function in \eqref{eq:1D_normal}, which is well-defined for odd $\ww$ with a slow decay rate due to the improved decay rate of the kernel in $y$; see the integral formula in \eqref{eq:ker} and the estimate  \eqref{eq:ker_decay}. 
In particular, the formulation \eqref{eq:1D_dyn} avoids the need to evaluate $\psi_{\al,x}(0), c_l, c_{\om}$.

\end{remark}

\begin{remark}[\bf Self-similar velocity with superlinear growth]
\label{rem:superlinear}
Due to the slow decay of the profile $\wwwa$, both the modified stream function 
$ \psio_{\alb}(\wwwa)$ and the self-similar velocity 
$\vvva = x + \psio_{\alb}(\wwwa)$ grow like $ x \lgp x$ for $x$ sufficiently large. This stands in sharp contrast to existing works on (nearly) self-similar blowup in incompressible fluids \cite{elgindi2019finite,elgindi2019stability},\cite{chen2019finite2,ChenHou2023a,ChenHou2023b} and related 1D models 
\cite{chen2019finite,chen2021HL,huang2024self,huang2025exact}, where the fluid velocity 
(corresponding to $\psi$ in \eqref{eq:1D_dyn0}) has a \emph{sublinear} growth. 
Thus, in our case, the nonlinear terms $ 2 V \pa_x W, V_x W$ in \eqref{eq:1D_dyn0}  \emph{cannot} be treated perturbatively for large $x$, which is a fundamental new difficulty in the present work.
To overcome this difficulty, we prove that the perturbation satisfies a \emph{faster} and stronger decay estimate, gaining a factor of $|\log x|^{-1/3}$ relative to the (approximate) profile.
\footnote{
This stands in sharp contrast to existing works on self-similar blowup analysis for fluid-type PDEs, including \cite{chen2019finite,chen2021HL,chen2019finite2,elgindi2019finite,ChenHou2023a,chen2024Euler,chen2024vorticity}, where the perturbation satisfies 
 \emph{slower} decay estimates than the profile while the nonlinear estimates still close. In our setting, closing the nonlinear estimates requires propagating \emph{faster} decay of the perturbation.
}

\end{remark}

 To the best of our knowledge, Theorem \ref{thm:main_1D} provides, within the class of 1D nonlocal models related to incompressible fluids, the first example  of nontrivial $C^{\infty}$ profiles with \emph{infinite} blowup rates, as well as the first example of self-similar blowup profiles with \emph{superlinear} growth of the self-similar velocity.

\subsubsection{Properties of the 1D profiles for the 3D lifting construction}

To lift the construction of the 1D singularities to 3D in \cite{chen2026eulerII}, we require 
the following two crucial properties of the 1D profiles.

\vs{0.05in}
\paragraph{\bf Stability of 1D profiles}

A crucial property of the 1D profile \(\waa\) is the contraction estimate in the
fixed-point formulation, with \(\lamcL<1\) uniformly in \(\al\) as
\(\al\to\f13\); see Theorems~\ref{thm:main_contract} and
\ref{thm:contra_lin_exact}. Moreover, in Theorem~\ref{thm:1D_profile_prop}, we
establish several estimates for \(\waa\), with implicit constants independent
of \(\al\in[\f13-\beps,\f13)\). Among them, the far-field cancellation estimates in item~(iv) of Theorem~\ref{thm:1D_profile_prop} are crucial for overcoming the logarithmic growth in several far-field estimates in~\cite{chen2026eulerII}.

We refer to these estimates, especially the contraction estimate, as the
stability properties of \(\waa\), since their implicit constants are
independent of \(\al\). As a result, we can use \(\e=\f13-\al\) as a small
parameter and perturb the estimates for \(\waa\) to estimate the 3D fixed-point
map.

\vs{0.05in}
\paragraph{\bf Arbitrary large spatial scaling exponent $\cls$}

The large spatial scaling $\cls \asymp (\tfrac{1}{3}-\al)^{-1}$
in \eqref{eq:1D_dyn0}, combined with incompressibility, induces strongly anisotropic self-similar flow $ (c_l r + U^r , c_l z + U^z)$  in the 3D profile equation \cite{chen2026eulerII},
\footnote{
Here, $(U^r, U^z)$ is the velocity 
associated with $ r^{\al} \waa $: 
  $-\D_{\R^5} \psi_{\mw{3D}} = r^{-1} \cdot r^{\al}\waa,  \ ( U^r, U^z) = \tf{1}{r} ( -\pa_z, \pa_r )(r^2 \psi_{\mw{3D}})$.
} 
with the $r$-directional component  much stronger than the $z$-directional one: $ \f{1}{r} (c_l r + U^r) \gtr \e^{-1} \f{1}{z} | c_l z + U^z | $, up to $\log$-correction. While the 1D model only captures 3D Euler \eqref{eq:Euler} in the $z$-direction along the axis $r=0$, 
the strong anisotropic structure allows us to extend the construction to the \emph{entire} range of $r>0$.  Moreover, the large scaling parameter $\cls$ provides a crucial large parameter
in the lifting construction in \cite{chen2026eulerII}.
See more discussion in \cite[Section 1.2.1]{chen2026eulerII}.

 In this work, we establish these two properties for the 1D profiles, thereby answering the second 
question in Question \ref{ques:1D_model}.

\vs{-0.05in}
\subsubsection{Related constructions}

Recently, Huang-Qin-Wang-Wei \cite{huang2024self,huang2025exact} developed an elegant Schauder fixed-point argument  to construct self-similar blowup profiles for gCLM model with \emph{any} \(a\leq 1\) and HL model. 
In comparison, our goal is not only to construct 1D self-similar blowup profiles, but also to establish the above two crucial \emph{quantitative} properties of the profiles for the 3D lifting construction in \cite{chen2026eulerII}. These properties are not directly provided by the Schauder fixed-point approach.

We note that \eqref{eq:1D} develops finite-time blowup from a large class of initial data $\om_0$. Consider $\om_0(x)$ that is odd in $x$ and  satisfies  $\om_0(x) \leq 0$ for $ x > 0 $; these properties are preserved by \eqref{eq:1D}. 
By symmetrizing the kernels and exploiting sign properties, one can \emph{analytically} estimate 
\[
\f{d}{dt} \pa_x \psia(t, 0) \leq - c_{\al}  (\pa_x \psia(t, 0))^2,
\quad  \pa_x \psia(t, 0) = 2 \al \int_0^{\infty} \om(y) y^{\al - 1} d y  < 0,
\]
with $ 0  < c_{\al} \les \f{1}{3} -\al$ for any $\al < \f{1}{3}$. Thus, the model blows up in finite time for any  $\al < \f{1}{3}$. 
 Since this blowup mechanism is not used in the 3D construction \cite{chen2026eulerII}, we do not pursue it here.

We also mention a question of independent interest: the dynamic
nonlinear stability of the 1D profile \(\bar W_\alpha\). This is different from the
\emph{finite codimension stability} of the 3D profiles established in~\cite{chen2026eulerII}. The contraction estimates for \(\cL_\alpha\) in Theorem~\ref{thm:contra_lin_exact} may be useful in this direction.

\subsection{Ideas of the proof}\label{sec:idea}

In this section, we outline the proofs of Theorems \ref{thm:reg_alb}, \ref{thm:1D_profile_prop}, and \ref{thm:contra_lin_exact}, which give the full versions of Theorems \ref{thm:main_1D} and \ref{thm:main_contract}. Below, we denote $\alb = \f13$.

A fundamental idea for obtaining \emph{sharp} self-similar blowup is to perturb 
a \emph{non-blowup profile}, an approach developed in our earlier works on gCLM, DG models 
\cite{chen2020slightly, chen2021regularity} and \cite{chen2023nearly}.

\subsubsection{\bf Construct $\f13$-profile in 1D}\label{sec:idea_step1} 

We recall the fundamental difficulties in constructing the $\f13$-profile from Remarks 
\ref{rem:infty}, \ref{rem:superlinear}, especially $V$ with a superlinear growth.
To overcome this difficulty, first, we derive \emph{a priori} the leading order asymptotics 
of the $\f13$-profile and capture it by constructing an approximate steady state $\bw$  to the dynamic equation of \eqref{eq:1D_dyn}  with $\bw =  |x|^{- 1/3 }  ( -6 + \ccb \log |x|^{-1/3 }  + \ccc \log |x|^{-2/3} )$  for $x \gg 1$,
where $\ccb, \ccc$ are given in \eqref{def:WF_para}.  To obtain the exact profile, we linearize \eqref{eq:1D_dyn0} around $\bw$ and 
reformulate the equation of perturbation $\td \ww = \wwd - \bw$, as a fixed point problem along the characteristic with a map $\cF$.
The central step is to prove Theorem \ref{thm:onto} that $\cF$ maps a ball into itself. 
The proof relies on the following key ingredients.

\vs{0.05in}
\paragraph{\bf Cancellations}
To overcome the log-growth in the estimate, we exploit several crucial cancellations to gain additional decay of order $|\log x|^{-b}$ with some $b>0$ for large $x$.

\begin{itemize}[leftmargin=1em]

\item  The key cancellation is the \emph{local-nonlocal} cancellation, which yields an additional decay
 factor $|\log x|^{-1}$ in estimating nonlocal terms. 
See Remark \ref{rem:int_wg} and Section \ref{sec:log_cancel} for further discussion.

\item {\bf Main terms in $\psio, \psio_x$.}
For $w(x) \sim c |x|^{-1/3}$ with large $x > 0$, a direct estimate yields 
$ \psio \sim x \log x, |\psio_x| \les \log x$. 
An important observation is that the factor $\log x$ only arises from the 
operator $\cJ = \cJ_{\alb}$ in \eqref{eq:Jw}. In particular, in Corollary \ref{cor:vel_est}, we prove 
\[
 | \tf{1}{x}\psio + 2 \alb \cJ(w)| 
 + |\psio_x + 2 \alb \cJ(w)|
 \les   \| w \|_{\bcX}, \quad   | \tf{1}{x} \psio  - \psio_x |
 \les |\lgp x|^{-1/3}  \| w \|_{\bcX} .
\]
For large $x$, we extract the main term $\cJ$ from the nonlocal terms 
$\psio, \psio_x$ for sharp estimates.

\item {\bf Asymptotics of $\bw$.}  By exploiting the explicit far-field asymptotics of the 
$\bw$ given in \eqref{eq:asym}, we obtain cancellations among the coefficients in \eqref{eq:err}.
For example, we gain a crucial decaying factor $|\log x|^{-4/3}$ from 
$ |\f{x \pa_x \bw}{\bw} + \f{1}{3}| \les |\log  x|^{-4/3}$. See 
\eqref{eq:decay:W} in Lemma \ref{lem:W_asym}.

\end{itemize}

\paragraph{\bf Nonlocal estimates}

Another main difficulty is that the term $\td \cR(w)$ in $\cL(w)$ \eqref{eq:F_lin} depends on $w$ nonlocally. 
To obtain sharp estimates of the nonlocal terms $ \psio(w)$, we use the scaling symmetry 
of the kernel and then partition the domain of the \emph{rescaled} integral. By bounding a number of $(<200)$ \emph{explicit} integrals, we obtain sharp estimates for $\psio, \pa_x \psio$. See more discussion in 
Section \ref{sec:u_est}.

\vspace{0.05in}
\paragraph{\bf Singularly weighted estimates}
To prove Theorem \ref{thm:onto}, we perform singularly weighted estimates. 
The singular weight $\vp$ with $\bb_1 < -1$ in \eqref{norm:X} plays a role similar to 
extracting a damping term for stability analysis of (nearly) self-similar blowup in 
\cite{ChenHou2023a,chen2019finite,chen2021HL}. By designing 
the weight $\vp$ in \eqref{norm:X}, we obtain a smaller constant $\lam$ for the contraction estimate \eqref{eq:lin_contra}. We choose $\max \bb_i =\f{1}{3}$ in \eqref{norm:X} and \eqref{def:mu_b} so that the weight $\vp$ in \eqref{norm:X} captures the decay :
$|w| \leq \vp^{-1} \nlinf{ w \vp }  \les \ang x^{-1/3} \nlinf{ w \vp } $.

We also establish a \emph{near-field contraction estimate}
 in Theorem \ref{thm:near_field_stab} for a linear operator $\cL_{\wwwa}$ in the fixed point map around the exact profile $\wwwa$. 

\vspace{0.05in}
\paragraph{\bf Computer-assistance}
We use computer assistance to construct the approximate profile $\bw$ and to verify some inequalities for the fixed-point argument. Due to the logarithmic growth of $V/x$ for the self-similar velocity (see Remark \ref{rem:superlinear}), we need to derive the first few leading-order terms of $\tf{1}{x} V(\bw)$ for the approximate profile with very large $x$. Moreover, we describe the computation for the stream function $\psio(\bw)$ of the approximate profile, 
 especially the \emph{explicit function $\psio(\bwf)$} \eqref{eq:W_rep}, with rigorous error control. The rigorous computation for $\psio(\bwf), \psio(\bwp)$ requires detailed discussion and expansions, such as Taylor expansion, and accounts for most of Appendix \ref{app:numerics}, 
apart from the standard setup and error estimate.

The full computer-assisted proof requires only light computation. Using a precomputed matrix—which otherwise takes about 30 minutes to generate—the complete verification finishes in less than 3 minutes. We refer to Appendix \ref{app:code} for implementation details and a link to the codes.

\vs{0.05in}

The remaining steps are purely analytic (pen-and-paper).

\vs{0.05in}
\paragraph{\bf Qualitative estimates and properties of the profiles}
After proving that the fixed point map $\cF$ maps a ball 
$\bDD$ in some weighted $L^{\infty}$ space  into itself, 
we then perform \emph{qualitative estimates} to show that $\cF$ is compact and continuous. 
A fixed point is obtained by a standard Schauder fixed-point argument. See Remark \ref{rem:Banach_Schauder}. We prove that the profile $\wwwa \in C^{\infty}$ in Theorem \ref{thm:reg_alb}
using the properties that $\psioa$ is more regular than $\wwwa$ and estimating $\wwwa$ using the profile equation \eqref{eq:1D_dyn}.

\vs{-0.05in}
\subsubsection{\bf Construct $\al$ -profile}\label{sec:idea_step2}

We perform careful \emph{perturbative estimates} to 
construct $\al$-profiles with $\al = \f13 - \e$ by exploiting the smallness of $\e \ll 1$.
A key difficulty is that the $\f13$-profile $\wwwa$ cannot be used directly as an approximate profile for \eqref{eq:1D_dyn} with $ \al < \f13$, since  it has a \emph{$O(1)$ relative residual error} for large $x$ and it does not capture the correct decay 
rate of $(\f13-\e)$-profile, which \emph{determines} the \emph{finite} blowup rate $(c_l, c_{\om})$ of the profile. A key step is to construct an approximate profile of the form
$\wa = \ang x^{\b} \wwwa$ for \eqref{eq:1D_dyn}, which captures the decay rate of $(\f13-\e)$-profile.
By carefully expanding the residual error for $\xx \gg 1$, we derive the main term, which determines $\b = - \f{\e}{8} + l.o.t.$, and obtain a small error with an improved decay rate. 
See Theorem \ref{thm:1D_error}.

We then perform a fixed point argument around $\wa$ 
and perturb the  \emph{near-field contraction property} of $\cL_{\wwwa}$  
 in Theorem \ref{thm:near_field_stab} for a linear operator $\cL_{\wwwa}$  to obtain similar contraction for a similar operator $\cL_{ \wa}$.
Combining this property and far-field contraction estimates for  $\cL_{ \wa}$, which requires $\kp < 1$ in \eqref{def:kp}, we construct the exact $\al$ -profile $\waa$. We further perturb the contraction property for $\cL_{ \wa}$ to obtain that for $\cL_{\waa}$ in Theorem \ref{thm:contra_lin_exact}. 
For \(\al<\alb\), the \(\e\)-improved decay of the integrand in \eqref{eq:1D_dyn0}, namely \(O(y^{-1-C\e})\), allows us to define the stream function \(\psio(\waa)\) through \eqref{eq:1D_dyn0}, with   \(\pa_x\psia(\waa), \tf{1}{x} \psia(\waa)  \) uniformly bounded. This contrasts with the logarithmic growth of \(\pa_x\psio_{\alb}(\waa), \tf{1}{x} \psio_{\alb}(\waa)  \). 
Then using normalization \eqref{eq:1D_normal:c}, we determine finite blowup rate $\bar c_{\om, \al, \mw{ 1D} }, \bar c_{l, \al, \mw{ 1D} }$ in \eqref{eq:wa_decay_1D} for the 1D model. These results are summarized in Theorem \ref{thm:1D_profile_prop}.

\vs{0.1in}

\paragraph{\bf Computer-assisted proofs in PDEs}
In recent years, there have been substantial developments in computer-assisted proofs for nonlinear PDEs using rigorous interval arithmetic. 
We highlight a few advances in incompressible fluids and related models, including the construction of nontrivial global smooth solutions to SQG \cite{castro2020global}, singularity formation from smooth data in 3D incompressible Euler \cite{ChenHou2023a,ChenHou2023b} and related models \cite{chen2019finite,chen2021HL}, and nonuniqueness of Leray--Hopf weak solutions 
for 3D incompressible Navier-Stokes without forcing \cite{hou2025nonuniqueness}. 
Computer-assisted proofs have also been used to construct imploding profiles for compressible Euler  \cite{buckmaster2022smooth} and the supercritical nonlinear wave equation \cite{buckmaster2024blowup}, as well as self-similar profiles for nonlinear Schrödinger and complex Ginzburg--Landau equations \cite{donninger2026self,dahne2024self}. Note that by imposing radial symmetry, the profile equations in  \cite{buckmaster2024blowup,buckmaster2022smooth,donninger2026self,dahne2024self} reduce to ODEs. 
See also the survey \cite{gomez2019computer}.

\subsection{Organization and notations}

The rest of the paper is organized as follows. In Section \ref{sec:ASS}, we construct an approximate profile for the $\f13$-profile and formulate the fixed-point map for the exact $\f13$-profile.
In Section \ref{sec:onto}, we develop nonlocal estimates and estimate the fixed-point map $\cF$.
In Section \ref{sec:schauder}, we construct the exact $\f13$-profile 
and prove a near-field contraction estimate.

The arguments in Sections \ref{sec:C_inf}-\ref{sec:1D_profile} are purely analytic. In Section \ref{sec:C_inf}, we  prove $C^{\infty}$-regularity estimates and several 
properties of the $\f13$-profile, and prove Theorem \ref{thm:reg_alb}. 
In Section \ref{sec:1D_profile_appr}, we construct the approximate profile for the $\al$-profile with $\al<\f13$. In Section \ref{sec:1D_profile}, we construct the exact $\al$-profile and prove the main Theorems \ref{thm:1D_profile_prop}, \ref{thm:contra_lin_exact} for its properties.
In Appendix \ref{app:basic}, we prove some basic estimates.
In Appendix \ref{app:numerics}, we present the numerical and analytic methods to estimate the approximate $\f13$-profile and its velocity rigorously. 
In Appendix \ref{app:proof}, we bound the integrals for the fixed point argument and 
discuss the implementation of the computer-assisted proof in Appendix \ref{app:code}.

We emphasize that Sections \ref{sec:onto}--\ref{sec:schauder}, 
Section \ref{sec:C_inf}, and Sections~\ref{sec:1D_profile_appr}--\ref{sec:1D_profile} can be read independently of one another. 
The computer-assisted estimates are confined to 
Section \ref{sec:onto}, Section \ref{sec:exist_profile}--\ref{sec:near_stable},
and Appendices \ref{app:numerics}--\ref{app:proof}, while Sections \ref{sec:contin}, \ref{sec:compact}, Sections \ref{sec:C_inf}--\ref{sec:1D_profile}, and Appendix \ref{app:basic} are purely analytic.

\vs{0.05in}

\paragraph{\bf Notations }
For each variable or operator, we list right of it the equation or result in which it is first defined.
We reserve  $ \alb, \al, \b, \e  $ for the parameters related to the regularity index $\al$ or $\e$:
\[
\alb \teq  \tf{1}{3}, \quad \e \teq \tf13 - \al, \quad   
\b: \mw{Thm} \ \ref{thm:1D_error}, 
    \quad \he \teq \tf13 - \al -\b = \e - \b : \eqref{def:para}.
\]

We introduce the following notations $\lgp(x)$, 
\beq\label{def:lgp}
  \lgp(x) = \log( x + 2 ),\quad \ang x = (1 + x^2)^{1/2},
\eeq
We have  $\lgp(x) \gtr 1$ for any $x \geq 0$.
We use $A \teq B$ to say that quantity A is defined to be equal to quantity/expression $B$.
We use $A \asymp B$ to denote $ C_1 A \leq B \leq C_2 A$ for some absolute constants $C_i > 0$.  We use calligraphic letters to denote operators or their values on the profile:
\[
\cK_{\bullet} : \eqref{eq:ker} , \ 
\cJa : \eqref{eq:Jw}, \  \cF : \eqref{eq:fix_map},
\  \cL, \, \cN, \, \cE : \eqref{eq:F_lin} , 
\  \cR,  \td \cR : \eqref{eq:lin_nloc},
\quad 
\cFR, \,  \cNab, \, \cEab, \, \cLab :  \eqref{eq:F_lin_ep} .
\]

We use special fonts of $C$ to denote specific constants and functions 
\[
 \mCC_{\bullet} : \eqref{eq:vel_const}, \quad \MCC_{\bullet} : \eqref{eq:vel_const}, \mw{Lem} \ref{lem:vel_est}, \quad \cca, \ccb, \ccc : \eqref{def:WF_para}.
\]

\section{Fixed point problem for the $\f13$-profile }\label{sec:ASS}

In this section, we first introduce the nonlocal operators. We then derive the far-field asymptotics of the profile, discuss the construction of an approximate $\frac13$-profile, and formulate the fixed-point map for the exact $\frac13$-profile.

\subsection{Kernels for nonlocal terms}

For $\al \in (0, \f13]$, including $\al = \f13$, we introduce the following symmetrized kernels related to 
$\psio(w), \pa_x \psio(w)$ with odd $w$ in \eqref{eq:1D_dyn0}, \eqref{eq:1D_normal}: 
\bseq\label{eq:ker}
\beq\label{eq:ker:b}
\bal
  K_{\al, 1}(x, y) & \teq |x + y|^{\al} - |x-y|^{\al} - 2 \al x y^{\al - 1}  , \\
  K_{\al, 2}(x, y) & \teq \al |x+y|^{\al-1} - \al \cdot  \sgn(x-y) |x-y|^{\al-1} - 2 \al y^{\al-1}  .
\eal
\eeq

We also introduce the truncated kernels $K_{\al, i, J}, K_{\al, \D}$ that are regular near $y =0$ 
\beq\label{eq:ker:c}
\bal
  K_{\al, 1, J}(x, y) & \teq |x + y|^{\al} - |x-y|^{\al} - 2 \al x y^{\al - 1} \one_{y \geq x} ,  \\
    K_{\al, 2, J}(x, y) &  \teq \al |x+y|^{\al-1} - \al \cdot  \sgn(x-y) |x-y|^{\al-1} - 2 \al y^{\al-1}  \one_{y \geq x},  \\ 
  K_{\al, \D}(x, y) & \teq  K_{\al, 2}(x, y) - x^{-1} K_{\al, 1}(x, y) .
\eal
\eeq

We define the operators associated with the above five kernels :
\beq
  \cK_{\bullet}(w)(x) = \int_{\R_+} K_{\bullet}(x, y) w(y) d y , \quad 
  \bullet \in \{  (\al, 1), (\al, 2) , (\al, 1, J) , (\al, 2, J) , (\al, \D) \} .
\eeq

For $w$ being odd in $x$, from the definition of $\psia$ in \eqref{eq:1D_dyn0} and $\psioa, V$ in \eqref{eq:1D_normal}, we can rewrite $\psio$ as 
\beq\label{eq:ker:a}
\bal
\psioa(w) &= \cK_{\al, 1}(w), \quad   \psio_{\al, x}(w)  = \cK_{\al, 2}(w), \quad V(w)   = x + \cK_{\al,1}(w) , \quad   V_x(w)  = 1 + \cK_{\al, 2}(w) , \\
\eal
\eeq
and view $\psioa(w),\pa_x \psioa(w)$ as linear operators for $w$. Throughout this paper, we use \(\psioa\) and \(\cK_{\alpha,1}\), as well as \(\partial_x \psioa\) and \(\cK_{\alpha,2}\), interchangeably. The \(\psioa\)-function is more informative, while estimates are more conveniently carried out using the \(\cK\)-operator. To simplify notation, we suppress the dependence of $\psi, \psio$ on $\al$ when no confusion arises.

\eseq

\bseq
We introduce the $\cJ_{\al}$-operator
 \beq\label{eq:Jw}
 \cJ_{\al}( w )(x) = \int_0^x y^{\al - 1} w d y.
 \eeq

For $w$ with a sufficiently fast decay, from the definition of $\psi$ in \eqref{eq:1D_dyn0}, we obtain
\beq\label{eq:iden_psiz_J}
 \pa_x \psia(0) = 2 \al \cJ_{\al}(w)(\infty).
\eeq

By definition, we have 
\beq\label{eq:cK_cKJ}
  \cK_{\al, 1, J}(w) = \cK_{\al, 1}(w) + 2 \al x \cJ_{\al}(w),
  \quad    \cK_{\al, 2, J}(w) = \cK_{\al, 2}(w) + 2 \al  \cJ_{\al}(w) .
\eeq
\eseq 

We use the subscript $J$ in kernels $\cK_{\al,i,J}$ \eqref{eq:ker:c} to indicate that the operator 
$  \cK_{\al, i, J}$ and $ \cK_{\al, i}$ differs by an operator related to $\cJa$.

\vs{0.05in}

\paragraph{\bf Scaling, decay, and integrability}
For $x, y \geq 0$, using the scaling symmetries, writing $z = \f{y}{x}$, and a simple calculations, we obtain the scaling 
\beq\label{eq:ker_scale}
\bal
  K_{ \bullet}(x, y) &= |x|^{\al} K_{\bullet}(1, z), 
  \quad  && \bullet \in \{ (\al, 1), (\al, 1, J)  \},  \\  
  K_{ \bullet}(x, y) &= |x|^{\al -1} K_{\bullet }(1, z), \quad && \bullet \in \{ (\al, 2), (\al, 2, J), (\al, \D)  \},  \\ 
 \eal
\eeq
and decay
\beq\label{eq:ker_decay}
 |K_{\al, 1}(1, z)|  \les \one_{z \leq 2 } z^{\al-1} + \one_{z > 2} |z|^{\al-3} , \quad | K_{\al,2}(1, z) |  \les \one_{z \leq 2} ( z^{\al-1} + |z-1|^{\al-1} )
 + \one_{z > 2} |z|^{\al-3}.
\eeq

The kernel $K_{\bullet}(x,y), \bullet \in \{ (\al, 1,J), (\al, 2, J), (\al, \D) \}$ 
satisfies similar estimates but are more regular near $y=0$. 
We refer these basic estimates 
of $K$ to Lemma \ref{lem:K_decay_basic}.  By subtracting $2 \al x y^{\al - 1}$ or $2\al y^{\al-1}$ in the kernels in \eqref{eq:ker:b}  and using the fast decay $|K_{\al, i}(x, y)| \les_x y^{\al-3}$ in $y$, 
we obtain that the operators $\cK_{\al, i}(W)(x)$ are well-defined for $W$ with a slow decay, e.g. $|W(y)| \les |y|^{-\al}$ for large $y$. 
Note that  $2 \al x y^{\al - 1}$ or $2\al y^{\al-1}$ arise from the integrand 
 in $\pa_x \psia(W)(0), x \pa_x \psia(W)(0)$.

\subsection{Asymptotics of the $\f13$-profile}\label{sec:ansatz}

In the rest of this section and in Sections \ref{sec:onto},  \ref{sec:schauder}, we construct the $\f13$-profile. We focus on  $\al =\alb = \f{1}{3}$ and simplify $\psio_{\alb}$ as $\psio$.  

A fundamental step is to derive the precise  asymptotic of the profile. We consider $W$ odd in $x$, 
\[
W(x) \leq 0 , \ \forall x \geq 0,  \quad W(x) \sim  - C_W |x|^{-\g}, C_W > 0  
\  \mbox{for large \ } x > 0.
\]

Inspired by numerical simulation, for $\al =\alb $, we aim to construct a profile 
with the asymptotics
\beq\label{eq:ansatz_V}
    \lim_{x\to \infty} V_x(x) =\infty, 
   \quad \lim_{x\to \infty} V / x  =\infty.
\eeq

Since $W, V$ are odd in $x$, we focus on $x > 0$. 
We present a heuristic derivation below and refer to the \emph{(a posteriori)} far-field estimates of the profile and error in Appendix \ref{app:nloc_very_far},\ref{app:cR_far} for rigorous 
derivations. We use $A \sim B$ to mean that $B$ is the main term of $A$ with $|A-B| = o(B)$ as $|x| \to \infty$.

\vs{0.05in}
\paragraph{\bf Ansatz of the decay rate} 

Under the ansatz \eqref{eq:ansatz_V}, the first crucial observation is that $\gamma = \f13$. 
In fact, if $\g > 1/3$, one can obtain  $\tf{1}{x} V, V_x \in L^{\infty}$ for $|W(x)| \les 
\ang x^{-\g}$ using the kernels \eqref{eq:ker}. This violates \eqref{eq:ansatz_V}.
If $\g < 1/3$, using the scaling analysis for the kernel $\cK_{\alb, 1}$ for $\psio$ in \eqref{eq:ker:b}, we obtain $\psio \sim C_\psi x^{1 + \alb- \g}$ for some $C_\psi \neq 0$. Since $\alb - \g >0$, we obtain 
$V = x + \psio \sim C_\psi x^{1 + \alb - \g}$ 
and $ \pa_x V \sim C_\psi (1 + \alb - \g) x^{\alb -\g}$ for $x$ large enough. Thus, using the equation \eqref{eq:1D_dyn}, we may match the asymptotics of both sides of \eqref{eq:1D_dyn}
\[
  2 C_\psi (-C_W) (-\g)  x^{\alb -\g -\g} = - (1-\alb)  \cdot (- C_W) C_\psi (1 + \alb - \g) x^{\alb -\g -\g} + l.o.t. .
\]
Since $\alb = \f13, C_W , C_\psi \neq 0$, we obtain $-2 \g = - \f23 ( \f43 -\g)$, which implies $\gamma = \f13$. 

For $|W| \les \ang x^{-1/3}$, by estimating the kernel $\cK_{\alb,i} = 
\cK_{\alb, i, J} - 2 \alb x \cJ_{\alb}$ (see Corollary \ref{cor:vel_est} Lemma \ref{lem:vel_al} with $\al =b =  \alb$ and $\psi \equiv 1$), we obtain
\beq\label{eq:asym_V1}
\bal
V & 
= x + \cK_{\alb, i, J} - 2 \alb x \cJ_{\alb}
=  O(x) - 2 \alb x \cJ_{\alb},  \quad 
 V_x  = 1 + \cK_{\alb, 2, J} - 2 \alb  \cJ_{\alb}
= O(1) - 2 \alb \cJ_{\alb}.
\eal
\eeq

From the definition of $\cJ_{\alb}$ in \eqref{eq:Jw}, for $W(x) \sim - C_W x^{-\g}$ with $ C_W >0$, we obtain
$\cJ_{\alb} \asymp - c^{\pr}\log x $ for $x \gg 1$. Therefore, we obtain 
\beq\label{eq:ansatz_V2}
 V/x, \quad  V_x \asymp \log x, 
 \quad \cJ_{\alb}  \asymp -\log x.
\eeq
for $x \gg 1$ and they becomes unbounded as $x\to 0$, which is consistent with the ansatz \eqref{eq:ansatz_V}.

Based on \eqref{eq:ansatz_V2},  we refine the asymptotics of $W$.  Multiplying $x^{1/3}$ on both sides of  \eqref{eq:1D_dyn}, we rewrite  \eqref{eq:1D_dyn} equivalently as 
\beq\label{eq:1D_dyn2}
  2 V \pa_x ( W x^{1/3}) = ( \tf83 - \tf 23 V_x + \tf23 \cdot \tf{V}{x}) W x^{1/3}
  = - \tf{2}{3} (  V_x - \tf{V}{x} - 4 ) \cdot W x^{1/3} .
\eeq

Since $V \asymp x \log x$ for large $x$, we consider the following ansatz for the asymptotics of $W$ 
\bseq\label{eq:asym}
\beq 
  W = x^{-1/3} (\cca + \ccb (\log x)^{\b_1} + O |\log x|^{2 \b_1}  ), 
  \quad \mbox{for \ } x \gg 1,
\eeq 
 with $\b_1 < 0$, $\cca \neq 0$. We perform the expansion for large $x$ and will derive that 
\beq\label{eq:asym_b}
\cca = -6, \quad  \b_1 = -\tf{1}{3}.
\eeq 
\eseq 

\paragraph{\bf Expand the kernel}

Recall the kernel $K_{\alb, \D}$ from \eqref{eq:ker:c}. For $x\geq 0$,
using a change of variable $y = x z$, 
and 
$K_{\D}(x ,  y) = x^{\al-1} K_{\D}(1, z)$, we obtain 
\[
\bal
    V_x - \f{V}{x} & = 
   \int_0^{\infty} K_{\alb, \D}(x, y) W(y)  d y
    = \int_0^{\infty} K_{\alb, \D}( 1,  z)  z^{-\alb}  \cdot ( W(x z) (xz)^{\alb} )  d z .
  \eal
\]

Since for fixed $z$, the asymptotics of $W$ in \eqref{eq:asym} implies
$W(x z) (xz)^{\alb} = \cca + \ccb |\log x|^{\b_1} + O( |\log x|^{2\b_1}) $, for $x \gg 1$, we obtain (see Appendix \ref{app:nloc_very_far} )
\beq\label{eq:ansatz_Vmix}
    V_x - \f{V}{x} 
  = (\cca + \ccb |\log x|^{\b_1} ) \cdot \int_0^{\infty} K_{\alb, \D}(1, z) z^{-\alb} d z + O(  |\log x|^{2 \b_1} ) .
\eeq

Using the definition of $K_{\alb, \D}$ from \eqref{eq:ker:b}, its decay estimates \eqref{eq:K_decay:mix} and a direct calculation yields a remarkable identity 
\begin{align}\label{eq:asym_int}
\int_0^{\infty} K_{\alb, \D}(1, z) z^{-\alb} 
& =  \int_0^{\infty} ( \alb |1 + z|^{\alb-1 } - \alb \cdot \sgn(1-z) |1-z|^{ \alb-1 } 
-  ( |1 +z|^{ \alb }  - |1-z|^{ \alb }  ) ) z^{-\alb} d z \notag \\
& = z^{1-\alb} ( |1-z|^{\alb} - |1+z|^{\alb } ) \B|_0^{\infty} 
= - 2 \alb  = - \tf23 .
\end{align}
As $z\to \infty$, we obtain $|z-1|^{\alb} - |1+z|^{\alb }
= - 2 \alb z^{\alb-1} + o(z^{\alb-1})$, which gives a nontrivial integral.

Under \eqref{eq:asym}, for large $x$, we obtain 
\[
V \asymp x \log x, \quad  | W x^{\alb}| \asymp 1, \quad 
 | \pa_x( W x^{\alb}) | \asymp x^{-1} (\log x)^{\b_1 - 1},
\quad  |  V   \pa_x( W x^{\alb}) | \asymp |\log x|^{\b_1}  .
\]
Hence, the left hand side in \eqref{eq:1D_dyn2} vanishes for large $x$,
which along with 
$ | W x^{\alb}| \asymp 1 $ implies $  V_x - \f{ V}{x} - 4  \to 0$ as $x \to \infty$. 
Using \eqref{eq:ansatz_Vmix} and \eqref{eq:asym_int}, we match the coefficients and obtain 
\[
 0 = \lim_{ x \to \infty} V_x - \tf{1}{x} V - 4 =  - \tf23 \cca  - 4 , \ \Longrightarrow  \  \cca = -6.
\]

To determine $\b_1$, we further match the next term. Using \eqref{eq:asym}, \eqref{eq:asym_V1}, the definition of $\cJ_{\alb}$ \eqref{eq:Jw}, for large $x$, we obtain the leading order terms 
\[
\bal
 \f{V}{x} &= O(1) - 2 \alb \int_0^x y^{\alb - 1} W(y) d y  = - 2 \alb \int_1^x \cca y^{-1}   d y  + O(|\log(x)|^{1+\b_1} )   = - 2 \alb  \cca \log(x)  + l.o.t.
 \eal
\]

Using the above estimate, for large $x$, we estimate the 
left hand side (LHS) of \eqref{eq:1D_dyn2}  as 
\[
\bal
  \mf{LHS}_{\eqref{eq:1D_dyn2} } & = - 4 \alb x \cdot (  \cca \log(x) +   l.o.t. ) \cdot  \pa_x( \ccb (\log x)^{\b_1} + l.o.t.)
   = - 4 \alb \cca \ccb \b_1 (\log x)^{\b_1}+ l.o.t. .
  \eal
\]

For the right hand side of \eqref{eq:1D_dyn2}, 
using \eqref{eq:ansatz_Vmix} and $- \f23 \cca = 4$, we get 
\[
\bal
  \mf{RHS}_{\eqref{eq:1D_dyn2} }
& = (- \f{2}{3})^2  \ccb  (\log(x))^{\b_1}   \cdot \cca + 
O( |\log x |^{2 \b_1} ) = \f{4}{9}  \cca \ccb  (\log(x))^{\b_1}  +O( |\log x |^{2 \b_1} ).
\eal
\]
Since $\b_1 < 0$, matching the term $(\log(x))^{\b_1} $, we obtain $\b_1$ in \eqref{eq:asym}:
\[
\bal
  - 4 \bar \al \cca \ccb \b_1 &= \f{4}{9} \cca \ccb, \ \Longrightarrow   \b_1 = - \f{1}{3}. 
  \eal
\]

The next-order expansion \emph{does not} determine $\ccb$; however, it implies
\footnote{
  We do not use identity \eqref{eq:relation_C2} in the proof, but impose it in the ansatz for  approximate profile to improve accuracy. 
}
\beq\label{eq:relation_C2}
\ccc  = \tf{1}{12} \ccb^2.
\eeq

\begin{remark}[\bf Far-field asymptotics]\label{rem:far_asym}
A distinctive feature of the nonlocal equation \eqref{eq:1D_dyn} with $\al = \f13$
is that the precise far-field asymptotics of $W$: $\cca x^{-1/3}$, can be derived \emph{explicitly}.
\footnote{
In comparison, for related models such as De Gregorio \cite{chen2019finite} and HL model \cite{chen2021HL}, if certain quantities, such as $W_x(0)$ or $u_x(0)$,  are normalized  at $x=0$ as in \eqref{eq:1D_normal}, $\cca,\b_1$ in the far-field asymptotics $W\sim \cca x^{\b}$ are not  explicit determined.
}
By enforcing the asymptotics in constructing the approximate profile, we obtain a much smaller \emph{relative} residual error $|\bar\cR|$ for large $x$, with decay
$|\bar\cR|\lesssim |\lgp x|^{-2/3}.$
This decay is crucial for overcoming the unboundedness of $ \f{1}{x} V$ and $V_x$ in \eqref{eq:ansatz_V} in the fixed-point argument. 
Notably, the exact value of  $\ccb$ for  $x^{-1/3}(\log x)^{-1/3}$ \eqref{eq:asym}  is not needed to obtain $|\bar\cR|\lesssim |\lgp x|^{-2/3}$; see Lemma \ref{lem:W_asym}.

\end{remark}

\subsection{Construction of approximate profile}\label{sec:appr_profile}

 We follow the works \cite{ChenHou2023a,ChenHou2023b,chen2019finite,chen2021HL} to construct an approximate blowup profile by solving the dynamic equation of \eqref{eq:1D_dyn} with $\al = \alb$
\bseq\label{eq:1D_dyn_solve}
\beq
 \pa_t W + 2 V \pa_x W = (3- \alb - (1-\alb) V_x) W .
\eeq
 In addition to $\pa_x V(0)=1$ \eqref{eq:1D_normal:c}, we further impose a normalization condition 
 at the origin
\beq\label{eq:1D_dyn_solve:b}
 \pa_x W(t, 0) = -1 ,
\eeq
\eseq
to remove a trivial scaling symmetry. See Footnote \ref{foot:wx0} for motivation. 
We then solve the above equation numerically for long enough time. Due to the odd symmetry of $W$, we only need to construct the approximate profile for $x \geq 0$.
Based on the \emph{ansatz} of the asymptotics \eqref{eq:asym} for $W$, 
for $x\geq 0$, we use the following basis representation for the approximate profile
\bseq\label{eq:W_rep}
\begin{gather}
 \bw = \bwp +\bwf , \\
    \bwf(x) \teq    \chi_1( x ) |x|^{- \f13 }  \big( \cca + \ccb (\log |x|)^{-\f13 } 
   + \ccc (\log |x| )^{- \f23 } \big), \quad 
    \chi_1(x) =  \f{  (x- z_0)^5}{ z_0^5 + (x- z_0 )^5} \one_{x\geq z_0} \notag 
\end{gather}
The explicit function $\bwf(x)$ captures the far-field asymptotics of $\bw $ given by \eqref{eq:asym}.
We design the mesh $ \mmx_{-1} < \mmx_{0} < 0 = \mmx_1 < ... < \mmx_{ \nnx + 2} $ 
\eqref{def:x_i}
 to discretize a large domain $[0, \mmx_{ N_{\bwp}}] $ and design
\beq\label{eq:W_rep:WP}
\bwp  = \sum \nolimits_{1\leq i \leq \nnx } 
a_i B_i(x) \in C^{2, 1} ,  \quad  \supp(\bwp)  \subset [ -\xed, \xed ], \quad 
 \ \xed \approx 8.8 \cdot 10^{27}
\eeq
\eseq
to captures the near-field of the profile with some coefficient $a_i$. 
Each basis $B_i(x) \in C^{2,1}$ is a 4-th order Bspline basis and is a piecewise cubic polynomial with a compact support. We refer more details of the mesh $\mmx_i$ and basis $B_i(x)$ and its property to
 Appendix \ref{sec:basis}.  In particular, $\bwp$ is a piecewise cubic polynomial 
 and has global $C^{2,1}$ regularity. We use the subscript $\mf{F}, \mf{P}$ to denote \textit{far-field}, \textit{polynomial}. 

 We cannot determine the exact value $\ccb$ and use an approximate value for it. 
 \footnote{
We first do the numerical simulation without imposing the explicit decay by choosing 
$ \bwf \equiv 0$. We then use the numerical solution to fits the decay rates, which provides 
an approximation for $\ccb$.
 }
From \eqref{eq:asym_b} and \eqref{eq:relation_C2}, we choose $\cca, \ccb, \ccc$ as  
\beq\label{def:WF_para}
\cca = -6, \quad \ccb \approx 1.5745 , \quad \ccc = \tf{1}{12} \ccb^2 ,
\quad z_0 = \mmx_{300} .
\eeq
We have imposed the decay rate $|\log x|^{-1/3}$ corresponding to $\b_1 = - \tf13 $ \eqref{eq:asym_b} in the form \eqref{eq:W_rep}. 
To obtain the values for the coefficients $a_i$ in \eqref{eq:W_rep}, we derive the equation for $w = W  - \bwf$
\[
 \pa_t w = F(w) \teq - 2 V \cdot \pa_x ( w + \bw) +
 (3- \al - (1-\al) V_x ) ( w + \bw) ,
 \quad V = \psio(w) + V(\bw).
\]

\paragraph{\bf Numerical scheme}
We use the following steps to update the values for $a_i$.
\bit[leftmargin=0.5em] 
\item \textbf{Step 0: Initialization.} Compute  nonlocal terms $V( \bwf)$ and $\bwf, \pa_x \bwf$ on the mesh $\mmx_i, 1\leq i \leq  \nnx$. Choose initial data $w_0  = - \f{x}{ (1 + 0.12 x^2)^{2/3}} - \bwf$. 
\footnote{
We choose odd initial data  $w_0 + \bwf$ that is negative for $x > 0$ decay no slower than $x^{-1/3}$. 
Apart from these requirements, no further fine-tuning of the initial data is needed.
}
Use the Bspline interpolation for $w_0$ to obtain the initial coefficients $a_i(0)$ in \eqref{eq:W_rep}.
Choose the time step $k$ according to suitable CFL condition. 

\item \textbf{Step 1: Update $w(t_n) \to w(t_{n+1})$.} 
Suppose that the coefficient $a_i( t_n)$ at $n$-th step and at time $t_n$
for $w(t_n)$ has been obtained.  To compute the forcing term $F(w)$ for a given $w$ in the representation \eqref{eq:W_rep} with coefficients $a_i$, we use the formula \eqref{eq:W_rep:WP} and take the derivative $\pa_x w$ on the basis function \eqref{eq:W_rep:WP} to evaluate $w$ and $w_x$ on the mesh $\mmx_i$. We evaluate the nonlocal terms $V( w), V_x(w)$ using the integrals \eqref{eq:ker}. 
With $F(w)$ being computed, we use a second-order Runge-Kutta method to discretize the temporal variable and update the PDE. 

\item \textbf{Step 2.} Repeat Step 1 until the \emph{relative} residual error $\cR = \f{F(w)}{ w + \bwf}$ becomes small enough and stop the computation. We determine $\bwp=w$ and set $\bw=\bwp+\bwf$ as the approximate profile.
\eit 
In Figure \ref{fig:solu_profile}, we plot the approximate profiles for $\bar W, \psio(\bar W) = \bv - x$. 
\begin{figure}[t]
\centering
\includegraphics[width=0.55\textwidth]{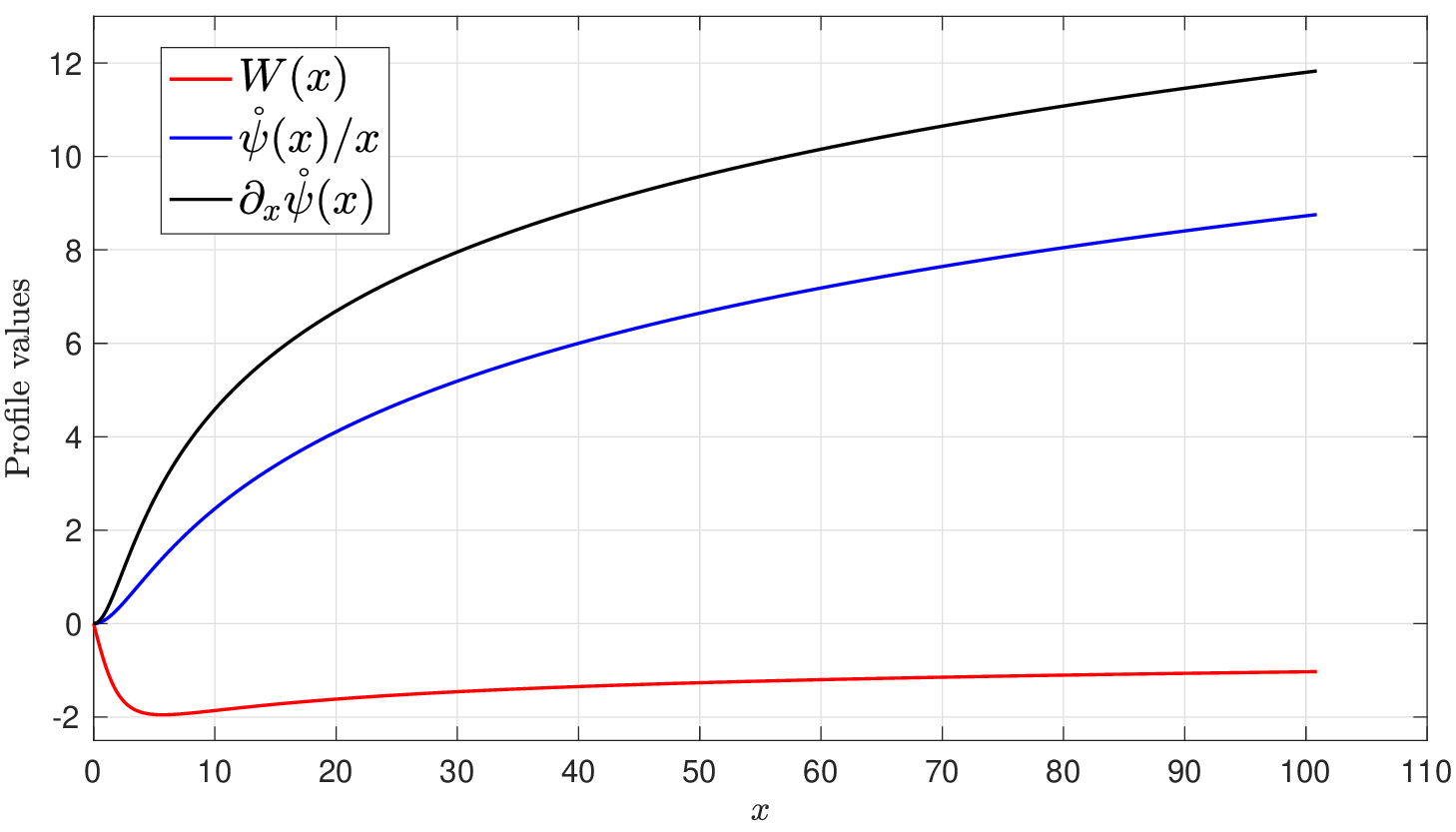}
\caption{
Approximate Profiles for \(\bar W\), \(\mathring{\psi}/x\), \(\partial_x \mathring{\psi}\) on grid points in $[0, 100]$. $\psio$ is the numerical evaluation of $\psio(\bw)$ with \emph{superlinear} growth.
}
\label{fig:solu_profile}
\end{figure}

\vs{0.1in}
\paragraph{\bf Truncation and Round-off  error}
The most difficult step in constructing the approximate profile 
is to compute the nonlocal terms 
$V(w), V_x(w)$ for $w= \bwf, \bwp$ with small error. 
For later fixed point argument, we require a small 
\emph{relative} residual error, which \emph{only} has a decay rates $O(|\log x |^{-2/3})$ 
for large $x$. To overcome this difficulty, we choose a large domain with $\xed \approx 8.8 \cdot 10^{27}$ and design the explicit functions $\bwf$ to capture the first three leading order far-field 
asymptotics in \eqref{eq:W_rep}. Choosing a large domain typically leads to significant round-off error, particularly in deriving the nonlocal terms $V, V_x$ \eqref{eq:ker}. 
Inspired by \cite{chen2021HL}, we compute the nonlocal terms using a combination of Gaussian quadrature rule and \emph{exact} integration to 
reduce the round-off error. We refer the methods and error estimates to Appendix \ref{app:nonlocal}-\ref{app:bwf_far}. For  rigorous proof, we control all truncation and round-off errors using  numerical analysis and \hyr[bd:intval]{\its Interval Arithmetic}. See Appendix \ref{app:piece_bound}.
These methods have gradually matured following the works \cite{chen2021HL,chen2019finite,ChenHou2023a,ChenHou2023b}. We emphasize that, in constructing the approximate profile 
$\bwp$ for \eqref{eq:W_rep}, we \emph{do not} need to track any error or \emph{use} interval arithmetic.

\begin{remark}[\bf $\pa_x \bw(0) \neq 1$ for approximate profile]\label{rem:wx0}

Due to numeric error, $\bw$ does not preserve 
\eqref{eq:1D_dyn_solve:b} $ \bw_x(0)=-1$;
instead, it satisfies \( |\bw_x(0) + 1| \leq 10^{-6} \) (see Lemma \ref{lem:W_asym}). The \emph{exact} 
identity  $ \bw_x(0)=-1$ \emph{is not used} 
in constructing the \emph{exact} profile $\wwwa$. 
On the other hand, we have $|\psio(\bw)(x) | \les |x|^2$ by Corollary \ref{cor:vel_est}.
Thus, $\pa_x V(\bw)(0)=1$ in \eqref{eq:1D_normal:b} holds for the numerical profile.

\end{remark}

\subsection{Fixed point map}\label{sec:fix_point_map}

In the rest of this section and Sections \ref{sec:onto}, \ref{sec:schauder}, we construct an exact profile around the approximated profile constructed numerically in 
Section \ref{sec:appr_profile}. A fundamental difficulty is that closing the nonlinear estimates typically requires control of the Lipschitz norm of the transport coefficient, which is \emph{unbounded} in this case, i.e. $V_x$ in \eqref{eq:1D_normal}, \eqref{eq:1D_dyn} is not uniformly bounded.
\footnote{
 This stands in sharp contrast to the canonical estimates for fluid-type equations, such as the continuation criteria based on $\nlinf{ \nabla \uu}$ for the Euler equations.  Thus, we cannot use methods similar to \cite{chen2019finite,chen2021HL} by proving the nonlinear stability of an approximate profile using energy method.
}
Instead, we rewrite the steady equation of \eqref{eq:1D_dyn} as a fixed point problem.

Let $\bw$ be the approximate profile. To present the derivation consistently, 
we retain $\alb$ and do not simplify some expressions  involving $\alb$, such as $ 3-\alb$. 
For a Lipschitz weight $\rho > 0$  to be chosen (see \eqref{eq:rho1}),  multiplying both sides of the profile equation \eqref{eq:1D_dyn} by $\f{\rho}{\bw}$ with $\al =\alb = \f13$, we obtain
\beq\label{eq:fix_map1}
 2 V \pa_x( \f{W}{ \bw } \rho ) 
 = \B( 3 - \alb - (1-\alb) V_x - 2 V \f{\pa_x \bw}{ \bw }  \B) 
 \f{W}{ \bw } \rho 
+ 2 V \cdot \pa_x \rho  \f{W}{ \bw} .
\eeq

We introduce the relative residual operator, which is linear in $W$:
\beq\label{eq:err}
 \cR(W) =  3 - \alb - (1-\alb) V_x(W) - 2 V(W) \f{\pa_x \bw}{ \bw} ,
\eeq
where $V = V(W) = x + \cK_{\alb, 1}(W)$  is a linear operator of $W$ with $\cK_{\al,i}$ defined in 
\eqref{eq:ker}. The name ``relative" reflects that $\cR(\bw)$ measures the 
residual error of the approximate profile $\bw$, relative to itself. In particular,
if $\bw$ is an \emph{exact} solution to \eqref{eq:1D_dyn}, we obtain $\cR(\bw)=0$. We decompose 
$W$ into the approximate profile $\bw$ and the perturbation $w$ satisfying
\[
 W = \bw + w, \quad  w \mbox{ \  is odd in \ } x,  \quad   \lim\nolim_{x \to 0} |x|^{-1} w(x) = 0.
\]

Cancelling $2 V \pa_x \rho$ on both sides, we rewrite \eqref{eq:fix_map1} as 
\bseq\label{eq:eqn_equiv}
\beq\label{eq:eqn_equiv:a}
\bal
  2 V  \pa_x( \f{w}{\bw} \rho)
  + 2 V \pa_x \rho
  & = \cR(W) \f{W}{\bw} \rho + 2 V \pa_x \rho (1 + \f{w}{\bw}) \, ,  \\
  \iff \qquad \qquad   2 V \pa_x( \f{w }{\bw} \rho) & = \cR(W) \f{W}{\bw} \rho + 2 V \pa_x \rho  \f{w}{\bw}  \, .
  \eal
\eeq
Since $ \lim_{x \to 0} \f{1}{ |x| } w(x) = 0$ and $\pa_x \bw(0) \neq 0$ 
by Lemma \ref{lem:bw_basic_sign}, integrating \eqref{eq:eqn_equiv:a}
from $0$ to $x$, we obtain 
\beq
   \f{w}{ \bw} \rho =  \int_0^x \f{ \cR(W) }{  2 V } \cdot \f{W}{\bw} \rho
   +  \f{\pa_x \rho}{\rho} \cdot   \f{w}{\bw} \rho  . 
\eeq
\eseq

We encode the relation on the right hand side as a map $\cF$ on $w$:
\bseq\label{eq:fix_map}
\beq\label{eq:fix_mapa}
  \cF( w) = \f{\bw}{\rho}  \int_0^x \B( \f{\cR( w + \bw)}{2 V} \cdot \f{w + \bw}{\bw} 
  \rho + \f{\pa_x \rho}{\rho} \cdot \f{ w }{\bw} \rho \B),
  \quad V =   x + \psio(w + \bw) 
\eeq
where $\psio$ is the modified stream function 
defined in \eqref{eq:ker}.

Then solving \eqref{eq:eqn_equiv} is equivalent to a fixed point problem
\beq\label{eq:fix_mapc}
 w =  \cF(w) ,  \quad w \mbox{ \  is odd in \ } x,  \quad   \lim\nolim_{x \to 0} |x|^{-1} w(x) = 0.
\eeq
\eseq

We note that the map $\cF$ depends on the choice of $\rho$ and choosing different $\rho$ may lead to 
different map. We suppress the dependence of $\cF$ on $\rho$ when no confusion arises.

\begin{remark}[\bf Integrating factor]\label{rem:int_wg}

The weight $\rho$ plays a fundamental role as an integrating factor. For a model problem 
$2 x \pa_x w = - c(x) w + R(w)$, 
where $c(x)$ is a given function and $R(w)$ denotes lower order terms, 
we first choose a weight $\rho$ to rewrite it equivalently 
\beq\label{eq:int_rho_motivate}
  2 x \pa_x( w \rho ) = R(w) \rho , \quad 2 x \pa_x \rho = c \rho.
\eeq

Following \eqref{eq:eqn_equiv} and \eqref{eq:fix_map}, we can rewrite it as a fixed point problem. 
The key advantage of the formulation \eqref{eq:int_rho_motivate} is that $\rho$ captures the sharp decay of $w$ and we do not need to estimate 
the \emph{linear} term $c(x) w$ in the fixed point argument.
While there are \emph{infinite many} different ways to reformulate solving \eqref{eq:1D_dyn} as a fixed point problem, without capturing the damping / growing effect of such a term, we cannot 
design the correct functional space to prove the fixed point argument.

For \eqref{eq:fix_map1}, we exploit a fundamental local-nonlocal cancellation.
\emph{On average}, the perturbation of 
the nonlocal term $- (1-\alb) V_x - 2 V \tf{\pa_x \bw}{ \bw}$
in $\cR(W)$ \eqref{eq:eqn_equiv}  behaves effectively like a local term: 
\[
- (1-\alb) \psio_x - 2 \psio \f{\pa_x \bw}{ \bw}  
= \f{4}{9} x^{\alb} w + l.o.t.
\]
for large $x$, where the lower-order terms (l.o.t.) gain an additional decay factor of $|\log x|^{-1}$.
The above relation is justified in the integral in \eqref{eq:fix_mapa} by applying an integration by part argument. 
 We choose an integrating factor $\rho$ precisely to captures the effect of this extra local terms. 
 Since $\rho \sim |\log x|^{1/3}$ is growing, this crucial weight allows us to prove 
 that the perturbation has a \emph{faster} decay rate $ x^{-1/3} |\log x|^{-1/3}$ than the 
approximate profile. See Remark \ref{rem:superlinear}.
 This is a crucial step to overcome the logarithmic growth of $V_x$ in \eqref{eq:ansatz_V}.  See more discussion in Section \ref{sec:log_cancel}.

\end{remark}

\subsection{Decomposition of the map}\label{sec:decomp1}
We use $W, V$ to denote the full solution, $\bw, \bv  = V(\bw)$ to denote the approximate solution, $w$ to denote the perturbation, and further introduce $\bar g$ as:
\beq\label{eq:nota1}
\bv = V(\bw), \quad 
  \bar g(x) \teq \f{ \bv(x)}{x} ,
  \quad 
  W = w + \bw,  \quad V = x + \psio(w + \bw) = \bv + \psio(w) 
\eeq

Since $\cR(\bw)$ is the residual error, we decompose the map \eqref{eq:fix_map} as 
\bseq\label{eq:F_lin}
\beq
  \cF(w) = \cL(w) + \cE(w) + \cN(w),
\eeq
where $\cL,  \cN, \cE$ denote the linear part, nonlinear part, and the error part, respectively, 
\beq
\bga
  \cL (w) \teq \f{\bw}{\rho}  \int_0^{x} \B( \f{ \td \cR(w)}{2 \bv} \rho
 + \f{\pa_x \rho}{\rho} \cdot  \f{w}{\bw} \rho \B)  ,  \\
 \cN \teq \f{\bw}{\rho}  \int_0^x \td \cR(w) \cdot \B( \f{W}{2V \bw} - \f{1}{2 \bv}  \B) \rho,
 \quad 
  \cE \teq  \f{\bw}{\rho} \int_0^x  \f{ \cR(\bw)}{2 V} \cdot \f{W}{\bw} \rho .
 \ega
\eeq
In the above derivation, we have further decomposed the residual operator $\cR(W) = \cR(\bw + w)$ into the residual error $\cR(\bw)$ and a linear operator $\td \cR$, which depends on $w$ linearly and nonlocally:
\beq\label{eq:lin_nloc}
\bga
      \cR(W ) = \cR(\bw) + \td \cR(w),  \\
      \cR(\bw) \teq  3 - \alb - (1-\alb) \bv_x - 2 \bv \f{\pa_x \bw}{ \bw} , \quad 
     \td \cR( w ) \teq -  \f{2}{3}  \psiox (w) - 2 \psio \f{\pa_x \bw}{\bw} .
  \ega
\eeq
\eseq

For the weight $\rho$ defined in \eqref{eq:rho1} and parameters $\vmu = (\mu_1, \mu_2,.., \mu_\nmu), \bb=(b_1, b_2,..,b_\nmu )$ with $\mu_i > 0, b_i \in [-1.2, \f13 ], b_1< ..<b_\nmu$ given in 
\eqref{def:mu_b}, Appendix \ref{app:para_wg}, we define the weight and norm 
\bseq\label{norm:X}
\beq
  || w ||_{\bcX} \teq ||  w \vp \rho ||_{L^{\infty}},
  \quad 
  \vp(x ) \teq \max_{i \leq \nmu} \mu_i^{-1} |x|^{b_i} ,
\eeq
To distinguish \(\bcX\)-norm  in \eqref{norm:X} for \(\alpha=\alb\) from  \(\cX\)-norm 
\eqref{norm:X_ep} used later for \(\alpha<\alb\), we use the subscript \(\alb\) for the former $\bcX$.  We define a $\d$-ball in the $\bcX$-norm
\beq
      \bD_{\d} \teq \{ w: || w||_{\bcX} \leq \d \},
\eeq
\eseq
where $\rho$ is the weight to be defined in \eqref{eq:rho1}. We aim to prove the following fixed point estimates. 

\begin{thm}\label{thm:onto}
There exists $\dfix >0$ such that, for $|| w ||_X \leq \dfix$, we have 
\beq\label{eq:onto}
   || \cF(w) ||_X \leq \dfix.
\eeq
Moreover, there exists $\lamalb \in (0,1)$, such that for any $w \in \bcX$, we have 
\beq\label{eq:lin_contra}
||\cL(w)||_{\bcX} \leq \lamalb || w||_{\bcX}.
\eeq
\end{thm}

\begin{thm}\label{thm:schauder}
The map $\cF: \bDD  \to \bDD$ is  continuous, and compact, 
with respective to the $\bcX$-norm. Moreover, $\cF(w)(x)$ is locally Lipschitz and satisfies $|\pa_x \cF(w)(x)| \les_R |x|^{\alb+1} $ for any $R >0$,
with uniform constants for any $w \in \bDD$.
\end{thm}

Theorems \ref{thm:onto} and \ref{thm:schauder} imply the following results.

\begin{thm}[\bf Existence of fixed point]\label{thm:fix_point}
The map $\cF$ admits a fixed point $w_{\alb} = \cF(w_{\alb})$
with  $ \| w_{\alb} \|_{\bcX} \leq \dfix$. 
The function $w_{\alb}$ is odd, locally Lipschitz and satisfies $|\pa_x w_{\alb}(x)| \les_R |x|^{\alb+1} $ for any $R >0$.
Moreover, $ \wwwa = \bw +  w_{\alb}$ solves the profile equation \eqref{eq:1D_dyn} with $\wwwa$ being locally Lipschitz. 
\end{thm}

In Theorem \ref{thm:reg_alb}, we further prove that the profile $\wwwa \in C^{\infty}$ and 
establish decay estimates.

\begin{proof}[Proof of Theorem \ref{thm:fix_point}]
By definition of $\bD_{\d}$ in \eqref{norm:X}, $\bD_{\d}$ is convex and closed for any $\d > 0$. 
By Theorems \ref{thm:onto} and \ref{thm:schauder}, $\cF$ continuously maps $\bD_{\dfix}$ into itself 
and $\cF(\bD_{\dfix})$ is compact. 
Applying the Schauder fixed-point theorem, we obtain that $\cF$ has a fixed point in $\bD_{\dfix}$. 

Since $w_{\alb} = \cF(w_{\alb})$, the regularity of $w_{\alb}$ follows from Theorem \ref{thm:schauder}.
Since $w_{\alb} \in \bcX$ and $ | \pa_x w_{\alb}| \les |x|^{1 + \alb}$, undoing the derivation in
\eqref{eq:fix_map1}-\eqref{eq:fix_map}, we prove that $\wwwa = \bw + w_{\alb}$ solves 
the profile equation \eqref{eq:1D_dyn} and it is locally Lipschitz.
\end{proof}

 The key step in our whole proof of Theorems \ref{thm:fix_point} is to prove that the linear operator $\cL$ is a contraction, i.e. \eqref{eq:lin_contra}. We treat the nonlinear and error terms $\cN, \cE$ perturbatively. We prove Theorem \ref{thm:onto} in Section \ref{sec:onto} and Theorem \ref{thm:schauder} in Section \ref{sec:schauder}.

\begin{remark}[\bf  Banach vs Schauder]\label{rem:Banach_Schauder}
Based on  Theorem \ref{thm:onto} 
that $\cF$ maps  $\bD_{\dfix}$ (a closed convex set) into itself,  and estimate \eqref{eq:lin_contra}, one may apply either Banach or Schauder fixed point argument to construct a solution. 
In the Banach approach, one must further show that
$\cF$ is a contraction. This follows essentially from 
the contraction estimate for $\cL$ \eqref{eq:lin_contra} since $\cE, \cN$  can be estimated 
 perturbatively.  
 We adopt the Schauder approach, since the map $\cF$ \eqref{eq:fix_map} gains a derivative. The remaining steps to show that 
$\cF$ is continuous and compact are \textit{purely qualitative},
 and they avoid tracking profile-dependent constants and verifying additional inequalities for the estimates of $\cN, \cE$, despite these being purely perturbative.

\end{remark}

\section{Estimates of the fixed point map}\label{sec:onto}

In this section, we prove that $\cL$ is a contraction and Theorem \ref{thm:onto}. We first develop nonlocal estimates in Section \ref{sec:u_est}, then we decompose $\cL$ in Section \ref{sec:lin_dec}, and estimate different terms in $\cL$ in the rest of the sections.

\subsection{Weighted $L^{\infty}$ estimates of velocity}\label{sec:u_est}

When  $\al = \alb$, we simplify  $\cJ_{\al}(w)$ \eqref{eq:Jw} as
 \beq\label{eq:Jw_alb}
 \cJ( w )(x) = \cJ_{\alb} = \int_0^x y^{\alb - 1} w d y. 
 \eeq

 For any scalar $b$, vectors $\vmu, \bb \in \R^n$, and interval $I \subset \R^+$, we define the constants 
\footnote{
  The subscript $0$ in the parameters $ \cpsidxo,  \cpsixo,   \mCC_{\D, 0}$ \eqref{eq:vel_const} indicates that $\mCC_{\bullet, 0}$ is an intermediate variable. 
}
\bseq\label{eq:vel_const}
\beq
\bal
    \cpsidxo( b, I) & \teq \tts{\int}_I | K_{ \alb, 1}(1, z) | \cdot |z|^{-b}  d z,
\qquad 
 \cpsidxJ(b, I)  \teq \tts{\int}_I | K_{ \alb, 1, J}(1, z) | \cdot |z|^{-b}  d z, \\
\cpsixo( b, I) & \teq \tts{\int}_I | K_{ \alb, 2}(1, z)| \cdot |z|^{-b}  d z, \qquad 
 \cpsixJ(b, I)  \teq \tts{ \int}_I | K_{ \alb, 2, J}(1, z) | \cdot |z|^{-b}  d z  ,
\\ 
   \mCC_{\D, 0}(b, I) & \teq \tts{\int}_I |K_{\alb, 1}(1, z) - K_{\alb, 2}(1, z) | \cdot |z|^{-b} d z .
  \eal
\eeq

Let $\cI = \{ I_i  \}_{i=1}^N$ be a collection of intervals that partition $\R_+$, i.e., $\cup_i I_i = \R_+$. Given parameters $\vmu, \bb \in \R^n$ and the 
partition $\cI$, we define the functions $\MCC_{\bullet}(\vmu, \bb, \cI, x)$ as
\beq
\bga
  \MCC_{\bullet}( \vmu , \bb, \cI, x ) \teq \tts{ \sum_{ I \in \cI } } \min_{ i\leq n}( \mu_i \mCC_{\bullet}( b_i, I)  |x|^{\alb - b_i} ), \\
\ega
\eeq
where $\bullet$ denotes one of the following five cases 
\[
\bullet \in \{ (\psio/x, 0),  (\psio/x, J), (\psio_x, 0), (\psio_x, J ) , (\D, 0) \} .
\]
By definitions, for any $1\leq i \leq \nmu$, where $\nmu$ denotes the length of the vector $\vmu,\bb$,
we obtain
\beq\label{eq:vel_const:upper}
  \MCC_{\bullet}( \vmu , \bb, \cI, x ) 
  \leq  \tts{ \sum_{ I \in \cI } }  \mu_i \mCC_{\bullet}( b_i, I)  |x|^{\alb - b_i} 
  = \mu_i \mCC_{\bullet}( b_i, \R_+) |x|^{\alb - b_i } .
\eeq

\eseq

Using the estimate of kernels from Lemma \ref{lem:K_decay_basic}, 
for kernels $\bullet \in \{ ( \psio/x, 0 ), (\psio_x, 0)  \}$, we obtain
\bseq\label{eq:vel_const_est}
\beq\label{eq:vel_const_est:a}
\bal
 | \mCC_{\bullet}(b, I)| & \leq    | \mCC_{\bullet}(b, \R_+)|  \les  \int_0^{\infty}
( \one_{z \leq 2} ( z^{\alb-1} + |z-1|^{\alb-1} )
 + \one_{z > 2} |z|^{\alb-3} ) |z|^{-b} d z \\
 & \les \int_0^{1/2} z^{\alb-1-b} d z+ 1 \les_b \f{1}{\alb - b }, \quad \forall \ b \ \in [-1.5, \alb) .
\eal
\eeq

For \(b\ge \alb\), these two constants become $\infty$, and the corresponding estimates are no longer useful
\beq\label{eq:vel_const_est:infty}
 | \mCC_{\bullet}(b, I)| = \infty, \quad \forall \, b \geq \alb, \quad  \bullet \in \{ ( \psio/x, 0 ), (\psio_x, 0)  \} .
\eeq
For other kernels $\bullet \in \{ ( \psio/x, J ), (\psio_x, J), (\D, 0)  \}$, we have \emph{uniform estimates} up to $b \leq 0.99$ \footnote{
The integral in 
\eqref{eq:vel_const_est:b} becomes unbounded as $b \to 1^-$. But we never choose $b \geq 1/2$.
}
: 
\beq\label{eq:vel_const_est:b}
\bal
 | \mCC_{\bullet}(b, I)| & \leq    | \mCC_{\bullet}(b, \R_+)|  \les  \int_0^{\infty}
( \one_{z \leq 2} ( 1 + |z-1|^{\alb-1} )
 + \one_{z > 2} |z|^{\alb-3} ) |z|^{-b} d z \les 1, \ \forall  \, b \ \in [-\f32, 0.99] .
\eal
\eeq
\eseq
The correction of the kernel $2\alb x y^{\alb-1} \one_{y < x}, 
2\al  y^{\alb-1} \one_{y < x}$ in $K_{\alb,1,J}, K_{\alb,2,J}$ \eqref{eq:ker:b}, \eqref{eq:ker:c} removes the singularity $z^{\alb-1}$ in 
\eqref{eq:vel_const_est:a}, resulting in uniform boundedness as $b \to \f13$.

\begin{remark}[\bf Intervals $I_i$ and functions $\MCC_{\bullet}( \vmu , \bb, \cI, x ) $]\label{rem:part}

To obtain a relatively sharp estimate for $\psio, \psio_x$, we only need to use a few intervals
$\cI = \{ I_i \}_{ i = 1}^N$ with $N = 6$ to partition $\R_+$. These intervals are given in \eqref{def:interval}. We choose the parameters $\vmu, \bb$ for the weight $\vp$ in \eqref{def:mu_b}.
With $\cI, \vmu, \bb$ being chosen, 
since the kernels have good sign properties and are not singular, it is elementary to derive rigorous upper bound for the 
 constants $ \mCC_{\bullet}(b, I)$ in \eqref{eq:vel_const}, which involves $5 \cdot \# \cI \cdot \# \vmu
 \leq 150$ 
 explicit integrals. 
 See Appendix \ref{app:sharp_constant}. Here, $\# A$ denote the number of element in $A$.
After bounding these constants, we obtain the \emph{explicit} upper bound 
for $\MCC_{\bullet}(\vmu, \bb, \cI, x)$ in  \eqref{eq:vel_const}. 
\end{remark}

Now, we state our main estimates for  $\psio, \psio_x$ defined in \eqref{eq:ker},

\begin{lem}\label{lem:vel_est}
Let $\alb = \f{1}{3}$. 
Suppose that $\vmu = (\mu_1, .., \mu_\nmu), \bb =(b_1,.,, b_\nmu )$ with $\mu_i > 0, b_i \in (- \f{5}{3}, 1)$
and $b_1 < b_2 < .. < b_\nmu$.  Recall  the $\bcX$-norm, weight $\vp = \max( \mu_i^{-1} |x|^{b_i} )$ from \eqref{norm:X}, and  the functions $  \MCC_{\bullet}( \vmu , \bb, \cI, x ) $ from  \eqref{eq:vel_const}. 

If $\max_i(b_i) < \alb$, 
\footnote{
We require $\max(b_i) < \alb$ since the constants $\cpsidxo( b_\nmu, I) ,  \cpsixo(b_\nmu, I) = \infty$ for $b_\nmu = \max b_i \geq \alb$ by \eqref{eq:vel_const_est:infty}.
}
we have the following estimates for $ \psio(w)$ \eqref{eq:ker:a} 
\bseq\label{eq:vel_est}
\beq\label{eq:vel_est:a}
\bga
\B|\tf{1}{x} \psio   \B| \leq \Cpsidxo(\vmu, \bb, \cI,  x )  || w \vp ||_{ L^{\infty} } , \quad 
               | \psio_x  | \leq \MCC_{\psio_x, 0}(\vmu, \bb, \cI,  x)  || w \vp ||_{ L^{\infty} } .         
\ega
\eeq

If $\max_i(b_i ) < 1$, we have the following estimates
\beq\label{eq:vel_est:b}
\bal
  \B|\tf{1}{x} \psio + 2 \alb \cJ\B| & \leq \CpsidxJ(\vmu, \bb, \cI,  x )  || w \vp ||_{ L^{\infty} }
        \quad 
                | \psio_x + 2 \alb \cJ  | \leq \CpsixJ(\vmu, \bb, \cI,  x)  || w \vp ||_{ L^{\infty} } , 
          \\         
     \B| \psio_x - \tf{1}{x} \psio  \B| &  \leq  \MCC_{\D, 0}(\vmu, \bb, \cI,  x) || w \vp ||_{ L^{\infty} } .
  \eal
  \eeq

Moreover, for any weight $\rho$ and parameter $x_1$ appearing in the $\bcX$-norm \eqref{norm:X} that 
satisfies  $\rho \geq 1$ and $x_1 > 10$,  we have the following improved decay estimates for $ \cJ(w)$ \eqref{eq:Jw_alb} and $\f{\psio}{x}, \psio_x$
\footnote{
  Since it is relatively simple to obtain piecewise sharp bound 
  for the integrand in the function $\MCC_J(\vmu, \bb, x)$ in \eqref{eq:J_est}, we do not partition the integral $\MCC_J(\vmu, \bb, x)$ as that in \eqref{eq:nonloc_idea}. 
}
\beq\label{eq:J_est}
\bal
|\cJ(w) | & \leq \MCC_J(\vmu, \bb, x) || w ||_{\bcX} ,  \quad 
 \MCC_J(\vmu, \bb,    x)  = \int_0^x y^{\alb-1}  \min_i (\mu_i |y|^{-b_i}) \rho(y)^{-1} d y, \\
       |\tf{1}{x} \psio  | &  \leq 
 \min( \Cpsidxo(\vmu, \bb, \cI, x)  , \ \CpsidxJ(\vmu, \bb, \cI, x) +2 \alb  \MCC_J(\vmu, \bb, x ) ) || w ||_{\bcX} \teq  \Cpsidx( \vmu, \bb, \cI,  x  ) || w ||_{\bcX},  \\
  | \psio_x | & \leq \min( \MCC_{\psio_x, 0}(\vmu, \bb, \cI,  x) ,   \ \MCC_{\psio_x, J}(\vmu, \bb, \cI,  x) + 2 \alb   \MCC_J( \vmu, \bb, x ) ) || w ||_{\bcX}
 \teq \MCC_{\psio_x}( \vmu, \bb,\cI, x  )  || w ||_{\bcX}.
\eal
\eeq

Suppose $ \max_i b_i = \alb$. 
For any $\g \in (0, 1)$, we have an improved decay estimate:
\begin{gather}\label{eq:est_vmix}
       | \psio_x - \tf{1}{x} \psio |
      \leq  \MCC_{\D}(\vmu, \bb, \cI,  x) || w ||_{\bcX},  
\end{gather}
where  $ \MCC_{\D}(\vmu, \bb,\cI, x)$ is defined as follows and $\mu_\nmu$ is the coefficient 
$\mu_\nmu^{-1} |x|^{\alb}$ in the weight $\vp$ \eqref{norm:X}
\[
     \MCC_{\D}(\vmu, \bb,\cI, x) \teq \min\B( (1 - x_1^{-\g} )^{\alb-2}
     \cdot x^{-\g(1 + \alb)  }  \mu_\nmu     
      + (1-\g)^{-1/3} \rho(x)^{-1} \MCC_{\D,0}(\vmu, \bb, \cI, x ), \  \MCC_{\D,0}(\vmu, \bb, \cI,  x) \B)  .    
\]

\eseq

\end{lem}

For large $x$, compared to the first estimate in \eqref{eq:vel_est:b}, the second estimate 
in \eqref{eq:J_est} has an extra 
crucial decaying factor $|\log x|^{-1/3}$ due to the weight $\rho$. Note that the functions $\MCC_J, \MCC_{\psio/x}$ depend on the weight $\rho$.  We suppress their dependence on $\rho$ when no confusion arises.

\begin{proof}
The main idea is to use the scaling symmetry of the kernel and estimate the integral directly.

Given a $m$-homogeneous kernel $K(y, z)$, i.e. $K(\lam y, \lam z ) = \lam^m K(y, z)$, for any $x > 0$, using $ |w| \leq \vp^{-1} \nlinf{ w \vp }$ and a change of variable $y = x z$,  we estimate 
\bseq\label{eq:nonloc_idea}
\beq
\bal
S & \teq | \int_{ \R_+ } K(x, y) w(y) d y |  \leq \nlinf{ w\vp}    \int_{ \R_+ } | K(x, y) | \vp(y)^{-1} d y \\
 & = \nlinf{ w \vp }   \int_{ \R_+ } | K(x, x z) | \vp(x z)^{-1} d (x z)  = \nlinf{ w\vp} |x|^{m+1} \int_{\R_+} K(1, z) \min_i( \mu_i  |x z|^{-b_i} ) d z .
 \eal
\eeq

It is quite technical to estimate the upper bound for all $x \geq 0$ effectively. 
\footnote{
One can obtain a sharper upper bound using a suitable dyadic partition in $x$; see, e.g., \cite{ChenHou2023a,ChenHou2023b}. 
Here we use simpler estimates, since the model is much simpler and the contraction estimates have a much larger margin.
}
Instead, we derive a weaker bound by partitioning $\R_+ = \cup_{j=1}^N I_j$, $\cI = \{ I_j \}_{j=1}^N $ and then estimate 
\begin{align}
  S & \leq \nlinf{ w \vp }   |x|^{m+1} \sum_{I \in \cI} \int_{I} K(1, z) \min_i( \mu_i  |x z|^{-b_i} ) d z
  \leq  \nlinf{ w \vp }    |x|^{m+1}  \sum_{I \in \cI} \min_{i} \B( \int_{I} K(1, z)  \mu_i  |x z|^{-b_i}  d z \B) 
  \notag \\
  & =  \nlinf{ w \vp } \sum\nolimits_{I \in \cI} \min_i\B(  \mu_i |x|^{m+1 - b_i} \int_{I} |K(1, z)| \cdot |z|^{-b_i} d z  \B) . \label{eq:nonlocal_idea:b}
\end{align}
\eseq

The estimate with the trivial partition $\cI = \{\R_+\}$ is not sharp enough for our estimate.
\footnote{
Since the weights $\mu_i | x |^{-b_i}$ have different behaviors, 
for different interval $I \in \cI$, the minimum in \eqref{eq:nonlocal_idea:b} may be attained by different $i$, yielding a much sharper estimate than that obtained with the trivial partition $I_1 = \R_+$.
}
See Remark \ref{rem:part}. For large $x$, we further improve the estimates using the weight $\rho$ in \eqref{norm:X}, \eqref{eq:rho1}.

\vs{0.05in}

\paragraph{\bf Proof of \eqref{eq:vel_est:a},  \eqref{eq:vel_est:b}}
We only prove the estimate of $\f{\psio}{x}$ as the proof for other estimates in \eqref{eq:vel_est:a},
\eqref{eq:vel_est:b} are similar. Since $K_{1,\alb}$ is $\alb$-homogeneous, using \eqref{eq:ker} and the above estimate \eqref{eq:nonloc_idea} with $K(x, y) = K_{1,\alb}(x, y), m = \alb$, we obtain 
\bseq\label{eq:vel_est_pf1}
\beq
\bal
  |\f{ \psio}{x}| & = \B| \f{1}{x} \int_{\R_+}  K_{\alb, 1 }(x, y)  w(y) dy \B|
\leq  |x|^{-1} 
 \nlinf{ w \vp } 
\sum_{I\in \cI} \min_i\big(  \mu_i |x|^{\alb +1 - b_i} \int_{I} |K_{\alb, 1}(1, z)| \cdot |z|^{-b_i} d z  \big) .
  \eal
\eeq
Using the notation \eqref{eq:vel_const}, we prove the estimate of $\f{ \psio}{x}$ in \eqref{eq:vel_est:a}:
\beq
    |\tf{ \psio}{x}| \leq |x|^{-1} \| w \|_{\bcX}
\sum\nolimits_{I \in \cI} \min_i\B(  \mu_i |x|^{\alb +1 - b_i} \mCC_{\psio/x,0}( b_i, I )  \B)
= \MCC_{ \psio/x, 0 }(\vmu, \bb, \cI, x)   \| w \|_{\bcX}  \, ,
\eeq
\eseq

\vs{0.05in}
\paragraph{\bf Proof of \eqref{eq:J_est}}

Recall $\cJ$ from \eqref{eq:Jw_alb}.
The estimate of $\cJ$ in \eqref{eq:J_est} follows from the pointwise estimate 
$|w| \leq \min_i(\mu_i |x|^{-b_i } ) \rho(x)^{-1} $ \eqref{norm:X}. 
Using estimate  \eqref{eq:J_est}, \eqref{eq:vel_est:b}, and triangle inequality, we obtain another estimate of $\psio $
\[
   | \psio | = 
   | ( \psio + 2 \alb x \cJ) | +  |2 \alb x \cJ| 
  \leq  ( |x| \MCC_{\psio/x, J}(\vmu, \bb, \cI, x) + 2 |\alb  x | \MCC_J(x)   ) \cdot \| w \|_{\bcX}.
  \]
Optimizing the above two estimates, we prove the estimate of $\psio$ in \eqref{eq:J_est}. 
The estimate of $\psio_x$ in \eqref{eq:J_est} is proved similarly.

\vs{0.05in}
\paragraph{\bf Proof of \eqref{eq:est_vmix}}

If $b_\nmu = \alb$, by definition of $\bcX$-norm \eqref{norm:X}, we have $|w| \leq \mu_\nmu |x|^{-\alb} \rho(x)^{-1}$. If $x < x_1$, since $\rho(x) = 1$ \eqref{eq:rho1} and $\rho^{-1} (1-\gam)^{-1} > 1$, the estimate \eqref{eq:est_vmix} follows from \eqref{eq:vel_est:b}. 

Next, we consider $x > x_1> 10$. 
Recall the kernel $K_{\D}$ of $\psio_x - \f{\psio}{x}$ from \eqref{eq:ker}.
We decompose
\[
\bal
| \psio_x - \f{\psio}{x}| & \leq \big( \int_0^{x^{1-\gam}} + \int_{x^{1-\gam}}^{\infty} \big) | K_{\D}(x, y) w(y) | d y \teq I_1 + I_2. \\
\eal
\]
For $y \geq x^{1-\g}$, by definition of $\rho$ \eqref{eq:rho1}, we get
\beq\label{eq:vel_est_pf2}
  \rho(y) \geq  \B( \f{\log x^{1-\g}}{\log x_1} \B)^{1/3} \vee 1
  \geq (1-\g)^{1/3} \cdot \B( \B( \f{\log x}{\log x_1} \B)^{1/3} \vee 1  \B)
  = (1-\g)^{1/3} \rho(x).
\eeq

Since $|w(y)| \leq \min(\mu_i |y|^{-b_i}) \rho(y)^{-1} || w ||_{\bcX}$, using \eqref{eq:vel_est_pf2} and following  \eqref{eq:vel_est_pf1}, we obtain
\[
  |I_2 | \leq (1-\g)^{-1/3} \rho(x)^{-1} 
  \int_0^{\infty} | K_{\D}(x, y) | \min(\mu_i |y|^{-b_i})  d y  \cdot || w ||_{\bcX}
  \leq (1-\g)^{-1/3} \rho(x)^{-1}  \MCC_{\D, 0}(\vmu, \bb, \cI, x) || w ||_{\bcX}.
\]
For $I_1$, changing $y = x z$, using $\rho \geq 1$ and $|w| \leq \mu_\nmu |x|^{-\alb} \rho(x)^{-1} || w ||_{\bcX} \leq \mu_\nmu |x|^{-\alb}  || w||_{\bcX}$, we get
\[
  \bal
  I_1 & \leq \mu_\nmu \int_0^{x^{1-\g}} | K_{\D}(x, y) \cdot y^{-\alb} | d y  \| w \|_{\bcX}
 = \mu_\nmu \int_0^{ x^{-\g} } | K_{\D}(x, x z) |\cdot  |x z|^{-\alb}  x d z  \| w \|_{\bcX}  \\
 & = \mu_\nmu  \int_0^{x^{-{\g}}} |K_{\D}(1, z) \cdot z^{-\alb} |  d z  \| w \|_{\bcX}  .
\eal
\]

From the definition of $K_{\D}$ \eqref{eq:ker}, since $\alb =\f13 \in (0, 1)$, using  $ 0\leq z \leq x^{-\g} \leq x_1^{-\g }< 1 $ and mean-value theorem with $|\xi| \geq 1- x_1^{-\g},
\forall  \, \xi \in [1-z, 1 + z]$, we have 
\[
\bal
  |K_{\D}(1, z)| & = \B| |1 + z|^{\alb} - |1-z|^{\alb} - \alb (  (1  + z)^{\alb-1} - |1-z|^{\alb-1}  ) \B|\\
  & \leq 2 \alb z ( (1 - x_1^{-\g} )^{\alb-1}
   +  (1 - \alb) (1 - x_1^{-\g} )^{\alb-2} ) 
   \leq    4 \alb  (1 - x_1^{-\g} )^{\alb-2} z.
   \eal
\]

Using the above two estimates and $\alb = \f{1}{3}$, we estimate $I_1$ as 
\[
\bal
I_1 & \leq  4 \alb (1 - x_1^{-\g} )^{\alb-2} \cdot \mu_\nmu \int_0^{x^{-\g}} z^{\alb-1} \cdot z d z || w ||_{\bcX}
= \f{  4 \alb  (1 - x_1^{-\g} )^{\alb-2} } { 1 + \alb } 
\cdot \mu_\nmu
x^{-\g(1 + \alb) }  || w ||_{\bcX}  \\
 & =     (1 - x_1^{-\g} )^{\alb-2} \cdot \mu_\nmu  x^{-\g  (1 + \alb) }  || w ||_{\bcX}  .
 \eal
\]
The above estimates of $I_1, I_2$ imply the first estimate in \eqref{eq:est_vmix}.
Combining the above estimate and \eqref{eq:vel_est:b}, we prove \eqref{eq:est_vmix}. 
\end{proof}

Lemma \ref{lem:vel_est} implies the following estimates for $\psio$ without explicit constants, which will be used to prove \emph{qualitative estimates} and Theorem \ref{thm:schauder}.

\begin{cor}[\bf Nonlocal estimates with implicit constants]\label{cor:vel_est}
Recall $\alb = \f{1}{3}$ and $\lgp x $ from \eqref{def:lgp}.  Let $\MCC_{\bullet}(x) =\MCC_{\bullet}(\vmu, \bb, \cI, x)$ be the coefficients  in \eqref{eq:vel_const}. 
For any weight $\vp$ in \eqref{norm:X} with parameters  $\vmu, \bb \in \R^{\nmu}$  satisfying 
$\mu_i>0, \max b_i =  \alb$ and $b_1 \in [-\f32, 0]$ (not necessary those in \eqref{def:mu_b}), we have the following estimates for $\MCC_{\bullet}, \psio$ 
\bseq\label{eq:coe_asym}
\beq
  \MCC_{\psio/x,J}(x) +  \MCC_{\psio_x,J}( x)  \les \min( |x|^{\alb - b_1},  1 ),
  \  \ 
   |\tf{1}{x} \psio+ 2\alb \cJ|  +  |\psio_x + 2 \alb \cJ|
    \les \min( |x|^{\alb - b_1},  1 ) \nlinf{w \vp} . \label{eq:coe_asymp:0} \\
\eeq

In addition, for the weight $\rho$ chosen in \eqref{eq:rho1} and the norm \eqref{norm:X}, we have
\begin{align}
  \MCC_{\psio/x}(x)  +  \MCC_{\psio_x}( x)  \les \min( |x|^{\alb - b_1},  |\lgp x|^{2/3} ),
  \  \
   |\tf{\psio}{x}|  +  |\psio_x|
    \les \min( |x|^{\alb - b_1},  |\lgp x|^{2 \alb} ) \nchib{ w } , \label{eq:coe_asymp:a} \\
\MCC_{\D}(x)  \les \min( |x|^{\alb - b_1},  |\lgp x|^{- \alb} ), \  \ 
 | \psio_x -\tf1x \psio| 
\les  \min( |x|^{\alb - b_1},  |\lgp x|^{- \alb} ) \nchib{ w } ,
\label{eq:coe_asymp:b}
\end{align}
\eseq
where the implicit constant depends on $\vmu, \bb$ in the norm \eqref{norm:X} 
and $x_1$ in the weight $\rho$ \eqref{eq:rho1}.

\end{cor}

Estimate \eqref{eq:coe_asym} show that for large $x$,
the leading order parts of $\tf{1}{x} \psio, \psio_x$ are given by $-2 \alb \cJ$.

\begin{proof}

Let $\{ \R_+\}$ be the trivial partition for $\R_+$. By definition \eqref{eq:vel_const}, 
for any partition $\cI$, we obtain 
\beq\label{eq:cor_vel_pf0}
  \MCC_{\bullet}(\vmu, \bb, \cI, x) \leq   \MCC_{\bullet}(\vmu, \bb, \{ \R_+\}, x).
\eeq

To simplify notation, we do not track constants that depend on $\vmu, \bb$ and simplify 
$\MCC_{\bullet}(\vmu, \bb, \cI,  x)$ as $\MCC_{\bullet}(x)$. Since 
$b_i \in [-3/2, \alb]$ for $1\leq i\leq \nmu$,
using estimate \eqref{eq:vel_const_est}, we obtain
\[
|\mCC_{\bullet}(b_i, \R_+ )|
 \les 1 ,\quad \forall \, \bullet \in \{   (\psio/x, J),  (\psio_x, J ) , (\D, 0) \} .
\]
Thus, using the definition of $\MCC_{\bullet}$ in \eqref{eq:vel_const} and $ \max_i b_i = \alb$, we obtain 
\beq\label{eq:cor_vel_pf1}
  |\MCC_{\bullet}( x)| \les \min( |x|^{\alb - b_1}, |x|^{\alb - \alb} )
  = \min(  |x|^{\alb - b_1}, 1 ) , \quad \forall \, \bullet \in \{   (\psio/x, J),  (\psio_x, J ) , (\D, 0) \}.
\eeq
Using the above estimates and \eqref{eq:vel_est:b}, we prove \eqref{eq:coe_asymp:0}.

From \eqref{eq:rho1}, we obtain  $\rho(x)^{-1} \les   |\lgp x|^{-\alb} \we 1$. 
We estimate $\MCC_J$ from \eqref{eq:J_est} 
\[
  \MCC_J(x) \les \int_0^x y^{\alb - 1} \min( |y|^{-\alb}, 1 ) \min(1, |\lgp y|^{-\alb } )  d y
  \les \min(1, |\lgp x|^{2\alb}).
\]

Applying  the above two estimates  to $\MCC_{\psio/x}, \MCC_{\psio_x}$ \eqref{eq:J_est}, we prove
\[
  |\MCC_{\psio/x}(x) | \les \min( |x|^{\alb - b_1},  |\lgp x|^{2 \alb} ),
  \quad 
  \MCC_{\psio_x}(x) \les \min( |x|^{\alb - b_1},  |\lgp x|^{2 \alb} ).
\]

Finally, we estimate $\MCC_{\D}$ from \eqref{eq:est_vmix}. Choosing $\g = \f{1}{2} $ in \eqref{eq:est_vmix} and using \eqref{eq:cor_vel_pf1} for $\MCC_{\D, 0}(x)$ and  $\rho(x)^{-1} \les   |\lgp x|^{-\alb} \we 1$, we prove 
\[
  \MCC_{\D}(x) \les \min( |x|^{\alb - b_1}, |\lgp x|^{-\alb}  ).
\]
We complete the proof.
\end{proof}

\subsection{Qualitative estimate of the approximate profile and error}\label{sec:ASS_est}

We have the following estimates for the profile $\bv, \bw$ constructed in Section \ref{sec:appr_profile}.
We use these estimates \emph{only} to prove \emph{qualitative estimates} and Theorem \ref{thm:schauder}. 
While these estimates are  \emph{qualitative}, they motivates several decompositions in estimating the fixed point map $\cF$ in later Section \ref{sec:log_cancel}, \ref{sec:lin_dec}.

\begin{lem}[\bf Non-degenerate (Computer-assisted)]\label{lem:bw_basic_sign}

Let $\bw = \bwp + \bwf \in C^{2,1}$ be the approximate profile constructed in Section \ref{sec:appr_profile}. 
The profile $\bw, \bv \teq V(\bw)$ satisfies
\[
  \pa_x \bw(0)  \in [-1 - 10^{-6}, -1 + 10^{-6}],  
  \quad \bw(x) \leq - 10^{-10} \cdot \min( |x|,  | x|^{-1}), \quad
   \bv(x)  \geq 10^{-10} \cdot x , \quad  \forall  \ x \geq 0 .
  \]

\end{lem}

We prove Lemma \ref{lem:bw_basic_sign} with computer-assistance, which is a trivial consequence of 
the (non-trivial) tight piecewise bounds for the profile.  The above lemma implies qualitative estimates :

\begin{lem}\label{lem:W_asym}
Let $\bw = \bwp + \bwf \in C^{2,1}$ be the approximate profile constructed in Section \ref{sec:appr_profile}, $\bv = V(\bw)$ be the velocity \eqref{eq:ker},
and $\bar g = \f{\bv}{x}$ as defined in \eqref{eq:nota1}. For any $x \geq 0$, we have 
\bseq\label{eq:decay}
\beq\label{eq:decay:W}
\bw \leq 0, \quad |\bw| \asymp \min( |x|, \ang x^{-\alb}), \quad  |\f{x \pa_x \bw}{\bw} + \f{1}{3}| \les |\lgp x|^{-4/3}.
\eeq
where $\lgp x $ is defined in \eqref{def:lgp}. For any $x \geq 1$,  $\bv$ satisfies
\beq\label{eq:decay:V}
\bga
\bv \gtr x \lgp x , \quad 
|\bv - 4 x \log x | \les |\lgp x|^{2/3} .
\ega
\eeq

Let $ \cR(\bw)$ be relative residual error defined in \eqref{eq:lin_nloc}. For 
$x \geq 0$, we have 
\beq\label{eq:lin_nloc:far}
 |\bv_x - \tf{1}{x} \bv - 4  | \les  |\lgp x|^{-1/3}, \quad
 |\cR(\bw)| \les \min( |x|^{1 + \alb},  |\lgp x |^{-2/3} ) .
\eeq 
In the above estimates, the implicit constants only depend on $\bw$ and are independent of $x$.

\eseq

\end{lem}

\begin{proof}

Since $ \bw \in C^{2,1}$ is odd, Lemma \ref{lem:W_asym} implies $\bw(x) < 0, \forall x > 0$
and $|\f{ x \pa_x \bw}{\bw}(x)| \les_R 1 $ for any $|x| \leq R$. 
From \eqref{eq:W_rep},
we obtain $\bw(x) = \bwf(x)$ for $x$ beyond the support of $\bwp$. 
Estimate \eqref{eq:decay:W} follows from a direct calculation for $\bwf$ \eqref{eq:W_rep}, 
which is omitted.

Lemma  \ref{lem:bw_basic_sign} implies \eqref{eq:decay:V} for any finite $x$. Since $\bv = x + \psio(\bw)$, using 
$ \nlinf{(x^{-1} + x^{\alb}) \bw} \les 1$ by \eqref{eq:decay:W} and 
using \eqref{eq:coe_asymp:0} in Corollary \ref{cor:vel_est} with $(\mu_1, \mu_2, b_1, b_2) = (1,1,-1, 1/3)$ and \eqref{eq:decay:W}, we obtain $\bv  = - 2 \alb  x \cJ(\bw) + O(x)$.
For large $x$, by expanding $\cJ(\bw)$ using the formula for $\bwf$ in \eqref{eq:W_rep} 
and $-2 \alb \cca = 4$, we  prove \eqref{eq:decay:V}. 

From \eqref{eq:lin_nloc} and \eqref{eq:ker:a}, we get $\cR(\bw) =    - (1-\alb) \psiox(\bw) - 2 \psio(\bw) \tf{\pa_x \bw}{ \bw} .$
Using Corollary \ref{cor:vel_est} with $(\mu_1, \mu_2, b_1, b_2) = (1,1, -1, 1/3)$,
we obtain estimate \eqref{eq:lin_nloc:far} for any $x, |x| \leq R $. 
We derive the \emph{quantitative} far-field estimates for $V(\bw)(x)$ in Appendix  \ref{app:nloc_very_far}  and prove the first estimate \eqref{eq:lin_nloc:far}. 
We derive the \emph{quantitative} far-field estimates for $\bar \cR(\bw)$ in 
Appendix \ref{app:cR_expand} and use the far-field estimates for $\bv$ 
to prove the decay estimates for $\bar \cR$ in \eqref{eq:lin_nloc:far}. 
While the \emph{qualitative} estimates in \eqref{eq:lin_nloc:far} are not 
difficult to prove using the explicit form for $\bwf$ \eqref{eq:W_rep} 
and by following Section \ref{sec:ansatz}, to avoid repetition we derive them from the more technical \emph{quantitative} estimates. 
\end{proof}

\subsection{Decomposition of linear operator}\label{sec:lin_dec}

The most difficult term in the linear part $\cL$ \eqref{eq:F_lin} is 
\beq\label{eq:cR_decay_moti0}
       \td \cR( w ) = -  \f{2}{3} \psiox (w) - 2 \psio( w) \f{\pa_x \bw}{\bw}  
       = -  \f{2}{3} \psiox (w) +  \f{2}{3} \cdot \f{ \psio}{x} - 2 \cdot \f{ \psio }{x} (  \f{x \pa_x \bw}{\bw}  + \f13)
\eeq

For large $x$, we exploit cancellations between $\psiox $ 
 and $\f{ \psio }{x}$, and in $    \f{x \pa_x \bw}{\bw}  + \f13$, as motivated in Section \ref{sec:idea_step1}. 
 For $w \in \bcX$, applying \eqref{eq:coe_asymp:b} to $ \psiox -\psio /x$ and \eqref{eq:coe_asymp:a} to $\psio/x$ 
from Corollary \ref{cor:vel_est}, using the estimates of $ \f{x \pa_x \bw}{\bw} + \f13, \bv$  from Lemma \ref{lem:W_asym}, and $\rho \les |\lgp x|^{1/3}$ \eqref{eq:rho1}, we obtain
\beq\label{eq:cR_decay_moti:b}
|\f{\td \cR(w)}{ \bv } \rho| \les |x|^{-1} |\lgp x|^{-1} \| w \|_{\bcX},
 \quad \f{\td \cR(w)}{ \bv } \rho  = - \f{2}{3} ( \psiox - \f{ \psio }{x} ) \cdot \f{\rho}{\bw}  + O( |\lgp x|^{-4/3} \| w \|_{\bcX} ),
 \eeq

The above term appears in weighted estimate for $\cL w$ in \eqref{eq:F_lin}.  Despite exploiting the cancellation, the above estimate for $\frac{\td \cR(w)}{\bv}\rho$ is not $L^1$-integrable and is therefore insufficient to close the estimate of $\cL w$ in \eqref{eq:F_lin}. 
We must exploit additional cancellations.

\subsubsection{Local-nonlocal cancellation}\label{sec:log_cancel}

To exploit further cancellation for large $x$,  we estimate $\td \cR$ differently.  For $x_1$ large enough to be chosen, we define
\beq\label{eq:rho1}
  \rho(x) = 1, \quad x \leq x_1,  \quad \rho(x) =  \B( \f{ \log x }{ \log x_1 } \B)^{1/3}, \quad  x > x_1.
\eeq
Clearly, $\rho$ is Lipschitz. 
For large $x$, we use the last expression in \eqref{eq:cR_decay_moti0} for $\td \cR$ 
to isolate the term $ \psiox - \f{ \psio }{x}$.
Since $\pa_x \rho(x) = 0$ for $x \leq x_1$,  we decompose the linear operator $\cL$ 
in \eqref{eq:F_lin} as follows 
\bseq\label{eq:lin}
\beq\label{eq:lin:a}
\bal
 \f{ \cL ( w) }{\bw} \rho
  & =  \int_0^{x_1 \wedge x}  \f{ \td \cR(w)}{2 \bv} \rho  d y
  + 
\int_{x_1 \wedge x }^x \f{  - \f{2}{3} ( \psioy - \f{ \psio }{y} )
 - 2 \f{ \psio }{y} ( \f{1}{3} + \f{ y  \bw_y}{\bw}  )
   }{2 \bv} \rho  d y
     + \int_{x_1 \we x}^x \f{\pa_y \rho}{ \rho } \cdot  \f{ w }{\bw} \rho  d y  \\
     &  \teq I + II + III .
\eal
\eeq

 Recall the notations $\bar g(x) = \f{1}{x}\bv(x) , \cJ$ from \eqref{eq:nota1}. 
We decompose the first integrand in $II$ as
\beq\label{eq:lin:c}
II  
  = -\f{1}{3} \int_{x_1 \wedge x }^x  \f{ ( \f{ \psio }{y} + 2 \alb \cJ )_y  }{\bar g} \rho
   + \f{1}{3} \int_{x_1 \wedge x }^x \f{2 \alb \cJ_y}{ \bar g} \rho 
   - \int_{x_1 \wedge x }^x \f{  \f{ \psio }{y} ( \f{1}{3} + \f{ y  \bw_y}{\bw}  )
   }{ \bv} \rho 
   \teq  II_1 + II_2 + II_3.
\eeq

Since $x_1$ is bounded and the integrand in $I(x)$ satisfies \eqref{eq:cR_decay_moti:b}, $I(x)$ is bounded for large $x$. From estimates \eqref{eq:cR_decay_moti0} and \eqref{eq:cR_decay_moti:b}, the integrand in $II_3$ is 
$L^1$-integrable. Below, we further exploit the cancellation between $III$ and $II_1, II_2$ to gain additional decay. Note that $II_i, III=0$ for $x \leq x_1$. 
\eseq
For $x > x_1$, we explore smallness from $II_1$ by performing integration by parts 
\beq\label{eq:lin1}
  II_1 = -   \f{ \f{ \psio }{y} + 2 \alb \cJ }{ 3 \bar g} \rho \B|_{x_1}^x
  + \f{1}{3} \int_{x_1}^x (\f{ \psio }{y} + 2 \alb \cJ) (\f{\rho}{\bar g})_y  d y.
\eeq

Using estimate \eqref{eq:decay:V}, definition of $\rho$ \eqref{eq:rho1}, and 
the bound \eqref{eq:coe_asymp:a}, one obtains 
\[
  | \f{\f{ \psio }{y} + 2 \alb \cJ }{ 3 \bar g} |  \les |\log y|^{-4/3},  \quad 
  | (\f{ \psio }{y} + 2 \alb \cJ) (\f{\rho}{\bar g})_y | \les
 \f{1}{ y } |\log y |^{-\f13 } \| w \|_{\bcX}  (\f{\rho}{\bar g})_y
   \les \f{1}{y} | \log y |^{-2}  \| w \|_{\bcX} .
\]
We gain an extra decay factor $|\log y|^{-1}$ for the second term compared to the estimate 
\eqref{eq:cR_decay_moti:b}.

It remains to estimate $II_2$ and $III$. Since $\cJ_y = w y^{\alb - 1}$ and 
\beq\label{eq:II2_Jx}
\f{2 \alb \cJ_y}{3 \bar g} \rho = \f{2 \alb w y^{\alb - 1}}{ 3 \bar g} \rho = \f{ 2 \alb w y^{\alb}}{ 3 \bv}  \rho 
= \f{2 \alb y^{\alb} \bw }{ 3 \bv} \cdot \f{w}{\bw} \rho, 
\eeq
we combine $II_2$ and $III$ as follows 
\beq\label{eq:lin2}
  II_2 + III = \int_{x_1}^x \B( \f{\pa_y \rho}{\rho}  + \f{2\alb y^{\alb} \bw }{3 \bv}  \B) 
  \cdot \f{w}{\bw} \rho .
\eeq

From \eqref{eq:decay} and Corollary \ref{cor:vel_est}, we obtain that $I, II_1, II_3$ are integrable as $x \to \infty$.

We choose $\rho$ \eqref{eq:rho1}, which captures the asymptotics of $w$, to obtain the crucial cancellation
\beq\label{eq:cancel_rho}
   \f{\pa_y \rho}{\rho} + \f{2 \alb y^{\alb} \bw}{3 \bv}
    = \f{1}{3} \cdot \f{1}{y} |\log y |^{- 1} + \f{2}{3} \cdot \f{1}{3} \cdot \f{-6}{4} ( y^{-1} |\log y|^{-1}  + O ( y |\log y|^{-4/3}) ) = O( y^{-1} |\log y|^{-4/3} ) ,
\eeq
for large $y$. Using the above crucial reformulation, we treat $I, II_1, II_2 + III, II_3$ perturbatively. 

\begin{remark} 

The cancellation in \eqref{eq:cancel_rho} is nonlocal since we need to perform the key integration by parts \eqref{eq:lin1}.  The slowest decay part in $II$ is governed by $II_2$ \eqref{eq:lin:c} with integrand 
$ \mw{RHS}_{\eqref{eq:II2_Jx}} $ depending on $w$ \emph{locally} \eqref{eq:II2_Jx}. 
Instead of bounding  $ \mw{RHS}_{\eqref{eq:II2_Jx}} $ directly, we introduce $\rho$ in the fixed point map \eqref{eq:fix_map} as an integrating factor. 
One may choose $\rho$ so that $\mw{LHS}_{ \eqref{eq:cancel_rho}}=0$.  We choose the explicit form in  \eqref{eq:rho1} to simplify the estimates for $\rho$.
See Remark \ref{rem:int_wg}.

\end{remark}

\subsection{Estimate linear terms}\label{sec:lin_est}

We estimate $\cL$ using the decomposition in Section \ref{sec:lin_dec} and nonlocal estimates in Lemma \ref{lem:vel_est}. 
Since the parameters $\vmu, \bb$ in \eqref{def:mu_b} and intervals $\cI=\{ I_i \}_{i\geq 1}$ in \eqref{def:interval} have been fixed, below we abbreviate the functions $\MCC_{\bullet}(\vmu,\bb,\cI,x)$ by $\MCC_{\bullet}(x)$ for $\bullet \in \{ \psio/x,\psio_x,\D, (\psio/x, J), (\psio_x, J) \}$ in Lemma~\ref{lem:vel_est}.

\vspace{0.05in}
\paragraph{\bf Estimate of $I$}

Recall $I$ from \eqref{eq:lin:a} and $\td \cR$ from \eqref{eq:lin_nloc}.  Since $\rho(x) = 1$ for $x \leq x_1$, we simplify 
\[
 I = \int_0^{x_1 \wedge x}  \f{ 1 }{2 \bv}  \cdot ( - \f{2}{3} \psioy - 2 \f{ \psio }{y} \cdot \f{y \bw_y}{\bw} )(y) d y,
\]

To exploit the cancellation between the integral of $\psioy$ and $ \psio$, we apply integration by parts to the integrand $  - \f13 \f{ \psioy }{ \bv}$
\[
  I =   -  \f{\psio}{3\bv} \B|_0^{x_1 \wedge x}  + \f{1}{3} \int_0^{x_1 \wedge x} \psio ( \f{1}{\bv} )_y
  - \int_0^{x_1 \wedge x } \f{ 2 \f{\psio}{ y } \cdot \f{ y  \bw_y }{\bw} }{2 \bv}
  =  - \int_0^{x_1 \wedge x}  ( \f{ y  \bv_y }{3 \bv} + \f{ y \bw_y }{\bw} ) \cdot \f{\psio}{ y \bv} d y
  - \f{1}{3} \cdot \f{ \psio(x\we x_1) }{\bv(x\we x_1)}   ,
  \]
where the boundary term at $0$ vanishes since $|\psio(x)| \les |x|^{1+\alb} \nchib{w} $ for $x$ near $0$ 
by Corollary \ref{cor:vel_est} and all functions in the integrand evaluate at $y$.

Recall $\bar g(x) = \f{\bv(x)}{x}$  from \eqref{eq:nota1}. Applying Lemma \ref{lem:vel_est} to $\psio(x)/x$, we obtain 
\beq\label{eq:est_I}
  |I| \leq 
   \B(  \f{ \MCC_{\psio/x}( x\we x_1 ) }{3  |\bar g( x\we x_1 )| }   + 
 \int_0^{x_1 \we x} \MCC_{\psio/x}(y) \B|   \big(\f{ y  \bv_y }{3 \bv} + \f{ y \bw_y}{\bw} \big) \f{1}{\bv} \B| d y \B) 
  \cdot || w ||_{\bcX}
  \teq B_{ I } \cdot || w ||_{\bcX}.
\eeq

\vspace{0.05in}
\paragraph{\bf Estimate of $II_1$}

Recall $II_1$ from \eqref{eq:lin1}. Applying Lemma \ref{lem:vel_est}
 and $\pa_y (\f{\rho}{\bar g} )
 = \f{  \rho_y \bar g -  \rho \bar g_y }{\bar g^2} $, we obtain 
\beq\label{eq:est_II1}
  |II_1| \leq \B(  | \f{\rho  \MCC_{\psio/x, J} }{3 \bar g }(x) | 
  +
| \f{\rho  \MCC_{\psio/x, J} }{3 \bar g}(x_1) | 
+ \f{1}{3} \int_{x_1}^x \MCC_{\psio/x, J}(y) \B| \f{\rho_y \bar g - \rho \bar g_y}{\bar g^2}   \B|  d y
\B)  \cdot || w ||_{\bcX}
\teq B_{II, 1}(x)  \cdot || w ||_{\bcX}.
\eeq

\vspace{0.05in}
\paragraph{\bf Estimate of $II_3$}

Recall $II_3$ from \eqref{eq:lin:c}. Applying Lemma \ref{lem:vel_est}, we obtain
\beq\label{eq:est_II3}
  |II_3| \leq 
  \int_{x_1}^x \B| \f{ \MCC_{\psio/x}(y) \cdot ( \f{1}{3} + \f{ y  \bw_y}{\bw}(y)  )  }{ \bv(y)} \B|  \rho(y) d y \cdot \nchib{ w }
  \teq B_{II, 3}(x)  \cdot || w ||_{\bcX}.
\eeq

\vspace{0.05in}
\paragraph{\bf Estimate of $II_2 + III$}

Recall $II_2 + III$ from \eqref{eq:lin2}. By definition of the $\chi$-norm \eqref{norm:X}, we obtain
\[
  |w| \rho(y) \leq || w ||_{\bcX} \min_i( \mu_i |y|^{-b_i}) .
\]
Using the above pointwise estimate, we obtain
\beq\label{eq:est_II_III}
  |II_2 + III | \leq \int_{x_1}^x \B| 
     \f{\pa_y \rho}{\rho} + \f{2 \alb y^{\alb} \bw}{3 \bv} \B| 
     \cdot \min_i( \mu_i |y|^{-b_i}) \f{1}{|\bw|} d y  \cdot || w ||_{\bcX} 
     \teq B_{II, 2, III}(x) \cdot || w ||_{\bcX} .
\eeq

\vspace{0.05in}
\paragraph{\bf Summary of the estimates}

Combining the above estimates and using the decomposition in \eqref{eq:lin}, we estimate 
the weighted norm of $\cL$ as
\beq
  | \rho \vp \cL w | 
  = \rho \vp \cdot \f{ |\bw|}{ \rho} |( I + II + III) |
  \leq \vp  | \bw | 
  ( \uds{  \eqref{eq:est_I}  }{ B_{I} } + \uds{ \eqref{eq:est_II1} }{ B_{II, 1} }  + \uds{ \eqref{eq:est_II3} }{  B_{II,3 } }
  + \uds{ \eqref{eq:est_II_III} }{ B_{II,2, III} } )(x) \cdot || w ||_{\bcX} .
  \label{eq:lin_est}
\eeq
where the label below $B_{\bullet}$ indicates the equation in which  $B_{\bullet}$ is defined.  Note that the coefficients, e.g. $B_{I}$, only depend on the parameters $\vmu, \bb, \cI$ 
via the functions \eqref{eq:vel_const} and the profile $\bw, \bv$.

\subsection{Estimate nonlinear and error terms}\label{sec:N_err}

Recall nonlinear $\cN$ and error terms $\cE$ from \eqref{eq:F_lin} 
\beq\label{eq:N_err_recall}
   \cN \teq \f{\bw}{\rho}  \int_0^x \td \cR(w) \cdot \B( \f{W}{ \bw } \cdot \f{1}{2V} - \f{1}{2 \bv}  \B) \rho,
 \quad 
  \cE \teq  \f{\bw}{\rho} \int_0^x  \f{ \cR(\bw)}{2 V} \cdot \f{W}{\bw} \rho .
\eeq

Using estimate \eqref{eq:J_est}, we obtain 
\bseq\label{eq:non_est1}
\beq
|\psio(y)| \leq  \MCC_{\psio/x}(y) y  \nchib{ w }, \quad 
\eeq
which implies 
\beq
  V(y) \geq \bv - |\psio| \geq \bv - \MCC_{\psio/x}(y) y  \nchib{ w },
  \quad |V(y)| \geq (\bv - \MCC_{\psio/x}(y) y  \nchib{ w } )_+,
\eeq
where we denote $a_+ \teq \max(a, 0)$. Since $W = \bw + w$ and $|| w ||_{\bcX} = || w \vp \rho ||_{L^{\infty}}$ \eqref{norm:X}, we have 
\beq\label{eq:non_est1:b}
|w \rho| \leq  \vp^{-1} \nchib{w}, \quad \B| \f{W}{\bw} \rho \B| \leq \B|\f{\bw + w}{\bw} \rho \B|
\leq
  \rho + \vp^{-1} |\bw|^{-1} || w||_{\bcX} .
\eeq
\eseq

Using the above pointwise bound for $V, W$, we estimate the error $\cE$ in \eqref{eq:N_err_recall} as 
\beq\label{eq:est_err}
  |\cE| \f{\rho}{\bw} \leq   \int_0^x \f{|\cR(\bw)| }{ 2 (\bv - \MCC_{\psio/x}(y) y  || w ||_{\bcX})_+ }
  \cdot  (\rho + \vp^{-1} |\bw|^{-1} || w||_{\bcX}) d y
  \teq B_{\cE}(x, || w ||_{\bcX}) .
\eeq
where all functions in the integrand evaluate at $y$.

\vs{0.1in}
\paragraph{\bf Estimate of $\cN$}

Next, we estimate $\cN$ in \eqref{eq:N_err_recall}.  Applying the two decompositions of $\td \cR(w)$ \eqref{eq:lin_nloc}:
\bseq\label{eq:est_cR}
\beq
\td \cR(w)(x) 
= - \f{2}{3} \psiox - 2 \f{  \psio }{x} \cdot \f{x \bw_x}{\bw}(x)
= ( - \f{2}{3} - 2 \f{x \bw_x}{\bw} ) \psiox
+ 2 \f{x \bw_x}{\bw}  ( \psiox - \f{ \psio }{x} ),
\eeq
applying Lemma \ref{lem:vel_est} to $\psiox, \f{ \psio }{x}, \psiox - \f{ \psio }{x}$, 
we bound $\td \cR(w)$ as
\beq\label{eq:est_cR:b}
\bga
  |\td \cR(w)(x) |  \leq B_{\cR}(x) || w ||_{\bcX}, \\
  \quad B_{\cR}(x) \teq \min\B( \f{2}{3} \MCC_{\psio_x}( x) 
 + 2 \B| x \f{\bw_x}{ \bw } \B| \MCC_{\psio/x}(  x ),
 \B| \f{2}{3} + 2 \f{x \bw_x}{\bw} \B| \MCC_{\psio_x}(  x) 
 + 2 \B| x \f{\bw_x}{ \bw } \B| \MCC_{\D}( x ) \B).
 \ega
\eeq
\eseq

Recall $\bar g(y) = \f{\bv(y)}{y}$ from \eqref{eq:nota1}. Using 
$W = \bw +w, V = \bv + \psio(w)$, we decompose 
\[
       \f{W}{2V \bw } - \f{1}{2\bv} 
   = \f{ \bw + w}{ 2 V \bw} - \f{1}{2 \bv}
   = \f{w}{2 V \bw} + \f{\bv- V}{2 \bv V}
   =  \f{w}{2 V \bw} - \f{ \psio  }{2 \bv V}
   = \f{1}{ 2 V} ( \f{w}{ \bw} - \f{ \psio }{ y \bar g } ) .
\]

Using \eqref{eq:non_est1} and the above decomposition, we estimate 
\[
\bal
  \B| \B(    \f{W}{2V \bw } - \f{1}{2\bv} \B) \rho(y) \B|
   & =  \B| \f{1}{ 2 V} ( \f{w}{ \bw} - \f{ \psio }{ y \bar g } ) \B| \rho
  \leq \f{1}{2|V|} (   \vp^{-1} \bw^{-1}   + \bar g^{-1} \MCC_{\psio/x}(y) \cdot \rho )|| w ||_{\bcX} \\
& \leq  
   \f{  \vp^{-1} \bw^{-1}   + \bar g^{-1} \MCC_{\psio/x}(y)\cdot \rho  }{ 2 ( \bv - \MCC_{\psio/x}(y) y || w ||_{\bcX} )_+ } \cdot || w ||_{\bcX}.
   \eal
\]
All the functions evaluate at $y$. Using the above estimates and the estimate of $\cR$ \eqref{eq:est_cR}, we prove 
\beq
|\cN \f{\rho}{\bw} |  \leq  \int_0^x \uds{ \eqref{eq:est_cR:b} }{ B_{\cR}(y) }
\cdot 
\f{  \vp(y)^{-1} |\bw(y)|^{-1}   + \bar g(y)^{-1} \MCC_{\psio/x}(y) \rho(y)  }{ 2 ( \bv(y) - \MCC_{\psio/x}(y) y || w ||_{\bcX} )_+ } 
d y
\cdot  
   || w ||_{\bcX}^2
   \teq  B_{\cN}(x, \| w \|_{\bcX})  \cdot    || w ||_{\bcX}^2. 
   \label{eq:est_non}
\eeq

Note that $B_{\cN}(x, || w ||_{\bcX}),
B_{\cE}(x, || w ||_{\bcX})$ only depend on the parameters $\vmu, \bb, \cI$ 
via the functions \eqref{eq:vel_const}, the profile 
$\bw, \bv$, and $|| w ||_{\bcX}$. Moreover, they are  increasing in $|| w||_{\bcX}$.

\subsection{Summary of the estimates}
Combining \eqref{eq:lin_est},\eqref{eq:est_err},and \eqref{eq:est_non}, we establish
\begin{align}
  |\cF(w) \vp \rho |  &\leq 
 |\cL + \cN + \cE| \vp \rho  \notag  \\
 & \leq     \vp  | \bw |  \B(  
  ( \uds{ \eqref{eq:est_I} }{ B_I}  + \uds{ \eqref{eq:est_II1} }{ B_{II, 1} } + \uds{ \eqref{eq:est_II3} }{ B_{II,3 }}
  + \uds{ \eqref{eq:est_II_III}  }{ B_{II,2, III} } ) || w ||_{\bcX}   
   + \uds{ \eqref{eq:est_non} }{ B_{\cN}(x, || w ||_{\bcX}) } \cdot  || w ||_{\bcX}^2
  + \uds{ \eqref{eq:est_err} }{  B_{\cE}(x, || w||_{\bcX} ) }    \B) \notag \\
 & \teq B_{\mf{tot}}(x, || w ||_{\bcX}) . 
  \label{eq:lin_tot} 
\end{align}
where the upper bound only depend on the parameters $\vmu, \bb, \cI$ via the functions \eqref{eq:vel_const}, the profile  $\bw, \bv$, and $|| w ||_{\bcX}$. 
The label below $B_{\bullet}$ indicates the equation in which  $B_{\bullet}$ is defined.
All of these functions are increasing in $\nchib{w}$. 

We verify the following lemma on the estimates with computer assistance and refer more details to Appendix \ref{app:proof}. 
In the left panel of Figure \ref{fig:fix_point_contraction}, we plot the \emph{rigorous} piecewise 
interval-arithmetic bounds for functions $\vp  | \bw | B_{\bullet}(x)$ in Lemma \ref{lem:close},  
relative to $\dfix$, over a large domain $[0, 10^{27}]$.

\begin{lem}[\bf Computer-assisted]\label{lem:close}
Let $\dfix = 3.5 \cdot 10^{-5} , \dfixa = 3.4 \cdot 10^{-5}$.
Let  $\vmu, \bb, x_1$ be the parameters chosen in \eqref{def:mu_b}, \eqref{def:x1},
and $\cI = \{ I_i\}_{i\geq 1}$ be the intervals chosen in \eqref{def:interval}. For any $x \geq 0$, the approximate profile $\bw, \bv$ satisfy 
\bseq\label{lem:profile}
\begin{gather}
  \MCC_{\psio/x}(\vmu, \bb,\cI, x) x \dfix \leq \tf{1}{2} \bv,  \label{eq:profile:a} \\
 \uds{ \eqref{eq:lin_tot} }{B_{\mf{tot}}(x, \dfix)} =
\vp  | \bw |  \B(  
  (  \uds{ \eqref{eq:est_I} } {B_I }+ \uds{ \eqref{eq:est_II1}} { B_{II, 1} } + \uds{ \eqref{eq:est_II3} }{ B_{II,3 } }
   + \uds{ \eqref{eq:est_II_III} } { B_{II,2, III} } ) \dfix  +\uds{ \eqref{eq:est_non}} { B_{\cN}(x, \dfix ) }  \dfix^2
  +\uds{ \eqref{eq:est_err} }{ B_{\cE}(x, \dfix ) }   \B) \leq \dfixa <  \dfix   .  \label{eq:profile:b}
\end{gather}
where the label below $B_{\bullet}$ indicates the equation in which  $B_{\bullet}$ is defined.
\eseq
\end{lem}

\begin{figure}[t]
\centering
\captionsetup{width=0.95\textwidth}

\begin{minipage}{0.54\textwidth}
\centering
\includegraphics[height=5.2cm]{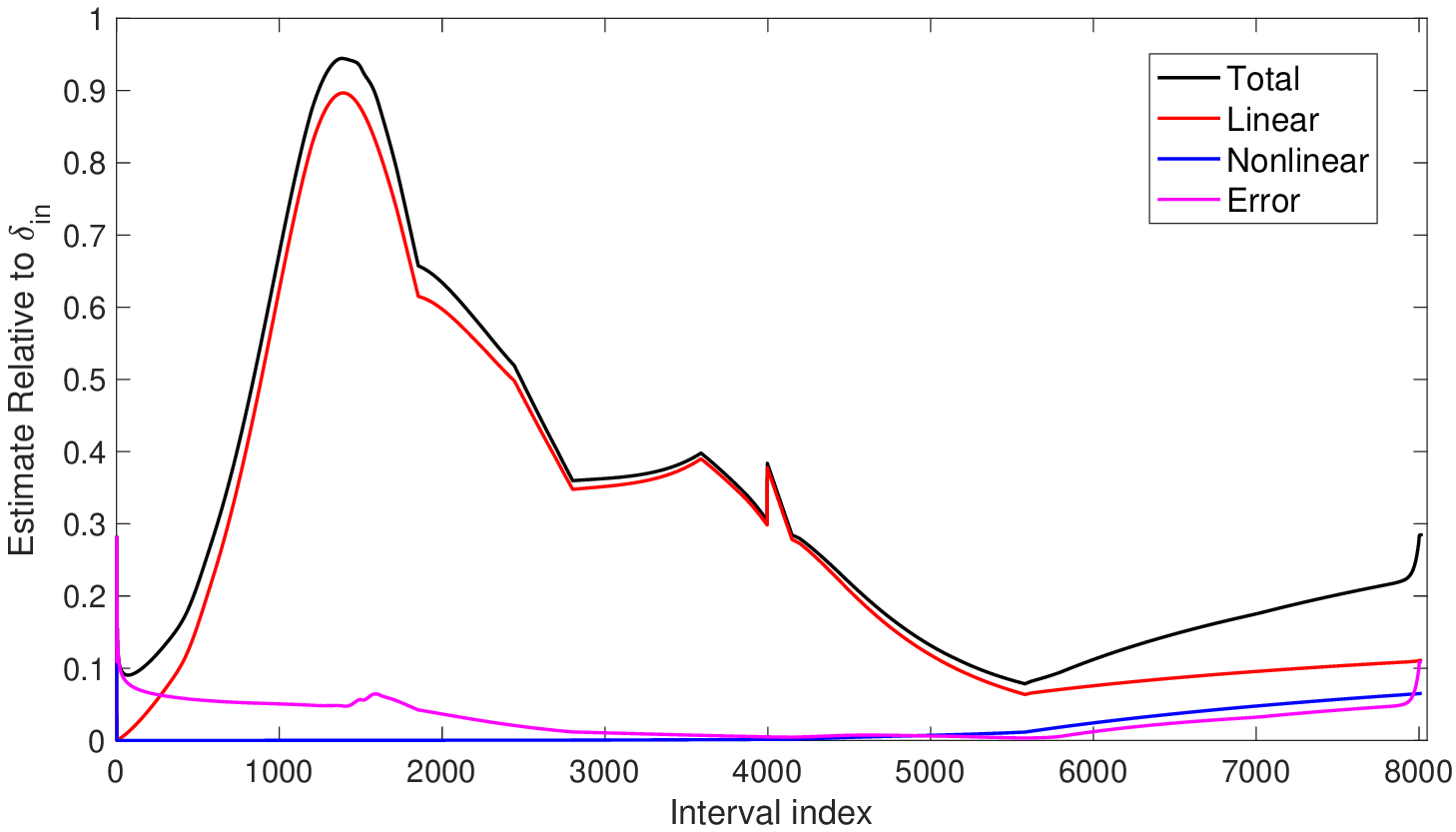}
\end{minipage}
\hfill
\begin{minipage}{0.44\textwidth}
\centering
\includegraphics[height=5.2cm]{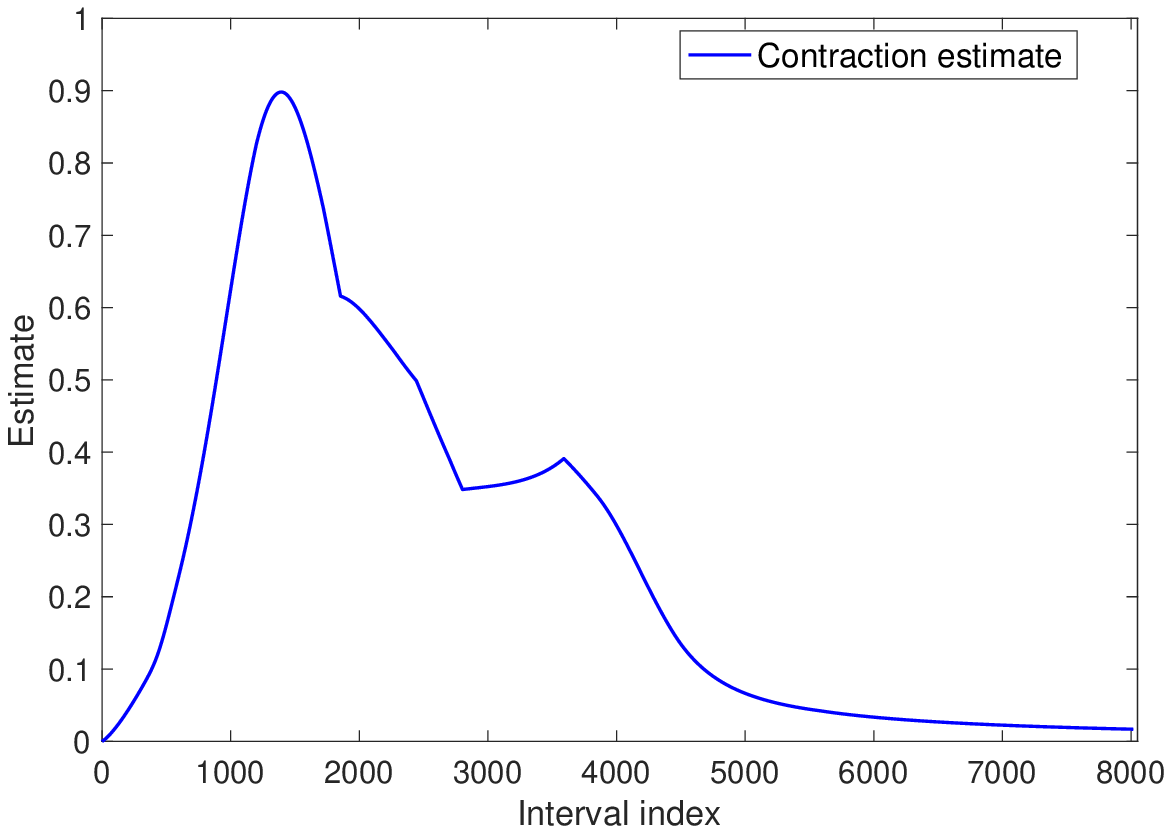}
\end{minipage}
\caption{  
Rigorous piecewise bounds over \(8000\) intervals covering \([0,10^{27}]\).
Left: estimates relative to \(\dfix\) for the fixed-point map 
in Lemma \ref{lem:close}, including  linear, nonlinear, and error bounds associated with
$\vp  | \bw | (B_I+B_{II,1}+B_{II,3}+B_{II,2,III})(x)$,  $\vp  | \bw | B_{\cN}(x,\dfix)\dfix$, \ $ \vp  | \bw | \dfix^{-1}  B_{\cE}(x,\dfix) .$ 
Right: estimate $B_\sst(x)$ for contration estimate in Lemma \ref{lem:profile2}.
}
\label{fig:fix_point_contraction}

\end{figure}

Using Lemma \ref{lem:close}, we are in a position to prove Theorem \ref{thm:onto}.

\begin{proof}[Proof of Theorem \ref{thm:onto}]

We fix an arbitrary $w$ with $|| w||_{\bcX}\leq \dfix$. Using \eqref{eq:non_est1} and \eqref{eq:profile:a}, we get 
\beq\label{eq:bound_V}
  V \geq \bv -   \MCC_{\psio/x}(\vmu, \bb, x) x \dfix \geq \bv / 2 \geq 0.
\eeq
Thus, $V(x) >0 $ for any $x > 0$. Since $B_{\cN}$ \eqref{eq:est_non}, $B_{\cE}$ \eqref{eq:est_err} are increasing in $|| w||_{\bcX}$, using estimate \eqref{eq:lin_tot} and \eqref{eq:profile:b} from Lemma \ref{lem:profile}, 
for any $w$ with $||w ||_{\bcX} \leq \dfix$, we obtain
\[
  \| \cF(w) \|_{\bcX} \leq B_{\mf{tot}}(x, \dfix) < \dfix,
\] 
and prove \eqref{eq:onto}. Since $\cL$ is linear, 
for any $w \in \bcX$, using \eqref{eq:lin_est} and \eqref{eq:profile:b}, we obtain
\[
\bal
  \| \cL(w) \|_{\bcX}
  \leq 
  \vp  | \bw | 
  (  B_{I}  + B_{II, 1} + B_{II,3 }
  + B_{II,2, III} ) || w ||_{\bcX} 
  \leq (\d_1/  \dfix) || w ||_{\bcX}  ,
  \eal
\]
with $\d_1/  \dfix < 1$ and prove \eqref{eq:lin_contra} with $\lam_{\alb} =  \dfixa/ \dfix $. 
We complete the proof of Theorem \ref{thm:onto}.
\end{proof}

\section{Qualitative properties of $\cF$ and a near-field contraction estimate}\label{sec:schauder}

In this section, we prove Theorem \ref{thm:schauder} using purely analytic arguments (pen-and-paper). Then we prove Theorem \ref{thm:near_field_stab} for a near-field contraction estimate of the linear operator 
around the \emph{exact} profile with computer-assistance.

The proof of Theorem \ref{thm:schauder} is divided into the following two lemmas.

\begin{lem}\label{lem:contin}
The map $\cF: \bD_{\dfix} \to \bD_{\dfix}$ is continuous with respect to the $\bcX$-norm.  
\end{lem}

\begin{lem}\label{lem:compact}
The map $\cF: \bDD \to \bD_{\dfix}$ is compact with respect to the $\bcX$-norm.
Moreover,  $\cF(w)(x)$ is locally Lipschitz and satisfies  $|\pa_x \cF(w)(x)| \les_R |x|^{\alb+1} $ for any $R >0$,
with constants independent of $w \in \bDD$.
\end{lem}

We prove Lemma \ref{lem:contin} in Section \ref{sec:contin} and Lemma \ref{lem:compact} in Section \ref{sec:compact} by purely analytic arguments.

Throughout this section, the implicit constants may depend on the profile $\bw,\bv$ and the parameters 
$\vmu, \bb$ in the norm \eqref{norm:X}, $x_1$ in \eqref{eq:rho1}, 
and the intervals $\cI$ chosen in \eqref{def:interval}. Since $b_1 \leq -1$, 
using Lemma \ref{lem:W_asym} and considering $x \leq 1$ and $x \geq 1$, we estimate the weight $\vp$ \eqref{norm:X} as
\beq\label{eq:wg_rat}
\f{\vp^{-1}}{|\bw| } 
\les 
  \f{  |x|^{-b_1} \we |x|^{-\alb}  }{ |x| \we |x|^{-\alb}}
  \les |x|^{-b_1 -1} \we 1 , 
   \quad 
  \vp |\bw|
  \asymp (|x|^{b_1} \vee |x|^{\alb} ) 
  ( |x| \we |x|^{ -\alb} )
\asymp |x|^{b_1 + 1} \vee 1.
\eeq

\subsection{Proof of continuity}\label{sec:contin}

In this section, we prove Lemma \ref{lem:contin} 
by proving  that $\cF$ is Lipschitz.

Consider two perturbation $w_1, w_2$ with $|| w_i||_{\bcX} \leq \dfix$. We adopt notations similar to \eqref{eq:nota1}
\beq\label{eq:nota_dif}
 V_i = \bv + \psio(w_i),
\quad W_i = \bw + w_i,  \quad  w_{\D} = w_1 - w_2.
\eeq

We shall estimate 
\beq\label{eq:dif_decomp}
\bal
\cF(w_1) - \cF(w_2) 
& = (\cL w_1 - \cL w_2) + (\cN(w_1) - \cN(w_2)) + (\cE(w_1) - \cE(w_2)) \\
& \teq \cL_{\D}(w_1, w_2) + \cN_{\D}(w_1, w_2) + \cE_{\D}(w_1, w_2).
\eal
\eeq

To simplify notation, we simplify $\cT_{\D}(w_1, w_2) = \cT_{\D}$ for $\cT = \cL, \cN, \cE$.

Using estimates \eqref{eq:non_est1} for $V, W_i$, \eqref{eq:wg_rat}, and estimate \eqref{eq:profile:a} from Lemma \ref{lem:profile}, we obtain 
\bseq\label{eq:Sch_est_VW} 
\begin{gather}
  V_i \geq \bv - \MCC_{\psio/x}(x) |x| \cdot|| w_i||_{\bcX} \geq 
  \bv - \MCC_{\psio/x}(x) |x| \dfix \geq \tf{1}{2} \bv, \\
  | \tf{W_i}{\bw} \rho| \les 
  |\rho| + \vp^{-1} \bw^{-1} || w_i ||_{\bcX}
  \les |\lgp x|^{\alb} + |x|^{-b_1 - 1} \we 1 \les |\lgp x|^{\alb},
\end{gather}
\eseq
where we have used $|| w_i||_{\bcX} \les 1$.

\vs{0.1in}

\paragraph{\bf Estimate of $\cL_{\D}$}
Since $\cL$ is linear, using \eqref{eq:lin_contra} in Theorem \ref{thm:onto}, we prove
\beq\label{eq:dif_lin}
\| \cL_{\D} \|_{\bcX} = \| \cL( w_1 - w_2) \|_{\bcX}
  \leq \| w_1 - w_2  \|_{\bcX}.
\eeq

To estimate $\cE_{\D}, \cN_{\D}$, we decompose a term appearing in both $\cE_{\D}$ and $\cN_{\D}$:
\bseq\label{eq:dif_rat}
\beq
  \cI_{\D}(w_1, w_2)
\teq \f{1}{2 V_1} \cdot  \f{W_1}{\bw} \rho 
  -   \f{1}{2 V_2} \cdot  \f{W_2}{\bw} \rho 
  =  \f{ (V_2 - V_1)}{2 V_1 V_2} \cdot  \f{W_1}{\bw} \rho 
  + \f{1}{2 V_2} \cdot \f{W_1 - W_2}{\bw} \rho.
\eeq

For any $w_i \in \bcX$ with $\| w_i \|_{\bcX} \leq \dfix$, applying estimate \eqref{eq:J_est} 
and Corollary \ref{cor:vel_est} to $V_2 - V_1 = \psio(w_1 - w_2) $, using $W_1 - W_2 = w_1 - w_2 = w_{\D}$,
estimates \eqref{eq:Sch_est_VW} for $\f{W_1}{\bw} \rho$ and $V_i$ we obtain 
\[
  |\cI_{\D}(w_1, w_2)(x)| \leq \B( \f{  x \MCC_{\psio/x} |\lgp x|^{ \alb }
} {2 |V_1 V_2|}
  + |\f{ 1 }{2 V_2} | \cdot \f{\vp^{-1}}{ | \bw | } \B) || w_{\D} ||_{\bcX}
  \les  \f{ 1 }{\bv}   \B( \f{  x \MCC_{\psio/x} |\lgp x|^{\alb}  } { \bv }
  +  \f{\vp^{-1}}{ |\bw | } \B) || w_{\D} ||_{\bcX} .
\]

Applying the estimates from Lemma \ref{lem:W_asym} to $\bw, \bv$, 
Corollary \ref{cor:vel_est} for $\MCC_{\psio/x}$, and the estimate \eqref{eq:wg_rat} for $(\vp \bw)^{-1}$, we obtain 
\beq
\bal
  |\cI_{\D}(w_1, w_2)(x)|  & \les   \f{  1 }{ x \lgp x }
  \cdot  \big( \f{  x \cdot (|x|^{\alb - b_1} \we |\lgp x|^{2\alb}  ) |\lgp x|^{\alb}  }{x \cdot \lgp x }
  + |x|^{-b_1 - 1} \we 1 
  \big) || w_{\D} ||_{\bcX} \\
  & \les  ( x \lgp x )^{-1} ( |x|^{\alb - b_1} \we 1 
  + |x|^{-b_1 - 1} \we 1 )  || w_{\D} ||_{\bcX}  \\
  & \les 
   ( x \lgp x )^{-1}   (|x|^{-b_1 - 1} \we 1)   || w_{\D} ||_{\bcX} .
\eal
\eeq
\eseq
where we recall $\lgp x \gtr 1$ from \eqref{def:lgp}.

\vs{0.1in}
\paragraph{\bf Estimate of $\cE_{\D}$}
Recall $\cE$ from \eqref{eq:F_lin}. 
Since $b_1 \leq -1$, applying \eqref{eq:dif_rat} 
for $\cI_{\D}$ and Lemma \ref{lem:W_asym} for $\cR(\bw)$, we estimate  the integrand in $\cE_{\D}$ as 
\beq\label{eq:dif_err_int}
  \bal
 I_{\cE, \D}(w_1, w_2)  & \teq  \f{\cR(\bw)}{2 V_1} \cdot  \f{W_1}{\bw} \rho 
  -   \f{\cR(\bw)}{2 V_2} \cdot  \f{W_2}{\bw} \rho 
  =  \cR(\bw) \cI_{\D} \, , \\
 |I_{\cE, \D}(w_1, w_2) | & \les (|x|^{1+\alb} \we  |\lgp x|^{- 2 \alb})   \cdot ( x \lgp x )^{-1} \cdot  (|x|^{-b_1 - 1} \we 1)   \nchib{ w_{\D} } \\
& \les ( |x|^{\alb -b_1 - 1} \we |x|^{-1} |\lgp x |^{-2 \alb - 1} )   \cdot  \nchib{ w_{\D} }.
  \eal
\eeq

In particular,  $I_{\cE, \D}$ is $L^1$-integrable. 
Using the above estimate, $b_1 \leq -1$ \eqref{def:mu_b},  and estimate of $\vp \bw$ \eqref{eq:wg_rat}, we obtain
\begin{align}
  |\cE_{\D} \rho \vp| &\les  \vp |\bw| \int_0^x I_{\cE, \D}(y) d y
  \les (|x|^{b_1+1} \vee 1) \int_0^x 
  |y|^{-b_1 - 1} \we |y|^{-1} |\lgp y |^{-2 \alb - 1}  d y || w ||_{\bcX} \notag \\
  & \les 
(|x|^{b_1+1} \vee 1)
  ( |x|^{-b_1 }  \we  1 ) || w ||_{\bcX}
  \les (|x| \we 1 ) || w ||_{\bcX} \les || w ||_{\bcX} . \label{eq:dif_err}
\end{align}

\paragraph{\bf Estimate of $\cN_{\D}$}

Recall $\cN$ from \eqref{eq:F_lin}. 
Since $\td \cR$ is linear, using the operator 
$\cI_{\D}(\cdot)$ in \eqref{eq:dif_rat}, we decompose the integrand in $\cN_{\D}$ as 
\beq\label{eq:dif_non_int}
\bal
  I_{\cN, \D} & \teq \td \cR(w_1) (  \f{W_1}{\bw} \cdot \f{1}{2 V_1} - \f{1}{2\bv} ) \rho
  -  \td \cR(w_2) (  \f{W_2}{\bw}  \cdot \f{1}{2 V_2} - \f{1}{2\bv} ) \rho \\
  & = (\td \cR(w_1) - \td \cR(w_2)) (  \f{W_1}{\bw} \cdot \f{1}{2 V_1} - \f{1}{2\bv} ) \rho
  + \td \cR(w_2) (  \f{W_1}{\bw} \cdot \f{1}{2 V_1}  -  \f{W_2}{\bw} \cdot \f{1}{2 V_2}   ) \rho  \\
& = \td \cR(w_{\D}) \cdot \cI_{\D}(w_1, 0) 
+ \td \cR(w_2) \cdot \cI_{\D}(w_1, w_2 )  .
\eal
\eeq

Since $\td \cR$ is linear, using estimate \eqref{eq:est_cR} for $\td \cR$,
 Corollary \ref{cor:vel_est} for the coefficients, and Lemma \ref{lem:W_asym} for the profile $\bw, \bv$, 
for $f = w_{\D} \in \bcX$ or $w_i \in \bcX$, we estimate 
\beq\label{eq:non_dif_est1}
\bal
   |\td \cR(f)|
   & \les \B( \B| 2 \f{x \bw_x}{\bw} + \f{2}{3} \B| \MCC_{\psio_x}(x)
   + \B| \f{x \bw_x}{\bw} \B| \MCC_{\D}(x) \B) \| f \|_{\bcX} \\
  & \les \B( (1 \we |\lgp x|^{-4/3} ) ( |x|^{\alb - b_1} \we  |\lgp x|^{2/3}  )
  + ( |x|^{\alb - b_1} \we  |\lgp x|^{-1/3}  ) \B)
  \| f \|_{\bcX} \\
  & \les ( |x|^{\alb - b_1} \we  |\lgp x|^{-1/3}  )
  \| f \|_{\bcX} . \\
\eal
\eeq

Note that  estimate \eqref{eq:dif_rat} holds for any $w_1, w_2 \in \bcX$ with $||w_i ||_{\bcX} \leq \dfix$. Applying \eqref{eq:nota_dif} with $w_2 = 0$, which implies $\psio(w_2) =0, w_{\D} = w_1, W_2 = \bw, V_2 = \bv$, we obtain 
\[
  | ( \f{W_1}{\bw} \cdot \f{1}{2 V_1} - \f{1}{2\bv} ) \rho |
  =  |\cI_{\D}(w_1, 0) |
  \les  ( x \lgp x )^{-1}   (|x|^{-b_1 - 1} \we 1)   || w_1 ||_{\bcX} . 
\]

Applying \eqref{eq:non_dif_est1} with $f = w_{\D}$ and $w_2$ 
and \eqref{eq:dif_rat} for $\cI_{\D}(w_1, w_2)$, we estimate $I_{\cN, \D}$  
in \eqref{eq:dif_non_int} as
\[
  |I_{\cN, \D}|
\les 
( |x|^{\alb - b_1} \we  |\lgp x|^{-1/3}  ) \cdot ( x \lgp x )^{-1}   (|x|^{-b_1 - 1} \we 1)
\cdot ( || w_{\D}||_{\bcX} || w_1 ||_{\bcX} 
+ || w_2||_{\bcX} || w_{\D}||_{\bcX}  ) .
\]

Since $||w_i ||_{\bcX} \les 1$, $1\les \lgp x $,
 and $b_1 \leq -1$ \eqref{def:mu_b},  we simplify the bound as 
\beq\label{eq:dif_non_int2}
    |I_{\cN, \D}| 
  \les ( |x|^{\alb - 2 b_1 - 2} \we x^{-1} |\lgp x|^{-4/3} )     || w_{\D}||_{\bcX} .
\eeq
In particular, $I_{\cN, \D}$ is $L^1$-integrable. Using the above estimate and 
estimate of $\vp \bw$ \eqref{eq:wg_rat},  we obtain
\beq\label{eq:dif_non}
\bal
  |\cN_{\D} \rho \vp|
  &\les \vp |\bw| \int_0^x |I_{\cN, \D}|
\les \vp |\bw| \int_0^x  ( |x|^{\alb - 2 b_1 - 2} \we x^{-1} |\lgp x|^{-4/3} )     || w_{\D}||_{\bcX}  \\
& \les ( |x|^{b_1+1} \vee 1 )
( |x|^{ \alb - 2 b_1 -1 }  \we 1)  || w_{\D}||_{\bcX}
\les ( |x|^{\alb- b_1} \we 1 )  || w_{\D}||_{\bcX}.
\eal
\eeq

Combing the estimates \eqref{eq:dif_lin}, \eqref{eq:dif_non}, \eqref{eq:dif_err}, we prove 
\[
||\cF(w_1) - \cF(w_2)||_{\bcX}
\leq
   ||( \cL_{\D} + \cN_{\D } + \cE_{\D} ) \rho \vp  ||_{L^{\infty}}
   \les || w_1 - w_2 ||_{\bcX}.
\]
In particular, $\cF$ is a continuous map from $\bD_{\dfix}$ to $\bD_{\dfix}$ with respect to 
the $\bcX$-norm.

\subsection{Proof of compactness}\label{sec:compact}

In this section, we prove Lemma \ref{lem:compact} on the compactness of $\cF$.

We adopt the decomposition in \eqref{eq:F_lin}
\[
  \cF(w) = \cL w + \cN(w) + \cE(w).
\]
To prove compactness,  we show that $\cF(w)$ is continuous 
near $\infty$ uniformly for $w \in \bDD$,  and $\cF(w)(x)$ is Lipschitz for 
$|x|\leq R$ with any $R>0$ by estimating the integrand in $\cF(w)$.

\vs{0.05in}

\paragraph{\bf Estimate of integrands in $\cL$} 
We fix $w$ with $|| w||_{\bcX} \leq \dfix <1$. 
We apply the decomposition of $\cL$ in Section \ref{sec:lin_dec}, \eqref{eq:lin1}, \eqref{eq:lin2}, \eqref{eq:lin} with $x_1 = 0$ so that $I =0$ and
\beq\label{eq:comp_lin_decom} 
\bal
  \f{\cL(w) \rho }{\bw} & =
  II_1 + (II_2 +III) + II_3
=  \cP_0(x) + \int_0^x I_{\cL}(y) d y , & \quad I_{\cL} &\teq \cP_1 + \cP_2 + \cP_3,   \\
  \cP_0(x) & \teq  -  \f{ \f{ \psio }{x} + 2 \al \cJ }{ 3 \bar g} \rho \B|_{0}^x, \ 
& \quad \cP_1(y)  & \teq \f{1}{3}(\f{ \psio }{y} + 2 \alb \cJ) (\f{\rho}{\bar g})_y ,  \\
\cP_2(y) & \teq  \big( \f{\pa_y \rho}{\rho}  + \f{2\alb y^{\alb} \bw }{3 \bv}  \big) 
   \f{w}{\bw} \rho, 
& \quad \cP_3(y) & \teq  \f{  - 2 \f{ \psio }{y} ( \f{1}{3} + \f{ y  \bw_y}{\bw}  )  }{2 \bv} \rho . 
\eal
\eeq

Recall $\bar g = \f{1}{x} \bv = 1 + \f{1}{x} \psio(\bw) $ from 
\eqref{eq:nota1}.  Since $ |\bw | \les \min( |x|, |x|^{-\alb} )$ by Lemma \ref{lem:W_asym}, using 
\eqref{eq:coe_asymp:0} in Corollary \ref{cor:vel_est} with $(\mu_1, \mu_2, b_1 , b_2) = (1,1, -1, 1/3)$
and triangle inequality,   we obtain
\[
\bal
& |\pa_x \bar g(x) |  = \f{1}{x} |\pa_x \psio(\bw) - \f{1}{x} \psio(\bw)| 
\les \f{1}{x} \cdot  \min( |x|^{\alb + 1},1) \nlinf{ \bw (|x|^{-1} + |x|^{\alb}) } \les \min( |x|^{\alb}, \ang x^{-1} ) .
\eal
\]
Using  $\rho$ from \eqref{eq:rho1} and $x_1 > 1$ \eqref{def:x1}, the above estimate and 
$\bar g \gtr \lgp x$ from \eqref{eq:decay:V}, we obtain
\begin{gather}\label{eq:rho_g}
|\f{\rho}{\bar g} | \les |\lgp x|^{-2/3}  ,   \\
\B| \pa_x (\f{\rho}{\bar g}) \B|
= \B| \f{\pa_x \rho \cdot \bar g - \rho \pa_x \bar g}{\bar g^2} \B|
\les  \f{  x^{-1} |\lgp x |^{\alb - 1 + 1} \one_{x \geq 1} 
+ ( |x|^{\alb} \we x^{-1}) |\lgp  x|^{\alb}
    }{  |\lgp x |^2 }
  \les 1 \we x^{-1} |\lgp x|^{\alb -2} ,  \notag
\end{gather}

Using Corollary \ref{cor:vel_est} for nonlocal terms, $\alb =\f13$, the above estimate, and $||w||_{\bcX} \les 1$, we estimate
\bseq\label{eq:est_lin_int}
\beq
\bal
  |\cP_1| & \les (|x|^{\alb - b_1} \we  1 )
  \cdot 
   (1 \we x^{-1} |\lgp x|^{\alb -2} )  || w||_{\bcX} \\
   & \les |x|^{ \alb - b_1 } \we  x^{-1} |\lgp x|^{\alb-2} 
   \les |x|^{\alb - b_1 -1} \we  x^{-1} |\lgp x|^{-1 - \alb}. 
\eal
\eeq

Applying Corollary \ref{cor:vel_est} for nonlocal terms and Lemma \ref{lem:W_asym} for the profile, we estimate $ \cP_3$ as 
\beq
  |\cP_3| \les   \f{ ( |x|^{\alb - b_1} \we |\lgp x|^{2\alb}  )  |\lgp x |^{-4/3}}{ x \lgp x }  \cdot |\lgp x|^{\alb} || w||_{\bcX}
  \les |x|^{\alb - b_1 - 1} \we x^{-1} |\lgp|^{-1 - \alb}.
 \eeq

For $\cP_2$, using Lemma \ref{lem:W_asym} for $\bv$, \eqref{eq:W_rep} for $\bw$, \eqref{eq:rho1}, and $| |\log  x | - |\lgp x| | \les x^{-1}$, we obtain 
\[
   \f{\pa_x \rho}{\rho} + \f{2\alb x^{\alb} \bw}{ 3 \bv }
   = \one_{x \geq x_1} \f{1}{3} \cdot \f{1}{x} |\log x|^{- 1}
   - \f{2}{9} \cdot \f{6}{4 x \lgp x} + O( x^{-1}  |\lgp x |^{-1 - \alb} )
   =  O( x^{-1} |\lgp x |^{-1 - \alb} ).
\]
Moreover, using \eqref{eq:wg_rat} for  $ \f{\vp^{-1}}{\bw}$ and $||w ||_{\bcX} \les 1$, we have 
\[
     |\f{\pa_x \rho}{\rho} + \f{2\alb x^{\alb} \bw}{ 3 \bv }|
\les (1 \we x^{-1})
+   \f{x^{1+\alb}}{x \lgp x}
 \les x^{\alb},
 \quad |\f{w}{\bw } \rho|
 \les \f{\vp^{-1}}{|\bw|} || w ||_{\bcX}
 \les  |x|^{-b_1 - 1}  \we 1.
\]

Using the above estimates and \eqref{eq:wg_rat} for  $ \f{\vp^{-1}}{\bw}$, we obtain
\beq
\bal
  |\cP_2| & 
  \les   ( |x|^{\alb} \we  x^{-1} |\lgp x |^{-1 - \alb} ) \cdot (|x|^{-b_1 - 1}  \we 1) 
   \les  |x|^{\alb -b_1 -1}  \we  x^{-1} |\lgp x |^{-1 - \alb} .
  \eal
\eeq
\eseq

Applying the estimates \eqref{eq:est_lin_int} to \eqref{eq:comp_lin_decom}, 
and using $b_1 \leq -1$, 
we establish 
\begin{align}\label{eq:est_lin_int2}
|I_{\cL}| = | \cP_1 + \cP_2 + \cP_3|
   \les |x|^{\alb -b_1 -1}  \we  x^{-1} |\lgp x |^{-1 - \alb}
   \les |x|^{\alb }  \we  x^{-1} |\lgp x |^{-1 - \alb}.
\end{align}

For $\cP_0$, applying \eqref{eq:coe_asymp:0} in Corollary \ref{cor:vel_est},
$\bar g \asymp \lgp x$ from Lemma \ref{lem:W_asym},  and \eqref{eq:wg_rat}, we obtain
\bseq\label{eq:est_lin_int3}
\beq
 |\cP_0  \vp \bw|  \les  \f{ ( |x|^{\alb - b_1} \we 1 ) }{ \lgp x   } \cdot |\lgp x|^{\alb}
 \cdot (|x|^{b_1 + 1} \vee 1 ) \les 
 |x|^{\alb + 1 } \we |\lgp x|^{\alb - 1 } .
\eeq

Using $\| w \|_{\bcX} \les 1, |w| \les |x| \cdot \| w \|_{\bcX}  \les |x|$, and Corollary \ref{cor:vel_est}, we obtain
\[
\bal
& |\f{\psio(w)}{x} + 2 \alb \cJ |   \les
|x|^{\alb+1} \| w \|_{\bcX}  \les |x|^{\alb+1} ,
\  | \pa_x (\f{ \psio(w)}{x} + 2 \alb \cJ  ) |
\les  \f{1}{x}   \B| \pa_x \psio - \f{\psio }{x}  \B| + |w x^{\alb-1}| 
\les |x|^{\alb} \nchib{w} \les |x|^{\alb} .
\eal
\]

 Using the above estimate and \eqref{eq:rho_g}, 
for $|x| \leq R$ with any $R > 0$, we obtain 
\beq
  |\pa_x ( \cP_0 )|  \les  |\pa_x ( \f{\psio}{x} + 2 \alb \cJ )| \cdot |\f{\rho}{\bar g}|
+ \B| ( \f{\psio}{x} + 2 \alb \cJ ) \cdot \pa_x ( \f{\rho}{\bar g} ) \B|  \\
\leq C(R) |x|^{\alb}.
\eeq

\eseq

\paragraph{\bf Estimate of integrands in $\cN, \cE$} 
We apply the estimates in Section \ref{sec:contin} with $ w_2= 0$. 
We fix an arbitrary  $w_1$ with $||w_1 ||_{\bcX} \leq \dfix < 1$. Using the notations \eqref{eq:nota1} and \eqref{eq:nota_dif}, we obtain 
\[
  W_2 = \bw, \quad V_2 = \bv, \quad w_2 = 0, \quad w_{\D} = w_1 - w_2 = w_1.
\]

Since $w_2 = 0$ and $\td \cR(w_2) = 0$, using \eqref{eq:dif_non_int} and \eqref{eq:dif_err_int} with $w_2 =0$, we obtain
\[
\bal
I_{\cE}(w_1) & \teq  \f{\cR(\bw)}{2 V_1} \cdot  \f{W_1}{\bw} \rho 
= \f{\cR(\bw)}{2 V_1} \cdot  \f{W_1}{\bw} \rho  
  -   \f{\cR(\bw)}{2 \bv} \cdot  \f{\bw}{\bw} \rho 
+  \f{\cR(\bw)}{2 \bv} \cdot  \f{\bw}{\bw} \rho  = I_{\cE, \D}(w_1, 0)
 +  \f{\cR(\bw)}{2 \bv} \rho
   , \\
I_{\cN}(w_1) & \teq \td \cR(w_1) (  \f{W_1}{\bw} \cdot \f{1}{2 V_1} - \f{1}{2\bv} ) \rho
= I_{\cN,\D}(w_1, 0),  \\
\eal
\]

Using \eqref{eq:dif_err_int}, \eqref{eq:dif_non_int2}, $|| w_1 ||_{\D} \les 1$, $b_1 \leq -1$,
and $\alb = \f{1}{3}$, we obtain
\[
\bal
|I_{\cE, \D}(w_1, 0)| & \les ( |x|^{\alb-b_1 -1} \we |x|^{-1} |\lgp x |^{- 2 \alb - 1}  ) || w_1 ||_{\bcX} 
\les |x|^{\alb} \we |x|^{-1} |\lgp x |^{- 1-\alb} ,
  \\
  |I_{\cN}(w_1)| & \les  ( |x|^{\alb - 2 b_1 - 2} \we x^{-1} |\lgp x|^{-4/3} )  || w_1||_{\bcX}
  \les |x|^{\alb} \we x^{-1} |\lgp x|^{-1 - \alb}.
  \eal
\]

Applying Lemma \ref{lem:W_asym} to $\bar \cR(\bw), \bv$ and \eqref{eq:rho1} to $\rho$, we obtain 
\[
  |\f{\cR(\bw)}{2 \bv} \rho|
  \les \f{|x|^{1+\alb} \we |\lgp x|^{-2 \alb}}{ x \lgp x } \cdot |\lgp x|^{\alb}
\les |x|^{\alb} \we  x^{-1} |\lgp x|^{-1 - \alb}.
\]

\paragraph{\bf Summary of the estimates}
 Combining the above estimates and \eqref{eq:est_lin_int2} for $I_{\cL}$, we obtain 
\beq\label{eq:compact_N_int}
\bal
I_{\cF}(w_1) & \teq I_{\cL}(w_1) + I_{\cE}(w_1) + I_{\cN}(w_1), \\
 |I_{\cF}(w_1)|
  &  \leq 
     |\tf{\cR(\bw)}{2 \bv} \rho| + |I_{\cL}| +| I_{\cE, \D}| + |I_{\cN}| 
     \les   |x|^{\alb} \we  x^{-1} |\lgp x|^{-1 - \alb},
\eal
\eeq
for any $ w_1$ with $|| w_1||_{\bcX} \leq \dfix$. Recall $\cN$ and $\cE$ from \eqref{eq:F_lin} and $\cL(w)$ from  $\f{\bw}{\rho} \times$
\eqref{eq:comp_lin_decom} 
\[
  \cL(w) = \f{\bw}{\rho} \B( \cP_0(w)  +\int_0^{x} I_{\cL}(w, y) d y \B),
  \ 
  \cN(w) = \f{\bw}{\rho} \int_0^x I_{\cN}( w, y) d y,
  \  \cE(w) =  \f{\bw}{\rho} \int_0^x I_{\cE}( w, y) d y.
\]

We introduce the weighted function
\beq\label{def:comp_cT}
  \cT(w)(x) = \cF(w)  \vp \rho =  (\cL(w) + \cN(w) + \cE(w)) \vp \rho
  = \vp \bw \B( \cP_0(w) +  \int_0^{x} I_{\cF}( w, y) d y \B).
\eeq

Recall that the weight $\vp$ in \eqref{norm:X} is piecewise smooth, $\vp = \mu_1 |x|^{b_1}$ for $|x| \les 1$, and is Lipschitz away from $0$ . Since  $|\bw| \les |x|$ by Lemma \ref{lem:W_asym}, 
for any $w \in \bDD$, combining the estimates \eqref{eq:est_lin_int3}, 
\eqref{eq:compact_N_int}, 
and using $ b_1 + 1 + \alb > 0$ by \eqref{def:mu_b}, we obtain
\beq\label{eq:comp_Lip_cT}
\bal
 |\pa_x \cT(w)(x)|
 & \les |\pa_x (\vp \bw) | \B( | \cP_0(w) |  +  \int_0^{x}  | I_{\cF}( w, y) | d y \B)
 + |\vp \bw | \cdot \B( |\pa_x \cP_0(w) | + | I_{\cF}( w, x) | \B) \\
 &\les_R |x|^{b_1} \cdot |x|^{\alb + 1}  + |x|^{  b_1 + 1} |x|^{\alb} \les_R |x|^{\alb + 1 +b_1} 
 \les_R 1.  
 \eal
\eeq
for any $R >0$. Since  $\cF(w) \rho$ is less singular than than $\cT(w)(x)$ near $0$,
similarly, we have 
\beq\label{eq:comp_Lip_cF}
  |\pa_x \cF(w)| = \B| \pa_x \B( \bw \B( \cP_0(w) +  \int_0^{x} I_{\cF}( w, y) d y \B) \B)  \B|
  \les_R |x|^{\alb + 1}.
\eeq

\paragraph{\bf Uniformly continuous near $ \infty$ }

From Lemma \ref{lem:W_asym}, the definition of $\vp$ \eqref{norm:X} and \eqref{def:mu_b}, 
and $\bw$ \eqref{eq:W_rep}, for $x > C$ with $C$ sufficiently large, we have 
\[
\vp = \mu_\nmu^{-1} |x|^{\alb}, \quad    \tf{\bw}{\rho}  \cdot \vp \rho = \bw \vp =  - 6 \mu_\nmu^{-1} + O( |\log x |^{-\alb} ) .
\]

Fix an arbitrary $w \in \bDD$.
Recall the formula of $\cT(w)$ from \eqref{def:comp_cT}. 
Since $\cT(w)(x)$ is continuous for $x\in \R$, 
$I_{\cF}$ are $L^1$-integrable by \eqref{eq:compact_N_int}, and $\cP_0 \bw \vp \to 0$ as $x \to \infty$ by \eqref{eq:est_lin_int3},  the limit $\cT(w)(\infty) =\lim_{x \to \infty } \cT(w)(x)$ exists. 
Using \eqref{def:comp_cT}, \eqref{eq:est_lin_int3},
and \eqref{eq:compact_N_int}, for large $ x>  C$, we have 
\beq\label{eq:cT_far}
\bal
  |\cT(w)(\infty) - \cT(w)(x)|  
& \les | \bw \vp \cP_0(x)| + |\bw \vp( x ) - \bw \vp(\infty)|   \B| \int_0^{\infty} I_{\cF}
   \B|
+ |\bw \vp( x ) | \cdot \B| \int_x^{\infty} I_{\cF} \B| \\
& \les  |\lgp x |^{-\alb}
+ \int_x^{\infty} y^{-1} |\lgp y|^{-1-\alb}  d y. \les  |\lgp x |^{-\alb} .
\eal
\eeq

\paragraph{\bf Compactness of $\cT(\cdot)$}

We can view $\cT(w)$ is an odd function in $x$. For any $R>0$, from estimate \eqref{eq:comp_Lip_cT}, $\cT(w)$ is equicontinuous on $[0, R]$, uniformly for $w \in \bD_{\dfix}$. From Theorem \ref{thm:onto}, for any $w$ with $||w||_{\bcX} \leq \dfix$, we obtain 
$| \cT(w)(x) | \leq \dfix$.  Now, we fix a sequence $\{ w_i \}_{i \geq 0}$ with $||w_i||_{\bcX} \leq \dfix$. 
For fixed $n$, applying Arzela-Ascoli theoren, up to extracting a subsequence, we obtain that 
$ \{\cT(w_i)\} $ is a Cauchy-Sequence in $L^{\infty}([0, n])$.
Using a diagonal argument,  up to extracting a subsequence, we obtain that 
$ \cT(w_i) $ is a Cauchy-Sequence in $L^{\infty}([0, n])$, for any $n \geq 1$. 
Since the limit $\cT(w_i)(\infty)$ exists for each $w_i$ by \eqref{eq:cT_far}
and $ |\cT(w_i)(\infty)| \les 1$, we can further assume that $\cT(w_i)(\infty)$ converges. 

Next, we show that $\{\cT(w_i)\}$ is a Cauchy-Sequence in $L^{\infty}([0, \infty))$. We fix a small $\e > 0$. Using \eqref{eq:cT_far}, by taking $N$ large enough, we obtain 
\[
\bal
   |\cT( w_i )(x) - \cT(w_i)(\infty)| 
   & \leq C |\lgp x|^{-\alb} 
   \leq C |\log N|^{-\alb} < \tf{1}{3} \e, \quad && \forall \ x \geq  \tf12 N.
\eal 
\]

Since $\{  \cT(w_i)\}$ is a Cauchy-sequence in $L^{\infty}( [0, N] )$, taking $M > N$ large enough and using the above estimates, for any $ i, j \geq M$, we obtain
\[
  |\cT(w_i)(x) - \cT(w_j)(x)| < \tf{2 \e}{3} +   ||\cT(w_i)(x) - \cT(w_j)(x)||_{L^{\infty}( [0, M] )}  < \e.
\]

Since $\e$ is arbitrary, we prove that  $\{\cT(w_i)\}$ is a Cauchy-Sequence in $L^{\infty}(\R)$.
Since $L^{\infty}(\R)$ is complete, up to extracting a subsequence, we obtain that $\cT(w_i) \to f_* \in L^{\infty}$. Thus, we prove that $\cF: \bD_{\dfix} \to \bD_{\dfix}$  is compact with respect to the $\bcX$-norm. From \eqref{eq:comp_Lip_cF}, we obtain that $\cF(w)(x)$ is locally Lipschitz and satisfies estimate \eqref{eq:comp_Lip_cF}.
We prove Lemma \ref{lem:compact}.

\subsection{Existence of the profile}\label{sec:exist_profile}

Using Theorem \ref{thm:fix_point}, we have constructed a fixed point $w_{\alb}$ to the map $\cF$ \eqref{eq:fix_map} with $ \nchi{w_{\alb}} \leq \dfix$. As a result, we construct an exact self-similar profile 
$( \wwwa, \vvva)$ to \eqref{eq:1D_dyn} in the form:
\beq\label{eq:profile_alb}
 \wwwa = \bw + w_{\alb}, \quad  \vvva = V( \wwwa) = x + \psio( \wwwa)
 = \bv + \psio(w_{\alb}).
\eeq

Recall the weight $\vp \gtr |x|^{-1} + |x|^{\alb}$ from \eqref{norm:X} and 
\eqref{def:mu_b} for the parameters.  From Theorem \ref{thm:fix_point}, we obtain 
\beq\label{eq:profile_waaa1}
\bal
   |\wwwa(x)| \leq |\bw| + |w_{\alb}|  \leq |\bw| + \vp^{-1} \nchi{ w_{\alb}}
   \les \min(|x|, |x|^{-\alb})  .
\eal 
\eeq
 
 Applying \eqref{eq:vel_est} in Lemmas \ref{lem:vel_est},  using \eqref{eq:profile:a}, 
 and \eqref{eq:decay:V} in Lemma \ref{lem:W_asym}, we obtain
\beq\label{eq:profile_vvva1}
\bal
| \psio( w_{\alb})|  & \leq \MCC_{\psio/x}(\vmu, \bb, \cI, x) x \nchi{w_{\alb}} \leq \MCC_{\psio/x}(\vmu, \bb, \cI, x) x \dfix, \\
 \vvva & \geq \bv -  \MCC_{\psio/x}(\vmu, \bb, \cI, x) x \dfix \geq \bv / 2 \gtr x \lgp x .
\eal
\eeq

 We define a linear operator around the exact profile which is similar to \eqref{eq:F_lin} with $\rho \equiv 1$ 
\beq\label{eq:lin_near:def}
\bga
    \cL_{\wwwa}( f) \teq \wwwa \int_0^x  \f{  \crbb(f)(y) }{ 2 \vvva(y) }  d y  , 
    \quad 
    \crbb(f) \teq - \f{2}{3} \psio_x(f) - 2 \psio(f) \f{\pa_x \wwwa}{\wwwa} ,
    \quad \psio(f) = \cK_{\alb, 1}(f) .
  \ega
\eeq

\subsection{Near-field contraction estimate of the linearized operator}\label{sec:near_stable}

We choose the parameter $\bb$ \eqref{def:mu_b} 
and parameters $\vmu_\sst, \bb_{\sst}$ in \eqref{def:mu_b_st}
with $\mu_{\sst, i} >0, \bb_{\sst, i} \in [-\f32, 0]$, and define the weight 
\beq\label{def:vpa}
  \vpa(x) = \max_{i} ( \mu_{\sst, i}^{-1} |x|^{\bb_{\sst, i}} ), 
  \quad  \bb_{\sst,1} = b_1 \in [ -\tf32, \, 0].
\eeq
The subscript $\sst$ stands for 
near-field since $\vpa$ does not grow as $\vp$ in $\bcX$-norm \eqref{norm:X}. 

We have the following near-field contraction estimate for $   \cL_{\wwwa}$ similar to \eqref{eq:lin_contra} in Theorem  \ref{thm:onto}.

\begin{thm}\label{thm:near_field_stab}

Let $\vpa$ be the weight defined in \eqref{def:vpa} with parameters $ \vmu_\sst, \bb_\sst$ given 
in \eqref{def:mu_b_st}.  For any $x \geq 0$ and any $w \vpa \in L^{\infty}$,  we have the following estimates 
\beq\label{eq:lin_near:cR_est}
\vpa(\xx)  |\cL_{\wwwa}(w)(x) | =  \vpa(\xx) |\wwwa(\xx) \B| \int_0^x \f{   \crbb(w)(\yy)   }{2 \vvva(\yy) } d \yy \B| \leq \lamst
   \| w \vpa \|_{L^{\infty}} ,  \quad
   \lamst \teq 0.95.
\eeq
\end{thm}

We refer to the above contraction estimate, valid for \emph{all} $x$, as the \emph{near-field contraction estimate}, since the weighted norm $\nlinf{w\vpa}$ does not control $w$ sharply for large $x$ compared with the $\bcX$-norm \eqref{norm:X}. The proof is similar to, but much simpler than, that of Theorem \ref{thm:onto}. The key reason is 
that since $\vvva \gtr x \log x$, given any weight $\vpa \asymp |x|^{ \bb_{\sst, 1}} + |x|^{\bb_{\sst,3}}$ with some $\bb_{\sst, 1}
\in [-\f32, -1], \bb_{\sst, 3} \leq 0$ (see \eqref{def:mu_b_st}),  using  \eqref{eq:vel_est:a} in Lemma \ref{lem:vel_est}, one  obtains the boundedness of $ \cL_{\wwwa}$ in the weighted space
\begin{align}\label{eq:near_field_stab_idea}
\vpa | \cL_{\wwwa}| &  \les \vpa |\wwwa| \int_0^x \f{ \min( |y|^{\alb-  \bb_{\sst,1}}, |y|^{\alb- \bb_{\sst, 3}} ) }{ y \lgp y } d y
\nlinf{ w \vpa}  \\
& \les ( |x|^{ \bb_{\sst,1}} + |x|^{ \bb_{\sst,3} } ) \min( |x|, |x|^{-\alb} ) \cdot \min( |x|^{\alb -  \bb_{\sst,1} }, |x|^{\alb- \bb_{\sst, 3} } ) \nlinf{ w \vpa }  \les  \min( |x|^{1 + \alb} , 1 )  \nlinf{ w \vpa} . \notag
\end{align}

The integrand becomes non-integrable as $a_2 \to \alb$, which is the key difficulty in the proof of Theorem \ref{thm:onto}.
Here, we do not need to design a special weight $\rho$ and exploit delicate cancellation 
for faster decay estimates as those in the proof of Theorem \ref{thm:onto}. 
Since we choose $\bb_{\sst,3}\leq 0$ in  \eqref{def:mu_b_st}, \eqref{eq:near_field_stab_idea}, we obtain much better estimates for large $y$ due to the $|\lgp y|^{-1}$ factor.

\begin{proof}[Proof of Theorem \ref{thm:near_field_stab}]

We fix any $w$ with $w \vpa \in L^{\infty}$. Applying integration by parts to $\pa_x \psio(w)$ (similar that for $I$-term in Section \ref{sec:lin_est},
with $\bw, \bv, x_1 \we x$ replaced by $\wwwa, \vvva, x$), we obtain
\[
 P\teq \int_0^x  \f{  \crbb(w)(y) }{ 2 \vvva(y) }  d y = 
\int_0^x  \f{  - \f{2}{3} \psioy(w) - 2 \psio(w) \f{\pa_y \wwwa}{\wwwa}   }{ 2 \vvva(y) }  d y  
= 
- \f{1}{3} \f{\psio(w)(x)}{ \vvva(x) }  - \int_0^x \big( \f{ y \pa_y \vvva}{3 \vvva}  + \f{y \pa_y \wwwa}{ \wwwa} \big)  \cdot \f{\psio(w)}{ y \vvva } d y  .
\]
where the boundary term vanishes at $x=0$ due to Corollary \ref{cor:vel_est}. Since $\wwwa, \vvva$ solves the profile equation \eqref{eq:1D_dyn} with $\al = \alb$ exactly, we obtain 
\[
  \f{ y \pa_y \wwwa}{\wwwa} = \f{ (3 - \alb -(1-\alb) \pa_y \vvva) y}{2 \vvva}
  = \f{ y ( \f43 - \f13 \pa_y \vvva) }{\vvva},
  \ \Rightarrow \  \f{ y \pa_y \vvva}{3 \vvva} +   \f{ y \pa_y \wwwa}{\wwwa}  = \f{4}{3} \cdot \f{y}{\vvva}.
\]

Using the above estimates, we obtain 
\[
P = - \f{1}{3} \f{\psio(w)(x)}{ \vvva(x) }  - \f43 \int_0^x \f{  \psio(w)}{ \vvva^2}(y) d y ,
\ \Rightarrow \ \cL_{\wwwa}(w) = - \wwwa\B( \f43 \int_0^x \f{  \psio(w)}{ \vvva^2}(y) d y + \f{1}{3} \f{\psio(w)(x)}{ \vvva(x)} \B).
\]

Applying estimate \eqref{eq:vel_est:a} in Lemma \ref{lem:vel_est} to  $\psio(w)$ with parameters 
 $\vmu_{\sst}, \bb_{\sst}, \cI$, we obtain 
 \footnote{
We remark that $w$ only belongs to $w \vpa \in L^{\infty}$ in Theorem \ref{thm:near_field_stab}, which is much weaker than  the $\bcX$-norm in \eqref{norm:X}.
} 
\beq\label{eq:linW_est1}
\bal
|\psio( w)(x) | & \leq \MCC_{\psio/x}( \vmu_\sst, \bb_\sst , \cI,  x) x   \nlinf{w \vpa} .
\eal
\eeq

Since $ \wwwa = \bw + w_{\alb}$ and $\nchi{w_{\alb}} = \nlinf{ w_{\alb} \vp \rho} \leq \dfix$, we obtain
\[
     |\wwwa| \leq |\bw| + \vp^{-1} \rho^{-1} \nlinf{ w_{\alb} \vp } \leq |\bw| + \dfix \vp^{-1} \rho^{-1}.
\]

From  \eqref{eq:lin_near:def}, we obtain $\cL_{\wwwa}(w)(x) = (\wwwa P)(x)$. 
Applying the estimate of $\vvva$ in \eqref{eq:profile_vvva1} (see \eqref{eq:profile_alb} for relation between $\bv$ and $\vvva$), the above upper bounds for $\wwwa$
and $\psio(w)$, we bound 
\footnote{
The parameters $\bb_\sst, \vmu_\sst$ 
in the numerators in \eqref{eq:linW_est} for the perturbation 
\emph{differ} $\bb, \vmu$ for the perturbation of the profile in the denominator in \eqref{eq:linW_est}.
}
\begin{align}\label{eq:linW_est}
&  |\vpa \cL_{\wwwa}(w)(x)| 
=  |\vpa\wwwa P(x)  |
\leq \vpa ( |\bw| + \dfix \vp^{-1} \rho^{-1}  )  |P(x)| \notag \\
 & \leq \vpa  ( |\bw| + \dfix \vp^{-1} \rho^{-1}) 
 \B( \f{1}{3} \cdot \f{  \MCC_{\psio/x}( \vmu_\sst, \bb_\sst , \cI,  x)  x}{
  \bv(x) -  \MCC_{\psio/x}(\vmu, \bb, \cI, x) x \dfix   }
  + \f{4}{3} \int_0^x  \f{  \MCC_{\psio/x}( \vmu_\sst, \bb_\sst , \cI,  y) y  }{ (  \bv(y) -  \MCC_{\psio/x}(\vmu, \bb, \cI, y) y \dfix  )_+^2 } d y \B)    \nlinf{w \vpa} \notag \\
& \teq B_{\sst}(x)     \nlinf{w \vpa} ,
\end{align}
where the function $B_{\sst}(x)$ only depends on the parameters $\vmu_\sst, \bb_\sst,  \vmu, \bb$ 
and the approximate profile $\bv, \bw$. 
Since  $\dfix$ is very small, the $\dfix$-term is treated perturbatively.
In Lemma \ref{lem:profile2}, we verify $\sup_x  B_\sst(x) < 1$. Applying Lemma \ref{lem:profile2} and \eqref{eq:linW_est}, we complete the proof. 
\end{proof}

\begin{lem}[\bf Computer-assisted]\label{lem:profile2}
Let $B_\sst(x)$ be the function defined in \eqref{eq:linW_est}
and $\dfix$ be as in Lemma \ref{lem:profile}. The approximate profile satisfies 
\[
  \sup\nolimits_{x \geq 0} B_\sst(x)\leq \lamst, \quad \lamst = 0.95.
\]
\end{lem}

The proof follows the same strategy and methods as those for Lemma \ref{lem:close}. We refer more details to Appendix \ref{app:proof_near_stab}. In the right panel of Figure \ref{fig:fix_point_contraction}, we plot the \emph{rigorous} piecewise interval-arithmetic
 bounds  for $B_\sst(x)$ over a large domain $[0, 10^{27}]$.

\section{Analytic $C^{\infty}$-regularity and sharp decay estimates of the profile}\label{sec:C_inf}

In this section, we prove higher-order regularity 
for the profile equation  \eqref{eq:1D_dyn} with general $\al$, and then use it together with Lemma~\ref{lem:W_asym} to establish the smoothness and decay of the $\tf13$-profile $\wwwa$ and $\vvva$ in Section \ref{sec:smooth_alb}. The proofs are purely analytic (pen-and-paper).

To distinguish nonlocal terms associated with Biot-Savart law with different $\al$, 
in this section, we denote $\psio_{\al}(f) = \cK_{\al, i}(f)$ \eqref{eq:ker}.  In particular, $\psio_{\alb}( f) $ is the same as $\psio$ used in 
Sections \ref{sec:ansatz}-\ref{sec:near_stable}.

We introduce the weighted norm that captures the decay of the derivatives 
{\small
\beq\label{norm:W}
  \| f \|_{\cW_{l}^k} \teq \sum\nolim_{0\leq i\leq k} \nlinf{ \la x \ra^{i+ l} \pa_x^i f },
  \quad \cW_l^{\infty} \teq \cap_{k\geq 0} \cW_l^k.
\eeq
}

We have the following relation between the smoothness of $w$ and $w/x$ for odd function $w$.

\begin{lem}\label{lem:norm_W}
For any $w$ odd in $x$, $k\geq 0$ and $l \in \R$, we have 
$
  || w||_{ \cW_{ l }^k } \les_k    || \f{w}{x} ||_{ \cW_{ l + 1}^k }.$
\end{lem}

\begin{proof}
The proof follows from estimating $ \pa_x^k w = \pa_x^k (\f{w}{x} \cdot x)$ using the Leibnize rule and \eqref{norm:W}
\end{proof}

We have the following higher order regularity estimates for the profile to \eqref{eq:1D_dyn} with 
general $\al$.

\begin{thm}[\bf $C^{\infty}$ regularity]\label{thm:reg}
Let $\e = \f13 - \al$ and $\al \in [ \f13 - \f{1}{100}, \f13]$. 
Suppose that $W$ is odd, locally Lipschitz, and solves the profile equation
\beq\label{eq:profile_eqn_recall}
 2 V \pa_x W = (3- \al - (1-\al) V_x) W, \quad V = x + \psioa( W), \quad \psioa(W) = \cK_{\al, 1}(W),
\eeq
as in \eqref{eq:1D_normal:c}, and  it satisfies the following estimates  
\bseq\label{eq:Cinfty_ass}
\beq
\bga
   |W(x)| \leq \s \min( |x|, \, |x|^{ - \g} ), \\ 
   |\pa_x \psioa(W)| + | \tf{1}{x} \psioa(W)  |  \leq \s \min( |x|, \cJe(x) ) , 
  \quad 1 + \tf{1}{x} \psioa(W) \geq \s^{-1} \cJe(x), 
\ega
\eeq
for any $x$, for some constant $\s> 0$  and $ \g \in [\tf13,  \ \tf13 + \tf{1}{100}]$, 
where $\cJe(x)$ is the function given by
\beq
  \quad     \cJe(x) \teq \min( \lgp x , \e^{-1} ) ,
\eeq
\eseq
and we have $\ang {\cJ_0}(x) = \lgp x$. Then the  profile satisfies $ W  \in \cW_{\g}^{\infty} \subset C^{\infty}$ and 
\bseq\label{eq:profile_Ck}
\begin{align}
  |\pa_x^k ( \tf{1}{x} W ) | \les_{k, \s} \la x \ra^{-k - \gam - 1} ,
  \quad  |\pa_x^k W| & \les_{k, \s} \ang x^{-k -\gam} , \qquad   \forall \, k \geq 0 ,  \label{eq:profile_Ck:a}
   \\
   |\pa_x^{k+1} \psioa |   & \les_{k, \s}\la x \ra^{-k + \al -\gam} ,
   \quad \forall \, k \geq 1 . \label{eq:profile_Ck:b}
\end{align}
\eseq
We remark that the implicit constants are \emph{independent of $\al, \e$}.

\end{thm}

We first establish high order nonlocal estimates in Section \ref{sec:high_nonlocal} 
and then prove  Theorem \ref{thm:reg} in Section \ref{sec:thm_reg_pf}.

\subsection{High order nonlocal estimates}\label{sec:high_nonlocal}
We have the following higher order estimates.

\begin{lem}\label{lem:vel_high}
 Suppose that $w$ is odd, $\al \in [\f13 - \f{1}{100}, \f13],  \g \in [\f13, \f13 + \f{1}{100} ]$.  Let $\psio_{\al} = \cK_{\al, 1}(w)$ be defined in \eqref{eq:ker:a}. For any $k\geq 1$, we have 
 \bseq\label{eq:vel_high} 
\begin{align} 
 |\pa_x^{k+1} \psioa |   & \les_k \la x \ra^{-k + \al -\gam}  || w ||_{\cWg^k}
   \les_k \la x \ra^{- k + \al -\gam}  \| \tf{w}{x} \|_{\cW_{\g+1}^k},  \label{eq:vel_high:a}  \\ 
    |\pa_x^{k+1} \psio - \pa_x^{k+1} \psio(0)|
    & \les_k |x| \cdot  \| \tf{w}{x} \|_{\cW_{\g+1}^k}. \label{eq:vel_high:b}
\end{align}
\eseq
\end{lem}

\begin{proof}

Below, implicit constants may depend on $k$, and we do not track this dependence to simplify notation. Since $\al$ and $\g$ are restricted to a fixed range, it is straightforward to verify that the constants below are independent of $\al$ and $\g$.

Recall $\psio$ from \eqref{eq:ker}. 
Since $\pa_x^{k+1} ( x y^{\al-1}) = 0$ for any $k\geq 1$, we obtain
\[
      \pa_x^{k+1} \psioa(x) =  \int K(x-y) \pa_y^k w(y) d y, \quad K(z) \teq - \al |z|^{\al - 1} \sgn(z) .
\]

For $|x| \leq 1$, since $ k \geq 1$, $\al \in [\f13 - \f{1}{100}, \f13]$, and $ \g \in [\f13, \f13 + \f{1}{100} ]$, we obtain 
\[
 |\pa_x^{k+1} \psioa(x)| \les_k || w ||_{\cWg^k}  \int |x- y|^{\al-1} \la y \ra^{-1 - \gam} d y 
\les || w ||_{\cWg^k}.
\]

Next, we fix $ |x| \geq 1$. We choose a smooth cutoff function $\chi$ with $\chi(z) = 1$ for $|z| \leq 1$ and $\chi(z) = 0$ for $|z| \geq 2$. We decompose 
\[
\bga
  \pa_x^{k + 1} \psioa
 =  \int (\chi_1(y) + \chi_2(y))  K(x-y)  \pa_y^k w( y ) d y \teq I_1 + I_2,  \quad 
  \chi_1(y) \teq \chi(  \f{4( y - x) }{|x| }  ),
\quad
  \chi_2 \teq 1 - \chi_1(y). 
  \ega
\]

The integrand in $I_1$ is supported in $y \in \{ |y-x| \leq |x|/2 \}$, where we have $\la y \ra \asymp \la x \ra$. We estimate 
\[
  |I_1| \les \int_{|y- x| \leq |x|/2 } |x - y|^{\al-1} \la y \ra^{-k - \gam } d y
  || w ||_{\cWg^k}
\les \la x \ra^{ -k - \gam } \int_{|y- x| \leq |x|/2 } |x - y|^{\al-1}  d y
  || w ||_{\cWg^k} \les \la x \ra^{-k + \al -\g}  || w ||_{\cWg^k}.
\]

 For $ i \geq 1$, we have 
\beq\label{eq:chi2_est}
  \supp( \pa_y^i \chi_2(y )) \subset \{ y: |x|/4 \leq |y-x| \leq |x| / 2 \},
  \quad  |\pa_y^i \chi_2(y)| \les |x|^{-i}.
  \eeq

Since the integrand in $I_2$ is supported away from $y=x$, 
 applying integration by parts and using the Leibniz rule, we estimate $I_2$ as 
\[
\bal
|I_2| & \les \int  |\pa_y^k ( \chi_2(y) |x-y|^{\al-1} \sgn(x-y)  ) | \cdot |w(y)| d y \\
& \les \sum_{i+j = k} \int_{|x-y| \geq |x|/4} 
|\pa_y^i \chi_2(y)| |x-y|^{\al-1- j}  \la y \ra^{-\g} d y  || w ||_{\cWg^k}
\teq \sum_{i+j =k } II_{i, j} \cdot || w ||_{\cWg^k}
\eal
\]

If $i \geq 1$, using \eqref{eq:chi2_est}, we have $| y |, | x-y | \asymp | x |$ for $y \in \supp(\pa_y^i \chi_2)  $, and hence
\[
  II_{i,j} \les \int_{ |y-x| \leq |x|/2 } |x|^{\al - 1 - j} \cdot |x|^{-i} \la x \ra^{- \g} d y
  \les \la x \ra^{-1 - i - j + 1 + \al - \g} \les \la x \ra^{-k + \al -\g}.
\]

If $ i =0$, we obtain $ j = k \geq 1$. We estimate $II_{0, k}$ as 
\[
\bal
  II_{0, k} &\les \int_{|x - y| \geq |x|/4} |x-y|^{\al - 1- k} \la y \ra^{- \g} d y
  & = ( \int_{ |x-y| \geq |x|/4, |y| \leq 2 |x| } + \int_{ |y| \geq 2 |x| } ) |x-y|^{\al - 1- k} \la y \ra^{- \gam} d y . 
  \eal
\]

Since for $|y| \geq 2|x|$, we have $|x-y| \asymp |y|$. Using 
$\al -1 - k-\g \leq -2$, we obtain 
\[
  II_{0, k} \les |x|^{\al-1-k} \int_{|y| \leq 2 |x|} \la y \ra^{- \gam}
  + \int_{|y| \geq 2 |x|} |y|^{\al - 1 - k - \gam} d y \les
|x|^{\al - 1-k} |x|^{1- \gam} + |x|^{-k + \al -\gam} \les |x|^{-k + \al -\gam}.
  \]
Combining the above estimates and using Lemma \ref{lem:norm_W}, we prove
\eqref{eq:vel_high:a} in Lemma \ref{lem:vel_high}.

\vs{0.1in}
\paragraph{\bf Proof of \eqref{eq:vel_high:b}}
Using estimate \eqref{eq:vel_high:a} and symmetry, we only need to consider $x \in [ 0, 1]$. Denote $ A = \f{w}{y}$. Firstly, we use $w= A \cdot y$ to rewrite $\pa_x^{k+1} \psioa$  :
\beq\label{eq:vel_high_pf2}
\bal
   \pa_x^{k+1} \psioa 
   & =  - \pa_x \int |x-y|^{\al} \pa_y^k (A  y )  d y
   =  - \pa_x \int |x-y|^{\al} ( y \pa_y^k A  + k \pa_y^{k-1} A  ) d y\\
   & = - \int \big( \pa_x |x-y|^{\al} \cdot y \pa_y^k A + k |x-y|^{\al} \pa_y^k A \big) d y 
   = \int \td K(x, y) \pa_y^k A d y,
   \eal 
\eeq
where $\td K$ is defined as 
\[
      \td K(x, y) = - ( \al \cdot \sgn(x- y)|x-y|^{\al-1} y  + k  |x-y|^{\al} ) .
\]

Since $w$ is odd, we obtain that $A$ is even. We further derive
\[
    \pa_x^{k+1} \psioa(x)  - 
      \pa_x^{k+1} \psioa(0)  = \int  \mr{K}(x, y) \pa_y^k A(y) d y, 
      \quad  \mr{K}(x, y)= \td K(x, y)  - \td K(0, y) .
\]
Note that $\td K(x, y)$ is $\al$-homogeneous. Using a direct calculation 
and Mean-value theorem, we yield
\[
  |K_{\mf{sym}}(x, y)| \les  (|x - y|^{\al - 1} + |y|^{\al-1}) x , \quad \forall | y| \leq 2 |x|,  \quad 
  |K_{\mf{sym}}(x, y)| \les x |y|^{\al - 1}, \quad |y| \geq 2 x .
\]

Since $k \geq 1$, $|x|\leq 1$, and $\al - \g \leq 0$, we obtain 
\[
  |  \pa_x^{k+1} \psioa(x)  - 
      \pa_x^{k+1} \psioa(0) | 
      \les || A||_{\cWg^k} \B( \int_{|y|\leq 2 |x| }  ( |x-y|^{\al-1} + |y|^{\al-1}) |x| d y
      + \int_{y \geq 2 x} x y^{\al -1} \la y \ra^{- \gam - 1} d y \B) 
    \les |x| \cdot || A||_{\cWg^k}.
\]
Since $A(x) = w/ x$, we complete the proof.
\end{proof}

\begin{lem}\label{lem:psi_wg}
Let $\psioa = \cK_{\al, 1}(w)$ be defined in \eqref{eq:ker}. Suppose $w$ is odd and $\al \in [\f13 - \f{1}{100}, \f13],  \g \in [\f13, \f13 + \f{1}{100} ]$. For any $k \geq 0$, we have 
\beq\label{eq:psi_wg_van}
\B|\pa_x^k ( \f{ \pa_x \psioa}{x} ) \B| +  |\pa_x^{k} (\f{\psioa}{x^2}) | \les_k 
\B| \B| \f{w}{x} \B|\B|_{\cW_{\gam+1}^k} .
\eeq

\end{lem}

\begin{proof}

For $k = 0$, using the definition of $\cK_{\al, i}$ in \eqref{eq:ker},
\[
  |w(x)| \les |x| \ang x^{-\gam-1} \| w / x \|_{ \cW_{\gam + 1}^0 } 
  \les |x|^{1-\al}\| w / x \|_{ \cW_{\gam + 1}^0 } ,
\]
the fact that $\cK_{\al,2}$ is $\al-1$-homogeneous by \eqref{eq:ker_scale}, 
and changing $z = y / x$, we estimate
\[
|\pa_x \psioa| \les \int_0^{\infty} \cK_{\al, 2}(x, y) w d y
\les \| \f{w}{x} \|_{\cW_{\gam+1}^0}   \int_0^{\infty} |\cK_{\al, 2}(x, y) |
 \cdot |y|^{1-\al} d y
\les  \| \f{w}{x} \|_{\cW_{\gam+1}^0} |x| \cdot \int_0^{\infty}|\cK_{\al, 2}(1, z)| z^{1-\al} d y.
\]
By the decay estimate of $K_{\al, 2}$ in \eqref{eq:K_decay_basic:b} in Lemma \eqref{lem:K_decay_basic},
the above integral is bounded. We prove the estimate for $\pa_x \psioa $ with $k=0$ in 
\eqref{eq:psi_wg_van}.  The proof for $\psioa$ in \eqref{eq:psi_wg_van} with $k=0$  is similar.

Next, we consider $ k \geq 1$ and focus on the estimate for $\psio_x$, as the proof for $\psio$ is similar. 

By definition of $\psio$ in \eqref{eq:ker}, we obtain $\psio(0) = \psio_x(0) = 0$.  We rewrite 
\[
 \pa_x^k (\f{ \pa_x \psioa}{x}) 
  = \pa_x^k ( \f{1}{x} I ),
  \quad 
  I \teq  \psio_{\al, x} - \sum\nolimits_{0\leq i \leq k} \f{1}{i!} \pa_x^{i+1} \psio_{\al}(0) x^i  .
\]
The higher order Taylor coefficient vanishes since 
$ \pa_x^k x^{i-1} = 0$ for $ 1 \leq i \leq k$. Using the integral formula of the Taylor expansion for $\psio_x$, we obtain
\[
  I = \int_0^x  \f{   \pa_x^{k+1} \psio(y) }{(k-1)!}  (x- y)^{k-1}  d y 
  - \f{1}{k!} \pa_x^{k+1} \psio(0) x^{k}
  = \int_0^x f(y) \f{ (x-y)^{k-1} }{(k-1)!} d y ,
  \]
where $f(y)$ is defined as 
\[
  f(y) \teq     \pa_x^{k+1} \psio(y) - \pa_x^{k+1} \psio(0)   .
\]

Next, we estimate $\pa_x^j I$ for any $ j \leq k$:
\[
  |\pa_x^j I(x) | 
  \les \one_{j \geq 1} | f(x)  ( \pa_x^{j-1}(x - y)^{k-1} )|_{y = x} |
  + \int_0^x | f(y) \pa_x^j (x - y)^{k-1} |  d y ,
\]
where the first term comes from the derivative acting on the integral $\int_0^x$.

Since $k \geq 1$, applying $|f(x)| \les_k |x| \cdot || w / x ||_{\cW^k_{\g + 1}}$ from Lemma \ref{lem:vel_high} and $j \leq k$, we obtain 
\[
    |\pa_x^j I(x) |  \les_k |x|^{1 + k-j} \cdot || \tf{w}{x}  ||_{\cW^k_{\gam + 1}}.
\]

Applying the Leibniz rule and the above estimate, we prove
\[
  |\pa_x^k \f{I}{x}|
  \les_k \sum\nolimits_{i+j =k} |\pa_x^i x^{-1} \cdot \pa_x^j I|
\les_k  \sum\nolimits_{i+j =k} | x^{-1-i} \cdot x^{k+1 - j} | \cdot \B\| \f{w}{x} \B\|_{\cW^k_{ \gam + 1}}
\les_k  \B\| \f{w}{x} \B\|_{\cW^k_{ \gam + 1}},
\]
and establish the desired estimate for $\pa_x \psioa$. The proof for $\psioa$ is similar.
\end{proof}

\subsection{Proof of Theorem \ref{thm:reg}}\label{sec:thm_reg_pf}
Now, we are ready to prove Theorem \ref{thm:reg}. 
 Denote $W = w + \bw$ and  $\psio = \cK_{\al,1}(W)$ the modified stream function associated with $W$ via \eqref{eq:ker:a}  and $V = x + \psio(W)$ \eqref{eq:ker}. 
 Due to symmetry, it suffices to prove the estimates \eqref{eq:profile_Ck} for $x \geq 0$.
Multiplying \eqref{eq:1D_dyn} with $\f{1}{x}$ , we obtain 
\bseq\label{eq:solu_induc}
\beq
\bga
    \pa_x ( \f{W}{x}) =  \f{ (3-\al) - (1-\al)  V_x - 2\f{V}{x} }{2 V} \cdot \f{W}{x}
    = \f{ - (1-\al) \psio_x -  2\f{\psio}{x} }{ 2 x (1 + \f{ \psio}{x}) }  \cdot \f{W}{x}.
  \ega
\eeq
We rewrite it as 
\beq
\pa_x Z = \f{H(W)}{G(W)} \cdot Z, \quad Z(x) \teq \f{W}{x}, \quad
     H(W) \teq  - (1-\al) \f{\psio_x}{x} - 2\f{\psio}{x^2},
    \quad G(W) \teq 2 ( 1 + \f{\psio}{x} ).
  \eeq
  \eseq

Recall $\lgp x$ from \eqref{def:lgp}. From the assumption \eqref{eq:Cinfty_ass}, we obtain
\beq\label{eq:induc_pf0:b}
   G(W)(x) \asymp \cJe(x), \quad 
   |\psio_x |, \ | \tf{1}{x} \psio | \les  |x| \we  \cJe(x).
\eeq

Since assumption \eqref{eq:Cinfty_ass} implies 
$ |\ang x^{\gam+1} W / x | \les 1$, we obtain $\| W / x \|_{ \cW_{\gam + 1}^0 } \les 1$. Next, we prove $W / x \in \cW_{\gam+1}^k$ using induction on $k \geq 0$.

\vs{0.1in}
\paragraph{\bf Inductive step $k\geq 0$} 
 Suppose that $Z \in \cW_{\gam+1}^k$ for $k\geq 0$ and $|| Z ||_{\cW^k_{\gam+1} } \les_k 1$. We shall prove $Z \in \cW_{\gam+1}^{k+1}$. 
Taking $\pa_x^k$ on \eqref{eq:solu_induc} and using Leibiniz rule, we have
\beq\label{eq:induc_goal}
  |\pa_x^{k+1} Z|
  \les \tts{\sum}_{i+j + l  = k} |\pa_x^i H| \cdot  |\pa_x^j (G^{-1}) | \cdot |\pa_x^l Z |
\eeq

For any $j\geq 0$, using chain rule for $G^{-1}$,  we derive
\beq\label{eq:induc_Ginv}
  |\pa_x^j G^{-1}| 
\les_k \one_{j=0} G^{-1} + \one_{j\geq 1} \sum\nolim_{ 1\leq n \leq j } G^{-(n+1)}  \prod\nolim_{ \sum_{  m \leq n } a_m = j, \ a_m \geq 1 } |\pa_x^{a_m} G| .
\eeq

For $|x| \leq 1$ and $0\leq s \leq k$,  since $ i \leq k$, 
 applying Lemma \ref{lem:psi_wg}, we estimate  $H$ and $\f{\psio}{x}$  as
\bseq\label{eq:induc_pf1}
\beq
 |\pa_x^i H| \les_k || Z||_{\cW_{\gam+1}^k} \les_k 1,  
 \quad |\pa_x^s \f{\psio}{x}|
 = |\pa_x^s  (\f{\psio}{x^2} \cdot x) | 
 \les_k |x  \pa_x^s (\f{\psio}{x^2} ) |
 + | \pa_x^{s-1}( \f{\psio}{x^2}  ) |
  \les_k || Z||_{\cW_{\gam +1}^k} \les_k 1.
\eeq

For $|x| \leq 1$, since $j\leq k$, applying the above estimate for $\psio / x$ and 
\eqref{eq:induc_pf0:b} for $G$, we estimate
\beq
    |\pa_x^j G^{-1}|  \les_k 
    \one_{j=0} + \one_{j\geq 1} \sum\nolim_{ 1\leq n \leq j }   \prod\nolim_{ \sum_{  m \leq n } a_m = j, \ a_m \geq 1 } 1  \les_k 1.    
\eeq
\eseq

For any $x \geq 1$, $ s \leq k$,  $p \in \{ 0, 1\}$ and $q \geq 1$, we apply Lemma \ref{lem:vel_high} 
for $\pa_x^n \psioa, n \geq 2$ and \eqref{eq:induc_pf0:b} 
for $\pa_x \psioa, \psioa$ to estimate 
\beq\label{eq:induc_pf2}
\bal
  | \pa_x^s ( \pa_x^p \psio \cdot x^{-q}) |
  & \les_{p, q} \tts{\sum}_{ m \leq s } |\pa_x^{m+p} \psio \cdot \pa_x^{s-m} x^{-q} | 
  \les_{p, q}  \tts{\sum}_{m\leq p} \cJe(x) \cdot |x|^{1-m-p} \cdot |x|^{-q - s +m} \\
  & \les_{p, q} \cJe(x) \cdot |x|^{1-p-q - s} ,
  \eal
\eeq
where we lose a large factor $\cJe(x)$ in the estimate for $\psioa, \pa_x \psioa$.

For $x \geq 1$ and $s \leq k$,  since $i\leq k$, applying \eqref{eq:induc_pf2} with $(p, q) = (1, 1)$
for $ \f{1}{x}\pa_x \psioa $, $(0, 2)$ for $\f{1}{x^2 } \psioa$, and $(0, 1)$ for $\f{1}{x} \psioa $
we estimate 
\bseq\label{eq:induc_pf3}
\beq
  |\pa_x^i H | \les_k |\pa_x^i ( \f{\psio_x }{x} ) | + 
  |\pa_x^i ( \f{\psio}{x^2} ) | \les_k \cJe(x) \cdot x^{- 1 - i},
  \quad  |\pa_x^s G | \les_k |\pa_x^s  \f{\psio_x }{x}  |
  \les_k \cJe(x) \cdot x^{- s}.
\eeq

Since $0\leq j , a_m\leq k$, applying the above estimates of $\pa_x^s G$ and \eqref{eq:induc_pf0:b} of $G$ to 
\eqref{eq:induc_Ginv}, we yield 
\beq
\bal
  |\pa^j G^{-1}|
& \les_k \one_{j=1} |\cJe(x)|^{-1}
+ \one_{j\geq 1} \sum_{n\leq j} | \cJe(x)|^{-(n+1)} 
 \prod_{ \sum_{  m \leq n } a_m = j, \ a_m \geq 1 }
 \cJe(x) \cdot x^{-a_m} \\
& \les_k 
\one_{j=1} |\cJe(x)|^{-1} 
+  \sum\nolim_{n\leq j} |\cJe(x)|^{-(n+1)}   \cdot | \cJe(x) |^{n} x^{- j}  
\les x^{-j} | \cJe(x)|^{-1}.
\eal
\eeq
\eseq

Combining estimates \eqref{eq:induc_pf1}, \eqref{eq:induc_pf3}, for $i, j \geq 0$, we obtain 
\beq\label{eq:induc_pf4}
  |\pa_x^i H \cdot \pa_x^j G^{-1}|
  \les \one_{|x| \leq 1} + \one_{|x| \geq 1} x^{-1 - i -j} 
  \les \la x \ra^{-1-j-i}.
\eeq

Since $l \leq k$, we obtain $|\pa_x^l Z| \les \la x \ra^{-l- \gam-1} || Z||_{\cW_{\gam +1}^k}$. Thus, applying 
\eqref{eq:induc_pf4} to \eqref{eq:induc_goal}, we prove 
\[
  |\pa_x^{k+1} Z | \les_k \tts{\sum}_{i+j+l = k} \, \la x \ra^{-1-i -j} \cdot \la x \ra^{-l- \gam -1} || Z||_{\cW_{\gam + 1}^k} \les_k \la x \ra^{-k-1- \gam - 1} ,
\]
which implies $Z \in \cW_{\gam + 1}^{k+1}$. By induction, we prove $Z = \f{1}{x} W \in \cW_{\gam + 1}^{k+1}$ for any $ k \geq 0$.  Using Lemma \ref{lem:norm_W}, we prove \eqref{eq:profile_Ck}.
Moreover, using estimates \eqref{eq:vel_high:a} for $ \psio_{\al}$, we prove 
\eqref{eq:profile_Ck}. We complete the proof of Theorem \ref{thm:reg}.

\subsection{Smoothness and sharp decay of the $\f13$-profile}\label{sec:smooth_alb}

In this section, we establish the \(C^\infty\) smoothness of the profile and prove the sharp decay estimates and logarithmic cancellations, which are crucial for later construction of profiles with \(\alpha < \tfrac13\).

\begin{thm}\label{thm:reg_alb}

Let $\wwwa$ be the odd self-similar profile constructed in Theorem \ref{thm:fix_point} 
and  $\vvva = x +  \psio_{\alb}(\bw) ,\psio_{\alb}(\bw)= \cK_{\alb, 1}(\bw)$ be the 
associated velocity. We have 
\beq\label{eq:bw_sign}
\pa_x \wwwa(0) = \pa_x \bw(0) \in [-1 - 10^{-6}, -1 + 10^{-6}], \quad \wwwa \leq 0 ,  \ \forall \, x \geq 0.
\eeq

\begin{enumerate}[label=(\roman*), leftmargin=1.6em]

\item \textsl{$C^{\infty}$-regularity}: We have $\wwwa, \vvva \in C^{\infty}$. 
Moreover, for any $k \geq 0$, 
\beq\label{eq:bw_smooth}
  |\pa_x^k \wwwa | \les_k \ang x^{-k -\alb},  \quad |\pa_x^{k+2} \psio_{\alb}| \les  \ang x^{-k-1}  .
\eeq

\item 
\textsl{Asymptotics}: 
For $x\geq 0$, the profile satisfies the following asymptotics
\bseq\label{eq:bw_est}
\beq\label{eq:bw_est:a}
\bal
| \wwwa| & \asymp \min( |x|, \ang x^{-\alb}), & \quad   | \wwwa + 6 x^{-1/3} | & \les x^{-1/3} | \lgp x|^{-1/3},  \\
  \vvva & \asymp x \lgp x  , &  \quad  |  \pa_x \vvva - \tf{1}{x} \vvva - 4| & \les  |\lgp x|^{-1/3} ,
  \eal
\eeq
and logarithmic cancellations:
\beq\label{eq:bw_est:b}
    |\f{\pa_x \bw}{\bw}| \les x^{-1}, 
    \quad     |\f{x \pa_x \bw}{\bw} + \f13|  \les |\lgp x|^{-4/3}, 
    \quad  | \pa_x  ( \f{ x \pa_x \bw }{\bw} ) | \les  \ang x^{-1} |\lgp x|^{-1}.
\eeq
\eseq

\end{enumerate}

\end{thm}

\begin{proof}

Recall from Theorem \ref{thm:fix_point} the locally Lipschitz solution $\wwwa = \bw + w_{\alb}$ to 
the profile equation  \eqref{eq:1D_dyn}.
We apply Theorem \ref{thm:reg} to prove $\wwwa \in C^{\infty}$. From \eqref{eq:profile_waaa1}, \eqref{eq:profile_vvva1}, 
we obtain
\[
   |\wwwa(x)| \les \min(|x|, |x|^{-\alb}). 
   \quad  1 +  \tf{1}{x}  \psioa(\wwwa )  = \tf{1}{x} \vvva  \gtr \lgp x .
\]

Using the definition of $\cJ_{\alb}(\wwwa)$ \eqref{eq:Jw_alb} and the above estimate for $\wwwa$, we bound 
\[
|\cJ_{\alb}(\wwwa)(x)  | \les \int_0^x y^{\alb -1 } (|y| \we |y|^{-\alb} ) d y \les \min( |x|, \lgp x ).
\]

Applying estimate \eqref{eq:coe_asymp:0} 
in Corollary \ref{cor:vel_est} with parameters $n = 2,  (b_1, b_2, \mu_1, \mu_2) = (-1, \alb, 1, 1)$
and the above estimates on $\wwwa$, we obtain 
\beq\label{eq:wwwa_prop_pf0}
\bal
  |\pa_x \psioab( \wwwa)  | + |\tf{1}{x} \psioab(\wwwa)| & \les 
    | \pa_x \psioab ( \wwwa) + 2 \alb \cJ_{\alb}(\wwwa)  | + |\tf{1}{x} \psioab + 2 \alb \cJ|  \\
& \les \min( |x|, 1) + \min( |x|, \lgp x) \les \min(|x|, \lgp x).
\eal
\eeq

Thus, the assumptions in \eqref{eq:Cinfty_ass} are satisfied  for $\al = \f13, \e=0, \gam = \alb$.
Moreover, using Theorem \ref{thm:fix_point}, we obtain that $\wwwa$ is locally Lipschitz.
Using Theorem \ref{thm:reg} and Lemma \ref{lem:norm_W}, we prove $\wwwa / x \in W_{\alb + 1}^{\infty}$and $\wwwa \in \cW_{\alb}^{\infty} \subset C^{\infty}$.

\vs{0.1in}
\paragraph{\bf Proof of \eqref{eq:bw_est:a}}

Recall $\bwp,\bw, \bwf$ from \eqref{eq:W_rep}. It suffices to consider 
$x$ large enough  so that 
$\bw = \bwf$ \eqref{eq:W_rep} and $\wwwa = w_{\alb} + \bw =w_{\alb} + \bwf$. Since $ \| w_{\alb} \|_{\bcX} \leq \dfix$, 
using definition of $\bcX$-norm \eqref{norm:X}, Corollary \ref{cor:vel_est} and estimate \eqref{eq:est_vmix}, we obtain
\beq\label{eq:wwwa_prop_pf1}
| w_{\alb}| \les \min( |x|^{-\bb_1} ,  \ang x^{-1/3} (\lgp x )^{-1/3} ) ,
\quad  |\psio_x(w_{\alb}) - \tf{1}{x} \psio( w_{\alb} )  | \les (\lgp x)^{-1/3}. 
\eeq
Since $\vvva = \vvva + \psio( w_{\alb})$ \eqref{eq:profile_alb}, using \eqref{eq:wwwa_prop_pf1}, the formula for $\bwf$ \eqref{eq:W_rep} and \eqref{eq:lin_nloc:far}, we prove
\beq\label{eq:wwwa_prop_pf2}
 |\wwwa + 6 x^{-\f13}| \les |\bwf + 6 x^{-\f13}| + |x|^{- \f13 } (\log x )^{-\f13}  \les   |x|^{- \f13} (\log x )^{- \f13} , 
 \quad | \pa_x \vvva-\tf{1}{x} \vvva - 4| \les (\lgp x )^{- \f13} .
\eeq

Since $\vvva = x + \psioab(\wwwa)$, combining the lower bound in \eqref{eq:profile_vvva1} and the estimate \eqref{eq:wwwa_prop_pf0},  we prove
$\vvva \asymp x \lgp x $ in \eqref{eq:bw_est:a}.

From \eqref{eq:wwwa_prop_pf2}, we obtain $\wwwa <0$ and $-\wwwa \asymp x^{-1/3}$ for $x> m$ 
with $m$ large enough. Since $\pa_x \bw(0) < -1$ by Lemma \ref{lem:bw_basic_sign},
 $\bb_1  < -1 $ \eqref{def:mu_b}, and $\wwwa \in C^{\infty}$,  we obtain 
\beq\label{eq:deri_wwwa0}
\pa_x \wwwa(0) = \pa_x \bw(0) \in [-1 - 10^{-6}, -1 + 10^{-6}],  \quad \wwwa \asymp -x 
\eeq
for 
 $0\leq x \leq m_1 < m$ with some small $m_1>0$. For $x \in [ m_1, m]$, using the profile equation 
\eqref{eq:1D_dyn} with $\al = \alb$  for $(\wwwa, \vvva)$,  $\vvva \gtr x \lgp x$ 
\eqref{eq:profile_vvva1}, and solving $\wwwa$ along the charactieristic, we obtain
\[
\wwwa(x) <0,  \quad \wwwa(x) \asymp \wwwa( m), \quad x \in [ m_1, m]. 
\]
Combining the above estimates and \eqref{eq:deri_wwwa0}, we prove \eqref{eq:bw_est:a} and \eqref{eq:wa_sign}.

\vs{0.1in}
\paragraph{\bf Proof of \eqref{eq:bw_est:b}}

Multiplying \eqref{eq:1D_dyn} by $x^{1/3}$ and using \eqref{eq:wwwa_prop_pf2} for $\vvva$, we obtain
\[
 2 \vvva \pa_x ( x^{\f13} \wwwa) = ( \f83 - \f23 \pa_x \vvva +  \f{2}{3 x} \vvva ) ( x^{ \f13 } \wwwa)  \, \Rightarrow \, \B| \f{\pa_x ( x^{1/3} \wwwa)  }{ 
 x^{1/3} \wwwa } \B| =  \B|\f{ 
4 - \pa_x \vvva +  \tf{1}{x} \vvva  }{3 \vvva}  \B|
 \les x^{-1} ( \lgp x)^{- \f43}.
\]  
Since $ \f{\pa_x ( x^{1/3} \wwwa)  }{  x^{1/3} \wwwa } = \f13 + \f{ x \pa_x \wwwa }{\wwwa} $, we prove the first two estimates in \eqref{eq:bw_est:b}.

Denote $g = \f{1}{x} \vvva$. Using the profile equation \eqref{eq:1D_dyn} and taking derivatives, we obtain
\[
  \f{x \pa_x \wwwa}{\wwwa} =  \f{\tf 43- \tf13 \pa_x \vvva }{  g },  \quad \pa_x (   \f{x \pa_x \wwwa}{\wwwa} ) 
  = \f{ \pa_x( \tf 43- \tf13 \pa_x \vvva ) }{g} - \f{(\tf 43- \tf13 \pa_x \vvva) \pa_x g }{g^2}.
\]

Using \eqref{eq:profile_Ck} with $\al = \g = \alb$, $\pa_{xx} \vvva = \pa_{xx} \psio_{\alb}(\wwwa)$, $\pa_x g = \f{1}{x}( \pa_x \vvva - \tf1x \vvva)$, and \eqref{eq:wwwa_prop_pf1}, we prove 
\[
 |\pa_x g | \les \ang x ^{-1}, \quad |\pa_{xx} \vvva| \les \ang x^{-1},
 \quad g \gtr \lgp x, \quad \Rightarrow | \pa_x (   \tf{x \pa_x \wwwa}{\wwwa} )  | \les \ang x^{-1} 
 |\lgp x|^{-1}.
\]
We complete the proof.
\end{proof}

\section{Analytic construction of the approximate $(\f13-\epsilon)$-profile }\label{sec:1D_profile_appr}

In this section,  we construct the approximate $C^{1/3-\e}$ profile by perturbing the $C^{1/3}$ profile $\wwwa $ constructed in Theorem \ref{thm:reg_alb}. 
We refer the motivation of this construction to  Section \ref{sec:idea_step2}. The proofs in this section and Section 
\ref{sec:1D_profile} are purely analytic (pen-and-paper).

Since $\wwwa, \vvva$ has been fixed and independent of $\al$, 
 in Section \ref{sec:1D_profile_appr} and \ref{sec:1D_profile},  we treat any constants depending on $\wwwa, \vvva$ as absolute constants. Recall the profile equation from \eqref{eq:1D_dyn}
  \beq\label{eq:1D_dyn1}
\bga 
 2 V \pa_x W = (3- \al - (1-\al) V_x) W  .
\ega 
\eeq
We adopt the notations in \eqref{eq:1D_normal} and impose the normalization conditions \eqref{eq:1D_normal:b}. Denote 
\bseq\label{def:para}
\beq
\bga
\alb \teq  \tf{1}{3}, \quad \e \teq \tf13 - \al,  
    \quad \he \teq \tf13 - \al -\b = \e - \b . \\
\ega
\eeq

We consider parameters in the following ranges 
\beq\label{def:eb_ineq}
 \e =  \alb - \al \in [0,  \tf{1}{1000}  ) , \quad  \b \in [-\e, 0] ,
 \quad \he = \e -\b \in [\e, 2 \e] .
\eeq
\eseq

For $\b \in [-\e, 0] $ to be chosen, we construct an approximate profile $(\wa, \va)$  by perturbing $\wwwa$ 
\beq\label{eq:wa}
\bal
\wa & = \wwwa (1 + x^2)^{\b / 2}, 
    & \quad \va &= x + \psioa(\wa) = x + \psia(\wa) - \psi_{\al,x}(\wa)(0) x, \\
    \cJab  & \teq \cJa(\wa) , & \quad \cJaa  & = (1 + \cJab^2)^{1/2} .
\eal
\eeq
Since $ -\alb + \b + \al-1 <0$, we obtain that the symmetrized integral for $\psia(\wa), \psi_{\al, x}(\wa)(0)$ converges and $\psi_{\al, x}(\wa)(0)$ is bounded.
We define the relative residual error associated with $(\va, \wa)$
\beq\label{eq:err_crab}
    \crab = (3-\al) -(1-\al) \pa_x {\va} - 2 \va \f{\pa_x \wa }{ \wa } .
\eeq

Formally, as $\al \to (\f{1}{3})^-$ and $\b \to 0$, we have $W_{\b} \to \wwwa$ and $\crab \to \cR_{\alb,0} = 0$.

The main result in this section is the following.

\begin{thm}\label{thm:1D_error}

There exists $ \beps_1 < \f{1}{1000}$ such that for any $\al$ with $\alb - \al = \e \in (0, \beps_1 ]$, there exists a unique solution $\b= \b(\al)$ in the range $[- \f{1}{2}\e , 0]$ to
\[
(2\al \e+4\al \b) \int_0^{\infty} \wwwa(y) \ang y^{\b} y^{\al-1} - 2 \b = - \f83,
\]
and $\beta $ satisfies 
\beq   \label{eq:beta_est}
\b  = \b(\al) = - \f{\e}{8} + O(\e^{4/3}) \in [-\f{\e}{2}, 0], 
   \quad \he = \e - \b \in [\e, \f32 \e] .
\eeq
Moreover, the error $\crab$ associated to  $W_{\b} = \wwwa \ang x^{\b}$ satisfies
\beq\label{eq:err_comp_R4}
  | \crab | 
  \les  \e  \min\B(  x, \  \min(  \lgp x, \e^{-1} )   |\lgp x|^{-1/3}   \B) 
  \les \e^{1/3}.
\eeq

\end{thm}

To analyze $\crab$, we compare the terms in $\crab$ and $\bar \cR$ one by one. Below, the reader should pay attention to the estimate for large $x$ since 
$\e$ is very small and $ | \la x \ra^{\e} - 1 | \les \e \log \ang x$ for $\e \log \ang x < 1$.

\begin{remark}[\bf Error $x \pa_x \crab$ is not small]

We cannot use $\wa$ as an approximate profile for constructing the 3D profile in 
\cite{chen2026eulerII} since we perform weighted $W^{1,\infty}$ estimates in
\cite{chen2026eulerII} and $ x \pa_x \crab$ is not small enough to close the estimates. 
We therefore construct the exact $\alpha$-profile in Section \ref{sec:1D_profile} with $x\pa_x\crab = 0$, which is crucial for the 3D construction in \cite{chen2026eulerII}.

\end{remark}

\subsection{Estimate and perturbation of velocity}

In this section, we estimate the difference between $\va$ and $\vvva$. 
We introduce the following functions to track error terms
\beq\label{def:err_del}
\bal
  \dda(x) &= \min( |x|^{  \al + 1} + |x|^{\al +\b + 1} , |x|^{-\he} ), \\
    \ddc(x) & = \min(\e \lgp x, 1) (\lgp x )^{-1/3} + \e \dda(x) + \e (\lgp x)^{-1/3}, \\
\eal
\eeq
where $\lgp x = \log(x + 2) $ for $ x \geq 0$ \eqref{def:lgp}.

Our key observation is that the 
velocity $V_{\al, \b}$ \eqref{eq:wa} can be approximated by 
$\al+\b$-velocity of $\wwwa$. These two terms have the same decay for large $x$. We have the following estimates of the difference.

\begin{lem}[\bf Comparison with $\alb$-nonlocal operator]\label{lem:vel_comp_gen}
Recall the kernel defined in \eqref{eq:ker}. For any $\g \in (\alb- \f{1}{500},\alb)$, $i=1,2$, 
 and $\e_{\g} = \alb - \g$, we have
\bseq\label{eq:vel_comp_gen}  
\begin{align}
  \B|  \cK_{ \g, i, J}( w ) -  |x|^{- \e_{\g}}  \cK_{\bar \al, i, J}( w ) \B| 
 & \les \e_{\g}  |x|^{2 -i} \min( |x|^{\g+1} || w x^{-1} ||_{\linf}, \ |x|^{- \e_{\g}} || w x^{\alb} ||_{\linf} ) , 
  \label{eq:vel_comp_gen:a}   \\
   |\cJ_{\g}(w)(x) - \cJ_{\alb}(w)(x)|  & \les \min( \e_{\g} |\lgp x|^2, \  \e_{\g} x , \ \lgp |x|  )
\cdot \max_{y \leq x} | w (|y|^{-1} + |y|^{\alb}) | ,
  \label{eq:vel_comp_gen:b} 
\end{align}
\eseq

\end{lem}

\begin{lem}[\bf Exchanging power weight with nonlocal operator]\label{lem:vel_comp}
Recall the kernel defined in \eqref{eq:ker}, $\dda$ from \eqref{def:err_del}, $\wa$ from \eqref{eq:wa}, and  $\e , \he$ from \eqref{def:para}:
\[
\dda(x) = \min(|x|^{\al + \b+1} + |x|^{\al + 1}, |x|^{-\he}),
\quad \wa= \wwwa \ang x^{\b},  \quad  \e = \bar \al - \al, \quad \he = \bar \al - \al - \b .
\]
Consider $i=1,2$. For  any $\e$ and $\b \in [-\e,0]$ satisfying \eqref{def:eb_ineq} 
and $x \geq 0$, we have 
\bseq\label{eq:vel_pow}
\begin{align}
  |   \cK_{\al, i, J}  ( \wa  ) - 
  \cK_{\al + \b , i, J } (\wwwa)  | & \les \e \cdot |x|^{2-i} \dda, 
  \label{eq:vel_pow:a} \\
|\cJ_{\al}( \wa  ) - \cJ_{\al + \b}(\wwwa) | & \les \e \min(x, 1) , \label{eq:vel_pow:b}   \\
 \B|  \cK_{\al + \b, i, J}( \wwwa) -  |x|^{-\he}  \cK_{\bar \al, i, J}(\wwwa) \B| 
& \les \e \cdot |x|^{2 -i} \dda(x),  \label{eq:vel_comp}  \\
|\cJ_{\al + \b}(\wwwa) - \cJ_{\alb}(\wwwa)|  & \les \min( \e \, |\lgp x|^2, \, \e \cdot x ,\, \lgp |x|  ). 
\label{eq:vel_pow:c} . 
\end{align}
\eseq

\end{lem}

We first estimate the  kernels in Section \ref{sec:compare_ker} and then prove 
Lemmas 
\ref{lem:vel_comp_gen}, \ref{lem:vel_comp} in Section \ref{sec:vel_comp}.

\subsubsection{Estimate the difference between kernels}\label{sec:compare_ker}

In this section, we estimate the difference between kernels $K_{\al, i}, K_{\al, i, J}$ with different index $\al$. Firstly, we have

\begin{lem}[\bf Approximation for power weight]\label{lem:pertb_pow}
For any $x > 1$, $z > 0$, and $|\b| \leq 1$, we get
\[
|x^{\al}  \la x z \ra^{\b } - x^{\al + \b}| 
\les |\b| x^{\al + \b} g_{\b}(z),\quad g_{\b}(z) \teq 
( \log \ang z + |\log z| ) \cdot ( 1 + z^{\b} + \ang z^{\b} ).
\]
\end{lem}

\begin{proof}
Denote $A = \f{x^2}{ 1 + (xz)^2}$. 
For $x>1$, we have $ \ang z^{-2} = \f{1}{1 + z^2} <  A < z^{-2}$. Since 
$|e^y - 1| \les |y| (e^y + 1) $  for any $y$, we bound 
\[
\bal
| ( \f{1 + (xz)^2}{x^2} )^{ \f{\b}{2} } -1|
= |A^{-\f{\b}{2}} - 1| 
\les |\b| \cdot |\log A| \cdot ( |A^{-\f{\b}{2}} + 1 ) 
 \les |\b| 
( \log \ang z + |\log z| ) \cdot ( 1 + z^{\b} + \ang z^{\b} ).
\eal
\]
Since $\ang {xz}^2 = 1 + (xz)^2$, multiplying the above estimate by $x^{\al+\b}$, we complete the proof.
\end{proof}

We have the following estimate for the kernel $K_{\al, i}$.
\begin{lem}[\bf Comparison between $\al-$kernels]\label{lem:K_decay}
For  $ |\al_i - \alb| < \f{1}{100}, i=1,2$ and any $z > 0$, the kernels defined in \eqref{eq:ker} satisfy  the following error estimates 
 \bseq\label{eq:K_err_decay}
 \begin{align} 
|K_{\al_1, 1, J}(1,z) - K_{\al_2, 1, J}(1,z) | & \les  |\al_1 -\al_2| 
(\one_{z \leq 3} + \one_{z> 3} z^{\alb + \f{1}{50} - 3} ),  \label{eq:K_err_decay:c} \\
| K_{\al_1, 2, J}(1, z) - K_{\al_2, 2, J}(1, z)| & \les
|\al_1 - \al_2| (  \one_{z \leq 3}  |z-1|^{\alb/2-1}  + \one_{z > 3}   z^{\alb + \f{1}{50} - 3}  ),  \label{eq:K_err_decay:d}  \\
|K_{\al_1, \D}(1, z) - K_{\al_2, \D}(1, z)| & 
\les  |\al_1 - \al_2| (  \one_{z \leq 3}  |z-1|^{\alb/2-1}  + \one_{z > 3}   z^{\alb + \f{1}{50} - 3}  ) .  \label{eq:K_err_decay:mix_err}
\end{align}
\eseq

\end{lem}

\begin{proof}

Since $\pa_{\al } x^{\al} = x^{\al} \log x$, for $|\al - \alb| < \f{1}{100}$ and $z \leq 3$, 
using the definitions in \eqref{eq:ker} and the integral formula in \eqref{eq:K_intform}, we obtain 
\[
\bal
  |\pa_{\al} K_{\al, 1, J}(1, z)| 
 &  \les  |1+ z |^{\al} \log |1 + z| 
  + |1-z|^{\al} | \log |1-z| \ | + z^{\al - 1} ( |\log z|+1) \cdot  \one_{z > 1}
\les 1 , \\ 
  |\pa_{\al} K_{\al, 2, J}(1, z)| 
  & \les  |1+ z |^{\al -1} \log |1 + z| 
  + |1-z|^{\al - 1} | \log |1-z|  \ | + z^{\al - 1} ( |\log z| + 1) \cdot \one_{z > 1}   \les  |1-z|^{\alb/2-1} ,
  \eal
\]
with implicit constants uniformly in $\al$ for $|\al - \alb| < \f{1}{100}$. 
 For $z > 3$ and $|\alb - \al| < \f{1}{100}$, using the integral formula in \eqref{eq:K_intform}, we obtain 
\[
      |\pa_{\al} K_{\al, 1, J}(1, z)|  
    +     |\pa_{\al} K_{\al, 2, J}(1, z)|  
   \les |z+1|^{\al -3 } \log |z+1|
  + |z-1|^{\al -3 } | \log |z-1 | \, |
  \les z^{\alb + \f{1}{50}-3} .
\]

Using the mean-value theorem and the above two estimates, we prove \eqref{eq:K_err_decay:c}, \eqref{eq:K_err_decay:d}.

Recall  $K_{\al, \D}(1, z) = K_{\al, 2, J}(1, z) - K_{\al, 1, J}(1, z) $ from \eqref{eq:ker:c}.
Using \eqref{eq:K_err_decay:c}, \eqref{eq:K_err_decay:d}, and triangle inequality,
we prove  \eqref{eq:K_err_decay:mix_err}. 
\end{proof}

\subsubsection{Proof of Lemmas \ref{lem:vel_comp_gen}, \ref{lem:vel_comp}}\label{sec:vel_comp}

Now, we are in a position to prove Lemmas \ref{lem:vel_comp_gen} and \ref{lem:vel_comp}.

\begin{proof}[Proof of Lemma \ref{lem:vel_comp_gen}]

\textbf{Proof of \eqref{eq:vel_comp_gen:a}} Recall the operator $\cK_{\g, i, J}$ and its associated kernel $K_{\g, i, J}$ from \eqref{eq:ker}. 
Since $K_{\g, i, J}$ is $(\g+ 1 - i)$-homogeneous, using a change of variable $y = x z$, we have 
\[
\bal 
  \cK_{\g, i, J} (w)(x) 
& = \int K_{\g,i, J}(x, y) w(y) d y 
= x^{\g + 2 - i}  \int  K_{\g, i, J}(1, z) w (x z) d z  \\
& = x^{\g + 2 -i} \int  (K_{\g, i , J}(1, z) - K_{ \bar \al, i, J}(1, z )) w(x z) d z 
+ x^{\g + 2 - i} \int K_{\bar \al, i, J} (1, z ) w(x z) d z \\
&  \teq I_1 + I_2.
\eal 
\]

 For $I_2$, using the definition of $K_{\g, i, J},\cK_{\g, i, J}$ from \eqref{eq:ker}, $\g - \alb =  - \e_{\g}$, and undoing the change of variables, we get 
\[
  I_2 = x^{ \g - \alb } \cK_{\alb, i, J}(w )(x)  = x^{- \e_{\g}} \cK_{\alb, i, J}(w )(x) .
\]

Let $ k = 1$ or $-\alb$. For $I_1$, applying the estimate \eqref{eq:K_err_decay:c} with $\al \rsa \g$, $|w(x z) |\les (xz)^k \nlinf{ w x^{-k} }$,
and $\alb- \g  = \e_{\g}$ from the statement in Lemma \ref{lem:vel_comp_gen}, we bound 
\[
\bal
  |I_1 | & \les  \e_{\g} x^{ \g + 2 - i} \int_0^{\infty} (\one_{z \leq 3} |z-1|^{\alb/2-1}  + \one_{z>3} z^{\alb + \f{1}{50}-3}  ) (xz)^k d z  \cdot \nlinf{ w x^{-k} }.
\eal
\]
Clearly, the integrand is integrable in $z$ for $k = 1$ or $-\alb$. Thus, we obtain 
\[
|I_1|  \les \e_{\g} x^{\g+ 2-i} x^k \cdot || w x^{-k} ||_{\linf}
= \e_{\g} x^{  2-i} \cdot x^{ k+ \g}  \nlinf{ w x^{-k} }.
\]

Combining the above two estimates for $ k = 1, -\alb$, and using $ \alb - \g = \e_{\g}$, we prove \eqref{eq:vel_comp_gen:a}.

\paragraph{\bf Proof of \eqref{eq:vel_comp_gen:b}}

We fix $x\geq 0$ and denote $M = \max_{y \leq x} |w (y^{-1} + x^{\alb})|$.
Since $\g \in (\alb- \f{1}{100}, \alb)$, using $|1 - y^{- b}| \les  b |\log y |  (1 + y^{-b})$ for  $y > 0, b>0$, 
and $|w(y)| \les \min(y, y^{-\alb}) M$ for $ y \leq x$, we obtain 
\[
\bal
  |\cJ_{\g}(w ) - \cJ_{\alb}( w )|
  & \les \int_0^x |w(y)| y^{\alb-1} | 1 - y^{- (\alb -\g)} | d y 
\les \int_0^x \min(y, y^{-\alb}) y^{\alb-1} | 1 - y^{- (\alb -\g)} | d y M
   \\
 &  \les |\g - \alb| \int_0^x \min(y, y^{-\alb}) y^{\alb-1}  |\log y| (1 + y^{- (\alb-\g) }) d y 
 \cdot M \\
& \les |\g-\alb|  \int_0^x  \one_{y\leq 2} +\one_{y >2} y^{-1} \log y d y
\les |\g - \alb|  \min( x,    | \lgp x|^2)  \cdot M .
  \eal
\]
Moreover, since $\g \leq \alb$, we have the trivial bound 
\[
   | \cJ_{\g}(w) | + |\cJ_{\alb}(w)| \les
   \int_0^x |w(y)| (y^{\alb-1} + y^{\g-1}) d y
  \les \int_0^x \min(y, y^{-\alb})(y^{\alb-1} + y^{\g-1}) d y M
\les \lgp x \cdot M.
\]
We prove the desired estimates. We complete the proof of Lemma \ref{lem:vel_comp_gen}.
\end{proof}

Next, we prove Lemma \ref{lem:vel_comp}.

\begin{proof}[Proof of Lemma \ref{lem:vel_comp}]

\textbf{Proof of \eqref{eq:vel_pow:a}} Recall the operator $\cK_{\g, i, J}$ and its associated kernel $K_{\g, i, J}$ from \eqref{eq:ker}.  Using a change of variable $y = x z$, we obtain
\[
\bal
 & \cK_{\al, i, J}(\wwwa \ang x^{\b})(x)
  - \cK_{\al + \b, i, J}(\wwwa)(x) 
   = \int K_{\al, i, J}(x, y)  \wwwa( y ) \ang y^{\b} 
  - K_{\al + \b, i, J}(x, y) \wwwa(y) d y \\
 & =  \int K_{\al , i, J}(1, z) x^{2-i} ( x^{\al} \la x z \ra^{\b} - x^{\al + \b} ) \wwwa(x z) 
+ ( K_{\al , i, J}(1, z) - K_{\al + \b, i, J}(1, z) ) x^{2-i + \al + \b}  \wwwa(xz) d z \\
& \teq II_1 + II_2.
\eal 
\]

For $II_2$ and $x > 1$, since $|\wwwa(xz) | \les \min(xz, |x z|^{-\alb})$ 
by \eqref{eq:bw_est}, using Lemma \ref{lem:pertb_pow}, we bound 
\[
  |II_1| \les |\b| x^{\al + \b + 2 - i} \int |K_{\al, i, J}(z)| g_{\b}(z)  \min(x z, |x z|^{-\alb}) d z .
\]

Using estimate \eqref{eq:K_decay} for $K$, we bound the integrand in $z$ as 
\[
  |K_{\al, i, J}(z)| g_{\b}(z) z^{\g} 
\les \min(1 + |z-1|^{\al-1}, z^{\al-3})  g_{\b}(z) z^{\g} , \quad \g = 1,  -\alb.
\]
which are $L^1$-integrable, with $L^1$-norm independent of $\al, \e$. Since $ \he = \alb - \al - \b$ \eqref{def:para},  we bound 
\[
  II_1  \les |\b| x^{\al + \b + 2 - i}  \min(x, x^{-\alb})
  = |\b| x^{ 2 -i } \min( |x|^{ \al + \b +1}, x^{ -\he } ).
\]

For $II_2$, using \eqref{eq:K_err_decay:c}, \eqref{eq:K_err_decay:d}, we bound
\[
\bal
  |II_2| 
& \les x^{2-i + \al + \b} |\b| \int (\one_{z \leq 3}  |z-1|^{\alb/2-1} + \one_{z > 3} z^{\alb + \f{1}{50}-3} ) \min( xz, |xz|^{-\alb} ) d z \\
& \les |\b| x^{2-i + \al + \b} \min(x, x^{-\alb})
= |\b| x^{ 2 -i } \min( |x|^{\al + \b +1}, x^{ -\he } ).
\eal
\]
Combining the above two estimates, we prove \eqref{eq:vel_pow:a}.

\paragraph{\bf Proof of \eqref{eq:vel_pow:b}}

From the definition of $\cJ$ \eqref{eq:Jw}, we obtain 
\[
  I \teq \cJ_{\al}(\wwwa \ang x^{\b})(x) - \cJ_{\al+\b}(\wwwa)(x)
   = \int_0^x \wwwa(y) ( y^{\al-1} \ang y^{\b} - y^{\al + \b -1} ) d y.
\]

Since $|e^a - e^b| \les |b-a| (e^{a} + e^b)$ and $|\b| < 1$, for $y \leq 10$, 
we obtain
\[
  |y^{\b} - \ang y^{\b} | 
  = | e^{\b \log y} - e^{\b \log \ang y} | 
  \les   |\b| ( |\log y| + |\log \ang y|) 
\cdot 
  |y^{\b} + \ang y^{\b} |
\les |\b| ( |\log y| + 1)  ( |y|^{\b } + 1 ).
\]
For $y > 10$, we obtain 
\[
    |y^{\b} - \ang y^{\b} | \les |\b| y^{\b-1} |y - \ang y| \les |\b| y^{\b-2}.
\]

Combining the above two estimates and using $|\wwwa(y)| \les \min(y ,|y|^{-\alb})$, we prove 
\[
    |I| \les |\b|
    \int_0^x \min(y, |y|^{-\alb}) y^{\al-1} (  ( | \log y | + 1 )  (|y|^{\b} + 1) \one_{y \leq 10} + y^{\b-2} \one_{|y|>10} )
    \les |\b| \min( x, 1 ).
\]

\paragraph{\bf Proof of \eqref{eq:vel_comp}, \eqref{eq:vel_pow:c}}

Recall $\dda(x) = \min( |x|^{  \al + 1} + |x|^{\al + \b + 1} , |x|^{-\he} )$ from \eqref{def:err_del}. We apply Lemma \ref{lem:vel_comp_gen} with $w = \wwwa, \g = \al + \b$ and using $|\wwwa| \les \min(x, |x|^{-\alb})$ \eqref{eq:bw_est} and $ \alb - \g = \alb -\al - \b = \he,  |\he| \les \e$ \eqref{def:para}. Applying \eqref{eq:vel_comp_gen}, we obtain 
\[
\bal
  |\cK_{\al + \b, i, J}(\wwwa) - |x|^{-\he} \cK_{\alb, i, J}(\wwwa)| 
& \les |\he| \cdot |x|^{2-i} \min( |x|^{ 1 + \g }, |x|^{ -(\alb - \g) } )
\les \e |x|^{2-i} \dda(x), \\
  |\cJ_{\al + \b}(\wwwa) - \cJ_{\alb}(\wwwa)| &  \les \min( \e \, |\lgp x|^2, \, \e \cdot x ,\, \lgp |x|  ).
\eal
\]
We complete the proof.
\end{proof}

\subsubsection{Estimates of velocity and kernels}\label{sec:vel_basic}

\begin{lem}\label{lem:vel_al}

Recall the operators $\cK_{\al, i, J}$ from \eqref{eq:ker}. For any $b \in [0, \f{1}{2} ]$, $\al \in [\f16, \f13]$, and weight $\Gam(x)$ satisfying $|\Gam(x) / \Gam(y)| \les 
 (\lgp(  x / y + y / x ) )^2$, 
we have 
\beq\label{eq:v_al_gen}
  \Gam(x) | \cK_{\al, i, J}(w)(x) | \les  x^{3-i + \al}  \ang x^{  - b-1} 
  || w  ( y^{-1} + y^b )  \Gam  ||_{L^{\infty}}, \quad i =1,2 , \\ 
\eeq
with implicit constants depending uniformly in $\al, \e$.

\end{lem}

\begin{proof}

Denote $f(y) = \min(y, |y|^{-b})$ and $D = \supp(w) \cap \R_+$. Using $\cK_{\al,i, J}$ from \eqref{eq:ker}, we obtain 
\[
   | \Gam \cK_{\al, i, J}(w)(x) |
  =  \Gam(x) |\int_0^{\infty} K_{\al,i, J}(x, y) w(y) d y|
\leq || w (|y|^{-1} + |y|^{b}) \Gam ||_{\linf} \Gam(x) \int_D |K_{\al, i, J}(x, y) |\Gam(y)|^{-1}
f(y) d y
\]

Using a change of variables $y = x z$, $|\Gam(y)/ \Gam(x)|  \les ( \lgp( z + z^{-1}) )^2$, 
$K_{\al, i, J}(x, x z) = x^{\al + 1 -i} K_{\al, i, J}(1, z)$, and 
estimate \eqref{eq:K_decay} in Lemma \ref{lem:K_decay_basic} for $K$, we bound the integral as 
\[ 
\bal
  I & \teq \Gam(x) \int_{  D } |K_{\al, i, J}(x, y)| \Gam(y)^{-1} f(y) d y
  =
    x^{2 + \al -i} \Gam(x) \int_{z \in x^{-1} D }  |K_{\al, i, J}(1,   z) \Gam(x z )^{-1} \min(x z, (x z)^{-b} ) | d z \\
  & \les x^{2 + \al -i} \int_{z \in x^{-1} D } F(z) \min(x z, (x z)^{-b} )  d z .
\eal
\]
where 
\[
  F(z) = \min( |z-1|^{\al-1} +1 , z^{\al-3}) 
  (\lgp (z + z^{-1}))^2 .
\]
Since $b \in [0, \f12], \al \in [\f16, \f13]$, we have $F z, F z^{-b} \in L^1(\R_+)$.
We prove \eqref{eq:v_al_gen}:
\[
  I \les  x^{2 + \al - i} \min( x, x^{-b} )
  \les x^{3  - i + \al} \ang x^{-b-1}.
\]
We complete the proof.
\end{proof}

\begin{lem}\label{lem:J_est}

Recall $\vvva = x + \cK_{\alb, 1}(\wwwa)$ from \eqref{eq:profile_alb}, $\cK_{\al, i}, \cJa$ from \eqref{eq:ker} and \eqref{eq:Jw}. 
For any $x \geq 0$, we have 
\bseq\label{eq:v_profile_est}
\begin{align}
   \tf{1}{x} | \vvva - x + 2 \alb x \cJ_{\alb}(\wwwa)|  
    + | \pa_x \vvva - 1 +  2 \alb \cJ_{\alb}(\wwwa)| & \les \min(1, x^{1 +\alb} ) \, , \label{eq:v_profile_est:a} \\
    \tf{1}{x} | \va - x+ 2 \al x \cJab|  +
      | \pa_x \va - 1 + 2 \al  \cJab|  & \les \min( \ang x^{-\he} , x^{1 + \al}) \, .
    \label{eq:v_profile_est:b}
\end{align}
\eseq

Denote $\e_{\g} = \alb - \g$. For any $\g < \alb$ with $|\alb - \g| \leq 10 \e$, and $x \geq 0$, we have
\bseq\label{eq:J_asym}
\begin{align}
  \cJ_{\g}(\wwwa) & = - 6 \cdot \e_{\g}^{-1} (1 - \ang x^{- \e_{\g} } ) + O( \min(\lgp x ,\e^{-1})^{2/3} ) \, ,    \label{eq:J_asym:a} \\
  |\cJ_{ \g }(\wwwa)|  &\les  \min(x, \lgp x ) ,
  \quad -\cJ_{ \g }(\wwwa) \gtr \min( |x|^{\g+1} ,  \, \e_{\g}^{-1} ( 1 - \ang x^{- \e_{\g} } )    ) .
  \label{eq:J_asym:b}
  \end{align}
\eseq

\end{lem}

\begin{proof}

\textbf{Proof of  \eqref{eq:v_profile_est} } Recall from \eqref{eq:ker} and \eqref{eq:wa} that 
\[ 
\bal
     \vvva - x + 2 \alb x \cJ_{\alb}(\wwwa)
   & = \cK_{\alb, 1, J}(\wwwa),  &
    \quad      \pa_x \vvva - 1 + 2 \alb  \cJ_{\alb}(\wwwa)  &= \cK_{\alb, 2, J}(\wwwa) , \\
  \va - x + 2 \al x \cJab  & = \cK_{\al, 1, J}(\wa),  &
\quad   \pa_x \va - 1 + 2 \al  \cJab & = \cK_{\al, 2, J}(\wa).
\eal
\]

Taking $\al = \alb, \b = 0$, we get $\wwwa = \wa, \vvva = \va$ by \eqref{eq:wa} and 
$\he = 0$ by \eqref{def:para}. Thus, estimate \eqref{eq:v_profile_est:b} implies \eqref{eq:v_profile_est:a}, and it suffices to prove \eqref{eq:v_profile_est:b}.   
Applying Lemma \ref{lem:vel_al} with 
$(\al, b, w) \rsa (\al, \alb - \b, \wa), \Gam \equiv 1$, and then using 
$ \al - (\alb - \b ) = - \he$ \eqref{def:para}, we prove \eqref{eq:v_profile_est:b}:
\[
  \tf{1}{x} |\cK_{\al, 1, J}(\wa) |
   + |\cK_{\al, 2, J}(\wa) |
   \les  x^{1 + \al} \ang x^{- (\alb -\b)- 1}
   \les  \min( x^{1 + \al}, \ang x^{-\he} ).
\]

\paragraph{\bf Proof of \eqref{eq:J_asym}}

For $x \geq 0$, using  \eqref{eq:bw_est:a} and $\g <  \alb$, we obtain 
\[
\bal
  \cJ_{\gam}( \wwwa )(x) & = 
  \int_0^x \wwwa y^{ \g-1} d y
 = - \one_{ x \geq 2} \int_1^x  6 y^{\g-1-\alb}  + 
\one_{ x \geq 2} \int_1^x  O( y^{\g-1- \alb} |\lgp y|^{- \alb} )  d y   
 + O(1) \\
 & \teq I_1 + I_2 + I_3.
\eal 
\]

Estimate \eqref{eq:J_asym:a} is trivial for $x \leq 2$.  Since $|x - \ang x| \les \ang x^{-1}$ and $y^{\g-1-\alb} \les 1$ for $y > 1$, using a direct calculation for $I_1$, a change of variable $t = \log y$ for $I_2$, 
and $\e_{\g} = \alb - \g$, we yield
\[
\bal
I_1 & = -6 \f{ 1 - \ang x^{-\e_{\g}} }{  \e_{\g}  },  \quad 
 |I_2|  \les \one_{x \geq 2} \int_{\log 2}^{\log x}
e^{ - \e_{\g} t} t^{- \alb} d t 
\les \one_{x \geq 2} \min(  |\log x|^{1-\alb}, \e_{\g}^{1-\alb}),
\eal 
\]
where the last inequality is obtained by considering $ \e_{\g} t \leq 1$ and $\e_{\g} t > 1$. Since $\alb = \f13$ and $ |x^{-\e} - \ang x^{-\e}| \les \e $ for $x>1$, combining the above estimates, 
we prove \eqref{eq:J_asym:a}. Using  $-\wwwa \asymp  \min( |x|, \ang x^{-\alb})$ from 
\eqref{eq:bw_est:a}, we prove 
\[
 - \cJ_{\gam}( \wwwa )(x)  = -\int_0^x \wwwa y^{ \g-1} d y
 \gtr \int_0^x y^{\g-1} \min( y, y^{-\alb}) d y \gtr \min( |x|^{ \g+1} , \e_{\g}^{-1} ( 1 - \ang x^{-\e_{\g}} ) ).
\]
The first estimate in \eqref{eq:J_asym:b} follows from 
$ |\wwwa| \les \min( |x|, |x|^{-\alb}) $ in \eqref{eq:bw_est:a} and $\g \leq \alb$.
We complete the proof.
\end{proof}

\subsection{Estimate of error}

In this section, we first estimate the error $\crab$ \eqref{eq:err_crab} in terms of $\al, \b$ and then determine $\b$ with a small relative error. The main difficulty is the estimate for large $x$.
Throughout this section, we consider $\b\in [-\e,0]$ as in \eqref{def:eb_ineq} 
and $x \geq 0$ since $\wa, \va$ are odd.

We recall the parameters from \eqref{def:para} and $\Vo_{\al, \b}, \cJab$ from \eqref{eq:ker} and \eqref{eq:wa}
\[
 \alb =  \tf{1}{3}, \quad \e = \alb - \al,  
  \quad \he = \e - \b,\quad \Vo_{\al, \b}(x) = V_{\al, \b}(x) - x,
  \quad \cJab = \cJa(\wa).
\]

Firstly, we rewrite $\crab$ \eqref{eq:err_crab} as
\bseq\label{eq:err_comp_R1}
\beq
\bal
& \crab  = 
(3-\al) - (1 - \al) \pa_x \va  - 2 \va \big( \f{\pa_x \wwwa}{\wwwa} 
+ \b  \f{x}{1 + x^2} 
\big)  \\
& = (3-\al) - \e \pa_x  \va 
- (1 - \bar \al) (  \pa_x \va - \f{\va}{x} )
- 2 \va \big( \f{\pa_x \wwwa}{\wwwa} + \f{1}{3 x}  \big) 
- 2 \va  \b \cdot \f{x}{1 + x^2}  \\
& = (3-\alb) - \e \pa_x \voa
- (1 - \bar \al) (  \pa_x \va - \f{\va}{x} )
- 2 \va \big( \f{\pa_x \wwwa}{\wwwa} + \f{1}{3 x}  \big) 
- 2 \va  \b  \cdot \f{x}{1 + x^2} , \\
\eal
\eeq
where  $\bar \al = \f{1}{3}$. Since $\wwwa$ is the exact profile to \eqref{eq:1D_dyn} with $\al=\alb=\f13$, we get
\beq
  0=  \bar \cR = \f{8}{3} - (1 - \bar \al)( \pa_x \vvva - \f{\vvva}{x} ) 
  - 2 \vvva \big( \f{\pa_x \wwwa}{\wwwa} + \f{1}{3 x}  \big) .
\eeq
\eseq
We estimate $\crab$ by comparing each term in $\crab$ with its counterpart in $\bar \cR$.

\subsubsection{Estimate of $V_{\al, \b}$}\label{sec:err1D_vel}

We use Lemma \ref{lem:vel_comp} to estimate the velocity $V_{\al, \b}$. From \eqref{eq:Jw}, we have
\beq\label{eq:err_comp_KV}
\bal
\cK_{\alb, 2, J}(\wwwa) 
  & = \pa_x \vvva - 1 +  2 \alb \cJ_{\alb }(\wwwa),\quad && \cK_{\alb, 1, J}(\wwwa) = \vvva - x   + 2 \alb x \cJ_{\alb}(\wwwa), \\
\cK_{\al, 2, J}(\wa) 
  & = \pa_x \va - 1 +  2 \al \cJ_{\al}(\wa),\quad  && \cK_{\al, 1, J}(\wa) = \va - x   + 2 \al x \cJ_{\al}(\wa), \\
\eal
\eeq

Recall the kernels from \eqref{eq:ker} and the relation \eqref{eq:Jw}. Using \eqref{eq:vel_pow}  in Lemma \ref{lem:vel_comp} and the above relation, we obtain 
\bseq\label{eq:err_comp_vel1}
\begin{align}
    \cK_{\al, 2, J}(\wa) & = \cK_{\al + \b, 2, J}(\wwwa) + O( \e  \dda(x))  = x^{-\he} \cK_{\alb, 2, J}(\wwwa) + O( \e \dda(x))  
    \\
    \cK_{\al, 1, J}(\wa) & =  \cK_{ \al + \b, 1, J}(\wwwa) + O( \e  |x| \dda(x)) 
   =  x^{-\he} \cK_{  \alb, 1, J}(\wwwa) + O( \e |x| \dda(x)) , \\
  J_{\al}(\wa) & = J_{\al + \b}(\wwwa) + O( |\e| \min(|x|, 1)).
\end{align}
\eseq

Using the above estimates and \eqref{eq:err_comp_KV}, we obtain
\begin{align}\label{eq:err_comp_vel2}
  \pa_x \va - \tf{ 1 }{x} \va
  &  
  = \cK_{\al,2, J}(\wa) - \tf{1}{x} \cK_{\al, 1, J}(\wa)   = x^{-\he}  ( \cK_{\alb, 2, J}(\wwwa)  - \tf{1}{x} \cK_{\alb, 1, J}(\wwwa) )   +  O( \e \dda(x) ) \notag \\
& = x^{-\he}  ( \pa_x \vvva - \tf{1}{x} \vvva ) +  O( \e \dda(x) ) .
\end{align}

Next, we estimate
\[
  \bal
\va - \vvva & = \big ( ( \va - x + 2 \al x \cJ_{\al }(\wa) )  
  - ( \vvva - x + 2 \bar \al x \cJ_{ \bar \al}(\wwwa)) \big)  - 2 \al x ( \cJ_{\al}(\wa) - \cJ_{\al + \b}(\wwwa))  \\
  & \quad -  2 \al  x ( \cJ_{\al + \b}(\wwwa) - \cJ_{\alb}(\wwwa))
   - 2 ( \al - \bar \al ) x \cJ_{\bar \al}(\wwwa) \teq I_1  +I_2 + I_3 + I_4.\\
  \eal 
\]

Using \eqref{eq:err_comp_KV}, \eqref{eq:err_comp_vel1}, $|\cK_{\alb, 1, J}(\wwwa)| \les 
\min(x^{2+\alb}, x)$ (see \eqref{eq:v_profile_est:a}), $\he > 0$ \eqref{def:para},  and $ |1 - x^{-\he}| \les |\he \log x|
(1 + x^{-\he})$, we estimate 
\[
\bal
  |I_1| & = |\cK_{\al,1, J}(\wa) - \cK_{\alb,1,J}(\wwwa)|
  \les |x^{-\he} - 1| \cdot |\cK_{\alb,1,J}(\wwwa)| + \e x \dda \\
& \les |\he \log x | (1 + x^{-\he}) \min( x^{2+\alb}, x ) 
+ \e x \dda  \les \he \min( x^2, x \lgp x)  + \e x \dda.
\eal
\]

For $I_2,  I_4$, since $\alb - \al = \e$, using \eqref{eq:vel_pow:b} and \eqref{eq:J_asym} from Lemma \ref{lem:J_est}, we estimate 
\[
\bga
  |I_2| \les |\e| \min(x^2, x), \quad |I_4| \les  \e x \min(x,  \lgp x) ,  \\
\ega
\]

Using estimate \eqref{eq:vel_pow:c} with $\g = \al + \b $ and $\alb - \g = \he$, we bound $I_3$ as 
\[
   |I_3| \les 
   x \min( \he \lgx^2 , \he x,  \lgx  ).
\]

Since $1 \les \lgp x$, combining the above estimate, and using $\he = \e -\b \in[\e, 2\e]$ \eqref{def:eb_ineq},  we establish 
\beq\label{eq:err_comp_vel3}
  |\va - \vvva | \les \e x \dda
  + \one_{x < 1} \e  x^2 
  + \one_{x > 1}  x \lgp x \cdot ( \min( \e \lgp x, 1) + \e ) .
\eeq

\subsubsection{ Estimate of $\crab$}\label{sec:err1D_R}

Comparing two identities in \eqref{eq:err_comp_R1}, and using \eqref{eq:err_comp_vel2}, 
we estimate $\crab$ as
\begin{align}
  & \crab  = \crab - \bar \cR \notag \\
&  = -\e \pa_x  \voa 
  - (1 - \bar \al)( x^{-\he}-1) ( \pa_x \vvva - \f{\vvva}{x} )  \notag 
  - 2 ( \va - \vvva) \big( \f{\pa_x \wwwa}{\wwwa} + \f{1}{3 x}  \big) 
- 2 \va  \b  \f{x}{x^2 + 1}
  + O( \e \dda) \notag  \\
& \teq II_1 + II_2 + II_3  + II_4 + O( \e \dda ) .  \label{eq:err_comp_R0} 
\end{align}

Below, we aim to estimate $\crab$ up to error $O(\ddc)$ with $\ddc$ defined in \eqref{def:err_del}.

\paragraph{\bf Estimate for $x \leq 2$}

Using \eqref{eq:v_profile_est:a} in Lemma \ref{lem:J_est}, $|x^{-\he} - 1| \les |x^{\he} + 1| \cdot |\he  \log x| $, and $\he \les \e$ by \eqref{def:eb_ineq}, we bound 
\[
\bga
   | ( x^{-\he}-1) ( \pa_x \vvva - \tf{\vvva}{x} ) |
  \les x^{1 + \alb} |  x^{-\he}-1|
  \les |\e \log x| 
  \cdot x^{1 + \alb} |  x^{-\he} +1| \les \e x , \\
    |\e \pa_x \voa | = |\e ( \pa_x \va -1) |
  \les \e x , 
  \quad |\va| \les x.
\ega
\]

 Using \eqref{eq:err_comp_vel3}, for $x \leq 2$, we estimate
\[
|II_3| \les | ( \va - \vvva) ( \tf{\pa_x \wwwa}{\wwwa} + \tf{1}{3 x}  )   |
\les \e  x^2 \cdot x^{-1} \les \e x.
\]

Combining these estimates, we prove
 \beq\label{eq:crab_near0}
   |\crab| \les \e x , \quad x \leq 2.
 \eeq

\paragraph{\bf Estimate for $x \geq 2$}

Recall $\dda,  \ddc$ from \eqref{def:err_del}. Using estimate  \eqref{eq:bw_est:a}:
\beq\label{eq:profile_bd1}
   |\f{\pa_x \wwwa}{\wwwa} + \f{1}{3 x}|
\les |\lgp x|^{-4/3} x^{-1}, \quad 
  |\pa_x \vvva - \f{\vvva}{x} - 4 | \les (\lgp x )^{-1/3} ,
\eeq
and \eqref{eq:err_comp_vel3} for $\va- \vvva$,  we bound $II_3$ in \eqref{eq:err_comp_R0} as
\[
\bal
  |II_3 |
& \les  \B( \lgp x \cdot  ( \min( \e \lgp x, 1) + \e  ) + \e \dda  \B) (\lgp x)^{-4/3} \\
& \les \big( \min( \e \lgp x, 1) + \e \big)  (\lgp x)^{-1/3} 
+ \e (\lgp x)^{-4/3} \les \ddc,  \\
\eal
\]

Since $x \geq 1$ and $\he >0$ \eqref{def:para}, we obtain 
\beq\label{eq:x_he}
   0 \leq 1 - x^{-\he} = 1 - e^{-\he \log x}\les \min(1, \he \log x )
  \les \min(1, \he \lgp x), \quad \forall x \geq 1.
\eeq
Using this estimate and  \eqref{eq:profile_bd1}, we estimate $II_2$ in \eqref{eq:err_comp_R0} as
\beq\label{eq:err_main_II2}
\bal
II_2 & = 
 -  4 (1 - \bar \al)( x^{-\he}-1) 
 - (1 - \bar \al)( x^{-\he}-1) ( \pa_x \vvva - \f{\vvva}{x} -4) 
 \teq II_{21} + II_{22} , \\
|II_{22} | & \les  |x^{-\he} - 1 | \cdot |\pa_x \vvva - \tf{1}{x} \vvva - 4| 
 \les  \min( \he \lgp x, 1 ) | \lgp x|^{-1/3} \les \ddc. 
 \eal
\eeq

We keep $II_{21}$ as a main term. Next, we estimate $II_1, II_4$ in \eqref{eq:err_comp_R0}. Using 
$\voa =\va - x$, Lemma \ref{lem:J_est} and \eqref{eq:err_comp_vel1},\eqref{eq:err_comp_KV}, we obtain
\[
\bal
  \pa_x \voa &= \cK_{\al, 2}(\wa) =  \cK_{\al,2,J}(\wa) 
  - 2 \al \cJab = O(\ang x^{-\he} +  \e \dda )   - 2 \al \cJab, \\
   \tf{1}{x} \va & = 1 - 2 \al \cJab + O( \ang x^{-\he } +\e \dda) . \\
\eal
\]
which along with the definition of $\ddc$ \eqref{def:err_del} implies 
\beq\label{eq:err_main_II1}
  II_1 = - \e    \pa_x \voa = 2 \al \e \cJab + O(\ddc).
\eeq

Since $|\cJab| \les \lgp x$ from Lemma \ref{lem:J_est}, for $x \geq 1$, we have 
\[
|  \va |  \les x \lgp x, \quad  \tf{x}{1 + x^2} = x^{-1} + O(\ang x^{-3}) ,
\] 
which implies 
\beq\label{eq:err_main_II4}
\bal
  II_4 &  = - 2 \b \tf{x}{1 + x^2} \va
  = - 2 \b \tf{1}{x} \va + O(|\b \va | x^{-3} ) \\
  & = - 2 \b (1 - 2\al \cJab)  + O(\e \dda ) + O(|\b| \ang x^{-1}) 
  = - 2 \b (1 - 2\al \cJab) + O(\ddc) .
\eal
\eeq

Combining the above estimates on $II_i$, we estimate $\crab$ as 
\beq\label{eq:err_comp_R2}
\bal
      \crab & = 2\al \e \cJab   - 4 (1 - \bar \al)(x^{-\he }-1)   - 2 \b(1 - 2 \al \cJab) 
  + O(\ddc) \\
  & = O(\ddc(x)) + (2 \al \e + 4 \al \b) \cJab - 4 (1 - \bar \al)( x^{-\he } -1) 
  - 2 \b.
\eal 
\eeq
We treat the above second term from $II_{21}$ \eqref{eq:err_main_II2},
$II_1$ \eqref{eq:err_main_II1}, $II_4$ \eqref{eq:err_main_II4} as the main term, 
and the first as an error, which is small when \( |\b|,|\e| \ll 1 \).

\subsection{Proof of Theorem \ref{thm:1D_error}}\label{sec:thm_1D_error}

Recall $\cJab = \cJa(\wa)$ from \eqref{eq:wa}. Denote 
\bseq\label{eq:beta}
\beq
  G(\b) \teq(2\al \e+4\al \b) \cJ_{\al}(W_{\b})(\infty) - 2 \b =    (2\al \e+4\al \b) \int_0^{\infty} \wwwa(y) \ang y^{\b} y^{\al-1} - 2 \b.
\eeq
We consider $ \b < \alb - \al = \e$. To obtain an error decaying in $x$, for a fixed $\al$, we solve the equation 
\beq
G(\b)   = -4 (1-\alb) =  -\tf{8}{3} .
\eeq
\eseq

With the above preparation, we are in a position to prove Theorem \ref{thm:1D_error}.

\begin{proof}[Proof of Theorem \ref{thm:1D_error}]

\textbf{Determining $\b$}. Consider $ \b \in I_{\e}\teq [- \f{1}{2}\e , 0]$. 
Since $\wwwa(y) < 0$ for $y >0$ \eqref{eq:bw_sign}, and $2 \b + \e \geq 0$, we obtain 
\[
  \pa_{\b} G(\b) = 4 \al \int_0^{\infty} \wwwa(y) \ang y^{\b} y^{\al-1} 
  + (2\al \e + 4 \al \b) \int_0^{\infty} \wwwa(y) \log \ang y \ang y^{\b} y^{\al-1} dy - 2 < 0.
\]

Thus $G(\b)$ is decreasing in $I_{\e}$. Next, we derive $G(\b)$. Since $\b \in [-\e/2, 0]$, we get $\he = \alb - \al - \b = \e - \b \in [\e, \f{3}{2}\e]$ and $|\he| \asymp |\e|$. Using \eqref{eq:vel_pow:b} and 
\eqref{eq:J_asym} with $\g = \al + \b$, for $\b \in I_{\e}$, we obtain 
\[
  \cJ_{\al}(\wa)(\infty)
  =  \cJ_{\al + \b}(\wwwa)(\infty) + O(\e)
  = \cca \he^{-1} + O(\e) + O(\he^{-2/3})
  = \cca  \he^{-1} + O(\e^{-2/3}) ,
\]
where $\cca = -6$. Using the above estimate and $|\b| \les \e$, we have the following estimate of $G(\b)$
\beq\label{eq:gb_est}
  G(\b) = (2\al \e + 4 \al \b) ( \cca \he^{-1} + O(\e^{-2/3}))
  = O(\e^{1/3}) +   \cca \f{ 2 \al \e + 4 \al \b}{ \e - \b }.
\eeq

Recall $ \cca =-6$ \eqref{def:WF_para}. Since $\alb - \al = \e < \f{1}{1000}$,  we obtain
\[
  G(0) = O(\e^{1/3}) + 2 \al \cca = -4 + O(\e^{1/3} + \e), \quad G( - \e / 2 ) =  O(\e^{1/3}) .
\]

For $\e < \beps_1$ with $ \beps_1$ small enough, we obtain $G(-\e/2) > - \f{8}{3} > G(0)$. Since $G(\b)$ is continuous in $\b$ and is decreasing in $I_{\e}$, for any $\e < \beps_1$, we obtain a unique solution $\b(\al) \in I_{\e}$ to \eqref{eq:beta}. 
Since $\e-\b \in [\e, \f{3}{2}\e]$, from \eqref{eq:gb_est}, the solution satisfies
\beq\label{eq:gb_est2}
 | G(\b)  - \cca \f{2 \al \e + 4 \al \b}{\e -\b} |
 =  | -\f{8}{3}  - \cca \f{2 \al \e + 4 \al \b}{\e -\b} |  \les \e^{1/3}  ,
 \Rightarrow   |6 (2 \al \e + 4 \al \b ) - \f{8}{3} (\e - \b)| \les \e^{4/3} .
\eeq

Solving the above inequality and using  $|\al - \alb| = \e < \f{1}{1000}$ and $\alb = \f{1}{3}$, we obtain 
\beq\label{eq:beta_est2}
\bga
\B| ( \f{8}{3} + 24 \al)  \b - ( \f83 - 12 \al ) \e \B|
\les \e^{4/3} ,  \
 \Longrightarrow \ \b = \f{2-9\al}{2 + 18 \al } \e + O(\e^{4/3}) = - \f{\e}{8} + O(\e^{4/3}),
  \ega
\eeq
and prove the estimate \eqref{eq:beta_est} and the results for $\beta$ in Theorem \ref{thm:1D_error}.

\vs{0.05in}
\paragraph{\bf Estimate of the error $\crab$}

Since $\b$ chosen above satisfies $\al < \alb, \b \in [-\e, 0]$ \eqref{eq:beta_est2}, 
for $|\al-\alb| = \e$ small enough, the assumptions on $\al, \b$ in Lemma \ref{lem:vel_comp} are satisfied. Thus, we obtain estimates in Sections \ref{sec:err1D_vel}, \ref{sec:err1D_R}.

Recall the estimates \eqref{eq:err_comp_R2}. Since $|\b| \les |\e|$ from \eqref{eq:beta_est}, we have 
\beq\label{eq:err_comp_R3}
 \crab  = \B( (2 \al \e + 4 \al \b) \cJ_{ \al}(\wa) - 4 (1 - \bar \al)( x^{-\he} -1)  -2 \b \B)
 + O(\ddc(x) ) 
 \teq I+   O(\ddc(x) ) .
 \eeq

Next, we estimate $I$.

\paragraph{\bf Case 1: $ \e \lgp x > 1$} Since $\b$ satisfies \eqref{eq:beta}, we obtain 
\[
  I=  (2 \al \e + 4 \al \b) ( \cJ_{\al}(\wa) - \cJ_{\al}(\wa)(\infty) )
   - 4 (1 - \bar \al) x^{-\he}. 
\]

We estimate 
\[
  \cJ_{\al}(\wa)(x) - \cJ_{\al}(\wa)(\infty) 
  =   -\int_x^{\infty} \cca y^{-\alb} y^{\al-1}  y^{\b} 
    -\int_x^{\infty} (\wwwa(y)\ang y^{ \b } - 
\cca y^{-\alb}  y^{\b}  ) y^{\al-1} d y \teq I_1 + I_2.
\]
For $I_1$, since $\al + \b - 1 -\alb = -1 - \he$, we obtain 
\[
  I_1 =  - \int_x^{\infty} \cca y^{ -1 - \he } d y = - \cca \he^{-1} x^{-\he}.
\]

Next, we estimate $I_2$. Using the asymptotics of $\wwwa$ in \eqref{eq:bw_est:a}, for $y > 2$, we have
\[
 | \wwwa  \ang y^{\b} - \cca y^{-\alb}   y^{\b} |
 \leq  | (\wwwa - \cca y^{-\alb}) \ang y^{\b }  |
 +| \cca y^{-\alb} (y^{\b} - \ang y^{\b})|
 \les y^{\b} y^{-\alb} (\log y)^{-1/3} .
\]

Thus, we obtain 
\[
  |I_2| \les \int_x^{\infty} y^{\b -\alb + \al - 1} (\log y)^{-1/3} d y
  \les \int_x^{\infty} y^{-\he - 1} (\log y)^{-1/3} d y
  \les (\lgp x)^{-1/3}  \he^{-1} x^{-\he} .
\]

Combining the above two estimates, using $\he = \e -\b$,  \eqref{eq:gb_est2}, 
and $\e \lgp x > 1$, we obtain
\[
\bal
  I & = (2 \al \e + 4\al \b) \cdot (- \cca  \he^{-1} x^{-\he}) 
  - 4(1-\alb) x^{-\he}
  + O (\e \cdot  (\lgp x)^{-1/3}  \he^{-1} x^{-\he} ) \\
  & =  ( -\cca \f{2\al \e+4 \al \b}{ \e - \b} - \f{8}{3} ) x^{-\he} + O( \lgp x^{-1/3} x^{-\he} )
  = O( \e^{1/3} x^{-\he} ).
\eal
\]

\paragraph{\bf Case 2: $\e \lgp x \leq 1$ and $x \geq 2$}

In this case, using \eqref{eq:vel_pow:b} and Lemma \ref{lem:J_est}, we obtain 
\[
\bal
  I & = O(\e^2) + (2\al \e +4 \al \b) \cJ_{\al + \b}(\wwwa) - 4(1-\alb) (x^{-\he} - 1) \\
  & = O(\e^2) + (2\al \e +4 \al \b) \cdot \cca \f{1 - x^{ -\he }}{ \he } + 4(1-\alb) (1 -x^{-\he} )
  + O(\e  (\lgp x )^{2/3} ) .
  \eal
\]
Since $|1 - e^{-t}| \les |1 + e^{-t}|  t$, for any $k >0$ and $x > 2$, we obtain 
\beq\label{eq:xe_est}
  x^{-\he} = e^{-\he \log x} 
\les_k \la \he \lgp x\ra^{-k}, 
\quad |1- x^{-\he}|  = | 1 - e^{-\he \log x} |  \les \min( \e \lgp x , 1), \quad |x| > 2.
\eeq

Using estimate \eqref{eq:gb_est2}, and the above two estimates,
we bound 
\[
  |I| \les \e^{1/3} |1 - x^{-\he}| + \e (\lgp x)^{2/3} 
\les \e^{1/3} \e \lgp x + \e (\lgp x)^{2/3}  \les \e (\lgp x)^{2/3}.
\]

Combining the estimates in both cases and using \eqref{eq:xe_est} with $k= 1/3$, we obtain 
\[
  |I| \les \min( \e |\lgp x|^{2/3}, \e^{1/3} x^{-\he} ) 
 \les \min( \e \lgp x , 1) (\lgp x)^{-1/3}.
\]

\paragraph{\bf Estimates of $\ddc$}

Recall $\ddc$ from \eqref{eq:err_comp_R3} and $\dda, \ddc$ from \eqref{def:err_del}.
For $x > 1$, since $\e \asymp \he$ and $|\b| \les \e$, using \eqref{eq:xe_est}, we have 
\beq\label{eq:est_dda}
\bga
    \e |\lgp x |^{-1/3}
   \les \min(1, \e \lgp x) | \lgp x|^{-1/3} \\
 \e \dda(x) \les \e \ang x^{-\he} 
  = \e e^{- \he \lgp x } \les \e \ang{ \he \lgp x}^{-1/3} 
  \les  \min(1, \e \lgp x) | \lgp x|^{-1/3}.
\ega
\eeq
Combining the above estimates of $I, \dda$ and the definition 
of $\ddc$ \eqref{def:err_del}, we prove estimate \eqref{eq:err_comp_R4} for $x \geq 2$.
For $x \leq 2$, we have estimated $\crab$ in \eqref{eq:crab_near0}, which implies 
\eqref{eq:err_comp_R4} for $x \leq 2$.
\end{proof}

\subsection{Estimates for the approximate profile}

The approximate profile $(\wa, \va)$ satisfies the following estimates, which will be used 
in Section \ref{sec:1D_profile} to construct the exact $\al$-profile.

\begin{thm}\label{thm:al_appr_profile}
Let $\beps_1, \b$ be chosen as in Theorem \ref{thm:1D_error} and $\e = \f13 - \al$. Define $\ga = \f{1}{x} \va$. 
There exists $ \beps_2 \in (0, \beps_1)$ such that for any $\e \in (0, \beps_2)$, we have 
 the following estimates for the profile. 

\begin{enumerate}[label=(\roman*), leftmargin=1.5em]

\item The function $\cJab =\cJa(\wa),  \cJaa = (1 + \cJab^2)^{1/2}$ defined in  \eqref{eq:wa} satisfies 
\beq\label{eq:Ja_hat}
\bal
\cJaa(x) &=  6 \cdot \he^{-1} ( 1 - \ang x^{-\he} ) +  O( \min( \lgp x, \e^{-1} )^{2/3} ), \quad  \cJaa \asymp \min( \lgp x, \e^{-1} )  .
\eal
\eeq
\item The approximate profile $ \wa$  satisfies 
\bseq\label{eq:wa_est}
\begin{align}
   \B| \f{x \pa_x \wa}{\wa} + \alb - \b \B| & \les | \lgp x|^{-1 - \alb}, 
   \quad  | \wa | \asymp  \min( |y| , \ang y^{-\he - \al} ) , \quad \wa \leq 0 , \ \forall \, x \geq 0     \label{eq:wa_est:a}   , \\
  \B| \pa_x \big( \f{ x \pa_x \wa}{\wa} \big) \B| & \les \ang x^{-1} |\lgp x|^{-1}  \label{eq:wa_est:b} ,
\end{align}
\eseq
 the associated velocity $\va$ satisfies,
 \bseq\label{eq:va_est}
 \begin{align}
 |\va -  2 \al x \cJaa| & \les x,  \quad  & | \pa_x \va - 2 \al \cJaa| & \les 1 ,   \label{eq:va_est:a}  \\
  \ga  = \tf{1}{x} \va  & \asymp   \cJaa,   \quad &  |\pa_x \va | & \les  \cJaa , 
 \label{eq:va_est:b}   \\
      |\pa_x \va - \tf{1}{x} \va  | & \les \ang x^{-\he} , \quad &
   | \pa_x \va - \tf{1}{x} \va - 4 \ang x^{-\he}| & \les \ang x^{-\he} ( |\lgp x|^{-1/3} + \e) .
    \label{eq:va_dif}
 \end{align}

Moreover,  we have the following estimates of the difference 
\beq\label{eq:va_dif2}
  |\va - \vvva| 
\les  \one_{x <1} \e x^2 + \one_{x > 1} x \lgp x \min( \e  \lgp x , 1 ).
\eeq

\eseq

\end{enumerate}

\end{thm}

\begin{proof} \textbf{Proof of \eqref{eq:Ja_hat}} Using estimate \eqref{eq:J_asym:a} with $\g =\al + \b$, 
estimate \eqref{eq:vel_pow:b}, , we obtain 
\[
  \cJab = \cJ_{\al + \b}(\wwwa) 
  + O( \e \min( |x|, 1 ) )
= - 6 \cdot \e_{\g}^{-1}  ( 1 - \ang x^{- \e_{\g}} ) + O( \min(\lgp x ,\e^{-1})^{2/3} )  +  O( \e \min( |x|, 1 ) ).
\]
Since $\e_{\g } = \alb - \g = \alb - \al - \b = \he$ by \eqref{def:para}, we prove 
the first estimate in \eqref{eq:Ja_hat}. Using \eqref{eq:J_asym:b} 
with $\g = \al + \b$ and the above relation, we prove 
the second estimate in \eqref{eq:Ja_hat}.

\paragraph{\bf Proof of \eqref{eq:wa_est:a} }

Since $\wa = \wwwa \ang x^{\b}$, we obtain
\beq\label{eq:wa_iden1}
 \f{x \pa_x \wa}{ \wa } = \f{x \pa_x \wwwa}{\wwwa} +\f{x \pa_x \ang x^{\b}}{\ang x^{\b} } 
 = \f{x \pa_x \wwwa}{\wwwa} + \b \f{ x^2}{\ang x^2}. 
\eeq

Using \eqref{eq:bw_est} and $|\b| \leq 1$ by \eqref{def:eb_ineq}, we obtain
\[
\bal
 \B|\f{x \pa_x \wa}{ \wa } + \alb - \b \B| &= \B| \f{x \pa_x \wwwa}{\wwwa}  + \alb + 
 \b  \f{x^2}{\ang x^2} - \b \B|
  \les |\lgp x|^{-1-\alb}
 + |\b| \cdot \ang x^{-2} 
 \les |\lgp x|^{-1-\alb}.
\eal
\]

Using \eqref{eq:bw_est} and $-\alb + \b =  -\al -\he $ by \eqref{def:para}, we obtain
$\wa \leq 0$ for $x \geq 0$ and 
\[
 |\wa | = \ang x^{\b} |\wwwa| \asymp \min(|x|, \ang x^{-\alb} ) \cdot \ang x^{\b}
 \asymp \min( |x|, \ang x^{-\alb + \b})
 = \min( |x|, \ang x^{ - \al -\he}).
\]

\paragraph{\bf Proof of \eqref{eq:wa_est:b} }
Using \eqref{eq:bw_est:b}, $\wa = \wwwa \ang x^{\b}$, identity \eqref{eq:wa_iden1}, 
and $|\b| \leq 1$ by \eqref{def:eb_ineq}, we prove
\[
 \B| \pa_x ( \f{x \pa_x \wa}{\wa} ) \B|
\leq \B| \pa_x ( \f{x \pa_x \wwwa}{\wwwa} ) \B| + \B| \b  \pa_x \tf{x^2}{\ang x^2} \B|
\les \ang x^{-1} |\lgp x|^{-1} +  \ang x^{-3} \les \ang x^{-1} |\lgp x|^{-1}.
\]

\paragraph{\bf Proof of \eqref{eq:va_dif}}
Using estimate \eqref{eq:err_comp_vel2} and \eqref{def:err_del}, 
and Theorem \ref{thm:reg_alb},  for any $x \geq 0$, we obtain 
\[
  |\pa_x \va - \tf{1}{x} \va - 4 x^{-\he}   |
  \leq | x^{-\he} ( \pa_x \vvva -\tf{1}{x} \vvva - 4)| +
  C \e \ang x^{-\he} 
  \les  | x^{-\he} |\lgp x |^{-1/3} | +    \e \ang x^{-\he} 
\]
Since $ |x -\ang x| \les \ang x^{-1} \les \ang \xx^{-\he} |\lgp x|^{-1/3}$, and $|\pa_x \va - \f{1}{x} \va| \les 1$ (from \eqref{eq:v_profile_est:b}), we prove 
\[
    |\pa_x \va - \tf{1}{x} \va - 4 \ang x^{-\he}   | 
    \les  \e \ang x^{-\he} + | \ang x^{-\he} |\lgp x |^{-1/3} |.
\]

\paragraph{\bf Proof of \eqref{eq:va_est:a} }

Since $ |x - \ang x | \les \ang x^{-1}$ for any $x$, we obtain
$ |\cJaa -\cJab| \les \cJaa^{-1}$. Thus, estimate \eqref{eq:va_est:a} follows from 
\eqref{eq:v_profile_est:b}. 

\paragraph{\bf Proof of \eqref{eq:va_dif2} }

Combining \eqref{eq:err_comp_vel3} and \eqref{eq:est_dda}, we prove \eqref{eq:va_dif2}.

\paragraph{\bf Proof of \eqref{eq:va_est:b} }

Since $\cJaa \gtr 1$ by \eqref{eq:Ja_hat}, estimate of $\pa_x \va$ in \eqref{eq:va_est:b} 
follows from Lemma \ref{lem:J_est} and \eqref{eq:Ja_hat}. Since the estimate \eqref{eq:va_est:a} just proved is uniform in $\e$, 
there exists $M$ large enough and independent of $\e$ such that 
 $\va \gtr x \cJaa$ for any $x \geq M$. 
 Since $\vvva \gtr x$ by \eqref{eq:bw_est:b}, taking $\e$ small enough and using 
 estimate \eqref{eq:va_dif2}, we obtain $\va \gtr x$ for $x \leq M$. Combining 
 these two estimates, we prove \eqref{eq:va_est:b}. 

 We complete the proof of Theorem \ref{thm:1D_error}.
\end{proof}

\section{Analytic construction of the exact $(\f13-\epsilon)$-profile}\label{sec:1D_profile}

In this section, we construct an exact $C^{\infty}$ profile for \eqref{eq:1D_dyn} with parameter $\al = \f13 -\e$ and sufficiently small $\e$ using a fixed point argument. The proof is purely analytic (pen-and-paper)

We choose \(\beta=\beta(\alpha)\) as in Theorem \ref{thm:1D_error}. We recall the approximate profile 
$\wa, \va$ from \eqref{eq:wa}. Since \(\alpha\) is fixed in most of the derivation, we suppress its dependence in the notation, for example writing \(\psioa\) as \(\psio\), and similarly for \(\cX\) and \(\cR\).

\subsection{Fixed point map}\label{sec:1D_fix_pt}

We follow Section \ref{sec:fix_point_map} to reformulate solving the profile equations \eqref{eq:1D_dyn}
\beq\label{eq:1D_dyn_recall}
 2 V \pa_x W = (3- \al - (1-\al) V_x) W .
\eeq
 as a fixed point problem.  Firstly, multiplying \eqref{eq:1D_dyn_recall} by $\f{1}{\wa} $, we obtain 
\[
  2 V \pa_x ( \f{ W }{\wa}  ) = (3 -\al - (1 -\al) V_x - 2 V \f{\pa_x \wa}{\wa} ) \cdot \f{W}{\wa} ,
  \quad V = x + \cK_{\al, 1}(W) .
\]
where $\cK_{\al, 1}$ is defined in \eqref{eq:ker}.  We introduce the residual operator, which is linear in $W$,
\beq\label{eq:err_ep}
 \cR(W) =  3 - \al - (1-\al) V_x - 2 V \tf{\pa_x \wa}{ \wa} , \quad V = x + \cK_{\al, 1}(W).
\eeq

\bseq\label{eq:eqn_equiv_ep}

We decompose $W = w + \wa$ into the approximate profile $\wa$ and perturbation $w$ with vanishing condition $ \lim_{x \to 0} \f{w(x)}{|x|} = 0$.  Since $\f{W}{\wa}(0) =1 + \lim_{x\to 0} \f{w}{\wa} (0) = 1$, integrating both sides, we obtain 
\beq
   \f{w}{ \wa}(x)
    = \f{W}{\wa}(x) -  \f{W}{\wa}(0)
    =  \int_0^x \f{ \cR(W) }{  2 V } \cdot ( 1 + \f{w}{\wa} ) d y ,
   \quad \lim_{x \to 0} \f{w(x)}{|x|} = 0.
\eeq
\eseq

We consider the following map $\cFR$
\bseq\label{eq:fix_map_ep}
\beq
  \cFR( w) = \wa  \int_0^x  \f{\cR( w + \wa)}{2 V} \cdot \f{w + \wa}{ \wa } ,
  \quad V =   x + \cK_{\al, 1}(w + \wa)  ,
\eeq

Then solving \eqref{eq:eqn_equiv_ep} is equivalent to a fixed point problem with the 
 vanishing condition 
\beq
 w =  \cFR(w) , \quad \lim\nolimits_{x \to 0} x^{-1} w(x) = 0.
\eeq
\eseq

\vs{0.05in}
\paragraph{\bf Constant for a far-field estimate}
 Recall the kernel $K_{\alb, \D}$ from \eqref{eq:ker}, which satisfies 
$K_{\alb, \D}(1, z)>0$ for $z>1$ 
and $K_{\alb, \D}(1, z) \leq 0$ for $z \in [0, 1)$ (see \eqref{eq:ker_sign}). 
Using these sign properties and
\[
 K_{\alb, \D}(1, z)  z^{-\alb}= \pa_z f(z), \quad f(z) = z^{1- \alb} ( |z-1|^{\alb} - |z+1|^{\alb} )
 \]
and $f(z) \to -2/3$ as $z \to \infty$ from \eqref{eq:asym_int}, we obtain the following identity:  \beq\label{def:cff}
  \bal
   \tts{ \int}_0^{\infty} | K_{\alb, \D}(1, z) | z^{ -1/3 } d z & = 
   \B| ( \tts{\int}_0^1 - \tts{\int}_1^{\infty} ) K_{\alb, \D}(1, z)  z^{ -1/3 } d z
    \B| 
     = |f(1) - f(0) - (f(\infty) - f(1))| \\
     &  = 2^{4/3} -2/3 \teq \cff,
    \eal    
  \eeq
 The constant $\cff \approx 1.85$ and satisfies $\cff \in (0,2)$. We will use $\cff < 2$ in the following proof.

\paragraph{\bf Choice of parameters}
We choose the following $\kp$-parameters for weight:
\beq\label{def:kp}
\cff \teq 2^{4/3 } -  \f23  \in(0,2), \quad 
  \kp = \f{\cff + 2}{4} \in ( \f{\cff}{2}, 1 ) \subset (0.8, 1), \ 
         \bbb = - b_1 = 1.2 ,\ 
  \kp_1 = \f{1-\kp}{1000}.
\eeq
with $\bb$ chosen in \eqref{def:mu_b}. We treat $\kp, \kp_1$ as absolute constants in the rest of the paper.

\vs{0.05in}
\paragraph{\bf Weights and norms}
Let $\kp, \kp_1$ be  as in \eqref{def:kp}. We define the weights $\vpb, \vpk$ and parameter 
\bseq\label{norm:X_ep}
\beq\label{def:vpb}
\vpb  = |\wa|^{-1} \ang x^{- \e_1} \cJaa^{- \kp} \vpk, \quad 
     \vpk = \exp\big( 9 \kp_1^{-1} \cdot \ang x^{-\e_1}  \big),     
      \quad \e_1 =  \kp_1 \e , 
\eeq
Recall the weights $\vpa$ from \eqref{def:vpa}. 
For some parameter $ \muc >  0$ to be chosen, we define the $\cX$-norm (different from the $\bcX$ norm in \eqref{norm:X}):
\beq
\bal
    || w ||_{\cX} & \teq  \max(  \muc || w ||_{\cXb}, \ \nlinf{ w \vpa }  ) 
    = \| w  \cdot \max( \muc \vpb , \vpa ) \|_{ \linff}  ,
\eal
\eeq
where $\vpa$ is the weight defined in \eqref{def:vpa} for contraction around the 
profile $\wwwa$ in Theorem \ref{thm:near_field_stab}. 
The subscript \(\mf{ctr}\) in \(\muc\) stands for contraction.
\eseq

Our main result in this section is the following.

\begin{thm}\label{thm:1D_solu}
There exists a constant $\muc$ small enough and $c( \muc ) < \f{1}{1000}$ small enough such that for any $\al = \alb - \e$ with $\e \in (0, c( \muc))$, the following statement holds true.  There exists absolute constants $ C_*$ and  $\lam_{\cF} \in (0, 1)$ independent of $\e,  \muc$ such that 
 for any $ \| w \|_{\cX} < \ddd = C_* \e  $, we have 
\beq\label{eq:1D_onto}
   || \cFR(w) ||_{\cX} < \ddd, \quad \ddd = C_* \e.
\eeq
For any $||w_1||_{\cX}, || w_2||_{\cX} < \ddd$, we have 
\beq\label{eq:1D_contra}
     || \cFR(w_1) - \cFR(w_2) ||_{\cX} < \lam_{\cF} || w_1 - w_2 ||_{\cX}. 
\eeq

In particular, the map $\cFR$ admits a fixed point $w$ with $|| w_{\al}||_{\cX} < \ddd$: $\cFR(w_{\al}) = w_{\al}$.

\end{thm}

 In Sections \ref{sec:1D_decomp}-\ref{sec:1D_fix_pt_pf}, we prove Theorem \ref{thm:1D_solu}. In Section \ref{sec:1D_prop}, we further establish several properties of the profile to 
 \eqref{eq:1D_dyn_recall}. Below, we discuss the weights and parameters in \eqref{norm:X_ep}.

\vs{0.1in}
\paragraph{\bf Estimates of weights}

From \eqref{eq:wa}, \eqref{def:para}, 
using definition of the weights  $\vpa$ \eqref{def:vpa}, $\vpb, \vpk$ in \eqref{def:vpb} and of parameters $\kp, \kp_1$ in \eqref{def:kp}, 
we obtain 
\bseq\label{eq:wg_equiv}
\beq\label{eq:wg_equiv:a}
\bga
 \vpa(x)  \asymp   |\xx|^{-\bbb} + 1 , \quad \vpk(x) \asymp 1 ,  \quad  \vpb  \asymp ( |x|^{-1} + |x|^{\alb - \b -\e_1} )  \cJaa^{-\kp} .
  \ega
\eeq 
Since $\bbb = -b_1\in (1, 1.3)$ by \eqref{def:kp}, \eqref{def:mu_b}, using \eqref{eq:wg_equiv:a} and the form of $\vpb$ in \eqref{norm:X_ep},
we obtain
\beq\label{eq:wg_equiv:b}
 \vpa \les |x|^{-\bbb+1} \vpb,
 \quad \max(\vpa, \vpb) |\wa| 
 \les  (|x|^{-\bbb+1} + 1 ) \ang x^{-\e_1} \cJaa^{-\kp}.
\eeq
Thus, for any function $f$, we have
\beq\label{eq:wg_nsingu}
\bal
|f| \max(\vpa, \muc \vpb)
\les | f \wa^{-1}|  \cdot \max(\vpa,  \vpb) |\wa|
\les | f \wa^{-1}| ( |x|^{-\bbb+1} + 1 ) \ang x^{-\e_1} \cJaa^{-\kp}.
\eal
\eeq
\eseq

From \eqref{eq:wg_equiv:a}, since $\muc |\xx|^{\alb} =1 $ when $|\xx| = \muc^{-3}$ and
$\cJaa \les \lgp x $ and since $(\alb - \b - \e_1) \cdot 2.9 < 0.99<1< 1.01< (\alb -\b-\e_1) \cdot 3.1 $, for $\muc < 1$ small enough and $\e < \e(\muc) \ll \muc$, we have
\beq\label{def:Rmu}
\bal
   \vpa(\xx) &> \muc \vpb(\xx) , \quad |\xx| \leq \Rmua =  \muc^{-2.9}, \\   
   \vpa(\xx) &< \muc \vpb(\xx) ,\quad |\xx| \geq \Rmub =  \muc^{-3.1},
  \eal
\eeq

\paragraph{\bf Range of $\he, \e_1$}
Recall $\he = \e - \b$ from \eqref{def:para}. From estimate \eqref{eq:beta_est}, 
by first requiring $\e$ small, we have the following range of different $\e$-parameters:
 \beq\label{ran:ep}
 \he   \in [ \f{9}{8} \e - C \e^{ \f43 }, \, \f{9}{8} \e + C \e^{ \f43 } ]
    \subset [ \f{8.9}{8} \e,  \ \f{9.1}{8} \e ],
    \quad \e_1 = \kp_1 \e \gtr \e.
 \eeq

Below, we list a few remarks about the weights and range of parameters.

\begin{remark}[\bf Two different scales]
The analysis of the fixed-point map $\cFR$ differs fundamentally in the two regimes $\e\log\xx\ll 1$ and $\e\log\xx\gtrsim 1$.  
 This is reflected in the estimates involving $\ang \xx^{ \he }, \min( \lgp x, \e^{-1})$ in Theorem \ref{thm:al_appr_profile}.  For $\e$ small enough and $x$ with  $\e \log x \ll 1$, 
we have $\ang x^{\e} = 1 + O(\e \lgp \xx)$
and $\cFR$ is a small perturbation to $\cF_{\alb, \R}$ for $\al = \f{1}{3} - \e$.
To control nonlocal terms perturbatively, especially $2(\alb- \b) ( \psiox - \f{\psio}{x} )$ appearing 
in $\td \cR(\ww)$ in \eqref{eq:lin_nloc_ep:b}, we design the decaying weight $\cJaa^{-\kp}$ in $\vpb$ 
\eqref{def:vpb}. 
We use $\cJaa^{-\kp}$ and the outgoing property, $\va > 0$, 
 to generate an effect  similar to an $O(1)$ damping term in the stability analysis
 \cite{chen2019finite,chen2021HL,ChenHou2023a}.

In the range of $\e \log x \gtr 1$, we have $\ang x^{\e} \gtr 1$. We use the weight 
$\ang x^{-\e_1}$ and $ \vpk$ to generate similar ``damping" effect to control 
 nonlocal terms perturbatively.
We choose the ratio \(\kp_1=\e_1/\e\) to be small \emph{relative} to \(1-\kp\) in order to close the fixed-point argument.
\end{remark}

\begin{remark}[\bf Role of $\vpk$]
While $\vpk$ is comparable to $1$ by \eqref{eq:wg_equiv}, 
\footnote{
The factor $9$ in  $ 9\kp_1^{-1}$ in $\vpk$ \eqref{def:vpb} can be replaced by any large \emph{absolute} constant. 
}
for $\e \log x \gtr 1$, we use $\vpk$ to generate a $O(1)$ ``damping" effect with main term given by
\[
   \f{2 V \pa_x \vpk}{\vpk} 
   =  4 \al x  \cJaa \cdot 9 \kp_1^{-1} \pa_x   \ang x^{-\e_1}  + l.o.t.
   = - 36 \al   \e \cJaa \cdot \ang x^{-\e_1} + l.o.t.,
   \quad  \e \cJaa \cdot \ang x^{-\e_1} \asymp  \ang x^{-\e_1},
\]
where l.o.t. denotes some much smaller lower order term and we use $\e_1 = \kp_1 \e$. 
Using  the above damping terms and similar terms arising from $\ang x^{-\e_1}$ in $\vpb$,
we control $(\e + 2 \b) \psiox$ in $\td \cR(w)$ perturbatively, 
\emph{without} requiring  $\kp_1 = \e_1 / \e$ to be bounded below.

\end{remark}

\begin{remark}[\bf Range of $\e_1$]
We choose $\e_1  = \kp_1 \e$ with an $\e$-independent constant $\kp_1 \gtr 1$, which may be \emph{arbitrary small},  to gain a factor $\ang x^{-\e_1}$ in various estimates, e.g. estimate of $(\e + 2 \b) \psiox$,  and to control the factor $\log x$ for large $x$. See Lemma \ref{lem:lgx_pow}.

\end{remark}

\begin{remark}[\bf Range of $\kp$ in \eqref{def:kp}]

The same proof works for any $\kp \in (\f{\cff}{2}, 1)$ with $ \tf{1}{2}\cff \approx 0.927$. We require \(\kp>\f{\cff}{2}\) to treat the nonlocal term \(\psiox-\f{\psio}{x}\) perturbatively; see \eqref{eq:1D_sharp_bound} and Remark \ref{rem:sharp_bound}. 
\footnote{
This range can be relaxed to \([0,1)\) by exploiting the more technical 
nonlocal cancellation in Section \ref{sec:log_cancel} and designing a weight \(\rho\) similar to \(|\log x|^{\alpha}\) in \eqref{eq:rho1}. To simplify the presentation, we do not use this refinement.
}
To close the fixed point argument for Theorem \ref{thm:1D_solu}, we require $\kp < 1$. 
We gain a small factor of order $\e^{c-\kp} $ with $ c \approx 1$ 
in nonlinear estimates so that we can treat nonlinear terms  perturbatively.

\end{remark}

\paragraph{\bf Idea of proof for Theorem \ref{thm:1D_solu} }

We estimate the map $\cFR$ for $x \in [0, \Rmub]$ by perturbing 
the near-field contraction estimate in Theorem \ref{thm:near_field_stab}. 
For large $x \geq \Rmua \gg 1$, by exploiting the far-field cancellations in Theorem \ref{thm:al_appr_profile}, we extract the main term $\f23 (\psiox - \f{1}{x} \psio)$ in \eqref{eq:lin_nloc_ep} in the linear estimate of $\cFR$, and use identity \eqref{def:cff} and $\cff < 2$ to control it. Other terms are small and we estimate them perturbatively. 
Combining estimates in these two overlapping regions, we prove Theorem \ref{thm:1D_solu}.

\subsection{Decomposition of the map}\label{sec:1D_decomp}
We denote $W$ and $V$ as follows with $ w$ being the perturbation 
\bseq\label{eq:1D_vw}
\beq
  V =   \va  + \psio(w)  , \quad  W = w + \wa, \quad \psio(w)  = \cK_{\al, 1 }(w),
  \quad  \va = x + \cK_{\al, 1}(\wa),
\eeq
and introduce $\ga$ and $\cJab$
\beq
    \ga \teq \tf{1}{x} \va, \quad \cJab \teq \cJ_{\al}( \wa ).
\eeq
\eseq

From \eqref{eq:err_ep}, we decompose the error as follows 
\bseq\label{eq:lin_nloc_ep}
\beq\label{eq:lin_nloc_ep:a}
\bal
      \cR(w + \wa )  & = \cR( \wa )  + \td \cR(w)  =      \crab + \td \cR(w), 
  \eal
\eeq
where we use $\crab$ from \eqref{eq:err_crab} and define $\td \cR$ as 
\beq
 \cR(\wa) = \crab,
 \quad     \td \cR( w )  \teq - (1-\al) \psiox (w) - 2 \psio  \tf{\pa_x \wa}{\wa}  .
\eeq

Using $\e = \alb - \al, 1-\al = 2 \alb + \e$, we rewrite $\td \cR(w)$ as 
\beq\label{eq:lin_nloc_ep:b}
\tcr(w) = - ( \e + 2 \b) \psiox - 2 (\alb - \b) ( \psiox  - \f{\psio}{x} )
  - 2 \psio \B(  \f{\pa_x \wa}{\wa} + (\alb - \b) \f{1}{ x} \B).
\eeq
\eseq

Using these notations, we decompose the fixed point map \eqref{eq:fix_map_ep} as
\bseq\label{eq:F_lin_ep}
\beq
  \cFR(w) = \cLab(w) + \cEab(w) + \cNab(w),
\eeq
where $\cLab,  \cNab, \cEab$ denote the linear part, nonlinear part, and the error part
\beq
\bga
  \cLab (w) \teq \wa  \int_0^{x}  \f{ \td \cR(w)}{2 \va }  ,  \\
 \cNab(w) \teq \wa  \int_0^x \td \cR(w) \cdot \B( \f{W}{2V \wa} - \f{1}{2 \va}  \B) ,
 \quad 
  \cEab(w) \teq  \wa \int_0^x  \f{ \cR(\wa)}{2 V} \cdot \f{W}{\wa}  , 
 \ega
\eeq
\eseq

In the following subsections, we will prove that $\cLab$ is a contraction in Theorem \ref{thm:contra_lin}, and estimate $\cNab, \cEab$ perturbatively.

\subsection{Estimates of nonlocal terms}
We have the following nonlocal estimates for $\cK_{\al,i}, \psioa$ \eqref{eq:ker}.

\begin{prop}\label{prop:vel_al}
Let \(\kp \in (\alb,1)\) be the parameter chosen in \eqref{def:kp}.
\footnote{
  The estimates in Proposition \ref{prop:vel_al} hold for much larger range of $\kp$. We restrict $\kp$ to this range since we only use $\kp$ in this range (see \eqref{def:kp}) and this assumption simplifies the proof.
}
 Let $\psio  = \cK_{\al, 1}(w), \psiox= \cK_{\al, 2}(w), \e =\alb -\al, \he =\alb -\al - \b$. 
For $w \in \cX$, we have 
\beq\label{eq:v_al_est}
   | \psiox + 2 \al \cJ_{\al}(w)| 
+   | \tf{1}{x} \psio(x) + 2 \al  \cJ_{\al}(w) |  \les 
 \cJaa^{\kp} x \ang x^{- \he - 1 + \e_1 } \nlinf{ w \vpb } .
\eeq

For $\psiox  - \f{ \psio }{x}$, we have the following estimates with a sharp constant for the main term
\beq\label{eq:vmix_sharp}
   |\psiox - \tf{1}{x} \psio | \leq ( 6 \cff \vpk^{-1} \cJaa^{\kp} + C \cJaa^{\kp - 1/3  } ) \ang \xx^{\e_1 - \he} \nlinf{ w \vpb    }  .
\eeq

The absolute constants in the above estimates are independent of $\e, \al, \b, \kp, R_1$.
For $\cJ_{\al}$, we have 
\bseq\label{eq:Ja_est}
\begin{align}
|\cJ_{\al}(w) |&   \leq  \B( 2 \vpk^{-1} |\cJaa(x)|^{\kp + 1} + C  | \cJaa(x) |^{\kp + 2/3}  \B)  || w ||_{\cX}, \label{eq:Ja_est:a}
 \\
|\cJ_{\al}(w) | & \les \min(x^{\al + 1},  \cJaa^{\kp + 1} )  \nlinf{ w \vpb} ,
\label{eq:Ja_est:b} 
\end{align}
\eseq
with some absolute constant $C$.

\end{prop}

\begin{remark}[\bf Explicit constants] \label{rem:expl_C1}
We use the explicit constant $6 \cff$ in the main terms in \eqref{eq:vmix_sharp} to prove the contraction property of the linear part of the map \(\cFR\). \footnote{Although we derive the explicit constant \(2\) in \eqref{eq:Ja_est:a}, we do not need this constant to be small. If \(2\) in \eqref{eq:Ja_est:a} is replaced by another constant \(C\), one can replace the weight \(\vpk\) by \(\vpk^{\,C+1}\) and still close the fixed point argument.}
See Remark \ref{rem:sharp_bound}.

\end{remark}

\begin{proof}

\textbf{Proof of \eqref{eq:v_al_est}} Recall the estimate of $\vpb$ from \eqref{eq:wg_equiv}. 
For $x > y$, we get
\[
  \lgp x = \log( x + 2 ) \leq \log( \f{x}{y} + 2 )( y + 2 )
  \leq \lgp y  + \lgp (\f{x}{y} + 2 )
  \les \lgp y \cdot \lgp \f{x}{y}, 
\]
For $x \leq y$, we have $\lgp x \leq \lgp y$. 
Since $\cJaa(x) \asymp \min( \lgp x, \e^{-1} )$, for any $\kp \in (0, 1)$, we obtain
\beq\label{eq:J_rat_bd}
  \f{\vp(x)}{\vp(y)} =  \f{ \cJaa^{\kp}(y) }{  \cJaa^{\kp}(x) } \les  (\lgp \f{x}{y} + \lgp \f{y}{x})^2 .
\eeq

Since $\psiox + 2 \al \cJa = \cK_{\al, 2, J},  \f{1}{x} \psio  + 2 \al \cJa = \cK_{\al, 1, J}$ \eqref{eq:Ja_hat}, using Lemma \ref{lem:vel_al} with $ b = \alb - \b - \e_1 $ and  $\Gam(x) = \cJaa^{-\kp}(x)$, we prove 
\[
\bal
     |\psiox + 2 \al \cJ_{\al}(w)| 
+   | \tf{1}{x} \psio(x) + 2 \al  \cJ_{\al}(w) |  & \les 
 \vp^{-1} x^{1+ \al} \ang x^{ - \alb + \b + \e_1-1} \nlinf{ w ( |x|^{-1} + |x|^{\alb - \b- \e_1 } ) \vp } \\
& \les \cJaa^{\kp} x \ang x^{ \al -\alb + \b + \e_1 - 1 } \nlinf{ w \vpb }. 
\eal
\]
Using $ \he =\alb - \al -\b$ \eqref{def:para}, we prove \eqref{eq:v_al_est}.

\vs{0.1in}

\paragraph{\bf Proof of \eqref{eq:vmix_sharp}}

Using the kernel $K$ in \eqref{eq:ker}, the definition of $\vpb$ in \eqref{norm:X_ep}, we obtain
\beq\label{eq:vel_al:pf1}
\bal
  | \psiox  - \f{\psio }{x} | & = \B| \cK_{\al, 2}(w)(x) - \f{1}{x} \cK_{\al, 1}(w)(x) \B|
 = \B| \int_0^{\infty} K_{\al, \D}(x, y) w(y) dy \B| \\
& \leq \nlinf{ w \vpb } \int_0^{\infty} K_{\al, \D}(x, y) \vpb^{-1}(y) dy,
\quad \vpb^{-1} =  |\wa(y)|  \cdot \ang y^{\e_1} \cJaa(y)^{\kp} \vpk^{-1}.
\eal
\eeq

Using \eqref{eq:bw_est} and the definition of \eqref{eq:wa}, for $y \geq 1$, we obtain
\[
  |\wa(y)|  \cdot \ang y^{\e_1} 
  = |\wwwa| \ang y^{\e_1 + \b}
  \leq (6 y^{-\alb} + C y^{-\alb} |\lgp y |^{-\alb} )  \ang y^{\e_1 + \b}
  \leq y^{-\alb +\b + \e_1 } (6 
+ C   |\lgp y |^{-\alb} ) .
\]
 Next, we estimate 
\beq\label{eq:vel_al_A}
A(y) = |\cJaa(y)|^{\kp} \vpk^{-1}(y).
\eeq
Clearly, $A(y)$ is increasing by definition \eqref{norm:X_ep}. For $y> x$, 
using mean-value theorem, $\vpk \asymp 1$ by \eqref{eq:wg_equiv}, and $\cJab = \cJa(\wa)$ \eqref{eq:wa}, for some $\xi \in [x, y]$, we obtain 
\[
  A(y) \leq A(x) + A^{\pr}(\xi) (y- x) ,
  \quad A^{\pr}(\xi) \les  |\cJaa(\xi)|^{\kp-1} |\pa_y \cJa(\wa)|
  + |\cJaa(\xi)|^{\kp}  |\pa_{\xi} \ang \xi^{-\e_1}|.
\]
From the definition of $\cJa$, $|\wa| \les \min(|x|, |x|^{-\alb + \b})$,
and $\cJaa \les \e^{-1}$, we obtain
  \[
  |\cJaa(\xi)|^{\kp}  |\pa_{\xi} \ang \xi^{-\e_1}|
  \les   |\cJaa(\xi)|^{\kp} \e \ang \xi^{-\e_1 -1}
  \les  |\cJaa(\xi)|^{\kp-1} |\xi|^{-1}, \quad |\pa_y \cJa(\wa)(\xi) | \les \wa(\xi) |\xi|^{\al-1} \les |\xi|^{-1}.
     \]
Since $A, \cJaa$ are increasing and $\xi \in [x, y]$ , we obtain
  \[
A^{\pr}(\xi) \les |\cJaa(\xi)|^{\kp-1} \xi^{-1}
\les \cJaa(x)^{\kp-1} x^{-1}.
  \]

Since $A(y) \leq A(x)$ for $y \leq x$, using the above estimate, we derive 
\[
  A(y) \leq A(x)  + C \one_{y> x} |\cJaa|(x)^{\kp-1} \cdot \tf{y-x}{x}
\]

Similarly, since $ \f13 < \kp < 1$ by \eqref{def:kp} and $\cJaa(y)$ is increasing, we obtain
\[
   |\cJaa(y)|^{\kp-\alb} \leq  
   |\cJaa(x)|^{\kp - \alb} +C  \one_{y > x} |\cJaa(x)|^{\kp - 4/3} \cdot \tf{y-x}{x}.
\]

Using $ 1 \les \cJaa(y) \les \lgp y$, $\alb = \f13$, 
$A(y) \les \cJaa(y)$, and combining the above estimate of $A(y)$ and $|\wa| \ang y^{\e_1}$, we obtain
\beq\label{eq:vel_al:pf2}
\bal
  \vpb(y)^{-1}
  & \leq  y^{-\alb +\b + \e_1 } (6 +  |\lgp y |^{- 1/3} )  A(y)
\leq  y^{-\alb +\b + \e_1 } ( 6 A(y) +  C |\cJaa(y) |^{\kp - 1/3 } ) \\
& \leq y^{-\alb +\b + \e_1 }  ( 6 A(x) + C |\cJaa(x)|^{\kp- 1/3}
+ C \one_{y > x}  |\cJaa(x)|^{\kp-1} \cdot \tf{y - x}{x}  ) .
\eal
\eeq

Recall from \eqref{eq:ker} that $ K_{\al, \D}(x, y)$ is $\al-1$-homogeneous. To obtain the upper bound of \eqref{eq:vel_al:pf1}, 
using a change of variable $y = x z$, we obtain
\[
  \int_0^{\infty} K_{\al, \D}(x, y) y^{-\alb +\b + \e_1 }   g( \f{y}{x} ) d y 
= x^{\al -\alb +\b + \e_1}  \int_0^{\infty} K_{\al, \D}(1, z) z^{-\alb +\b + \e_1 }  g( z  ) d z,
\]
where $g(z) = 1$, or $g(z)= \one_{z > 1} (z-1)$. Using the decay estimate \eqref{eq:K_decay:mix}, we bound 
\[
\bal
   \int_0^{\infty} | K_{\al, \D}(1, z )|  z^{-\alb +\b + \e_1 } ( 1 +  \one_{ z > 1} (z-1 ) ) d z \les 1.
\eal
\]

Applying the above estimates and \eqref{eq:vel_al:pf2} to \eqref{eq:vel_al:pf1}, and using $ \al - \alb + \b = -\e + \b = -\he$ by \eqref{def:para}, we obtain
\[
  |\psiox  - \f{\psio }{x}|  \leq x^{-\he + \e_1} 
  \B( 6 A(x) \cdot  I +  C \cJaa^{\kp- 1/3 } \B),
  \quad 
  I \teq \int_0^{\infty} |K_{\al, \D}(1, z) z^{-\alb +\b + \e_1} | d z.
\]

Finally, we evaluate the integral $I$ using $\pa_a s^a = \log s \cdot s^a$, mean value theorem,
$|\b| , \e_1 \les \e $ by \eqref{eq:beta_est} and \eqref{norm:X_ep},  we obtain
\[
|z^{-\alb +\b + \e_1} - z^{-\alb}|
\les |\log z| \cdot \e ( z^{-\alb + \e_1 + \b} + z^{-\alb} ),
\]
which along with the estimates \eqref{eq:K_decay:mix} and \eqref{eq:K_err_decay:mix_err}
on $K_{\al, \D}$ imply 
\[
\bal
  & |K_{\al, \D}(1, z) z^{ -\alb +\b + \e_1 } - K_{\alb ,\D}(1, z) z^{-\alb}|
 \leq | K_{\al, \D}(1, z) ( z^{ -\alb +\b + \e_1 } -z^{-\alb}  )|
+ | K_{\al, \D}(1, z) -  K_{\alb ,\D}(1, z) | z^{-\alb} \\
& \qquad \qquad \les \e |\log z|    ( z^{ -\alb +\b + \e_1 } + z^{-\alb} )  
\min( |z-1|^{\al-1}, z^{\al-3} )   + \e z^{-\alb}  (  \one_{z \leq 3}  |z-1|^{\alb/2-1}  + \one_{z > 3}   z^{\alb + \f{1}{50} - 3}  ).
\eal
\]
The above term is $L^1$-integrable with upper bound of the integral uniformly in $\al$.
Using the above estimate and the bound of $\cff$ \eqref{def:cff}, we obtain
\[
  I \leq C \e + \tts{\int}_0^{\infty}  | K_{\alb ,\D}(1, z) z^{-\alb} |
\leq C \e + \cff.
\]

Since $ \e \cJaa^{\kp} \les \cJaa^{\kp -1}  $ by \eqref{eq:Ja_hat}, 
and $x^{-\he + \e_1} \leq \ang x^{-\he + \e_1} + O(\ang x^{-1})$ for $x \geq 1$, using the above estimates and $A(x)= |\cJaa(x)|^{\kp} \vpk^{-1}(x)$ from \eqref{eq:vel_al_A}, we prove 
\[
   |\psiox - \tf{\psio }{x}| \leq \ang x^{-\he + \e_1}  ( 6 \cff |\cJaa(x)|^{\kp} \vpk^{-1}(x)  
   + C |\cJaa(x)|^{\kp- 1/3} ),
\]
which is \eqref{eq:vmix_sharp}, for $x \geq 1$. For $x \leq 1$, \eqref{eq:vmix_sharp}  follows from \eqref{eq:v_al_est} and triangle inequality.

\vs{0.1in}

\paragraph{\bf Proof of \eqref{eq:Ja_est:a} }

  Using \eqref{eq:bw_est}, $\wa = \wwwa \ang x^{\b} $ \eqref{eq:wa}, 
  and $ |x^{\b} - \ang x^{\b}| \les \ang x^{\b - 2} \les \ang x^{-1}$ for $x > 1$, we obtain
  \[
    |\wa(y)| \leq \min( C |y|, 6 |y|^{-1/3 + \b} + C |y|^{-1/3 + \b} |\lgp  y|^{-1/3} ) .
  \]

Using the definition of $\cJa$ \eqref{eq:Jw},  $\cJaa(y) \les \lgp y$ \eqref{eq:Ja_hat}, 
$\al \leq \f13 < \kp$ by \eqref{def:kp},  the properties that $\cJaa, \vpk^{-1}$ are increasing by \eqref{def:vpb}, and $\nlinf{w \vpb }$-norm \eqref{norm:X_ep}, we obtain
\[
\bal
  |\cJa(w) | &\leq \int_0^{x}  y^{\al-1} |\wa|  \cJaa^{\kp} \ang y^{\e_1} \vpk(y)^{-1} dy \cdot 
\nlinf{w \vpb } \\
& \leq \B(C +  6 \one_{x > 1}   \int_1^{\ang x}  ( |\cJaa(y)|^{\kp} + C |\cJaa(y)|^{\kp-1/3} )  y^{\al - 1 - 1/3 +\b + \e_1} 
  d y \B) \cdot \vpk^{-1}(x) \nlinf{w \vpb } \\
& \leq \big( C + \f{6}{\he -\e_1} (1 - \ang x^{- (\he - \e_1)})\cdot \vpk^{-1}(x) ( |\cJaa(x) |^{\kp}
+ C |\cJaa(x)|^{\kp-1/3}  ) \big)  \nlinf{w \vpb } ,
\eal 
 \]  
where $\he = \f{1}{3} - \al - \b$ is defined in \eqref{def:para}.  Using \eqref{eq:Ja_hat}, $ \he - \e_1 < \he$, 
and $\he , \he - \e_1 \asymp \e$, we obtain 
\[
  \cJaa + C \cJaa^{2/3} \geq  \tf{6}{\he } (1 - \ang x^{- \he } ) 
  \geq \tf{6}{\he } (1 - \ang x^{- (\he - \e_1) } ) .
\]

Since $\cJaa \gtr 1,\vpk \asymp 1$, and $ \e_1 \leq \f{1}{2 } \he $ by \eqref{ran:ep}, combining the above two estimates, we prove 
\[
\bal
  |\cJa(w) | & \leq \f{\he}{\he - \e_1} ( \cJaa \vpk^{-1}  + C \cJaa^{2/3} )
(|\cJaa^{\kp}(x) + C |\cJaa(x)|^{\kp-1/3} )  \nlinf{w \vpb }  \\
& \leq 2 ( \cJaa^{\kp+1} \vpk^{-1}  + C \cJaa^{\kp+2/3} ) \nlinf{w \vpb } .
\eal 
\]

\paragraph{\bf Proof of \eqref{eq:Ja_est:b} }

A direct calculation yields 
\[
   | \cJ_{\al}(w) | \leq \int_0^x y^{\al-1 } |w| d y \les \int_0^x |\wa| y^{\al-1} \lgy^{\kp} d y \cdot\nlinf{ w \vpb } \teq I \cdot \nlinf{ w \vpb }.
\]
Since $|y^{\al-1} \wa | \les \min(y^{\al},  y^{-\he})$, we get $| \cJ_{\al}(w) | \les x^{1+\al} 
\nlinf{ w \vpb }$ and prove the first bound in \eqref{eq:Ja_est:b}. For the second bound, from \eqref{eq:wa_est:a}, we obtain $|\wa| y^{\al-1} \les \ang y^{-1 - \he}$. Using change of
variables $t = \log z$, $s = \he t$, and $\kp > 0$, 
for $x >2$, we yield 
\[
I 
\les \int_0^{x} \ang z^{- \he - 1} |\lgp z|^{  \kp } d z
  \les 1 + \int_{\log 2}^{\log x } e^{- \he t } t^{\kp} d t
  \les   \he^{-k-1} \one_{x >2} \int_{\he \log 2}^{ \he \log x } e^{- s } s^{\kp} d s
  \les \he^{-k-1} \min( \he \log x, 1)^{k+1}.
\]
Since $\he \asymp \e$ by \eqref{ran:ep}, we prove $I \les \cJaa^{\kp+1}$ 
and thus prove the second bound in \eqref{eq:Ja_est:b}. 
\end{proof}

\subsection{Estimates in $x \leq \Rmub$}

Recall $\Rmua <\Rmub$ from \eqref{def:Rmu}.  For $x \leq \Rmub$, we show that the difference between
 $\crbb(w)$ in \eqref{eq:lin_near:def} and $\td \cR_{\al}(w)$ in \eqref{eq:lin_nloc_ep} is very small, and then use Theorem \ref{thm:near_field_stab} to estimate $\td \cR(w)$.  Below, we consider $x \leq \Rmub$.

\subsubsection{Difference between various terms}
Recall that $\Rmub$ only depends on $\muc$.

\vs{0.1in}
\paragraph{\bf Difference between power}
From \eqref{eq:beta} and $\b \log \ang x < 1$, for $x \leq \Rmub$, we have 
\beq\label{eq:x_ep}
\bal
  |\ang x^{\pm \b} - 1|  &\leq C(\muc) |\b| \log \ang x
   \leq C(\muc) |\b| ( \ang x - 1 ) \leq  C(\muc) \e \, x^2   ,\\
  |x^{ \e} - 1|  & \leq C(\muc)\, |\b  \log  x  | \, (x^{\e} + 1) .
 \eal
\eeq

\paragraph{\bf Difference between velocity}

Recall from \eqref{eq:ker}:
\beq\label{eq:v_cKa}
\psio  = \cK_{\al, 1}(w), \quad \psiox  = \cK_{\al, 2}(w),
\quad  \psiox  + 2\al \cJa(w) = \cK_{\al, 2, J}(w),
\quad  \psio  +  2\al x \cJa(w) = \cK_{\al,1,J}(w).
\eeq

For $x \leq \Rmub$, applying \eqref{eq:va_dif2}, we obtain 
\bseq\label{eq:near_dif_bv}
\beq
  |\va-  \vvva| \les    \one_{x \leq 1} \e x^2 + \one_{x > 1} \e x |\lgp \Rmub|^2
  \les \e C(\muc) \, x.
\eeq

Using the above estimates and $\va \gtr x$ by \eqref{eq:va_est}, for $|\xx| \leq \Rmub$,  we obtain
\beq\label{eq:near_v_rat}
 \f{ |\va -\vvva| }{ |\va| }   \les \e C(\muc) x \cdot x^{-1} \les \e C(\muc ),
 \quad \Longrightarrow \quad 
    \f{|\vvva|}{|\va|}   \leq 1 + \e C(\muc). 
\eeq
\eseq

Applying \eqref{eq:v_al_gen} with $\Gam = \cJaa^{-\kp}, b = \alb - \b$ and \eqref{eq:Ja_est}, 
for any $x \leq \Rmub$, we obtain
\beq\label{eq:near_est1}
\bal
|\cK_{\al, i, J}(x)|
 &  \leq  C x^{3-i+\al} \cJaa^{\kp}  || w||_{\cX}
 \les C(\muc)  x^{3-i+\al} || w||_{\cX},  \\
|\cJ_{\al}(w)(x)| & \leq C(\muc)  x^{\al+1}  || w ||_{\cX}, \\
|\cK_{\al, i}(x)| 
& \les |\cK_{\al, i, J}(x)| + x^{2-i} |\cJ_{\al}(w)| 
\les C(\muc) x^{3 + \al -i}  || w ||_{\cX}. \\
\eal 
\eeq

Next, we compare $\cK_{\al,i}(w)$ and $\cK_{\alb, i}(w)$. 
Using the relation \eqref{eq:Jw}, we decompose 
\[
  \cK_{\al,i}(w) - \cK_{\alb,i}(w) 
  =  ( \cK_{\al,i, J}(w) - \cK_{\alb,i, J}(w) )  
  - (2 \al x^{2-i} \cJa(w)  - 2 \alb x^{2-i} \cJ_{\alb}(w) ) = P_1 + P_2.
\]

Using Lemma \ref{lem:vel_comp_gen} with $\g = \al$, $\e = \alb - \al$, 
$|| w x^{-1}|| \les || w||_{\cX}$, and \eqref{eq:near_est1}, for $x \leq \Rmub$, we obtain 
\[
\bal
  |P_1|  & \leq |\cK_{\al,i,J}(w) (1 - x^{\e} )|
  + |x^{\e} ( \cK_{\al,i,J}(w) - x^{-\e} \cK_{\alb, i, J}(w) ) | \\
 & \les C(\muc)    x^{3+\al-i}  | 1 - x^{\e}|  \cdot || w ||_{\cX} 
 + \e x^{\e} x^{2-i} x^{\al + 1} || w ||_{\cX} \\
\eal
\]
Using \eqref{eq:x_ep} and $x^{\al} |\log x|  (x^{\e} + 1) \leq C(\muc)$ for $x \leq C(\muc)$,  we further bound
\[
 |P_1|  \les  C(\muc) \e   x^{3-i}  || w||_{\cX}
 \]

For $P_2$, applying \eqref{eq:vel_comp_gen:b} with $\g = \al$ in Lemma \ref{lem:vel_comp_gen}, 
$\max_{y \leq \Rmub} |w(y)|(|y|^{-1} + |y|^{\alb}) \leq C(\muc) || w||_{\cX}$ \eqref{eq:wg_equiv} 
(only the $L^{\infty}$-norm in region $y\leq x$ is needed in \eqref{eq:vel_comp_gen:b}), and \eqref{eq:near_est1}, 
for $x \leq \Rmub$, we obtain 
\[
\bal
  |P_2| 
& \leq 2 | (\al - \alb) x^{2-i} \cJa(w)| + 2 \alb x^{2-i} | \cJa(w) - \cJ_{\alb}(w)| \\
& \leq C(\muc) \e  x^{3-i+\al}  || w||_{\cX}
+  C(\muc) \e \cdot x^{3-i} || w||_{\cX}
\leq C(\muc) \e \cdot x^{3-i} || w||_{\cX} .
\eal
\]

Combining the above estimates on $P_1, P_2$, we bound 
\beq\label{eq:near_dif_v}
  |   \cK_{\al,i}(w) - \cK_{\alb,i}(w) | \les  C(\muc) \e \cdot x^{3-i} || w||_{\cX} .
\eeq

\paragraph{\bf Difference between profiles}

Using \eqref{eq:beta},\eqref{eq:wa},\eqref{eq:x_ep}, and $|\wwwa| \gtr C(\muc) x$ for $x \leq \Rmub$,  we get 
\bseq\label{eq:near_dif_w}
\beq
\bal
   |\wa(x) - \wwwa(x) |  &= |(\ang x^{\b} - 1) |\wwwa| 
  \leq  C(\muc) \e x^2  |\wwwa|,  \\
  |\wa^{-1} - \wwwa^{-1}| & = | (\ang x^{-\b} - 1) \wwwa^{-1} | 
  \leq C(\muc) \e x^2     |\wwwa^{-1}| 
  \les  C(\muc) \e  x . \\
\eal 
\eeq
For $x \leq \Rmub$, from the above estimates, we derive 
\beq\label{eq:near_w_rat}
\tf{|\wa|}{|\wwwa|}
\leq 1 +  |\tf{\wa - \wwwa}{\wwwa} | 
\leq 1 + C(\muc) \e.
\eeq
\eseq

\paragraph{\bf Difference between $ \crbb$ and $\td \cR(w)$}

Recall  $\crbb$ from \eqref{eq:lin_near:def} and $\td \cR(w)$ from \eqref{eq:F_lin_ep}
\beq\label{eq:near_recall_R}
\bal
     \crbb(w) & = - (1-\alb) \psio_{\alb, x} - 2 \psio_{\alb} \tf{\pa_x \wwwa}{\wwwa},
   \quad  \psio_{\alb}(w) = \cK_{\alb, 1}(w),    \\
    \td \cR( w )  & = - (1-\al) \psio_x (w) - 2 \psio  \tf{\pa_x \wa}{\wa} ,
    \quad \psio(w) = \cK_{\al, 1}(w).
\eal
\eeq
Here, we simplify $\psioa$ as $\psio$, and keep the subscript only for \(\psio_{\alb}\) to distinguish it from \(\psio=\psio_{\al}\).

Using the definition and $\e = \alb - \al$, we decompose
\[
\bal
     \crbb(w) - \td \cR(w)
       & = - (1-\alb) ( \psio_{\alb, x} - \psio_x )- 2 (\psio_{\alb} - \psio ) \f{\pa_x \wwwa}{\wwwa}
       + \e \psio_x 
       - 2 \psio( \f{\pa_x \wwwa}{\wwwa} - \f{\pa_x \wa}{\wa} ) \\
     &  = P_1 + P_2 + P_3 + P_4.
  \eal
\]

Using \eqref{eq:wa} and \eqref{eq:near_est1}, we obtain
\[
  |P_4| = \B|2 \psio( \f{\pa_x \wwwa}{\wwwa} - \f{\pa_x \wa}{\wa} )\B|
  = \B|2 \psio \cdot \f{\pa_x \ang x^{\b}}{ \ang x^{\b} } \B|
  \les |\b \psio \cdot \f{x}{ \ang x^2 } | \leq C(\muc) \e x || w||_{\cX}.
\]

Using $ |\wwwa| \gtr C(\muc) x$ for $x \leq \Rmub$ and \eqref{eq:near_dif_v}, we obtain 
\[
  |P_2| \leq C(\muc) \e x^2 \cdot x^{-1} || w||_{\cX} \leq C(\muc) \e x  || w||_{\cX}.
\]

Using \eqref{eq:near_dif_v} and \eqref{eq:near_est1}, we obtain
\[
  |P_1 +  P_3|
  \les C(\muc) ( \e x  + \e x  ) || w||_{\cX} 
\les C(\muc) \e x || w||_{\cX} .
\]

Combining the above estimates, for $x \leq \Rmub$, we prove
\bseq\label{eq:near_dif_R}
\beq
      |\crbb(w) - \td \cR(w) | 
    \leq C(\muc) \e x || w||_{\cX} .
\eeq
Moreover, using \eqref{eq:near_est1}, we obtain 
\beq
 |\td \cR(w)| \leq C(\muc) ( |\psiox| + \tf{1}{x} |\psio| )\leq C(\muc) x^{1+\al}  || w||_{\cX}.
\eeq

\eseq

Using \eqref{eq:near_dif_R} and \eqref{eq:near_dif_bv}, 
and $\va, \vvva \gtr x$ by \eqref{eq:va_est}, \eqref{eq:bw_est:a}, 
for $x \leq \Rmub$, we bound 
{\small
\beq\label{eq:near_dif_rat}
 \B| \f{ \crbb(w)}{\vvva} - \f{\td \cR(w) }{\va} \B|
  \leq \B| \f{\crbb(w) - \td \cR(w) }{\vvva} \B|
  +  \B| \f{\td \cR(w) (\va- \vvva)}{ \va \vvva} \B|
\leq C(\muc) \e \| w \|_{\cX}.
\eeq
}

\subsubsection{Estimate of $\cLab(w)$}\label{sec:near_Lwa}

Using estimate \eqref{eq:near_dif_rat} for the difference, Theorem \ref{thm:near_field_stab} for $ \crbb(w)$, $ \nlinf{w \vpa} \leq \nchi{w}$ \eqref{norm:X_ep}, and the estimate of ratios in \eqref{eq:near_w_rat}, 
for any $x \leq \Rmub$, we bound 
\beq
\bal
   \vpa |\wa| \B| \int_0^x & \f{  \td \cR(w)  }{ 2 \va }  d y \B|
 \leq \vpa (1 + \e C(\muc)) |\wwwa|  \B( \B| \int_0^x \f{ \crbb(w) }{2 \vvva} d y \B| 
+  \int_0^x C(\muc) \e \| w \|_{\cX} d y \B) \\
& \leq (1 + \e C(\muc)) ( \lamst +  C(\muc) \e \vpa |\wwwa| x  )  \| w \|_{\cX}  \\
  & \leq    (1 + \e C(\muc)) ( \lamst 
   + C(\muc) \e |x|^{2 - \bbb} ) \| w \|_{\cX}   \leq ( \lamst + C_{ \eqref{eq:1D_near_R:est} }(\muc) \cdot \e )  \| w \|_{\cX}, 
\eal
   \label{eq:1D_near_R:est}
\eeq
where in the third inequality, we have used $|\wwwa| \les |\xx| $ from \eqref{eq:bw_est:a} in Theorem \ref{thm:reg_alb}, and $\vpa \les_{\muc} x^{-\bbb}$ by \eqref{eq:wg_equiv}.
In the last inequality, we have used $\lamst < 1$ 
from Theorem \ref{thm:near_field_stab}.

\subsection{Estimates in $x \geq \Rmua$}

Recall the nonlocal term $\td \cR(w)$. We use the decomposition in \eqref{eq:lin_nloc_ep:b}
\beq
  \tcr(w) = - ( \e + 2 \b) \psiox - 2 (\alb - \b) ( \psiox - \f{\psio}{x} )
  - 2 \psio \B(  \f{\pa_x \wa}{\wa} + (\alb - \b) \f{1}{ x} \B)
  \teq I_{\eqref{eq:cR_1D:decomp}, 1} + I_{\eqref{eq:cR_1D:decomp},2} + I_{\eqref{eq:cR_1D:decomp}, 3}.
  \label{eq:cR_1D:decomp}
\eeq

Below, we use Proposition \ref{prop:vel_al} to estimate each term and consider $x \geq 0$ without loss of generality.
Clearly, using Proposition \ref{prop:vel_al} , \eqref{eq:wa_est} on $\wa$, \eqref{eq:va_est} on $\va$, for $x \leq 2$, we obtain
\beq
   |I_{\eqref{eq:cR_1D:decomp}, i}(x)| \les  x \nlinf{ w \vpb}.
   \label{eq:cR_est:near0}
\eeq

\subsubsection{Estimate of $I_{\eqref{eq:cR_1D:decomp}, 1}, I_{\eqref{eq:cR_1D:decomp}, 2}$}
Although $I_{\eqref{eq:cR_1D:decomp}, 1}$ contains a small factor $\e$, we do not gain any small factor in its estimates 
for large $x$ due to the growing bound in the estimate of $\cJa$ term.

Below, we use Proposition \ref{prop:vel_al} to derive the main terms in the upper bound with slowest decay. Using \eqref{eq:Ja_est:a}, \eqref{eq:v_al_est}, 
and $\al \leq 1/3$, we obtain
\beq
\bal
 |I_{ \eqref{eq:cR_1D:decomp}, 1} | & = (\e + 2\b)  |\psiox| \leq  (\e + 2\b) ( 2\al |\cJa(w) | + |\psiox + 2\al \cJa(w)| )  \\
   & \leq (\e + 2\b)  4 \al \cdot (  \vpk^{-1} \cJaa^{\kp+1}   + C \cJaa^{\kp + 2/3} ) 
\nlinf{ w \vpb }.
  \eal
  \label{eq:cR_est:I1}
\eeq

For $ I_{\eqref{eq:cR_1D:decomp}, 2}$, using \eqref{eq:vmix_sharp} and $|\b| \les \e$, $\e \cJaa \les 1 $ by \eqref{eq:Ja_hat}, we bound 
\beq
\bal
  |I_{\eqref{eq:cR_1D:decomp}, 2} | &\leq  2 (\alb - \b )  (6  \cff \vpk^{-1} \cJaa^{\kp} + C \cJaa^{\kp - 1/3 } ) \ang \xx^{\e_1 - \he} \nlinf{ w \vpb    }   \\
  & \leq   ( 4  \cff \vpk^{-1}  \cJaa^{\kp}   + C \cJaa^{\kp - 1/3  } ) \ang \xx^{\e_1 - \he} \nlinf{ w \vpb    }  .
\eal
\label{eq:cR_est:I2}
\eeq

Recall the weight $\vpb$ from \eqref{norm:X_ep}. To estimate the integral, we compare the above bound with 
\beq
  M_{ \eqref{def:1D_damp_al} }  = 2 \va \pa_x ( |\wa| \vpb )^{-1}
  = 2 \va   \pa_x ( \ang x^{\e_1} \cJaa^{\kp} \vpk^{-1}). 
  \label{def:1D_damp_al}
\eeq
To compute $  M_{ \eqref{def:1D_damp_al}} $, we evaluate 
\[
\bal
 \f{   M_{ \eqref{def:1D_damp_al} } }{ \ang x^{\e_1} \cJaa^{\kp}  \vpk^{-1} } & 
= 2 \va  \f{ \pa_x \big( \ang x^{\e_1} \cJaa^{\kp}   \exp( - 9\kp_1^{-1}  \ang x^{-\e_1} )  \big)  }{   \ang x^{\e_1} \cJaa^{\kp}  \exp( -  9 \kp_1^{-1}  \ang x^{-\e_1} )  }
= 2 \f{\va}{x} ( \f{ x \pa_x \ang x^{\e_1}}{\ang x^{\e_1}} + \f{x \pa_x \cJaa^{\kp}}{\cJaa^{\kp}}
- \f{9}{\kp_1} x \pa_x \ang x^{-\e_1}
 ) .
\eal
\]

Using $\e_1 = \kp_1 \e$ and a direct calculation, we yield 
\[
  \bal
\ang x^{-\e_1}( x \pa_x \ang x^{\e_1} ) & = \e_1 x^2 \ang x^{-2}
= \e_1 + O(\e \ang x^{-2}), \\
- \f{9}{\kp_1} x \pa_x \ang x^{-\e_1} & = 
\f{ 9 \e_1 }{\kp_1} x^2 \ang x^{-\e_1-2}
= 9 \e x^2 \ang x^{-\e_1-2}
= 9 \e \ang x^{-\e_1} +  O(\e \ang x^{-2}), 
  \eal
\]

Using the estimates \eqref{eq:Ja_hat} on $\cJaa$, $\cJab = \cJa(\wa)$ from \eqref{eq:wa}, 
$\wa = \wwwa \ang x^{\b}<0$ for $x>0$ \eqref{eq:bw_sign}, $\cJaa \les \lgp x$ by \eqref{eq:Ja_hat},  and estimate on $\wwwa$ \eqref{eq:bw_est}, we obtain
\[
\bal
\f{ x \pa_x \cJaa^{\kp}}{\cJaa^{\kp}}
& = \kp \f{ |\cJa(\wa)| \cdot x \pa_x |\cJa(\wa)| }{\cJaa^2}
= \kp \f{|\cJa( \wa )|}{ \cJaa^2} \cdot x \cdot |\wa| x^{\al- 1} \\
& = \kp (\cJaa^{-1} + O(\cJaa^{-2} ) )  \cdot ( 6 x^{-\alb } + O( x^{-\alb} (\lgp x)^{-1/3} ) ) x^{\al + \b} \\
& = ( 6 \kp \cJaa^{  - 1}  + O( \cJaa^{-4/3} ) ) x^{-\alb + \al  + \b} 
= ( 6 \kp \cJaa^{  - 1}  + O( \cJaa^{-4/3} ) ) x^{ -\he} ,
\eal
\]
where we use $\alb - \b -\al = \e - \b = \he $ by \eqref{def:para} in the last identity.
From \eqref{eq:va_est} and \eqref{eq:Ja_hat}, we obtain $\f{1}{x}\va = 2 \al  \cJaa +    O(1)
= 2 \alb \cJaa + O(1)$.  Using $\cJaa \les \min( \lgp x , \e^{-1} )$,  
 $\al =\f{1}{3} - \e$, and $\e_1 \les \e$ \eqref{norm:X_ep}, we derive the main terms 
\[
\bal
 \f{   M_{ \eqref{def:1D_damp_al} }  }{ \ang x^{\e_1} \cJaa^{\kp} \vpk^{-1}  } & = 2 ( 2 \alb  \cJaa +    O(1 ) )  \cdot \B( \e_1 + 9 \e \ang x^{-\e_1} +  O(\e \ang x^{-2}) + ( 6 \kp \cJaa^{  - 1}  + O( \cJaa^{-4/3} ) ) x^{ -\he} \B) \\
  & = 4 \alb \e_1 \cJaa
  + 24 \alb \kp \cdot  x^{ -\he }
  +  36 \alb \e \ang x^{-\e_1}
  + O(\cJaa^{-1/3} x^{ -\he } + \e \cJaa^{2/3} ) .
\eal
\]
Multiplying $\ang x^{\e_1} \cJaa^{\kp} \vpk^{-1}$ and using $\alb = \f13$, we obtain
\beq
\bal
    M_{ \eqref{def:1D_damp_al} } & = 4 \alb  \e_1 \cJaa^{\kp+1} \ang x^{\e_1} \vpk^{-1}
  + 8 \kp \cdot \cJaa^{\kp} \ang x^{ -\he + \e_1} \vpk^{-1}
  + 12 \e  \cJaa^{\kp} \vpk^{-1} \\
  & \quad +   O( \cJaa^{\kp-1/3}  x^{\e_1 -\he} + \e \cJaa^{2/3 + \kp} \ang x^{\e_1}   ) .
\eal
\label{eq:cR_est:M}
\eeq

Recall that $\kp \in (\f{\cff}{2}, 1) \subset (0.8, 1)$ from \eqref{def:kp}. We denote
\beq\label{def:lam_kp}
  \lam_\kp = \f{\cff}{2 \kp} \in (0.8, 1).
\eeq

We compare the main term in $M$ and those in $I_1, I_2$. Recall the range of parameter \eqref{ran:ep} and \eqref{def:kp}. 
Recall $\b =- \f{\e}{8} + O(\e^{4/3})$. 
For $\e$ small enough, we have $\e + 2 \b \in [0, \e]$  and 
\beq\label{eq:1D_sharp_bound}
\bal
0.8 < \tf{1}{2} \cff = \lam_{\kp} \kp < \kp,
\quad  (\e + 2 \b) \cdot 4 \al
< 4 \alb \e < \lam_{\kp} \cdot 12 \e.
  \eal
\eeq

Therefore, comparing $ |I_{ \eqref{eq:cR_1D:decomp}, 1} |,  |I_{ \eqref{eq:cR_1D:decomp}, 2} |$ \eqref{eq:cR_est:I1}, \eqref{eq:cR_est:I2} with $M$ in \eqref{eq:cR_est:M}, we obtain
\[
 |I_{ \eqref{eq:cR_1D:decomp}, 1} |
 +  |I_{ \eqref{eq:cR_1D:decomp}, 2} |
\leq ( \lam_{\kp} M + C_{err}(x)  ) \nlinf{w \vpb} ,
\]
where $C_{err}(x)$ denote the error part 
\[
  C_{err}(x) = 
  C \cJaa^{\kp-1/3}  x^{\e_1 -\he} + C \e \cJaa^{2/3 + \kp} \ang x^{\e_1}  .
\]

Recall the definition of $M$ from \eqref{def:1D_damp_al}. 
Since $ (|\wa| \vpb)^{-1}$ is increasing by \eqref{norm:X_ep}, integrating the above bound and using 
\eqref{eq:cR_est:near0}, we prove 
\[
\bal
 \int_0^x \f{| I_{\eqref{eq:cR_1D:decomp}, 1}| + | I_{\eqref{eq:cR_1D:decomp}, 2} |}{ 2 \va }
  & \leq  \big( C
  + \one_{x >2} \int_2^x \f{\lam_{\kp} 2 \va \pa_x ( |\wa|^{-1} \vpb^{-1} ) + C_{err}(y)}{ 2 \va } d y 
  \big)
  \nlinf{ w \vpb }  \\ 
& \leq  \big( \lam_{\kp} ( |\wa| \vpb)^{-1}   + C +\one_{x >2} \int_2^x \f{C_{err}(y)}{2 \va} d y  \big) \nlinf{ w \vpb } \, .
\eal
\]

Since $\cJaa(y)$ is increasing in $y$, using $\va \gtr x \cJaa$ by \eqref{eq:va_est}, $\he -\e_1 \gtr \e$ by \eqref{ran:ep}, the above bound of $C_{err}(y)$, and \eqref{eq:log_ineq_J:a} with $k = \kp -4/3 > -2/3, \kp -1/3 > 0$, we bound 
\[
\bal
  \int_2^x \f{C_{err}(y)}{2 \va} d y
  & \les \int_2^x \f{ |\cJaa(y)|^{\kp-1/3} \ang y^{\e_1 - \he} 
+ \e |\cJaa(y)|^{2/3 + \kp} \ang y^{\e_1}
  }{y \cJaa(y) } d y  \\
  & \les \int_2^x \f{ |\cJaa(y)|^{\kp-4/3} \ang y^{\e_1 - \he} 
  }{y } d y 
   + \e |\cJaa(x)|^{ \kp-1/3} \int_2^x \ang y^{\e_1-1} d y
  \\
 & \les | \cJaa(x)|^{\kp-1/3} + |\cJaa(x)|^{\kp-1/3} \ang x^{\e_1}
 \les  |\cJaa(x)|^{\kp-1/3} \ang x^{\e_1}.
\eal
\]

Recall $|\wa| \vpb \asymp \ang x^{-\e_1} \cJaa^{-\kp}$ from \eqref{eq:wg_equiv}. Combining the above estimates, multiplying $|\wa| \vpb$, and using $\kp > \f{1}{3}$ by \eqref{def:kp}, we prove 
\beq
\bal
  |\wa| \vpb \int_0^x \f{| I_{\eqref{eq:cR_1D:decomp}, 1} | + | I_{\eqref{eq:cR_1D:decomp}, 2} |}{ 2 \va } 
& \leq (\lam_{\kp} + C \cJaa^{-\kp} + C \cJaa^{-1/3} )  \nlinf{ w \vpb } \\
& \leq (\lam_{\kp} +  C \cJaa^{-1/3} )  \nlinf{ w \vpb }.
\label{eq:cR_est:I12}
\eal
\eeq

\begin{remark}[\bf Sharp bounds]\label{rem:sharp_bound}

We derive estimates with explicit constants for the main terms in 
\eqref{eq:vmix_sharp} to establish the first bound in \eqref{eq:1D_sharp_bound} with some $\lam_{\kp} < 1$. 
We use $\lam_{\kp} < 1$ to further prove the contraction estimates. 
See also Remark \ref{rem:expl_C1}.
\end{remark}

\subsubsection{Additional decay estimates}

The upper bound in \eqref{eq:cR_est:I12} does not decay in $x$.
Using \eqref{eq:v_al_est} and \eqref{eq:Ja_est:b} on $\psio, \cJa$, 
and $\e \cJaa \les 1$ by \eqref{eq:Ja_hat}, we bound 
\bseq\label{eq:cR_decay_bd}
\beq
\bal
  |  I_{\eqref{eq:cR_1D:decomp},1} | + |  I_{\eqref{eq:cR_1D:decomp},2}| 
  \les |\psiox + 2 \al \cJa| + |\cJa(w)| + | \tf{1}{x} \psio + 2 \al \cJa |
\les  \min( x,  \cJaa^{\kp+1} )  \| w \vpb \|.
\eal
\eeq

Since $\cJaa$ is increasing in $y$, integrating the above bounds, using $|\wa| \vpb = \ang x^{-\e_1} \cJaa^{-\kp} \vpk$, $\va \gtr x \cJaa(x)$, $\e_1 / 2 \gtr \e$ by \eqref{ran:ep},
and Lemma \ref{lem:lgx_pow}, we bound 
\beq
\bal
|\wa| \vpb \int_0^x \f{  \min( y,  \cJaa^{\kp+1} ) }{ y \cJaa } d y 
& \les \ang x^{-\e_1} |\cJaa(x) |^{-\kp} \cJaa(x)^{\kp} ( 1 + \int_1^x \f{1}{y} )   \\
& \les \ang x^{-\e_1} \lgp x \les \e^{-1} \ang x^{-\e_1/2}.
\eal
\eeq

\eseq

\subsubsection{Estimate of $ I_{\eqref{eq:cR_1D:decomp},3}$}
For $ I_{\eqref{eq:cR_1D:decomp},3}$, using Proposition \ref{prop:vel_al} and $\he > \e_1$ by \eqref{ran:ep}, we bound 
\[
\bal
  |\tf{1}{y} \psio | 
 &  \leq   |\tf{1}{y} \psio + 2 \al \cJa(w) | + |2\al \cJa(w)|
  \les  ( \cJaa^{\kp}   \ang y^{-\he + \e_1} +   \cJaa^{\kp+1}  )  \nlinf{ w \vpb }   \les   \cJaa^{\kp+1}  \nlinf{ w \vpb }.
  \eal
\]

Using the above estimate, \eqref{eq:wa_est} on $\wa$, and $\cJaa(y) \les \lgp y$ by \eqref{eq:Ja_hat} and $\kp \in (0, 1)$, we estimate $ I_{\eqref{eq:cR_1D:decomp},3}$ as
\bseq\label{eq:cR_est:I3}
\beq
\bal
| I_{\eqref{eq:cR_1D:decomp},3}(y)| 
&= \B| 2 \f{\psio}{y} \B(  \f{ y \pa_y \wa}{\wa} + (\alb - \b)  \B) \B|  \les   \,  \cJaa^{\kp+1}  
\lgy^{-1-\alb}  \nlinf{ w \vpb } \\
& \les \cJaa |\lgp y|^{\kp - 1- \alb} \nlinf{ w \vpb }.
\eal
\eeq

Recall $ |\wa| \vpb = \ang x^{-\e_1} \cJaa^{-\kp} \vpk$ from \eqref{norm:X_ep}. 
Using lower bound of $\va$ in \eqref{eq:va_est} and \eqref{eq:cR_est:near0}, 
$\kp > \f{1}{2} > \alb$ by \eqref{def:kp}, 
$\e_1 \gtr \e$ by \eqref{ran:ep}, and Lemma \ref{lem:lgx_pow} on $\cJaa$, we bound the integral as
\beq
\bal
 |\wa| & \vpb  \B| \int_{0}^x  \f{  I_{\eqref{eq:cR_1D:decomp},3} (y)  }{ 2 \va}   d y \B|
  \les |\wa| \vpb \B(  1 +  \int_{2}^x \f{   |\cJaa(y)| \cdot \lgy^{-1-\alb + \kp}  }{ y \cJaa }   d y \B) \cdot \nlinf{ w \vpb }  \\
 & \les |\wa| \vpb \B(  1 +  \int_{2}^x \f{    |\log y|^{-1-\alb + \kp}  }{ y  }   d y \B) \cdot \nlinf{ w \vpb }  \\
 & \les \ang x^{-\e_1} \cJaa^{-\kp} |\lgp x|^{-\alb + \kp} \nlinf{ w \vpb } 
 \leq C_{ \eqref{eq:cR_est:I3} }  \ang x^{-\e_1/2}  |\cJaa(x)|^{-1/3 } \nlinf{ w \vpb } .
\eal
\eeq
\eseq

Optimizing estimates  \eqref{eq:cR_est:I12}, \eqref{eq:cR_decay_bd} for $ I_{\eqref{eq:cR_1D:decomp},1},  I_{\eqref{eq:cR_1D:decomp},2}$, and applying \eqref{eq:cR_est:I3} to \eqref{eq:cR_1D:decomp}, 
we prove 
\beq
\bal
   |\wa| \vpb \int_0^x \f{   |\td \cR(w)(y)|  }{2 \va} dy 
   & \leq \min( \lam_{\kp} +   C_{ \eqref{eq:1D_far_R:est}  }
    \cJaa^{-\f13} , \ C \e^{-1} \ang x^{- \f12 \e_1} ) \cdot \nlinf{ w \vpb },
   \  \lam_{\kp} < 1. 
\eal
\label{eq:1D_far_R:est} 
\eeq

\vs{0.1in}

\paragraph{\bf Summary}
We combine estimates \eqref{eq:1D_far_R:est} and \eqref{eq:1D_near_R:est} to establish the contraction estimates for the linear part. 
Since $\lam_{\kp}, \lamst < 1$ by Theorem \ref{thm:near_field_stab} and \eqref{def:lam_kp}, we first choose $\muc$ sufficiently small so that $R_{\muc,1}$ associated with $\muc$ is large enough by \eqref{def:Rmu}, and then take $\e \leq c(  \muc )$ to ensure
\beq
\lam_{\kp} +  C_{ \eqref{eq:1D_far_R:est}  } | \cJaa( R_{\muc, 1})|^{-1/3} < \f{1 + \lam_{\kp}}{2}  ,
 \quad  
\lamst +  C_{ \eqref{eq:1D_near_R:est} }(\muc) \e < \f{1+ \lamst }{2} .
 \label{eq:cR_est:para}
\eeq

Recall $ \| w \|_{\cX} = \nlinf{ w \max(\vpa, \muc \vpb) }$.  For $x \leq \Rmub$, using \eqref{eq:1D_near_R:est} and \eqref{eq:cR_est:para}, we prove 
\bseq\label{eq:cR_tot:est}
\beq
  |\wa|\vpa \B| \int_0^x \f{  \td \cR(w) (y)  }{2 \va} d y\B|
  \leq ( \lamst + 
 C_{ \eqref{eq:1D_near_R:est} }(\muc) \e )   \| w \|_{\cX} 
  \leq \f{ 1 + \lamst }{2}  \| w \|_{\cX} .
  \label{eq:cR_tot:est1}
\eeq

For $x > \Rmua$, using  \eqref{eq:1D_far_R:est} and \eqref{eq:cR_est:para}, we prove 
\begin{align}
    \muc |\wa| \vpb  \int_0^x \f{  |\td \cR(w)(y)|  }{2 \va} d y 
   & \leq 
 \min( \lam_{\kp} +   C_{ \eqref{eq:1D_far_R:est}  }
    |\cJaa(R_{\muc, 1}) |^{-1/3} , \ C \e^{-1} \ang x^{-\e_1/2} ) \cdot \muc \nlinf{ w \vpb } \notag \\
  & \leq 
  \min( \tf{1}{2} (1 + \lam_{\kp} ), \, C \e^{-1} \ang x^{-\e_1/2} )  \| w \|_{\cX} .
  \label{eq:cR_tot:est2}
\end{align}
\eseq

When $x \in [\Rmua, \Rmub]$, both estimates in \eqref{eq:cR_tot:est} hold. Recall $\cLab(w)$ from \eqref{eq:F_lin_ep}. Since $\vpa \geq \muc \vpb$ for $x \leq \Rmua$, $\vpa < \muc \vpb$ for $x \geq \Rmub$ by \eqref{def:Rmu}, 
combining estimates in \eqref{eq:cR_tot:est}, we prove

\begin{thm}\label{thm:contra_lin}

 Let $ \lamst$, $\beps_2$, \(\lam_{\kp}\) be as defined in 
 Theorems \ref{thm:near_field_stab},  \ref{thm:al_appr_profile}, and \eqref{def:lam_kp}, respectively,
Let $\e = \f13 - \al$, $\muc, \beps_2$ be the parameter in the $\cX$-norm \eqref{norm:X_ep}, 
and in Theorem \ref{thm:1D_error}. There exists an absolute constant $\muc >0$ and a small $\beps_4(\muc) \in (0, \beps_3)$,  
such that for any $\e \in (0, \beps_4 )$,  we have
\[
\bal
| \max(\vpa, \muc \vpb)\cLab(w)(x) | & \leq  \max(\vpa, \, \muc \vpb) |\wa| \B| \int_0^x \f{  \td \cR(w)(y)   }{2 \va} d y \B| \leq \min( \lam_{\cX} , C \e^{-1} \ang x^{-\e_1} ) \| w \|_{\cX}, \\
  \eal
  \]
for any $x\geq 0$, where 
$\lam_{\cX} \teq \tf12 + \tf12 \max( \lam_{\kp},  \lamst ) 
= \tf12 + \tf12 \max( \f{\cff}{2 \kp} ,  0.95 ) \in (0, 1) $.
\end{thm}

For $x \geq \Rmub$, the second bound $C \e^{-1} \ang x^{-\e_1}$ in Theorem \ref{thm:contra_lin} follows from \eqref{eq:cR_tot:est2}. For $x \leq \Rmub$, 
since $\ang x^{\e_1} \les 1$, by choosing $C$ large enough so that $C \ang x^{-\e_1} \geq 1> \f{1 + \lamst}{2}$, 
the bound $C \e^{-1} \ang x^{-\e_1}$ follows trivially from \eqref{eq:cR_tot:est}.

For the rest of this section, we fix \(\muc\) as in Theorem \ref{thm:contra_lin}. Below, the implicit constants may depend on \(\muc\), and this dependence is suppressed.

\subsection{Estimate of error term $\cEab$}

In this section, we estimate $\cEab$ from \eqref{eq:F_lin_ep} perturbatively.

Recall $\kp < 1$ from \eqref{def:kp}. Firstly, we impose the following constraint on $\ddd$ in Theorem \ref{thm:1D_solu} 
\beq\label{eq:ineq_ddd1}
   \ddd <  \e^{ (\kp + 1 ) /2 },
\eeq

For $|| w ||_{\cX} < \ddd $ with \eqref{eq:ineq_ddd1} and $\e$ small enough, 
since $\cJaa \les \e^{-1}$ \eqref{eq:Ja_hat}, using Proposition \ref{prop:vel_al} and \eqref{eq:va_est}
, we obtain
\beq\label{eq:ass_v}
\bal
  |\psio(w)|  &\leq C x \cJaa^{\kp+1} || w ||_{\cX} \leq C x \cJaa \e^{-\kp} \ddd 
  \leq  \tf{1}{2} \va  , \\
   V  & = \va + \psio(w)
 \geq x ( 2 \al \cJaa - C \cJaa^{2/3})) - C x \cJaa^{\kp + 1} || w||_{\cX} \geq \tf{1}{2} \va. 
  \eal
\eeq

Using the norm \eqref{norm:X_ep}, we bound 
\beq\label{eq:lin_est_w}
   |\f{W}{\wa} | \leq 1+  |\f{w}{\wa}| 
  \les 1 + \ang x^{\e_1} \cJaa^{\kp} || w ||_{\cX}
\eeq

Using \eqref{eq:err_comp_R4} on $\cR(\wa)$, \eqref{eq:va_est}, \eqref{eq:ass_v} on $V$,  we bound 
\beq\label{eq:lin_err_est}
\bal
  | \f{ \cEab  }{ \wa } |
  & = |\int_0^x  \f{ \cR(\wa)}{2 V} \cdot \f{W}{\wa}  |
   \les \int_0^x \f{ \e \min( \cJaa \lgy^{-1/3}, y ) }{ y \cJaa}  (1 + \ang y^{\e_1} |\cJaa(y)|^{\kp} 
   \| w \|_{\cX} ) .
\eal
\eeq

Using $1\les \cJaa(y) \les \lgy$ by \eqref{eq:Ja_hat}, we bound the integrand 
\beq\label{eq:lin_err_est_int1}
F(y) \teq \f{ \e \min( \cJaa \lgy^{-1/3}, y ) }{ y \cJaa}  
  \les \f{ \e \min( \lgy^{-1/3} , y ) }{ y } .
\eeq
The second bound $y$ in the minimum is only used for small $y$, e.g. $y \leq 10$. 

We estimate the integral with and without $||w||_{\cX}$ 
separately. 
Since $\cJaa$ is increasing, we have
\beq\label{eq:lin_err_est_int2}
\bal
  P_1 & \teq \int_0^x F(y) d y 
\les \int_0^x \f{ \e \min( \lgy^{-1/3} , y ) }{ y } d y
\les \e \min( x, \lgx^{2/3} ), \\
P_2 & \teq \int_0^x F(y) \ang y^{\e_1} |\cJaa(y)|^{\kp} dy 
\les |\cJaa(x)|^{\kp}   \int_0^x \f{ \e \min( \lgy^{-1/3} , y ) }{ y } \ang y^{\e_1} d y .
\eal
\eeq

If $x \leq 2$, we bound $P_2$ as 
\[
  P_2 \les  \e |\cJaa(x)|^{\kp} x.
\]

If $x > 2$, we bound $P_2$ as 
\[
  P_2 \les  \e |\cJaa(x)|^{\kp} \B( 1 + \int_2^x y^{\e_1 - 1}  |\log y|^{-1/3} 
( \one_{ \e \log y < 1}  + \one_{\e \log y \geq 1}  ) d y  \B)
\]

Since $\e_1 \asymp \e$ \eqref{norm:X_ep}, if $0 < \e \log y < 1$, we get $y^{\e_1} = e^{  \e_1 \log y }  \les 1$.
If $\e \log y > 1$, we get $|\log y|^{-1/3} < \e^{1/3}$. It follows 
\[
  P_2 \les 
  \e |\cJaa(x)|^{\kp} \B( 1 + \int_2^x y^{-1} |\log y|^{-1/3} + y^{\e_1-1} \e^{1/3} \B)
\les    \e |\cJaa(x)|^{\kp} ( \lgx^{2/3} + \e^{-2/3} \ang x^{\e_1} ),
\]
where we use $1  \les \lgp x$ in the last inequality.  Therefore, we estimate $P_2$ as 
\beq\label{eq:lin_err_est_int3}
  P_2 \les   \e |\cJaa(x)|^{\kp} \min( x, \, \lgx^{2/3} + \e^{-2/3} \ang x^{\e_1} ).
\eeq

Applying estimates of $P_1, P_2$ to \eqref{eq:lin_err_est}, for any $x \geq 0$, we prove
 \[
 \bal
   |  \cEab  \wa^{-1}|
   &  \leq   |P_1| + |P_2| \cdot || w ||_{\cX}   \\
 &  \les  \e \min(x, \lgx^{2/3}) +  \e |\cJaa(x)|^{\kp} \min( x, \, \lgx^{2/3} + \e^{-2/3} \ang x^{\e_1} ) || w||_{\cX} \\
  \eal
  \]

Since we have chosen $\muc$ in Theorem \ref{thm:contra_lin} and treat it as an absolute constant, Applying the above estimate,\eqref{eq:wg_nsingu}, $\bbb < 2$ from \eqref{def:vpa}, 
and Lemma \ref{lem:lgx_pow} with $\e_1 \gtr \e$ by \eqref{ran:ep}, we bound 
\[
 \bal
   |  \cEab  \max(\vpa, \muc \vpb) | 
 & \les  |  \cEab  \wa^{-1} | \cdot ( |x|^{-\bbb+1} + 1 ) \ang x^{-\e_1} \cJaa^{-\kp} \\
 &  \les \B( \e \min(x, \lgx^{2/3}) +  \e |\cJaa(x)|^{\kp} \min( x, \, \lgx^{2/3} + \e^{-2/3} \ang x^{\e_1} ) || w||_{\cX}  \B) \\
 & \qquad  \cdot ( |x|^{-\bbb+1} + 1 ) \ang x^{-\e_1} \cJaa^{-\kp} \\
 & \les \e \cJaa^{2/3 - \kp}
 + \e \ang x^{-\e_1} ( |\lgp x|^{2/3} + \e^{-2/3} \ang x^{\e_1} ) || w||_{\cX} .
  \eal
\]

Since $\kp > 2/3 $ by \eqref{def:kp}, using $\e_1 \gtr \e$ by \eqref{ran:ep} and  Lemma \ref{lem:lgx_pow} again, we prove
  \beq\label{eq:lin_err_est2}
      |  \cEab  \max(\vpa, \muc \vpb) |  \les \e +  \e^{1/3} || w||_{\cX} .
\eeq

\subsubsection{Contraction estimates of $\cEab$}

For later proof of contraction estimate, we estimate $\cEab( w_1 ) - \cEab(w_2)$ for two solution 
$w_1, w_2$ with $||w_i || < \ddd$ and $\ddd$ satisfying \eqref{eq:ineq_ddd1}. 
Consider 
\beq\label{eq:solu_contra}
  W_i = \va + w_i, \quad  V_i = \va + \psio( w_i), 
  \quad |w_i | < \ddd, \ i = 1,2.
\eeq

Using the definition of $\cEab(w_i)$ \eqref{eq:F_lin_ep}, we have 
\beq\label{eq:lin_contra_err1}
\bga
  \f{ \cEab(w_1) - \cEab(w_2)}{\wa} 
   = \int_0^x \f{ \cR(\wa) }{\wa} \cdot ( \f{W_1}{V_1} - \f{W_2}{V_2} ),\\
 \wa^{-1} (\f{W_1}{V_1} - \f{W_2}{V_2} )
  = \wa^{-1} ( \f{W_1 - W_2}{V_1} - \f{W_2(V_1 - V_2) }{V_1 V_2 } )
  = \wa^{-1} (  \f{w_1 - w_2}{V_1} - \f{W_2  \cdot \psio(w_1 - w_2) }{V_1 V_2} ).
\ega
\eeq

Using \eqref{eq:ass_v}, \eqref{eq:lin_est_w}, and \eqref{eq:va_est} we obtain
\[
  | \psio(w_i)| \leq \f{1}{2} \va \leq V_i,  \quad V_i \asymp x \cJaa, \quad  |\f{w_i}{\wa}| \les 1 + \ang x^{\e_1} \cJaa^{\kp} || w_i||_{\cX}. 
\]

Using the norm \eqref{norm:X_ep}, Proposition \ref{prop:vel_al} for $\psio(w_1 - w_2) $, 
and $|| w_i||_{\cX} < \ddd $, we obtain 
\[
\bal
    \B|\f{w_1 - w_2}{ \wa V_1} \B| 
  & \les \f{ \ang x^{\e_1}  \cJaa^{\kp} }{x \cJaa} || w_1 - w_2||_{\cX}, 
  , \\
|\f{W_2( \psio(w_1) - \psio(w_2) ) }{ \wa V_1 V_2} | 
& \les  (1 + \ang x^{\e_1} \cJaa^{\kp} || w_2||_{\cX} ) \f{ x \cJaa^{\kp+1}  }{ x^2 \cJaa^2 } || w_1 - w_2 ||_{\cX} \\
\eal
\]
Since $|| w ||_{\cX} < \ddd < \e^{\hk}$ \eqref{eq:ineq_ddd1} and $\cJaa \les \e^{-1}$, we get $\cJaa^{\kp} \nu \les 1$. Thus, we further obtain
\[
\bal
|\f{W_2 \cdot \psio(w_1 - w_2)  }{ \wa V_1 V_2} | 
 \les (1 + \ang x^{\e_1} ) \f{  x \cJaa^{\kp+1}  }{ x^2 \cJaa^2 } || w_1 - w_2 ||_{\cX}
 \les \f{  \ang x^{\e_1} \cJaa^{\kp}  }{ x \cJaa } || w_1 - w_2 ||_{\cX}.
\eal
\]

Applying the above estimates to \eqref{eq:lin_contra_err1}, we bound 
\beq\label{eq:lin_contra_err2}
\bal
\B|  \wa^{-1} (\f{W_1}{V_1} - \f{W_2}{V_2} ) \B|
   \les  \f{  \ang x^{\e_1} \cJaa^{\kp}  }{ x \cJaa } || w_1 - w_2 ||_{\cX}.
\eal
\eeq

Using \eqref{eq:err_comp_R4} on $\cR(\wa)$ and the above estimate, we bound 
\[
\B|   \f{ \cEab(w_1) - \cEab(w_2) }{\wa}  \B|
\les \int_0^x  \f{ \e \min( \cJaa \lgy^{-1/3}, \, y )  }{ y \cJaa }
\cdot 
     \ang y^{\e_1}  \cJaa^{\kp}    || w_1 - w_2 ||_{\cX} .
\]

From \eqref{eq:lin_err_est_int1}, \eqref{eq:lin_err_est_int2},  we observe that the above integral is the same as $P_2$ in \eqref{eq:lin_err_est_int2}. Thus, using \eqref{eq:lin_err_est_int3}, we prove 
\[
\bal
 |    (\cEab(w_1) - \cEab(w_2))  \wa^{-1}  |
  & \les  \e |\cJaa(x)|^{\kp} \min( x, \, \lgx^{2/3} + \e^{-2/3} \ang x^{\e_1} ) || w_1 - w_2 ||_{\cX} .
\eal
\]

We treat $\muc$ chosen in Theorem \ref{thm:contra_lin} as an absolute constant.
 Applying the above estimate, estimate \eqref{eq:wg_nsingu}, $\bbb < 2$ from \eqref{def:vpa}, 
and Lemma \ref{lem:lgx_pow} with $\e_1 \gtr \e$ by \eqref{ran:ep}, we bound 
\beq\label{eq:lin_contra_err}
\bal
  |    (\cEab(w_1) & - \cEab(w_2))  \max(\vpa, \muc \vpb)  | \\
 & \les  |    (\cEab(w_1) - \cEab(w_2))  \wa^{-1}  | \cdot ( |x|^{-\bbb+1} + 1 ) \ang x^{-\e_1} \cJaa^{-\kp }    \\
 & \les \e \min( x, \, \lgx^{2/3} + \e^{-2/3} \ang x^{\e_1} ) 
 \cdot ( |x|^{-\bbb+1} + 1 ) \ang x^{-\e_1}   || w_1 - w_2 ||_{\cX}    \\
 & \les \e ( \lgx^{2/3} + \e^{-2/3} \ang x^{\e_1} ) \ang x^{-\e_1}   || w_1 - w_2 ||_{\cX}    \\
 & \les \e ( \e^{-2/3} + \e^{-2/3} ) || w_1 - w_2 ||_{\cX}  
 \les \e^{1/3} || w_1 - w_2 ||_{\cX}  .
 \eal
 \eeq

\subsection{Estimate of nonlinear term $\cNab$}

Recall $\cNab$ from \eqref{eq:F_lin_ep},
\beq\label{eq:1D_non_recall}
\bal
   \cNab  & = \wa  \int_0^x \td \cR(w) \cdot \B( \f{W}{2V \wa} - \f{1}{2 \va}  \B) ,
\eal
\eeq
and $W = \wa + w , V = \va + \psio$ from \eqref{eq:1D_vw}. We decompose
\[
  \f{W}{ V \wa  } - \f{1}{\va}
  = \f{ W \va - V \wa  }{V \wa \va }
  = \f{ (W - \wa) \va + (\va - V) \wa  }{V \wa \va }
  = \f{w}{V \wa} - \f{\psio}{V \va}.
\]

Using $V \asymp \va$ \eqref{eq:ass_v} and the norm \eqref{norm:X_ep}, we obtain 
\[
  |\f{w}{ V \wa }| \les     \f{ \ang x^{\e_1} \cJaa^{\kp} }{ x \cJaa  }  || w ||_{\cX},
  \quad  |\f{\psio}{V \va} | \les \f{x \cJaa^{\kp+1}}{ x^2 \cJaa^2 } || w ||_{\cX}
\les \f{ \cJaa^{\kp}}{ x \cJaa } || w ||_{\cX}.
\]

Since $\ang x^{\e_1} \geq 1$, combining the above two estimates, we derive 
\beq\label{eq:lin_non_coe}
    | \f{W}{ V \wa  } - \f{1}{\va} | \les   \f{ \ang x^{\e_1} \cJaa^{\kp-1} }{ x   }  || w ||_{\cX}.
\eeq

\paragraph{\bf Estimate of $\cR$}
Recall the decomposition \eqref{eq:cR_1D:decomp} on $ \tcr(w)$. Combining estimates 
\eqref{eq:cR_est:near0},  \eqref{eq:cR_est:I1},  \eqref{eq:cR_est:I2}, \eqref{eq:cR_est:I3}, which holds for any $x \geq 0$, using $\vpk \asymp 1$ by \eqref{eq:wg_equiv},
$ |\b| \les \e$ \eqref{eq:beta_est}, 
and $\nlinf{ w \vpb} \les \nchi{ w}$ \eqref{norm:X_ep} ($\muc$ is fixed in Theorem \ref{thm:contra_lin}), we bound
\beq\label{eq:lin_non_cr}
   |\tcr(w)| \les  x \ang x^{-1}  \big(  \e \cJaa^{\kp+1}
     + \cJaa^{\kp} \ang x^{-\he + \e_1} 
     + \cJaa^{\kp+1} \lgx^{-4/3} \big) \nchi{ w} .
\eeq

For $x \leq 1$, the above estimate follows from \eqref{eq:cR_est:near0};
for $x > 1$, it follows from  \eqref{eq:cR_est:I1},  \eqref{eq:cR_est:I2}, \eqref{eq:cR_est:I3}.

Applying \eqref{eq:lin_non_coe} and \eqref{eq:lin_non_cr} to \eqref{eq:1D_non_recall}, we obtain 
\begin{align}\label{eq:lin_non_int1}
  |\f{\cNab }{\wa}| & \les \int_0^x 
     ( \e \cJaa^{\kp+1}
     + \cJaa^{\kp} \ang y^{-\he + \e_1} 
    + \cJaa^{\kp+1} \lgy^{-4/3} ) \f{y}{\ang y} \cdot \f{\ang y^{\e_1} \cJaa^{\kp-1} }{y} 
    d y || w ||_{\cX}^2 \notag  \\   
  & \teq (S_1 + S_2 + S_3) || w ||_{\cX}^2 ,
\end{align}
where $S_i$ denote the integral of the $i$-th summand in the integrand. All the functions in the integrands evaluate at $y$. 
Clearly, we have a trivial bound 
\beq\label{eq:lin_non_est1}
    |\cNab \cdot  \wa^{-1}(x)| \les x || w ||_{\cX}^2 , \quad x \leq 1.
\eeq

Next, we consider large $x$. Since $\cJaa$ is increasing, 
we bound 
\bseq\label{eq:lin_non_est:S}
\beq
\bal
  |S_1|  = \e \int_0^x  \cJaa^{\kp+1 + \kp - 1  } \ang y^{\e_1 - 1} d y
\les |\cJaa(x)|^{2\kp } \e \int_0^x \ang y^{\e_1 - 1} d y
\les |\cJaa(x)|^{2\kp } \ang x^{\e_1} .
\eal
\eeq

Recall $\he - 2 \e_1  \geq \f{(8.9-5)\e}{8} \gtr \e$ by \eqref{ran:ep}. 
Using \eqref{eq:log_ineq_J:a} with $k = 2 \kp - 1  > -1$,  we estimate $S_2$ as 
\beq
  |S_2| = \int_0^x \cJaa^{2 \kp - 1 } \ang y^{-\he + 2\e_1 - 1} d y
  \les  |\cJaa(x)|^{2 \kp} .
\eeq

For $S_3$, since $\cJaa, \ang y$ is increasing, we bound 
\beq
\bal
  |S_3| & = \int_0^x \cJaa^{2\kp } \lgy ^{-4/3} \ang y^{\e_1 -1} d y
\les |\cJaa(x)|^{2\kp } \ang x^{\e_1} \int_0^x \lgy^{-4/3} y^{-1} dy  \\
& \les |\cJaa(x)|^{2\kp } \ang x^{\e_1}.
\eal
\eeq
\eseq

Combining the above estimates and using $\cJaa \les \e^{-1}$ by \eqref{eq:Ja_hat}, we prove 
\beq\label{eq:lin_non_est2}
\bal
  |\cNab  \wa^{-1}| & \les  \one_{x \leq 1} x + \one_{x > 1}  | \cJaa(x)|^{2\kp }  \ang x^{\e_1}
   || w||_{\cX}^2   \les \min(x, \  \e^{-\kp}  |\cJaa(x)|^{\kp } \ang x^{\e_1}
  ) || w||_{\cX}^2 .
  \eal 
\eeq

Since we have chosen $\muc$ in Theorem \ref{thm:contra_lin} and treat it as an absolute constant, applying the above estimate, estimate \eqref{eq:wg_nsingu}, and $\bbb < 2$ from \eqref{def:vpa}, we bound 
\beq\label{eq:lin_non_est3}
\bal
  |\cNab \max(\vpa, \muc \vpb)  |
 &  \les   |\cNab  \wa^{-1}| \cdot (|x|^{-\bbb+1} + 1) \ang x^{-\e_1} \cJaa^{-\kp} 
  \les \e^{-\kp} || w||_{\cX}^2 .
\eal
\eeq

\subsubsection{Contraction estimates of $\cNab$}

For later proof of contraction estimate, we estimate $\cNab( w_1 ) - \cNab(w_2)$ for two solution 
$w_1, w_2$ with $||w_i || < \ddd$ and $\ddd$ satisfying \eqref{eq:ineq_ddd1}. 
We adopt the same notation as in \eqref{eq:solu_contra}. Using the definition of $\cNab$ \eqref{eq:1D_non_recall}, we obtain
\[
  \f{(\cNab(w_1) - \cNab(w_2))  }{\wa}
  =  \int_0^x ( \td \cR(w_1) - \td \cR(w_2) ) \cdot \B( \f{W_1}{2V_1 \wa} - \f{1}{2 \va}  \B) 
    + \tcr(w_2) \B( \f{W_1}{2V_1 \wa}  -  \f{W_2}{2V_2 \wa} 
   \B) 
  \teq \cN_{\D, 1} + \cN_{\D, 2}.
\]

From the definition \eqref{eq:lin_nloc_ep}, $\tcr(w)$ is linear in $w$. Thus, using \eqref{eq:lin_non_cr}, we obtain 
\[
\bal
|\tcr(w_1) - \tcr(w_2)|
& = |\tcr(w_1 -w_2)| \\
& \les  x \ang x^{-1} (  \e \cJaa^{\kp+1}
     + \cJaa^{\kp} \ang x^{-\he + \e_1} 
     + \cJaa^{\kp+1} \lgx^{-4/3}) || w_1 - w_2 ||_{\cX}.
\eal
\]

Applying the above estimate, \eqref{eq:lin_non_coe} with $(V, W, w) \rsa (V_1, W_1, w_1)$,  we bound 
\[
\bal
  |\cN_{\D, 1}| 
  \les & \int_0^x 
     ( \e \cJaa^{\kp+1}
     + \cJaa^{\kp} \ang y^{-\he + \e_1} 
    + \cJaa^{\kp+1} \lgy^{-4/3} ) \f{y}{\ang y} \cdot \f{\ang y^{\e_1} \cJaa^{\kp-1} }{y} d y \\
 &  \cdot    || w_1 - w_2||_{\cX}   || w_1 ||_{\cX} .
 \eal
\]

For $\cN_{\D, 2}$, applying \eqref{eq:lin_contra_err2} and \eqref{eq:lin_non_cr} with 
$w \rsa w_2$, we bound 
\[
\bal
  |\cN_{\D, 2}| 
   \les & \int_0^x 
     ( \e \cJaa^{\kp+1}
     + \cJaa^{\kp} \ang y^{-\he + \e_1} 
    + \cJaa^{\kp+1} \lgy^{-4/3} ) \f{y}{\ang y} 
      \cdot    \f{  \ang y^{\e_1} \cJaa^{\kp}  }{ y \cJaa } d y \\
     &   \cdot  ||w_2||_{\cX} || w_1 - w_2 ||_{\cX} ,
\eal
\]
where the functions in the integrand evaluate at $y$.

The above $y$-integrals are the same as those in \eqref{eq:lin_non_int1}, which have been estimated in 
\eqref{eq:lin_non_est1}-\eqref{eq:lin_non_est2}. Thus, applying \eqref{eq:lin_non_est2}, 
 $||w_i||_{\cX} < \ddd$ \eqref{eq:solu_contra}, 
and $\cJaa \les \e^{-1}$ by \eqref{eq:Ja_hat}, we bound 
\[
\bal
  |(\cNab(w_1) - \cNab(w_2) )  \wa^{-1} |
  & \leq  |\cN_{\D, 1}| + |\cN_{\D,2}| \\
  & \leq  \min(x, \   |\cJaa(x)|^{2\kp } \ang x^{\e_1}
   ) || w_1 - w_2||_{\cX}   ( || w_1 ||_{\cX} + ||w_2||_{\cX} ) \\
& \les  \min(x,  \ \e^{-\kp}   |\cJaa(x)|^{\kp } \ang x^{\e_1}
 ) \ddd || w_1 - w_2||_{\cX}    .
\eal
\]

We treat $\muc$ determined in Theorem \ref{thm:contra_lin} as an absolute constant.  Applying the above estimate, estimate \eqref{eq:wg_nsingu}, and $\bbb < 2$ from \eqref{def:vpa},  we prove
\begin{align}\label{eq:lin_contra_non}
  |(\cNab(w_1) - \cNab(w_2) ) \max(\vpa, \muc \vpb) |
   & \les   |(\cNab(w_1) - \cNab(w_2) )  \wa^{-1} | \cdot (|x|^{-\bbb+1} + 1) \ang x^{-\e_1} \cJaa^{-\kp}  
   \notag \\ 
  & \les \e^{-\kp} \nu  || w_1 - w_2||_{\cX}    .
\end{align}

\subsection{Proof of Theorem \ref{thm:1D_solu}}\label{sec:1D_fix_pt_pf}

We are in a position to prove Theorem \ref{thm:1D_solu}.

\subsubsection{Proof of into property  \eqref{eq:1D_onto}}

In this section, we prove \eqref{eq:1D_onto}. Recall the $\cX$-norm from \eqref{norm:X_ep} 
\beq\label{eq:norm:X_recall}
   || f ||_{\cX} = \nlinf{ f \max(\vpa, \muc \vpb) }.
\eeq
where  $\muc$ is determined in Theorem \ref{thm:contra_lin}.

Applying estimates \eqref{eq:lin_err_est2} on $\cEab$ and \eqref{eq:lin_non_est2} on $\cNab$, we bound 
\beq
  \| \cEab + \cNab(w) \|_{\cX}  = \nlinf{ (\cEab + \cNab) \max(\vpa, \muc \vpb) } 
   \leq  C_{ \eqref{eq:lin_1D_pf3}  } ( \e + \e^{1/3} || w||_{\cX} + \e^{-\kp}|| w||_{\cX}^2  ) ,
   \label{eq:lin_1D_pf3}
\eeq
for some absolute constant $C_{ \eqref{eq:lin_1D_pf3}  } $ independent of $  \e$. Recall the estimate of $\cLab$ from Theorem \ref{thm:contra_lin}
\[
  || \cLab(w) ||_{\cX} < \lam_{\cX} || w||_{\cX}, \quad \lam_{\cX} < 1,
\]

Now we choose 
\beq
\ddd =  C_{ \eqref{eq:ddd_set} } \e, 
 \quad C_{ \eqref{eq:ddd_set} } = \f{2 C_{ \eqref{eq:lin_1D_pf3}  }  }{1 -\lam_{\cX}} + 1, 
 \quad \Longrightarrow \ C_{ \eqref{eq:lin_1D_pf3}  }  \e < \f{1}{2} (1 - \lam_{\cX}) \ddd.
 \label{eq:ddd_set}
\eeq
Since $ C_{ \eqref{eq:ddd_set} }$ is an absolute constant, $\nu$ satisfies \eqref{eq:ineq_ddd1} for $\e$ small enough. 
Since $\kp < 1$ by \eqref{def:kp}, for $|| w||_{\cX} < \ddd$, by further requring $\e$ small and  combining 
\eqref{eq:lin_1D_pf3} and \eqref{eq:ddd_set}, we prove \eqref{eq:1D_onto}:
\[
\bal
  || \cLab + \cEab + \cNab||_{\cX}
& \leq \lam_{\cX} \ddd 
+ \f{1}{2} (1 - \lam_{\cX}) \ddd 
+  C(  C_{ \eqref{eq:lin_1D_pf3}  }  , C_{ \eqref{eq:ddd_set} } ) (  \e^{1/3} \cdot \e + \e^{-\kp} \e^{ 2  }  )  \\
  & \leq \tf{1}{2}(1 + \lam_{\cX} ) \ddd + 
  C(  C_{ \eqref{eq:lin_1D_pf3}  } , C_{ \eqref{eq:ddd_set} } ) (  \e^{4/3}  +  \e^{ 2 - \kp  }  ) 
   <  \tf{1}{3}( 2 + \lam_{\cX} ) \ddd < \ddd.
\eal
\]

\subsubsection{Proof of contraction property}\label{sec:1D_prop}

Next, we prove the contraction property \eqref{eq:1D_contra}. Consider 
\beq\label{eq:solu_contra2}
  W_i = \va + w_i, \quad  V_i = \va + \psio(w_i), \quad |w_i | < \ddd, \ i = 1,2.
\eeq

Recall the $\cX$-norm from \eqref{eq:norm:X_recall}. Applying \eqref{eq:lin_contra_err} for $\cEab$ and \eqref{eq:lin_contra_non} for $\cNab$, and $\ddd = C_* \e$ from \eqref{eq:ddd_set}, we bound 
\[
\bal
  ||  (\cEab + \cNab)(w_1) - (\cEab+ \cNab)(w_2) ||_{\cX}
 &  \leq C (\e^{1/3} + \e^{-\kp} \ddd  )  || w_1 - w_2||_{\cX}   \leq \bar C_2 (\e^{1/3} + \e^{1-\kp} ) || w_1 - w_2||_{\cX} , 
\eal
\]
for some absolute constant $\bar C_2$.  

Since $\cLab$ is linear, using  the estimate of $\cLab$ from Theorem \ref{thm:contra_lin}, we have 
\[
    || \cLab(w_1) - \cLab(w_2) ||_{\cX} < \lam_{\cX} || w_1 - w_2||_{\cX}, \quad \lam_{\cX} < 1.
\]

Since $ \lam_{\cX}< 1$, combining the above two estimates and taking $ \e$ small enough, we prove 
\[
  ||\cFR(w_1) - \cFR(w_2) ||_{\cX}
  \leq   ||  (\cEab + \cNab + \cLab)(w_1) - (\cEab+ \cNab + \cLab)(w_2) ||_{\cX}
  < \f{1}{2} (\lam_{\cX} + 1) || w_1 - w_2||_{\cX} , 
\]
Choosing $\lam_{\cF} =\f{1}{2} (\lam_{\cX} + 1) < 1$, we prove \eqref{eq:1D_contra}. 

Using \eqref{eq:1D_onto}, \eqref{eq:1D_contra}, and the Banach fixed point theorem, 
we construct a fixed point solution to 
$\cFR(w) = w$ with $||w||_{\cX} < \ddd$. We complete the proof of Theorem \ref{thm:1D_solu}.

\subsection{Regularity and sharp decays of the profile}

Let $w_{\al}$ be the fixed point constructed in Theorem \ref{thm:1D_solu}. 
We have 
\beq\label{eq:1D_profi_small}
  \| w_{\al} \|_{\cX}  < \nu \les \e . 
\eeq
Denote $\waa, \vaa$ 
\beq\label{def:1D_profi}
  \waa = \wa + w_{\al},\quad \vaa = \va + \psio(w_{\al})
  = x + \cK_{\al,1}(\waa) .
\eeq
Then $\waa, \vaa$ solve the profile equation \eqref{eq:1D_dyn_recall}
\beq\label{eq:profi_waa}
   2 \vaa \pa_z \waa = (3- \al - ( 1-\al ) \pa_z \vaa) \waa .
\eeq

Applying Proposition \ref{prop:vel_al} with $ w = w_{\al}$, using triangle inequality and \eqref{eq:va_est},  we obtain 
\beq\label{def:1D_profi:b}
\bal
   |\tf{1}{x} \psio( w_{\al})| + |\pa_x \psio(w_{\al})|
 \les \cJaa^{\kp+1}  \| w \|_{\cX} \les     \cJaa \e^{-\kp} \ddd 
  \les \e^{1-\kp} \cJaa , \\
   \vaa  \geq \tf{1}{2} \va > 0, \quad \vaa(x) \asymp \cJaa(x) \cdot |x| ,
   \quad |\pa_x \vaa | \les \cJaa.
   \eal
\eeq

From the boundedness of $\waa, \pa_x \vaa$, and the above lower bound of $\vaa$, we obtain that $\waa$ is Lipschitz.  Since $\vaa = x + \psioa(\waa)$, using the above estimates, we obtain
\beq\label{eq:1D_profi_psio}
  |\pa_x \psio_{\al}( \waa)|  + | \tf{1}{x} \psio_{\al}(\waa) |
  \les 1 + | \pa_x \vaa| +  | \tf{1}{x} \vaa|
  \les \cJaa.
\eeq

\begin{thm}\label{thm:1D_profile_prop}
Let $\kp$ be the parameter chosen in \eqref{def:kp}. Let $\beta =\beta(\al)$, \(\beps_2\), 
$\muc$ be chosen as in Theorems \ref{thm:1D_error}, \ref{thm:al_appr_profile}, \ref{thm:1D_solu}, respectively.  There exists $\beps_4 > 0$ such that for any $\e = \f13 - \al \in (0, \beps_4)$, the following results hold. 
There exists a $C^{\infty}$ $\al$-profile $\waa$ to \eqref{eq:1D_dyn} and it  satisfies 
$ \| \waa - \wa \|_{\cX} \les  \nu$, 
\beq\label{eq:wa_sign}
\bal
 \pa_x \waa(0) = \pa_x \wwwa(0) \in [-1 - 10^{-6}, -1 + 10^{-6}] , \quad   \waa(x) < 0,  \quad  \forall \, x > 0 , \\
    \nlinf{  (\waa - \ang x^{\b(\al) } \wwwa) (   |x|^{-\bbb} + |x|^{\alb - \b - \kp_1 \e } \cJaa^{-\kp}  ) }  \les \e , \quad \bbb = 1.2 , \eal
\eeq

\begin{enumerate}[label=(\roman*), leftmargin=1.5em]

\item \textsl{(Asymptotic)}.
 There exists a constant $c_{\wwwa} >0$ such that 
\begin{align}
      \exp( -c_{\wwwa} \cdot \e^{1- \kp }  )  \cdot    \ang x^{-c_{\wwwa} \cdot \e^{2 - \kp}}    
       |\wa| & \leq |\waa|  \leq        \exp( c_{\wwwa} \cdot \e^{1-\kp }  )  
        \cdot \ang x^{ c_{\wwwa} \cdot \e^{2-\kp}}  |\wa| , \label{eq:wa_refine:b}  
\end{align}
with $e^{ c_{\wwwa} \e^{2-\kp}} \les 1$, for some constant $c_{\wwwa} >0$ independent of $\al$.
In particular, we have
 \beq\label{eq:wa_upper_lower}
 \bga
     |\xx| \we |\xx|^{ -\hau } \les |\waa(\xx)| \les |\xx| \we |\xx|^{-\hal}, \\
           \hau =  \f13 + \f18 \e + c_{\wwwa} \e^{2-\kp}, 
     \quad \hal = \f13 + \f18 \e - c_{\wwwa} \e^{2 - \kp} .
     \ega
 \eeq 
The exponents $\b, \hal, \hau$ satisfy 
\beq\label{ran:ep2}
 \he = \e -  \b   , \  \hal - \al   , \  \hau - \al  \in [ \f{9}{8} \e - C \e^{ 2 -\kp }, \, \f{9}{8} \e + C \e^{ 2 -\kp } ]
     \subset [ \f{8.9}{8} \e,  \ \f{9.1}{8} \e ] .
\eeq

\item \textsl{(Regularity)}. The profile $\waa(x)$ has  $ C^{\infty}$-regularity is odd in $ x$, and satisfies 
\bseq\label{eq:waa_reg}
\begin{align}
  |\pa_x^k \waa |  & \les_k \ang x^{- k  -\hal },    \quad  \forall  \ k \geq 0,  \label{eq:waa_reg:a} \\
 \qquad  | \pa_x^{k+1} \vaa |  & \les_k \ang x^{-k + \al -\hal} , \quad  \forall \  k \geq 1.
 \label{eq:waa_reg:b}
\end{align}
\eseq

\item \textsl{(Scaling)}. Recall $\pa_x \psi(f)(0) = 2 \al \cJa(f)(\infty)$ from \eqref{eq:iden_psiz_J}. Denote
\bseq\label{eq:wa_decay_1D}
\beq\label{eq:wa_decay_1D:a}
\bar c_{l,\al  } = 2 - 4\al \cJa( \waa)(\infty), \quad  \bar c_{\om, \al}  = 2 + (1 - \al) 2 \al \cJa(\waa)(\infty) .
\eeq
We have 
\beq
\cJa(\waa)(\infty) = - \tf{16}{3 } \e^{-1} + O(\e^{2-\kp}), 
   \quad \bar c_{l,\al } \, , \quad  - \bar c_{\om, \al } \asymp \e^{-1},
\eeq
and 
\beq\label{eq:wa_decay_1D:c}
      \f{ \bar c_{\om, \al  }}{ \bar c_{l, \al  } }= -   \f{1}{3} - \f{\e}{8} 
      +  O (\e^{2 - \kp})
\in  [ - \f13- \f{\e}{7} ,- \f13- \f{\e}{9}]  , \ 
 \bar c_{l,\al } 
       = \f{64}{9 \e }  + O(\e^{-\kp}) , 
    \  \bar c_{\om, \al  } = - \f{64}{27 \e }  +  O(\e^{-\kp}). 
\eeq

\eseq

\item \textsl{(Log-improvement)}. We have the following estimates on $\waa$
\begin{align}
   |\f{x \pa_x \waa}{\waa} + \f{1}{3} | & \les \e  + |\lgp x|^{\kp-2}
     , \qquad  |  \f{ \pa_x \waa}{\waa} |  \les  |x|^{-1} . 
\label{eq:1D_prof_est2:c}
\end{align}
and 
\beq\label{eq:1D_prof_dxx} 
   | \pa_x  ( \f{x \pa_x \waa}{\waa} )  |  \les 
   \ang x^{-1} \cJaa^{-1} , \quad 
   | x\pa_x  ( \f{\pa_x \waa}{\waa} )  | \les |x|^{-1}  , 
\eeq

In particular, there exists $C_{\wwwa}$ only depending on $\wwwa$, such that \beq\label{eq:1D_prof:low}
   x  \f{ \pa_x \waa}{\waa} + \f{1-\al}{2} \geq - C_{\wwwa} \f{|\xx|}{\ang \xx } ( \e^{2-\kp} + |\lgp x|^{\kp - 2}  )  .
\eeq

\item \textsl{Estimates for $\vaa$} Recall $\he = \e - \b$ from \eqref{def:para}. We have the following estimates for $\vaa$
\begin{align}
   |\pa_x \vaa - \tf{1}{x} \vaa - 4 \ang x^{-\he} | \les 
   \ang x^{-\he}  |\lgp x |^{-1/3}  +  \e^{1- \kp} , \label{eq:1D_prof:va} \\
   \vaa \geq \tf{1}{2} \va > 0, \quad \vaa(x) \asymp \cJaa(x) \cdot |x| .
   \label{def:1D_profi:b2}
\end{align}
\end{enumerate}
We emphasize that all the implicit constants in the above estimates are independent of $\al, \e$.

\end{thm}

\begin{remark}[\bf $ \waa$ is close to $ \wa$]
 The above estimates except for the scaling \eqref{eq:wa_decay_1D} essentially show that $\waa$ is sufficiently close to  $\wa$, and \(\waa\) satisfies estimates analogous to those in Theorem \ref{thm:al_appr_profile} for the approximate profile \(\wa\). Moreover, \(\waa\) solves \eqref{eq:1D_dyn} \emph{exactly}.
\end{remark}

\begin{proof}

We first prove \eqref{eq:wa_upper_lower}, \eqref{ran:ep2}, and then prove \eqref{eq:wa_sign}.

\textbf{Proof of \eqref{eq:wa_upper_lower},\eqref{ran:ep2} and estimate between $\waa, \wa$}.
Multiplying the profile equation \eqref{eq:1D_dyn} with $\f{1}{\wa}$ and using \eqref{eq:err_ep}, we obtain 
\beq\label{eq:waa_pf1}
  2 \vaa \pa_x (\f{\waa}{\wa} )= \cR(\waa) \f{\waa}{\wa}, 
  \quad  \cR(\waa) = (3 - \al - (1-\al) V_x(\waa) - 2 V(\waa) \f{\pa_x \wa}{\wa} ) .
\eeq

Using \eqref{eq:lin_non_cr} with $w = w_{\al}$ and Theorem \ref{thm:1D_error}, we have the following estimates 
\[
\bal
  |\cR( \waa )| = |\crab| +| \td \cR(w_{\al})|
    \les &  x \ang x^{-1} (  \e \cJaa^{\kp+1}
     + \cJaa^{\kp} \ang x^{-\he + \e_1} 
     + \cJaa^{\kp+1} \lgx^{-4/3}) || w_{\al} ||_{\cX} \\
     &  +   \e  \min\B(  x, \  \min(  \lgp x, \e^{-1} )   |\lgp x|^{-1/3} \B).
  \eal
\]

Since $|| w_{\al}||_{\cX} \les \e$, using the estimate 
$\cJaa \asymp \min( \lgp x , \e^{-1} )$ \eqref{eq:Ja_hat} and $\kp > \f{2}{3}$ by \eqref{def:kp}, we bound 
\beq\label{eq:cR_waa:est}
\bal
    |\cR( \waa )| & \les x \ang x^{-1} 
    \big( \one_{\e \lgp x \leq 1}  |\lgp x|^{\kp} \e 
     + \one_{\e \lgp x > 1} \e^{1-\kp}  
      \big)  \\
    & \les x \ang x^{-1} \e \cdot \min(\lgx, \e^{-1} )^{\kp} 
     \les x \ang x^{-1} \e \cJaa^{\kp} \les x \ang x^{-1} \e^{1 -\kp}.
  \eal
\eeq

Using $\vaa \asymp x \cJaa$ from estimate \eqref{def:1D_profi:b}, 
and $\cJaa^{-1} \les \e^{-1} + \lgp x$ by \eqref{eq:Ja_hat}, we obtain
\beq\label{eq:wa_refine:a}
\bal
 \B| \f{\pa_x \waa}{\waa} - \f{\pa_x \wa}{\wa} \B| & =   \B| (\f{\waa}{\wa})^{-1}  \cdot \pa_x (\f{\waa}{\wa}  \B| = \B| \f{\cR(\waa)}{ 2 \vaa  } \B| 
  \les \ang x^{-1} \e \cJaa^{\kp-1} \\
  & \les \ang x^{-1} \e ( \e^{1-\kp} + |\lgp x|^{\kp-1}).
\eal
\eeq

Next, we estimate the integral 
\bseq\label{eq:S1_est1}
\beq
  S   \teq \int_0^{x}   \f{\cR(\waa)}{ 2 \vaa  }(y) d y  , \\
  \eeq
Using estimate \eqref{eq:cR_waa:est} on $\cR(\waa)$ and $\vaa(y) \gtr y \min( \lgp y, \e^{-1} )$ by
\eqref{def:1D_profi:b}, we bound 
  \beq
\bal
|S| & \les \int_0^x \ang x^{-1} \e \min(\lgx, \e^{-1} )^{\kp-1}  
\les \int_0^x \ang x^{-1} \e (\lgx^{\kp-1} \one_{ \e \lgp x < 1} +  \e^{1-\kp} )  \\
 & \les  \e^{1 -\kp} + \e^{ 2 - \kp } \log \ang x .
\eal 
\eeq
  The above integral can be obtained by discussing $\e \lgp x \leq 1$ and $ \e \lgp x > 1$.
\eseq

Since $ \f{\waa}{\wa}|_{x =0} = 1$, integrating the above estimates from $0$ to $x$, 
and using $\cJaa \les \e^{-1}$, we prove 
\beq\label{eq:wa_refine:b:pf}
\bal
  \f{\waa}{\wa} &= \exp( S(x)  ) \geq \ang x^{ - c\cdot  \e^{2-\kp} } \exp(- c \cdot \e^{1-\kp} ) , 
  \quad   
 \f{\waa}{\wa} = \exp( S(x)  ) \leq \ang x^{  c \cdot  \e^{2-\kp} } \exp( c \e^{1-\kp} )  .
 \eal
\eeq
for some constant $c= c(\kp)>0$. It follows \eqref{eq:wa_refine:b}.

Since $\b(\al) = - \f18 \e + O(\e^{4/3})$ \eqref{eq:beta_est}, 
and $\wa = \ang x^{\b} \wwwa$ \eqref{eq:wa}, estimate \eqref{eq:wa_upper_lower} follows from \eqref{eq:wa_refine:b}.

Estimate \eqref{ran:ep2} on $\hal, \hau, \b$ follows from 
\eqref{eq:wa_upper_lower}, \eqref{eq:beta_est} and by choosing $\e$ small enough.

\vs{0.1in}
\paragraph{\bf Proof of \eqref{eq:wa_sign} }
Since $\waa = \wa + w_{\al} \in C^{\infty}, \wa \in C^{\infty}, \pa_x \wa(0) = 
\pa_x \wwwa(0)$ by \eqref{eq:wa}, and $|w_{\al}| \les |x|^{\bbb}$ by $w_{\al} \in \cX$ and the definition of 
$\cX$-norm \eqref{norm:X_ep}, we obtain $\pa_x \waa(0) = \pa_x \wwwa(0) $. 
Using estimate for  $\pa_x \wwwa(0) $ from \eqref{eq:bw_sign}, we obtain estimate \eqref{eq:wa_sign} 
for $ \pa_x \waa(0) < 0$, and $\waa < 0$ for small $x>0$. Using \eqref{eq:wa_upper_lower}, we prove $\waa<0$ for any $x>0$.

Since $\e_1 = \kp \e $ \eqref{norm:X_ep}, from \eqref{eq:wg_equiv:a}, 
we obtain $\max(\vpa, \muc \vpb) \asymp |x|^{-\bbb} + |x|^{\alb - \b -\e_1}   \cJaa^{-\kp} $. 
Using $ w_{\al} = \wwwa- \ang x^{\b} \wwwa$ and $\nchi{w_{\al}} \les \e$,
we prove the second estimate in  \eqref{eq:wa_sign}.

\vs{0.1in}
\paragraph{\bf Proof of \eqref{eq:waa_reg} }

Recall that $\waa$ is locally Lipschitz from the discussion below \eqref{def:1D_profi:b}. 
Since $\waa$ solves the profile equations \eqref{eq:profi_waa}, combining the 
estimates in \eqref{eq:wa_upper_lower}, \eqref{eq:1D_profi_psio}, and \eqref{def:1D_profi:b},
we obtain
\[
 |\waa| \les \min( |x|, \ang x^{-\hal}), 
 \quad | \tf{1}{x} \psio_{\al}(\waa) | +|\pa_x \psio_{\al}(\waa)| \les \cJaa,
 \quad  1 + \tf{1}{x} \psio_{\al}(\waa) = \tf{1}{x} \vaa \gtr \cJaa.
\]

Therefore, the assumptions \eqref{eq:Cinfty_ass} in Theorem \ref{thm:reg} are satisfied 
for $\s$ being some large absolute constant, $\g = \hal$. Using 
\eqref{eq:profile_Ck} in Theorem \ref{thm:reg}, we prove \eqref{eq:waa_reg:a}. 
Since $\vaa = x + \psio_{\al}(\waa)$ and $\pa_x^{k+1} \vaa = \pa_x^{k+1} \psio_{\al}(\waa)$ for $k \geq 1$, using \eqref{eq:waa_reg:a} and Lemma \ref{lem:vel_high}, we prove \eqref{eq:waa_reg:b}.

\vs{0.1in}
\paragraph{\bf Proof of \eqref{eq:wa_decay_1D} }

Recall $\waa =\wa + w_{\al}$ with $ \| w_{\al} \|_{\cX} \les \e$. 
Using the estimates \eqref{eq:Ja_est:b} $\| w_{\al} \vpb \|_{L^{\infty}} 
\les \|  w_{\al}\|_{\cX} \les \e$, and \eqref{eq:Ja_hat}, we obtain
\[
|\cJa( w_{\al})(\infty)|
\les \cJaa^{1 + \kp}  \|  w_{\al}\|_{\cX}
\les \e^{-1-\kp } \e \les \e^{-\kp},
\quad  \cJa( \wa)(\infty) = - \tf{6}{\he} + O(\e^{-2/3}) , 
\]
where $\he = \e - \b = \tf{9}{8} \e + O(\e^{4/3})$ by \eqref{def:para} and 
\eqref{eq:beta_est}. Thus, for $\e>0$ small enough, 
using the definition of 
$\clone$ in \eqref{eq:wa_decay_1D} and  the above estimate, we obtain 
\[
\bal
 \cJa(\waa)(\infty) &= - 6 \he^{-1} + O(\e^{-\kp}) 
 = - \tf{16}{3 } \e^{-1} + O(\e^{-\kp}) ,  \quad \he =  \tf{9}{8} \e + O(\e^{4/3}), \\
\clone &= -4 \al \cJa(\waa)(\infty) + O(1)
 = \tf{64}{9} \e^{-1}   + O(\e^{-\kp}) \asymp \e^{-1},\\
  \cwone & = (1-\al) 2 \al \cJa(\waa)(\infty) + O(1)
 = - \tf{64}{27} \e^{-1} + O(\e^{-\kp}) \asymp - \e^{-1}.
\eal
\]

Using the definition of 
$\clone, \cwone$, and $\e = \f13 - \al$, we obtain
\[
 3 \cws + \cls  = 8 + (1-3\al) 2 \al  \cJa(\waa)(\infty) = 8 + 3 \e \cdot 
 \tf{2}{3} \cdot ( - \tf{16}{3} \e^{-1} ) + O(\e^{1-\kp} )
 = - \tf83  + O(\e^{1-\kp} ).
\]

For $\e$ small enough, using the above estimates, we prove \eqref{eq:wa_decay_1D:c}:
\beq\label{eq:cw_cl_rat}
\bal
  \f{ \cwone}{ \clone } + \f13 & = \f{ 3 \cws + \cls}{ 3 \cls}      
   = \f{  - \tf83 + O(\e^{1-\kp})  }{  \f{64}{3} \e^{-1} + O(\e^{- \kp}) }
  = - \f{\e}{8} + O(\e^{2-\kp}) \in [- \f{\e}{7} ,- \f{\e}{9}] ,
\eal
\eeq

\vs{0.05in}

\paragraph{\bf Proof of \eqref{eq:1D_prof_est2:c}}
Dividing \eqref{eq:waa_pf1} by $\f{\vaa}{x} \cdot \f{\waa}{\wa }$, we obtain
\beq\label{eq:waa_iden2}
 \f{\cR(\waa)}{ \vaa / x }=
  \f{2 x \pa_x ( \waa / \wa) }{  \waa / \wa } 
  =  \f{2 x \pa_x \waa}{\waa} - \f{2 x \pa_x \wa}{\wa}.
\eeq

Using \eqref{eq:cR_waa:est} for $\cR$, \eqref{def:1D_profi:b} for $\vaa / x$, 
$ -\b = \f{\e}{8} + O(\e^{4/3})$ by \eqref{eq:beta_est}, and \eqref{eq:wa_est:a}, we obtain
\[
\B| \f{ x \pa_x \waa}{\waa} + \f{1}{3} + \f{\e}{8} \B| \les \B| \f{x \pa_x \wa}{\wa}
+ \f{1}{3} - \b \B|  +  \f{ |\cR(\waa) | }{ |\vaa / x| } + \e^{4/3}
\les |\lgp x|^{-4/3} + \e^{4/3} + \f{ \e \cJaa^{\kp} }{\cJaa}
\]
Since $\cJaa^{-1} \les \e + |\lgp x|^{-1}$ and $\kp \in (0.8, 1)$ by \eqref{def:kp}, using Young's inequality, we prove: 
\beq\label{eq:1D_prof_est2:b}
\B| \f{ x \pa_x \waa}{\waa} + \f{1}{3} + \f{\e}{8} \B| \les \e ( \e^{1-\kp} + |\lgp x|^{\kp-1} )
+ |\lgp x|^{-4/3} + \e^{4/3}  \les \e^{2 -\kp} + |\lgp x|^{2-\kp}. 
\eeq

Using triangle inequality, $\e^{2-\kp} \les \e \leq 1, $ 
 $|\lgp x|^{-1} \les 1$, and the above estimate, we prove  \eqref{eq:1D_prof_est2:c}.

\vs{0.1in}

\paragraph{\bf Proof of \eqref{eq:1D_prof_dxx}}
Taking $D_x$ on both sides of \eqref{eq:waa_iden2} and using \eqref{eq:wa_est:b}, we obtain
\[
 | \pa_x \f{x \pa_x \waa}{\waa}| \les  | \pa_x \f{x \pa_x \wa}{\wa}|
 + | \pa_x  \f{\cR(\waa)}{ \vaa / x }|
 \les  \ang x^{-1} |\lgp x|^{-1} +  | \pa_x  \f{\cR(\waa)}{ \vaa / x }| .
\]

Since $\al - \hal < 0 $ by \eqref{ran:ep2}, using \eqref{eq:waa_reg:b} with $k=1$, we obtain
\beq\label{eq:waa_iden3}
\bal
 | \pa_x \f{\vaa}{x}| & = | \f{ \pa_x \vaa - \tf{1}{x} \vaa}{x}|
 = \B| \f{1}{x^2} \int_0^x \pa_{yy} \vaa(y) \cdot y d y \B|
 \les \f{1}{x^2} \int_0^x \min( |y|, \ang y^{\al-\hal} ) d y \\
& \les \min( 1,  \ang x^{\al-\hal-1} )
 \les \ang x^{-1}.
\eal
\eeq

Using \eqref{def:1D_profi:b}, the above estimate, $\al-\hal < 0$ by \eqref{ran:ep2}, and \eqref{eq:cR_waa:est}, we obtain
\[
\B| \pa_x  \f{\cR(\waa)}{ \vaa / x } \B| 
\les \B| \f{\pa_x \cR(\waa)}{  \vaa / x  } \B| + \B| \f{\cR(\waa) \pa_x (\vaa/x)}{ (\vaa/x)^2 } \B|
\les \f{ | \pa_x \cR(\waa) | }{\cJaa} + \f{ \ang x^{-1}  }{ \cJaa^2} .
\]

Using the formula for $\cR(\waa) $ in  \eqref{eq:wa_iden1} with $\vaa = V(\waa)$, we obtain
\[
\bal
 |\pa_x \cR(\waa)| & \les |  \pa_{xx} \vaa | + | \pa_x ( \f{ \vaa}{x} \cdot \f{x \pa_x \wa}{\wa} ) |
 \les |  \pa_{xx} \vaa | + | \pa_x ( \f{ \vaa }{x} ) \cdot \f{x \pa_x \wa}{\wa}   |
 +  |  \f{ \vaa}{x}  \cdot \pa_x \f{x \pa_x \wa}{\wa}  | \\
\eal
\] 

Using estimate \eqref{eq:waa_iden3} and \eqref{def:1D_profi:b} for $\vaa / x$,
\eqref{eq:waa_reg:b}, bound \eqref{eq:wa_est} for $\wa$, 
 and $\cJaa(x) \les \lgp x$, we estimate 
\[
  | \pa_x \cR(\waa)|  \les  \ang x^{\al -\hal - 1} + \ang x^{-1}
 + \cJaa \cdot  \ang x^{-1} |\lgp x|^{-1} \les \ang x^{-1}.
 \]

Combining the above estimates, we prove the first estimate in \eqref{eq:1D_prof_dxx}.
\[
 | \pa_x \tf{x \pa_x \waa}{\waa}| \les \ang x^{-1} ( |\lgp x|^{-1} + \cJaa^{-1} + \cJaa^{-2} ) \les 
\ang x^{-1} \cJaa^{-1} .
\]

The second estimate in \eqref{eq:1D_prof_dxx} follows from the above estimate and 
\eqref{eq:1D_prof_est2:c}, and triangle inequality.

\vs{0.1in}
\paragraph{\bf Proof of \eqref{eq:1D_prof:low} }

Denote $p(x) =   x  \f{ \pa_x \waa}{\waa} + \f{1-\al}{2}$. 
Since  $\waa \in C^3$ with estimates \eqref{eq:waa_reg:a} uniformly in $\al$, $\waa(0) = 0$, $\pa_x \waa(0) < -\f12$ \eqref{eq:wa_sign}, using Taylor expansion, we obtain
\[
  p(x) \geq \tf12 + \tf{1-\al}{2 } \geq \tf12 ,  \quad \forall \, |x| \leq c,
\]
with an absolute constant $c>0$ independent of $\al$.

Since $\al = \f{1}{3} -\e$, $\f{1-\al}{2} = \f{1}{3} + \f{\e}{2}$, for $|x| \geq c$, using \eqref{eq:1D_prof_est2:b}, we obtain
\[
  p(x) =  x  \f{ \pa_x \waa}{\waa} + \f{1-\al}{2}
  \geq x  \f{ \pa_x \waa}{\waa} + \f{1}{3} + \f{\e}{8} 
\geq - C ( \e^{ 2- \kp } + |\lgp x|^{ \kp-2} )
\geq - 
C  \f{|\xx|}{\ang \xx }( \e^{2-\kp} + |\lgp x|^{\kp-2} ).
\]
Combining the above estimates, we prove \eqref{eq:1D_prof:low}.

\vs{0.1in}
\paragraph{\bf Proof of  \eqref{eq:1D_prof:va}, \eqref{def:1D_profi:b2} }

Using $\vaa =\va + \psio$, estimate \eqref{eq:va_dif} for $\va$, Proposition \ref{prop:vel_al} for $\psio$, and triangle inequality, we bound 
\[
\bal
   |\pa_x \vaa - \tf{1}{x} \vaa - 4 \ang x^{-\he} |
& \leq    |\pa_x \va - \tf{1}{x} \va - 4 \ang x^{-\he} |
+ |\psiox - \tf{1}{x} \psio| \\
& \les \ang x^{-\he} ( |\lgp x|^{-1/3} + \e ) + \cJaa^{\kp} \ang \xx^{-\he + \e_1} || \om_{\al} ||_{\cX}.
\eal
\]
Since $\cJaa \les \e^{-1}$, $||\om_{\al} ||_{\cX} \les \e$, 
and $\he > \e_1 > 0$ by \eqref{ran:ep}, we prove  \eqref{eq:1D_prof:va}:
\[
\bal
     |\pa_x \vaa - \tf{1}{x} \vaa - 4 \ang x^{-\he} | & \les 
      \ang x^{-\he} |\lgp x|^{-1/3} + \e  + \e^{1-\kp}  \ang \xx^{-\he + \e_1}   \les       \ang x^{-\he}  |\lgp x|^{-1/3}   + \e^{1-\kp} .
\eal
\]

Estimate \eqref{def:1D_profi:b2} follows from \eqref{def:1D_profi:b}. We conclude the proof of Theorem \ref{thm:1D_profile_prop}.
\end{proof}

\subsection{Contraction estimates around the $\al$-profile}
We introduce the weight and norm 
\beq\label{def:vpcc}
\vpcc  = |\waa|^{-1}
\ang x^{ c_{ \wwwa} \e^{2 -\kp} } \ang x^{- \e_1} \cJaa^{- \kp} \vpk ,
\quad  \| w \|_{\cXc} \teq  \nlinf{ \max(  \vpa, \muc \vpcc ) w }.
\eeq
where $\vpk$ is defined in \eqref{def:vpb} 
and $\muc$ is the parameter in $\cX$-norm \eqref{norm:X_ep} and is chosen in Theorem \ref{thm:contra_lin},
and $c_{ \wwwa}$ is the constant in \eqref{eq:wa_refine:b}. We define the linearized operator around 
the profile $(\waa, \vaa)$ analog to $\cLab(w)$ in \eqref{eq:F_lin_ep} 
\beq\label{def:cLa}
\bga
  \cLa (w) \teq \waa  \int_0^{x}  \f{  \cRa(w)}{2 \vaa }  ,  
  \quad  \cRa( w )   = - (1-\al) \psiox (w) - 2 \psio  \f{\pa_x \waa}{\waa}  , 
 \ega
\eeq
where $\psio(w) = \cK_{\al, 1}(w)$ defined in \eqref{eq:ker}. We extend Theorem \ref{thm:contra_lin} as follows 
\begin{thm}\label{thm:contra_lin_exact}
Let $\kp_1$ be defined in \eqref{def:kp}. Let $\beps_4$ be the parameter, and let $(\waa, \vaa)$ be the profile constructed in Theorem \ref{thm:1D_profile_prop}. There exists $\beps_5 \in (0, \beps_4)$ such that for any $ \e \in (\beps_5, 0)$, we have
\[
 |\max(  \vpa, \muc \vpcc ) \cLa(w)| \leq  
  \max(  \vpa, \muc \vpcc ) \B| \waa  \int_0^{x}  \f{  \cRa(w)  }{2 \vaa }  \B|
\leq  \min( \lamcL , C \e^{-1} \ang x^{- \kp_1 \e /2} )
\| w \|_{\cXc}  ,
\]
where $\lamcL =\f12 (\lam_{\cX} +1) = \tf34 + \tf14 \max( \f{\cff}{2 \kp} ,  0.95 ) \in (0, 1) $,
and $\cff, \kp$ are the constants defined in \eqref{def:kp}.

\end{thm}

\begin{proof}

First, we decompose $\cLa$ as 
{\small
\beq\label{eq:lin_exact_dec1}
   \cLa(w) = \waa \int_0^{x}  \f{  \td \cR(w) }{2 \va } 
+  \waa \int_0^{x} 
 \f12( \f{\cRa(w)}{ \vvva} - \f{ \td \cR(w)}{\va} )
\teq \waa \cdot ( \cL_1 + \cL_2).
\eeq
}
We denote by
\beq\label{def:Ge}
c_\e  = \exp( c_{\wwwa} \cdot \e^{1-\kp }  ) ,\quad  G_{\e} = \ang x^{ c_{\wwwa}  \e^{2-\kp}}  ,
\eeq
the functions and constants in \eqref{eq:wa_refine:b}, which are perturbations of $1$. Using the definitions of $\vpcc$ in \eqref{def:vpcc} and $\vpb$ in \eqref{def:vpb}, 
estimate \eqref{eq:wa_refine:b} on the ratio between $\wa, \waa$,  we compare
\beq\label{eq:lin_contra_pf1}
 |\waa| \leq c_{\e} G_{\e} |\wa|,  \quad 
\vpb \leq c_{\e} \vpcc, \quad \vpcc \leq \vpb c_{\e} G_{\e}^2 .
\eeq

We multiply the power weight $G_{\e}$ in $\vpcc$ in \eqref{def:vpcc} so that 
$\vpcc$ is stronger than $\vpb$ and the norm $\cXc$ is stronger than $\cX$ in \eqref{norm:X_ep}. Using Theorem \ref{thm:contra_lin} and the above estimates, we estimate $\cL_1$ as 
\beq\label{eq:cL1_est}
\bal
 & \max(\vpa, \muc \vpcc) |\waa| \B| \int_0^x \f{\td \cR(w)}{ 2 \va} \B|
  \leq (c_{\e} G_{\e})^3 \max(\vpa, \muc \vpb)\cdot |\wa|  \B| \int_0^x \f{ \td \cR(w) }{ 2 \va} \B| \\
 & \qquad \leq  (c_{\e} G_{\e})^3  \min( \lam_{\cX} , C \e^{-1} \ang x^{-\e_1} ) \| w \|_{\cX} 
\leq   (c_{\e} G_{\e})^4 \min( \lam_{\cX} , C \e^{-1} \ang x^{-\e_1} ) \| w \|_{\cXc}.
 \eal
\eeq

\paragraph{\bf Estimate of $\cL_2$}
Recall $\waa =\wa + w_{\al}$ from \eqref{def:1D_profi}. Using  $\td \cR(w)$ from \eqref{eq:lin_nloc_ep:a}, we decompose 
\[
 \f{\cRa(w)}{ \vvva} - \f{ \td \cR(w)}{\va}
 = \f{\cRa(w) - \td \cR(w)}{ \vvva} + \f{\td \cR(w) ( \va - \vvva  )}{ \vvva \va}
 = - \f{2 \psio(w)}{\vvva}  (  \f{\pa_x \waa}{\waa} - \f{ \pa_x \wa}{\wa} )
 -  \f{\td \cR(w) \psio(w_{\al}) }{ \vvva \va} \teq I + II.
\]

For $I$, using $\vaa \gtr x \cJaa$, estimate \eqref{eq:ass_v} for   $\psio(w)$,
 estimate \eqref{eq:wa_refine:a}, $\kp \in (1/2, 1) $ \eqref{def:kp}, we bound 
 \beq\label{eq:lin_exact_I}
|I(y) | \les \f{ y \cJaa^{\kp+1} \nchi{w}}{ y \cJaa } \cdot \ang y^{-1} ( \e  |\lgp y|^{\kp-1} + \e^{2 - \kp} )
\les   \f{ \cJaa^{\kp}}{ \ang y }    ( \e  |\lgp y|^{\kp-1} + \e^{2 - \kp} ) \nchi{w} .
 \eeq
For $II$, using \eqref{eq:lin_non_cr} for $\td \cR(w_{\al})$, 
$\nchi{w_{\al}} \les \e$, \eqref{eq:ass_v} , and $\vvva, \va \gtr x \cJaa$ \eqref{def:1D_profi:b2}, \eqref{eq:va_est}, we bound 
\begin{align}\label{eq:lin_exact_I}
|II(y)| & \les
  \f{y}{ \ang y}  \big(  \e \cJaa^{\kp+1}
     + \cJaa^{\kp} \ang y^{-\he + \e_1} 
     + \cJaa^{\kp+1} \lgp y^{-4/3} \big) \nchi{ w} \cdot \f{ y \cJaa^{\kp+1} \nchi{ w_{\al}}}{ y^2 \cJaa^2} \notag \\
     & \les \ang y^{-1} \cJaa^{\kp}  \cdot \e \cJaa^{\kp-1} \nchi{w} 
     \les \ang y^{-1} \cJaa^{\kp}  \cdot  ( \e  |\lgp y|^{\kp-1} + \e^{2 - \kp} ) \nchi{w}  .
\end{align}

Since  $\cXc$-norm is stronger than $\cX$-norm with $c_{\e} \les 1$ and 
$\cJaa$ is increasing, plugging in the above estimates for $I, II$ in $\cL_2$ \eqref{eq:lin_exact_dec1}, we bound 
\[
\bal
 |\cL_2 |  \les |\cJaa(x)|^{\kp}\| w \|_{ \cXc} \int_0^x \f{ ( \e  |\lgp y|^{\kp-1} + \e^{2 - \kp} )  }{\ang y} d y   
 \les |\cJaa(x)|^{\kp}\| w \|_{ \cXc} \min( \e x,  \e |\lgp x |^{\kp} + \e^{2 -\kp} \lgp x ).
  \eal
\]

Using  \eqref{eq:wg_equiv:b} for $\max(\vpa, \vpb) |\wa|$,  \eqref{eq:lin_contra_pf1},  the above estimate, and Lemma \ref{lem:lgx_pow} , we obtain
\[
\bal
 |\max(\vpa, \muc \vpcc) \waa \cL_2 |
& \les   G_{\e}^3 
 \max(\vpa, \muc \vpb) |\wa \cL_2 |
 \les   G_{\e}^3  ( |x|^{-\bbb+1} + 1) \ang x^{-\e_1} \cJaa^{-\kp} |\cL_2| \\
 & \les  G_{\e}^3  ( |x|^{-\bbb+1} + 1) \ang x^{-\e_1}  
\| w \|_{ \cXc} \min( \e x, \e |\lgp x |^{\kp} + \e^{2 -\kp} \lgp x ) \\
& \les  
 G_{\e}^3   \ang x^{-\e_1}  
\| w \|_{ \cXc}  ( \e |\lgp x |^{\kp} + \e^{2 -\kp} \lgp x + \e )
\les  G_{\e}^3 \ang x^{-\e_1/2}  \e^{1 - \kp} \| w \|_{ \cXc}  .
 \eal
\]

Recall $\e_1 \asymp \e$  \eqref{ran:ep}. We have $G_{\e}^3 \les  \ang x^{\e_1/2} $, 
$c_{\e}\leq 1 + C \e^{1-\kp}$ from \eqref{def:Ge} for $\e$ small enough, and
\[
\bal
G_{\e} & \leq \exp( C \e^{ \mhk}  ) 
\leq 1+ C \e^{\mhk}, \  &  \mbox{for} \ \e^{ 1 + \mhk} \lgp x \leq 1, \\
\quad    G_{\e}(x)^5 \ang x^{- \f{\e_1}{2} }  & \les x^{ - \f{\e_1}{4}} 
\leq e^{-C \e^{-\mhk} } \ll \e^2 , 
& \mbox{for}  \ \e^{ 1 + \mhk} \lgp x > 1 .
\eal
\]

Combining the above estimates, \eqref{eq:cL1_est} for $\cL_1$, the above estimate for $\cL_2$, and choosing $\e$ small enough, we prove the desired estimates.
\end{proof}

\subsection{Proof of Theorems \ref{thm:main_1D}, \ref{thm:main_contract}  }

Theorem \ref{thm:main_contract} follows directly from Theorem \ref{thm:contra_lin_exact}. 
The results in Theorem \ref{thm:main_1D} other than \eqref{eq:blowup_scaling} 
follows from Theorem \ref{thm:reg_alb} and Theorem \ref{thm:1D_profile_prop}. 

It remains to prove \eqref{eq:blowup_scaling}. 
Since $\al = \f13 - \e$,  using \eqref{eq:wa_decay_1D}, 
we obtain 
\[
\bal
  \al +  \tf{\cws}{ \cls}
 & = \al  - \tf13  - \tf18 \e + O(\e^{2-\kp})
= - \tf{9 \e}{8}  + O(\e^{2-\kp}),
\quad   \al +  \tf{\cws}{ \cls} =  [ - \tf{10}{9} \e, -\tf{8}{7} \e] , \\
   \clss  &  = \f{ -\cls}{  \al \cls + \cws}
   = - (  \al +  \f{\cws}{ \cls} )^{-1}
   =  \f89 \e^{-1} + O(\e^{-\kp}), 
   \quad \clss \asymp \e^{-1},  \\
         - (\al \cls + \cws) & = \f{\cls}{\clss} 
   =  \f{ \f{64}{ 8 \e} + O( \e^{-\kp}) }{  \f89 \e^{-1} + O(\e^{-\kp}) }
   = 8 + O( \e^{1-\kp}),
\eal
\]
 Similarly, we obtain $\cwss = -\f{8}{27} \e^{-1} +  O(\e^{-\kp}) $. 
Using the rescaling relation 
\eqref{eq:1D_rescale}, we prove \eqref{eq:blowup_scaling}.

\vs{0.15in}
\noindent\textbf{Acknowledgments.}
The work of JC was partially supported by NSF Grant
DMS--2408098.

\appendix 

\section{Basic estimates for kernels and the $\cJaa$-function}\label{app:basic}

In this appendix, we present some basic estimates for the kernels 
$\cK_{\alb, i}$ \eqref{eq:ker} and for the function $\cJaa$ \eqref{eq:Ja_hat}.
We have the following simple estimates for the kernels $K_{\al,i}$.

\begin{lem}\label{lem:K_decay_basic}
Let $K_{\al,i }, K_{\al, i, J}$ be the kernels in \eqref{eq:ker}.
For  $ |\al - \alb| < \f{1}{100}$ and any $z > 0$, we have 
\bseq\label{eq:K_decay}
\begin{align}
 |K_{\al, 1}(1, z)|  &\les \one_{z \leq 3 } z^{\al-1} + \one_{z > 3} |z|^{\al-3} ,
\label{eq:K_decay_basic:a}
  \\
 | K_{\al,2}(1, z) | & \les \one_{z \leq 3} ( z^{\al-1} + |z-1|^{\al-1} )
 + \one_{z > 3} |z|^{\al-3}, \label{eq:K_decay_basic:b} \\
  |K_{\al, 1, J}(1, z)| & \les \min( 1, z^{\al-3} ),  \label{eq:K_decay:a}  \\
   | K_{\al, 2, J}(1, z) |  + |K_{\al, \D}(1, z)| & \les \min( |z-1|^{\al-1}, z^{\al-3}  ) , \label{eq:K_decay:mix} 
 \end{align}
\eseq
and the following sign estimates for $z > 1$
\beq\label{eq:KJ_sign}
  K_{\al, i, J}(1, z) > 0, \quad K_{\D}(1, z) = K_{\al, 2, J}(1, z) -K_{\al, 1, J}(1, z) > 0 , 
 \ \forall  \, z > 1.
\eeq

\end{lem}

\begin{proof}

Below, we fix $\al$ with $|\al - \alb| \leq \f{1}{100}$.

\vs{0.1in}
\paragraph{\bf Proof of \eqref{eq:K_decay:a}-\eqref{eq:K_decay:mix} }

 Recall the definition of kernels from \eqref{eq:ker}. For $z \leq 3$, we have
\beq\label{eq:K_decay_pf1}
\bal
  | K_{\al, 1, J}(1, z) | 
  & \les |1 + z|^{\al} + |1- z|^{\al} + z^{\al-1} \one_{z > 1} \les  1,  \\
  | K_{\al, 2, J}(1, z) | 
  & \les |z+1|^{\al - 1}  + |z-1|^{\al-1} + z^{\al - 1} \one_{z > 1} 
  \les |z-1|^{\al-1} .
\eal
\eeq

For $z > 1$, by definition, we have $K_{\al, i}(1, z) =  K_{\al,i,J}(1, z)$. Using integrations by parts, we obtain 
\bseq\label{eq:K_intform}
\beq
\bal
K_{\al, 1, J}(1, z) & = |1 + z|^{\al} - |1-z|^{\al} - 2 \al z^{\al-1}
= \tts{ \int}_0^1 \pa_s (  (z+s)^{\al} - (z-s)^{\al} - 2 \al s z^{\al-1} ) d s  \\
& = \tts{\int}_0^1 \pa_s^3 (  (z+s)^{\al} - (z-s)^{\al} - 2 \al s z^{\al-1} ) \cdot \f{ (1-s)^2 }{2 } d s  \\
& = \al (\al-1) (\al - 2) \int_0^1  (  (z+s)^{\al-3} + (z-s)^{\al-3} )  \cdot \f{ (1-s)^2 }{2 } d s , \\
K_{\al, 2, J}(1, z) & =  \al |1+z|^{\al-1} + \al  |1-z|^{\al-1} - 2 \al z^{\al-1}   
=  \al \tts{\int}_0^1 \pa_s (   (z+s)^{\al-1} + (z-s)^{\al-1} ) d s  \\
& = \al \tts{\int}_0^1 \pa_s^2 (  (z+s)^{\al-1} + (z-s)^{\al-1}  ) \cdot  (1-s) d s  \\
& =  \al (\al-1)(\al-2) \int_0^1   ( (z+s)^{\al-3} + (z-s)^{\al-3}  ) \cdot  (1-s) d s  .
 \eal
\eeq
\eseq

It is easy to check that the boundary terms at $s=0, 1$ vanish. 
Using \eqref{eq:K_intform}, for $z > 3$, we obtain 
\beq\label{eq:K_decay_pf2}
  |K_{\al, 1, J}(1, z)| \les z^{\al-3},  \quad | K_{\al, 2, J}(1, z)| \les z^{\al - 3}.
\eeq

Combining \eqref{eq:K_decay_pf1} and \eqref{eq:K_decay_pf2}, we prove \eqref{eq:K_decay:a} and the estimate for $K_{\al, 2, J}$ in \eqref{eq:K_decay:mix}. 
From the definition \eqref{eq:ker:c}, we have 
\bseq\label{eq:K_decay_pf3}
\begin{align}
  K_{\al, \D}(1, z) & = K_{\al, 2}(1, z) -K_{\al, 1}(1, z) 
  = K_{\al, 2, J}(1, z) -K_{\al, 1, J}(1, z) , \label{eq:K_decay_pf3:a} \\
 K_{\al, i}(1, z) & = K_{\al, i, J}(1, z )  - 2 \al z^{\al-1} \one_{z < 1}, \qquad i = 1, 2 .
 \label{eq:K_decay_pf3:b} 
\end{align}
\eseq
Thus, estimate \eqref{eq:K_decay:mix} for $K_{\al, \D}$ follows from 
the estimates for $K_{\al, i, J}$ in \eqref{eq:K_decay} and \eqref{eq:K_decay_pf3:a}.
Estimate \eqref{eq:K_decay_basic:a} 
follows from \eqref{eq:K_decay:a} and \eqref{eq:K_decay_pf3:b}, 
while \eqref{eq:K_decay_basic:b} follows from \eqref{eq:K_decay:mix} and 
\eqref{eq:K_decay_pf3:b}.

Using the formula \eqref{eq:K_intform}, $\al(\al-1)(\al-2) >0,  z \pm s > 0$ for $z >1$,  $ 1 - s \in [0,1]$, 
and $ 1- s \geq \f{ (1-s)^2}{2}$, we prove the sign properties \eqref{eq:KJ_sign}.
\end{proof}

Recall $\cJaa \asymp \min( \lgp x, \e^{-1} )$ from \eqref{eq:Ja_hat}. 
The following lemma compares $\cJaa$ and $\lgp x$:
\begin{lem}\label{lem:lgx_pow}

For any $k>0, \ell > 0$ and $\e < 1$, we have 
 \[
   \ang x^{-k \e} |\lgp x|^\ell \les_{k,\ell} |\cJaa(x)|^\ell  \les  \e^{-\ell}.
 \]
\end{lem}

Note that the constant is independent of $\e$. The proof follows easily from $\ang x^{-k \e} = e^{ -k \e \log \ang x }, \lgp x \asymp \log (1 + \ang x)$ and discussing $ \e \log \ang x < 1 $ and $ \e \log \ang x > 1$.

We have the following estimates for the $\cJaa$-integral.

\begin{lem}
Let $\cJaa$ be as in \eqref{eq:Ja_hat}. For $ k > -1$ and $c > 0$, we obtain
\begin{align}\label{eq:log_ineq_J:a}
 \int_0^x  \ang y^{ - c \e - 1 }  \cJaa^{k}(y) d  y
 \les_{c, k} \min( |x|, \cJaa^{k+1}(x) ) .
\end{align}

\end{lem}

\begin{proof}
The first bound in \eqref{eq:log_ineq_J:a} by $C_k |x|$ is trivial.  We define
$H(x) = 1 + \int_0^x \ang y^{-c \e - 1} d y$ and obtain
\[
H(0) = 1, \quad  H(x) \asymp_c \min( \e^{-1}, \lgp x ),
  \quad \pa_x H(x) = \ang x^{- c \e - 1} > 0 .
\]
Since $\cJaa \asymp \min( \e^{-1}, \lgp x) \asymp_c H$ \eqref{eq:Ja_hat}, and $k+1>0$,   we bound 
\[
\bal
  I & = \int_0^x \ang y^{-c \e - 1} \cJaa^k(y) d y 
\les_{c, k} \int_0^x \pa_y H(y) \cdot H(y)^k d y   \les_{c, k} 1 +  H(x)^{k+1}
\les \cJaa^{k+1}(x).
\eal
\]
We prove  \eqref{eq:log_ineq_J:a}.  
\end{proof}

\section{Numerics and piecewise bounds for the profiles}
\label{app:numerics}

In this appendix, we discuss the basis representation for the profile $\bw$, the computation of the associated nonlocal terms $V(\bw)$, and their rigorous piecewise bounds. 

\subsection{Parameters for the weight and estimates}\label{app:para_wg}

In this section, we present all the parameters in the estimates.
We choose the parameters $\mu_i , \bb_i$ appeared in the norm \eqref{norm:X} as follows 
\beq\label{def:mu_b}
\vmu = ( 1, 4, 30, 300, 5000 ), \quad 
\bb = ( -1.2, \, -0.5, \, 0, \,  0.2, \, 1/3 ) ,\quad \nmu = 5.
\eeq 
and the parameters $\mu_{\sst, i}, \bb_{\sst,i}$ appeared in the norm \eqref{def:vpa} as follows
\beq\label{def:mu_b_st}
\vmu_{\sst} = ( 1 , 4, 30 ), 
 \quad \bb_\sst = (-1.2, -0.5, 0), \quad \nst = 3.
\eeq

We choose the parameters for the fixed point argument in Lemma \ref{lem:close}.
\beq\label{def:delta_bw}
 \dfix = 3.5 \cdot 10^{-5} , \quad \dfixa = 3.4 \cdot 10^{-5} 
\eeq

We choose the following intervals $\cI = \{ I_i\{_{i\geq 1 }$ to partition $\R_+$ and use them in Lemma \ref{lem:vel_est} 
\beq\label{def:interval}
\cI = \{ I_i \{_{i=1}^4, \, 
 I_1 = [0, 1], \, I_2 = [1,2], \, I_3 = [2, 10], \,  I_4 = [10, 100], \, I_5 = [100, 10^5],   I_6 = [10^5, \infty). 
\eeq

We choose the parameter $x_1$ in the definition of $\rho$ \eqref{eq:rho1} as 
\beq\label{def:x1}
 x_1 = \mmx_{1000} \in [ 3708, 3709]. 
\eeq 
where $\mmx$ is the mesh in \eqref{def:x_i}. We choose $\mmx_1$ as a grid point on the mesh to simplify the estimates; its value is a floating-point number rather than a simple integer.

\subsection{Discretization and Bspline basis}\label{sec:basis}

In this subsection, we discuss the mesh and the Bspline basis functions used in \eqref{eq:W_rep}. 
We design increasing mesh $\mmx_i, 1 \leq i \leq \nnx + 5$ with the following properties 
near $x=0$ and in the far-field:
\bseq\label{def:x_i} 
\beq
 \mmx_i   = \tf{1}{ 384} (i-1) ,  \ 1 \leq i \leq  75, \quad  x_{i + 1200} = x_{1200} r_1^{i}, \  1 \leq i 
 \leq  805, \quad  \nnx = 2000, \   r_1 \approx 1.06,
\eeq
with $ \mmx_{\nnx+5} > 10^{20}$. We use the middle part of the mesh $\mmx_{i}$ to glue the 
 mesh near $x=0$ and in the far-field. We apply the filter $\mmx_i \to \mmx_{\mw{new},i} = \tf{1}{4} (\mmx_{i-1} + 2 \mmx_i + \mmx_{i+1})$ for $ 2\leq i\leq \nnx-1$ several times to obtain the mesh with $\mmx_1 = 0$ 
 and smoother in $i$. We introduce $\mmx_{0}, \mmx_{-1}$ 
 as follows and obtain
\footnote{
We design the above mesh so that (a) \(\mmx_i\) varies smoothly with \(i\); (b) \(\mmx_i\) is linear in \(i\) and finer near \(x=0\); (c) the spacing \(\mmx_{i+1}-\mmx_i\) is increasing; and (d) the ratio \(\frac{\mmx_{i+1}}{\mmx_i}\) grows exponentially. 
In the far field, we use an exponentially growing mesh \(\mmx_i\) to cover a large domain with relatively few points, while keeping the discretization error small by requiring the growth rate \(\frac{\mmx_{i+1}}{\mmx_i}-1=r_1-1\) to remain moderate. Properties (a) and (b) improve the approximation by the B-spline basis for moderate \(x\).
} 
\beq
 \mmx_{-1} < \mmx_{0} <   0 = \mmx_1 < \mmx_2< .. < \mmx_{\nnx+5}, \quad \mmx_{0} \teq - \mmx_2, 
 \quad  \mmx_{-1} \teq - \mmx_3.
\eeq
\eseq
Since these properties are not used in the proof and many other choices of \(\mmx_i\) are possible, we omit the explicit formula. The parameters \(h_i\) and \(N_i\) are also not used in the proof.

\vs{0.05in}
\paragraph{\bf Refine mesh for verification} 
For rigorous verification, we refine the mesh \eqref{def:x_i} $\mmx$ to obtain denser mesh $X$ 
with floating point values
\bseq\label{def:XX}
\beq
  X_{4(i-1) + 1} = \mmx_i, \ 
   X_{ 4(i-1) + j +1} = \mmx_i + \tf{j}{4} (\mmx_{i+1} - \mmx_i) , \
j= 1, 2,3, \   \quad \nnX = 4 (\nnx+4)+1= 8017,
\eeq
for any $1\leq i \leq \nnx + 4$, and $ X_{ \nnX} = \mmx_{\nnx+5}$. In particular, $X_1 = 0$, and for any $j$, we obtain $[X_j, X_{j+1}] \subset [\mmx_i, \mmx_{i+1}]$ for some $i$. We denote the last grid point in the mesh as $\xed$. From the definition of $\bwp$ \eqref{eq:W_rep}, the support property \eqref{eq:Bspline_prop1}, and \eqref{def:x_i}, we get 
\beq
 \xed \teq \mmx_{\nnx + 5} = X_{\nnX} > 10^{20},  \quad \supp( \bwp) \cap \R_+ \subset [0, \mmx_{\nnx+2} ] \subset   [ -\xed, \xed].
\eeq
\eseq

We emphasize that the Bspline basis representation \eqref{eq:Bspline}, \eqref{eq:W_rep} 
is based on the \emph{coarse mesh} $\mmx$. We evaluate various functions on the refined mesh $X_{\cdot}$.

\paragraph{\bf Parameters $z_0$ for $\bwf$} 
We choose $z_0$ in \eqref{eq:W_rep} for $\bwf$ and $z_L$ as 
\beq\label{def:z_0}
z_0 = \mmx_{300} \approx 1.937 , 
\quad z_{L_1} = 10^{10}, 
\quad z_L =  4 \cdot 10^{12}. 
\eeq

\subsubsection{Bspline basis and representation}

To approximate \emph{odd} function in a large domain, e.g. $\bwp$  \eqref{eq:W_rep}, we use Bspline representation:
\beq\label{eq:Bspline_rep}
F(x)  = \tts{\sum}_{1\leq i \leq \nnx }  a_i B_i(x) \in C^{2, 1} ,  \quad  \supp(F)  \subset [ -\xed, \xed ], \quad  \ \xed \approx 8.8 \cdot 10^{27} .
\eeq
with coefficients $a_i \in \R$. We choose the Bspline basis $B_i(x)$ as 
\footnote{
By definition of $\mmx$ \eqref{def:x_i},  for $i\geq 3$ and $|j|\leq 2$, we get $x_{i+j} \geq  0$.
Thus, $\mf{Bs}(x ;   \{ \mmx_{i +j} \}_{ -2\leq j \leq 2} )  =  \mf{Bs}(x ;   \{ \mmx_{i +j} \}_{ -2\leq j \leq 2} ) - \mBs(x ;  \{ - \mmx_{i +j} \}_{ -2\leq j \leq 2} ) =  $ for any $x \geq 0$, and it can be naturally extended to a $C^{2,1}$ odd function in $\R$.
}
\bseq\label{eq:Bspline}
\beq
\bal
 B_i(x)  & = C_{B_i} ( \mf{Bs}(x ;   \{ \mmx_{i +j} \}_{ -2\leq j \leq 2} ) -
 \mBs(x ;  \{ - \mmx_{i +j} \}_{ -2\leq j \leq 2} ) ) , &&  i= 1, \, 2 , \\
B_i(x ) & = C_{B_i}   \mf{Bs}(x ;   \{ \mmx_{i +j} \}_{ -2\leq j \leq 2} ) &&  3 \leq i \leq \nnx.
\eal 
\eeq
where $C_{B_i}$ is some parameter depending on $ \{ \mmx_{ i+j} \}_{ -2 \leq j \leq 2}$ to improve 
the condition number.
Given the parameters $\ss \in \R^5$,  $\mBs(x; \ss)$ is the standard $4$-th order Bspline basis function 
{\small
\beq\label{eq:Bspline:a}
 \mf{Bs}(x ; \ss  ) \teq  \sum \nolimits_{ 1 \leq l \leq 5} \f{ 4 ( s_l - x)^3}{  m_l(\ss) } \one_{s_l -x \geq 0}, 
 \quad 
 m_l(\ss) \teq \prod \nolimits_{1\leq j\leq 5, \, j \neq l}  (s_l - s_j) .
\eeq
}
where $ x_+ = \max(x, 0)$. We can rewrite the Bspline function equivalently as 
\footnote{
Given $\ss$, this identity follows from the identity $\sum_{1\leq l \leq 5}  \f{ ( s_l - x)^3}{  m_l(\ss) }  \equiv 0$ for any $x$. Using Lagrangian interpolation identity 
 $p(x) \equiv \sum_{1\leq l \leq 5} p(s_l) \f{ \prod_{ j \neq l} (x-s_j) }{ \prod_{ j\neq l} (s_l -s_j) } $ 
for any polynomial $p$ with degree $\leq 4$, 
choosing $p(x) = x^k, k\leq 3$, and comparing the coefficients of $x^4$ on both sides of the identity, 
we yield  $0 = \sum_{1\leq l \leq 5} p(s_l) ( \prod_{ j\neq l}  (s_l -s_j)  )^{-1}
= \sum_{1\leq l \leq 5} \ s_l^k m_l(\ss)^{-1}$.
}
{\small
\beq\label{eq:Bspline:b}
  \mf{Bs}(x ; \ss  )  =  -  \sum \nolimits_{ 1 \leq l \leq 5} \f{ 4 ( s_l - x)^3}{  m_l(\ss) } \one_{s_l -x < 0} . 
\eeq
}
\eseq
The above equivalent formula implies 
\beq\label{eq:Bspline_prop1}
\bal
\mBs(x; -\ss) &=  \mBs(-x; \ss), \quad \forall \  \ss \in \R^5, \quad x \in \R,  
\quad \supp(\mBs(\cdot ; \ss ))  \subset [s_1, s_5],  \\
\quad B_i(x)
& =  C_{B_i} ( \mf{Bs}(x ;   \{ \mmx_{i +j} \}_{ -2\leq j \leq 2} ) -
 \mBs(-x ;  \{  \mmx_{i +j} \}_{ -2\leq j \leq 2} )) ,  \quad  i= 1, 2 .
\eal 
\eeq
Thus,  \( B_i(x) \) in \eqref{eq:Bspline} is odd for \( i=1,2 \). Moreover, $B_i(x)$ and \( \bwp \) in \eqref{eq:Bspline_rep} have global \( C^{2,1} \) regularity, and are cubic polynomials on  \(  [x_j, \mmx_{j+1}] \); hence they belong to \( C^{\infty}([\mmx_j, \mmx_{j+1}]) \) for any \( j \).

\subsection{Rigorous piecewise bounds}\label{app:piece_bound}

To obtain rigorous bound for a function $f(x)$, we derive the piecewise upper and lower bounds using the following methods. 

\subsubsection{Piecewise bounds in bounded domain}

\vs{-0.05in}
\subsubsection*{\bf Interpolation Estimates}\label{bd:interp}
Given a function $f$, using standard linear interpolation, we obtain
\beq\label{eq:lin_interp}
 f(x) =\tf{b-x}{b-a} f(a) +   \tf{x-a}{b-a}  f(b) + \cE(f), \quad  \cE(f) = f^{\pr \pr}(\xi) \tf{(x-a)(x-b)}{2} , \quad  \forall  \ x \in [a, b],
\eeq
for some $ \xi \in (a, b)$. 
Since $ |(x-a)(x-b)| \leq \f{(b-a)^2}{4}$, we obtain the second order error estimates 
\bseq\label{eq:err_est}
\beq\label{eq:err_est:2nd}
f(x) \in [ \min(f(a), f(b)) - f_{\mw{2,err}} , \ \max(f(a), f(b) + f_{\mw{2, err}} ) ],
\quad    f_{\mw{2,err}} \teq \tf{(b-a)^2}{8} \| \pa_x^2 f \|_{ L^{\infty}(a,b)} , 
\eeq
for any $ x \in [a, b]$. By considering $x \in [a, \f{b+a}{2}], [ \f{b+a}{2}, b]$ separately, we obtain the first order estimate 
\beq\label{eq:err_est:1st}
 f(x) \in [ \min(f(a), f(b)) - f_{\mw{1, err}} , \ \max(f(a), f(b) + f_{\mw{1, err}} ) ],
\quad   f_{\mw{1, err}} \teq \tf{(b-a)}{2} \| \pa_x f \|_{ L^{\infty}(a,b)} ,
\eeq
\eseq
for any $x \in [a, b]$. Below, we first discuss how to obtain a rigorous bound for 
$\pa_x^k f$ with a large $k$.

\vs{-0.05in}
\paragraph{\bf Interval Arithmetic}\label{bd:intval}

For rigorous computer-assisted proof, we use interval arithmetic  \cite{rump2010verification,moore2009introduction}, and use the MATLAB toolbox INTLAB (version 11 \cite{Ru99a}) for the interval computations.  In the verification step, every real number $p$ in MATLAB is represented by an interval $[p_l,p_r]$ that contains $p$, where $p_l,p_r$ are \emph{precise} $16$-digit floating-point numbers. Every computation of real number summation, multiplication or division is performed using the interval arithmetic, so the output is an interval $[P_l,P_r]$ that rigorously contains the exact value $P$.

For an explicit function $f$, for example one involving $x^a$, $\log x$, linear combinations, products, max, or min of such terms, we enclose the values of $f(y)$ on a bounded interval $[a,b]$ by applying interval arithmetic in INTLAB to $f([a,b])$. The output is an interval $[f_l, f_r]$, where $f_l$ and $f_r$ are floating-point numbers and may equal $\pm\infty$,   such that $f(y)\in [f_l,f_r]$ for all $y\in [a,b]$.

We remark that we only use basis interval arithmetic and do not apply it directly to bound more complicated functions, e.g. integrals and nonlocal terms such as $\cK_{\alb, i}$ in \eqref{eq:ker}.

\vs{-0.05in}
\subsubsection*{\bf Method A.1: Bspline representation}\label{bd:Method1}
For $f$ representing as \eqref{eq:Bspline_rep} or \eqref{eq:Bspline} with $\ss = \{ \mmx_{i +j} \}_{ |j|\leq 2} $, e.g. $f =\bwp$  \eqref{eq:W_rep:WP}, 
 from the properties in \eqref{sec:basis}, 
 $f$ is a cubic polynomial on $[\mmx_i, \mmx_{i+1}]$. Thus, we obtain $\pa_x^k f(x) = 0$ 
for any $x \in [\mmx_i, \mmx_{i+1}]$ and any $k\geq 4$,
and $\pa_x^3 f(x) \equiv \mw{const}$ for any $x \in [\mmx_i, \mmx_{i+1}]$.

\vs{-0.1in}
\subsubsection*{\bf Method A.2: Explicit functions}\label{bd:Method2}

When $f$ is an explicit function, e.g. $ \bwf , \chi_1$ in \eqref{eq:W_rep}, we first derive the 
\emph{exact} formulas for the  derivatives $\pa_x^i f, i\leq 5$ using symbolic computation. 
Then by applying \hyr[bd:intval]{\its Interval Arithmetic}, we obtain a rigorous enclosure of a high order derivatives 
\[
\pa_x^k f( [a, b]) \subset [ f^l_{k}(a, b), f^u_{k}(a, b) ],
\]
where $l, u$ indicate \emph{lower} and \emph{upper}, and $f^l_{k}(a, b), f^u_{k}(a, b)$ are some floating point numbers.

\vs{-0.1in}
\subsubsection*{\bf Method A.3: Linear Combinations, Products, and Quotients}\label{bd:Method3}

When \( f= g + p \) is a linear combination of an explicit function \( g \) and a piecewise polynomial \( p \) that are cubic on each interval \( [\mmx_i, \mmx_{i+1}] \), e.g.
$f=\bw$ in \eqref{eq:W_rep}, \eqref{eq:Bspline_rep}  within  $[a, b] \subset [\mmx_i, \mmx_{i+1}]$, since $\pa_x^k p(x) =0$ for $k\geq 4$, we 
obtain high order derivative bounds using $\pa_x^k f = \pa_x^k g(x)$ and \hyr[bd:Method2]{\its Method A.2} for the explicit function $g$.

After we obtain high order derivative bounds for $\pa_x^k f$,
by evaluating $\pa_x^l f(a), \pa_x^l f(b), 0 \leq l \leq k-1$ using the 
basis representation \eqref{eq:Bspline_rep} for the piecewise polynomials $p(x)$, and 
the symbolic formulas for explicit functions $g(x)$, and applying 
\hyr[bd:interp]{\its Interpolation Estimates} \eqref{eq:err_est:1st} 
for $l= k-1$ and  \eqref{eq:err_est:2nd} for $l\leq k-2$, we obtain rigorous piecewise 
bounds for $\pa_x^k f$. 
\footnote{
Compared with applying \hyr[bd:intval]{\its Interval Arithmetic} directly to enclose $\pa_x^l f([a,b])$ for $l\leq k-1$, the method based on \eqref{eq:err_est} yields a tighter bound for $\pa_x^l f$, with $O((b-a)^2)$ error away from the grid values at $a,b$, whereas \hyr[bd:intval]{\its Interval Arithmetic} typically gives only $O(b-a)$ error.
}
Note that we use the identity
\[
\partial_x^3 p(x) \equiv \tfrac{1}{b-a} ( \partial_x^2 p(b) - \partial_x^2 p(a) )
\]
to evaluate \(\partial_x^3 p(x)\) when \(p\) is a cubic polynomial on \( [a,b] \subset [\mmx_i, \mmx_{i+1}] \) (so that \(\partial_x^2 p\) is linear).

If $f = gh$ for  functions $g,h$ in the classes discussed above, using the Leibniz rule 
\beq\label{eq:Leibiz}
 \pa_x^k f = \pa_x^k (g h) = \tts{ \sum\nolimits_{0\leq i\leq k} }\tts{ \binom{k}{i}  }\pa_x^i g \pa_x^{k-i} h,
\eeq
the piecewise bounds for $g,h$, we obtain the piecewise bounds for $\pa_x^k f$. Moreover, using the formulas for $g,h$ and Leibniz rule,  we obtain the formulas for $\pa_x^i f, i\leq k$. Then we apply the above method to obtain tight bounds for $\pa_x^i f, i\leq k-1$.
The case $f  = g/h$ is estimated similarly.

\subsubsection{Far-field estimates of profiles}\label{app:profile_far_field_bound}

The previous methods apply to derive piecewise bounds on bounded domain. For large $x \geq \xed$, we obtain $\bwp = 0$ and $\bw = \bwf$ \eqref{eq:W_rep}. 
We decompose $\bwf$  
\beq\label{def:bwfa}
\bwf(x) = \chi_1(x) x^{-1/3} \bwfa,
\quad \bwfa \teq  \cca + \ccb | \log |x| |^{-1/3}  + \ccc | \log |x| |^{- 2/3 } . 
\eeq
where $\mf{A}$ is short for asymptotics. Below, we derive the bound $ |\pa_x^k  \bwf| \leq C_k x^{-1/3 -k}$ for $k \geq 0$ and $x > 10^{10}$
by estimating $\pa_x^k  (x^p (\log x)^q)$ for $x^{-1/3} \bwfa$ and $\pa_x^k \chi_1(x)$.

\vs{0.05in}
\paragraph{\bf Far-field bounds for $y^p (\log y)^q$}

We fix $p$ and $q \leq 0$. We first derive the expansion recursively
\[
 \pa_x^i (y^p (\log y)^q) =  y^{p-i} ( \tts{\sum}_{ 0 \leq j\leq i} C_{i, j} (\log y)^{q- j} ).
\]

For $i=0$, we obtain $C_{0,0} = 1$ and $C_{0, j} = 0, j \geq 1$. Suppose the we have the above expansion for $i\geq 0$. Taking $\pa_x$, we obtain
\[
\pa_x^{i+1} (y^p \log y^q)
= y^{p-i-1} (p-i) \cdot 
(\tts{\sum}_{ 0 \leq j\leq i} C_{i, j} (\log y)^{q- j} )
+ y^{p-i-1} ( \sum\nolimits_{ 0 \leq j\leq i} C_{i, j} (q - j) (\log y)^{q- j - 1}  ) .
\]

Thus, we derive the coefficients 
\bseq\label{eq:bound_asym_tail}
\beq
C_{i+1,j}=\one_{j\leq i} (p-i) \cdot C_{i, j} + \one_{j\geq 1} (q-(j-1) )  \cdot C_{i, j-1} , \quad  \forall  \ 0\leq j \leq i+1,
\eeq
and $ C_{i+1, j} = 0, \ \forall j > i+1$. Using the above coefficients and $q\leq 0$, for any $y \geq L \geq e$, we bound 
\beq
 |\pa_x^i (y^p (\log y)^q)| \leq y^{p-i} \cdot \tts{\sum}_{0\leq j \leq i} |C_{i, j}|  (\log L)^{q - j} \teq C_{\log , p, q, i} y^{p-i}.
\eeq
\eseq

\paragraph{\bf Far-field bounds for cutoff}

Recall from \eqref{eq:W_rep}:
\beq\label{def:chi_0}
\chi_1(z) = \chi_0( \f{z}{z_0}  - 1),
 \quad 
 \chi_0(z) \teq \f{z^5}{1 + z^5} \one_{z > 0} = ( 1- \f{1}{1 + z^5}) \one_{z > 0}.
\eeq
To estimate $\pa_z^k \chi_0(z)$ for large $z \geq 10$, we first derive the formula for $p_k$ in
\[
\pa_z \chi_0 =  \f{5 z^4}{ (1 + z^5)^2}, \quad   \pa_z^k \chi_0(z) = \f{p_k(z)}{ (1 + z^5)^{k+1}} ,
\quad p_1(z) = 5 z^4,
\]
recursively for $k\geq 1$ and $z > 0$. Suppose that $p_k$ is given for $k \geq 1$. Taking $\pa_z$, we obtain
\[
 \pa_z^{k+1} \chi_0(z) = \f{ (1 + z^5)\pa_z p_k(z) - 5 (k+1) z^4 p_k(z) }{ (1 +z^5)^{k+2}},
 \ \Rightarrow  \ p_{k+1}(z) =  (1 + z^5)\pa_z p_k(z) - 5 (k+1) z^4 p_k(z) . 
\]

In particular, we obtain $\deg p_{k+1} = \deg p_k + 4$ for $k\geq 1$ and thus $\deg p_k = 4 k$. 
By extracting the coefficients of $p_k$ using symbolic derivations, for $z\geq L \geq 10, k \geq 1$, we obtain 
the explicit bounds
\[
|p_k(z)| = |\sum\nolimits_{ 0\leq i \leq 4k} a_{4k, i} z^i |
\leq z^{4 k} \sum \nolimits_{ 0\leq i \leq 4k} |a_{4k, i}| L^{i-4k}
\teq z^{4k} C_{\chi_0, k, L},  \quad
\Rightarrow \  |\pa_z^k \chi_0|\leq C_{\chi_0, k, L} z^{-5-k} .
\]

Using the relation \eqref{def:chi_0} and chain rule, for $z \geq L > 11 z_0$, 
we get $z / z_0 - 1 \geq L /z_0 - 1  \geq 10$ and 
\begin{align}
 | \pa_z^k \chi_1( z) | & = z_0^{-k} |(\pa_z^k \chi_0 )( \f{z-z_0} {z_0} )| 
 \leq z_0^{-k} C_{\chi_0, k, \f{L}{z_0} -1 } ( \f{z-z_0}{z_0})^{-k-5} 
  \leq z_0^{-k} C_{\chi_0, k, \f{L}{z_0} -1 }  (   \tf{L-z_0}{L z_0} )^{-k-5} z^{-k-5} \notag \\
  & \teq C_{ \eqref{eq:bound_chi_tail},  \chi_0, k, z_0, L}  z^{-k-5},
\label{eq:bound_chi_tail} 
\end{align}
where we have used $\f{ z -z_0}{z_0}  \geq  z \f{L-z_0}{L z_0}$ in the last inequality 
for $z \geq L > 11 z_0$.

\vs{0.1in}
\paragraph{\bf Far-field bounds for $\bwf$}

For $x \geq \xed$, since $\bw = \bwf = \chi_1(x)  x^{-1/3}( \cca + \ccb (\log x)^{-1/3} + (\log x )^{-2/3} )$ 
and since $\chi_1 \leq 1$, for any $x \geq \xed \geq L \geq 11 z_0$, applying \eqref{eq:bound_asym_tail} 
with $p = -\f13, q= 0, -\f13, -\f23$, estimate \eqref{eq:bound_chi_tail}, triangle inequality, and  Leibniz rule, we derive 
\beq
   |\pa_x^k \bwf(x) | \leq C_{ \eqref{eq:deri_WF_far}, z_0,  k,  L }  x^{-1/3 -k}.
   \label{eq:deri_WF_far}
\eeq
for some explicit constant $ C_{ \eqref{eq:deri_WF_far}, z_0,  k,  L }$.

When $k=0$, from \eqref{def:bwfa}, since $\cca = -6< 0, \ccb , \ccc > 0$, $\bwfa$ is decreasing in $x$.
Thus, for $x \geq L> 10$, we get 
\beq\label{eq:bwfa_far}
\bwfa  < -6 + |\ccb| + |\ccc| < -3 < 0, \  \Rightarrow  
 \ \bwfa(x)  \ \in [ \cca,  \bwfa(L)] \subset [-6, -3] .
\eeq

 From definition of $\chi_1$ in \eqref{eq:W_rep},  $\chi_1(x) \in [0,1]$ is increasing in $x$.  Thus, for any $x \geq L$, we get 
\begin{align}\label{eq:deri_WF_far_k0}
   \chi_1(x) & \in [ \chi_1(L), 1],   \\
 \bwf(x) x^{1/3} & = \chi_1 \bwfa(x) \in [ \chi_1(L), 1] \bullet [ \cca, \bwfa(L)], 
 \ 
 | \bwf(x) | \leq x^{-1/3} \max( |\cca|, |\bwfa(L)| ) \notag
\end{align}
where we use $[ a, b] \bullet [ c, d]$ to denote 
the enclosure for $ x y, x \in [a,b], y \in [c, d]$.

\subsection{Estimates and expansion of the kernels}

In this subsection,
we expand and estimate the kernels for later use in computing the nonlocal terms. Recall the kernels $\cK_{\alb,i}$ from \eqref{eq:ker}.
\beq\label{eq:ker_recall_app}
\bal
  K_{\alb, 1}(x, y) & = |x + y|^{\alb} - |x-y|^{\alb} - 2 \alb x y^{\alb - 1}  ,  \\
  K_{\alb, 2}(x, y) & = \alb |x+y|^{\alb-1} - \alb \cdot  \sgn(x-y) |x-y|^{\alb-1} - 2 \alb y^{\alb-1}  .
\eal
\eeq

\subsubsection{Taylor expansion}\label{app:kernel_TL}
We consider $x,y \ge 0$.  When $x $ and $y$ are comparable, 
we evaluate the kernel $K_{\alb, i}$ using \eqref{eq:ker_recall_app}. When $x \ll y$ or $ y \ll x$, however, a direct evaluation of the kernel $K_{\alb, i}$ in \eqref{eq:ker_recall_app} may suffer from significant round-off error, 
In these two regimes, we expand the kernels to reduce the round off error.

We fix a parameter $r \geq 2$ denoting the ratio between $x, y$. To simplify the notation, we introduce 
\beq\label{def:factor}
   \AAF(a, i) \teq \tts{ \prod\nolimits_{0\leq j \leq i-1} } (a-j), \quad \mf{A}(a,0) = 1.
\eeq

\paragraph{ \bf Case 1: $y > r x,r>2$}

We denote $s = \f{x}{y}$. By assumption, we get $s < r^{-1} < 2^{-1}$. Using the scaling symmetry, we obtain
\[
K_{\alb, 1}(x, y) = y^{\alb} ( (s+1)^{\alb} - (1- s)^{\alb} - 2 \alb s)
= y^{\alb}( g_1(s) - g_1^{\pr}(0) s) , \quad g_1(s) \teq (s+1)^{\alb} - (1- s)^{\alb} .
\]

Clearly,  $g_1$ is odd in $s$ and  $\pa^{2i} g_1(0)=0$ for any $i\geq0$. Thus, for any $k\geq 3$, using Taylor expansion, we obtain
\[
\bal
   g_1(s) - g_1^{\pr}(0) s & = \sum\nolimits_{ 3 \leq  i \leq k, \ i \mw{ \ is \  odd} } \, \f{1}{i!} \pa^i g_1(0)  s^i
    +  \f{1}{(k+1)!} \pa^{k+1} g_1( \xi )    ( s^{k+1} - (-s)^{k+1} ) . \\ 
  \eal
\]

Using notation \eqref{def:factor}, $|s| \leq r^{-1}$, 
and $\pa^j g_1(s) = \AAF(\alb, j) ( (1 + s)^{\alb-j} - (-1)^j (1-s)^{\alb-j} ) $, we obtain
\[
   g_1(s) - g_1^{\pr}(0) s = \sum\nolimits_{ 3 \leq  i \leq k, \ i \mw{ \  odd} }
   \f{2 \AAF(\alb, i)}{ i!} s^i + 
      g_{1, k, \mw{err}},  \
   |g_{1, k, \mw{err}}| 
   \leq \f{2  \mf{A(\alb, k+1)} }{ (k+1)!} s^{k+1} (1 - \f1r )^{\alb-k-1}.
  \]
We treat the first term on the right hand side as the main term, and $g_{1, k, \mw{err}}$ as the error.  
When $x \ll y$, the leading order term in $\cK_{\alb, 1 }(x, y) = y^{\alb} ( g_1(s) - g_1^{\pr}(0) s)$ is given by $C \f{x^3}{y^3} \cdot y^{\alb}$. Due to the round-off error, a direct evaluation of $\cK_{\alb,1}(x, y)$ cannot capture the small factor $\f{x^3}{y^3} $. It is crucial to use the above expansion.

Similarly, for $y > r x > 2 x$, we denote $s = \f{x}{y}$ and rewrite $\cK_{\alb, 2}(x,y)$ in \eqref{eq:ker_recall_app} as 
\[
 \cK_{\alb, 2} (x, y) = \alb y^{\alb-1} ( (1 + s)^{\alb-1} + (1-s)^{\alb- 1} - 2 )
 =  y^{\alb-1} \cdot \alb ( g_2(s) - g_2(0) ), \quad g_2(s) \teq (1 + s)^{\alb-1} + (1-s)^{\alb-1} .
 \]

Since $g_2(s)$ is even in $s$ for $|s| < r^{-1} < \tf12$, using Taylor expansion and
$ \pa^i g_2(s) = \AAF(\alb-1, i) ( (1 + s)^{\alb-1 -i} +  (-1)^i (1-s)^{\alb-1 - i} )$, 
for any $k\geq 2$, we obtain
\[
 g_2(s) - g_2(0)  = \sum\nolim_{ 2\leq i \leq k, \ i \mw{\  even} } \f{ 2 \AAF(\alb-1, i)}{ i!} 
 s^i  + g_{2, k, \mw{err}}, 
 \quad  |g_{2, k, \mw{err}}|
  \leq \f{ 2 \AAF(\alb-1, k+1)}{ (k+1) ! } s^{k+1} (1 - \f1r  )^{ \alb-2-k}.
\]

\vs{0.05in}
\paragraph{\bf Case 2: $x > r y, r>2$}

The cancellation in the kernels \eqref{eq:ker_recall_app} comes from the first two terms. 
Since $x > 2 y$, using the kernel $\cK_{\alb,i ,J}$ in \eqref{eq:ker:c}
and introducing $z = \f{y}{x} < r^{-1}$, we rewrite
\begin{align}\label{eq:deri_K_J_12_0}
\cK_{\alb, 1}(x, y) &= \cK_{\alb,1, J}(x, y) - 2 \alb x y^{\alb-1}, 
\quad \cK_{\alb, 1, J}(x, y) = x^{\alb} ( (1 + z)^{\alb} - (1-z)^{\alb} ) \teq x^{\alb} f_1(z), \\
 \cK_{\alb, 2}(x, y) &= \cK_{\alb,2, J}(x, y) - 2 \alb  y^{\alb-1}, 
 \quad \cK_{\alb, 2, J}(x, y) =  x^{\alb-1} \cdot \alb ( (1 + z)^{\alb-1}  - (1 - z)^{\alb-1} )
=  x^{\alb-1} \cdot  f_2(z). \notag
\end{align}

Both $f_1(z)$ and $f_2(z)$ are odd functions in $z, |z| < r^{-1}\leq \tf12$,
and we have $\pa^{2j} f_i(0) = 0$. Thus, for $k \geq 2$, using Taylor expansion for $f_i$ around $z=0$, we expand
\[
\bal
f_1(z) &= \sum\nolimits_{ 1 \leq i \leq k ,\  i \mw{ \ odd} }  \f{ 2 \AAF(\alb, i)}{ i!} z^i
+ f_{1, k, \mw{err}} , && | f_{1, k, \mw{err}} |  \leq \f{2 \AAF(\alb, k+1)}{ (k+1)!} z^{k+1}
(1 - r^{-1})^{\alb - k-1}  , \\
f_2(z) & = \sum \nolimits_{ 1 \leq i \leq k ,\  i \mw{ \  odd}}  \f{ 2\alb \AAF(\alb-1, i)}{ i!} z^i 
+ f_{2 , k, \mw{err}} , \quad 
  && | f_{2, k, \mw{err}} |   \leq \f{2 \alb \AAF(\alb-1, k+1)}{ (k+1)!} z^{k+1}
(1 - r^{-1})^{\alb - k- 2}  .
\eal
\]

We choose $k \leq 40$ relatively large, and treat $f_{i, k, \mw{err}}$ as the error parts. When $x > r y$, we use the above expansion to evaluate $K_{\alb, i, J}$ and then 
subtract the terms $ 2 \alb x y^{\alb-1}$ or $2 \alb y^{\alb-1}$ to evaluate 
$K_{\alb,i}$ with rigorous error control. 

\subsubsection{Derivative bounds} \textbf{ Case 1 $y > r x, r\geq 2$}.
Denote $z = \f{y}{x}$. In this case, we obtain $z \geq 2$ and
\[
\bal
K_{\alb, 1}(x, y) &= K_{ \alb, 1, J}(x, y) = 
x^{\alb  }  K_{\alb, i, J}(1, z) ,  && \pa_y^k K_{\alb, 1}(x, y)  = x^{\alb  - k} \pa_z^k  K_{\alb, 1, J}(1, z)  , \\
 K_{\alb, 2}(x, y) & =
 K_{\alb, 2, J}(x, y) 
= x^{\alb  -1 }  K_{\alb, 2, J}(1, z)  ,   
&& \pa_y^k K_{\alb, 2}(x, y)  = x^{\alb -1 - k} \pa_z^k  K_{\alb, 2, J}(1, z) .
\eal
\]

Applying the expansion of $K_{\alb, i, J}$ in \eqref{eq:K_intform} and using $z \geq r \geq 2$, 
$z-1 \geq ( 1- r^{-1} )z$, we bound 
\[
\bal
  | \pa_z^ k K_{\alb, 1, J}(1, z) | 
 & \leq 
 \B| \alb (\alb-1) (\alb - 2) \int_0^1  ( \pa_z^k (z+s)^{\alb-3} + \pa_z^k (z-s)^{\alb-3} )  \cdot \f{ (1-s)^2 }{2 } d s \B|  \\
  &  \leq \tf13 | \AAF(\alb-3 , k)  \AAF(\alb, 3)  |   (z-1)^{\alb-3- k} 
\leq \tf13 | \AAF(\alb-3 , k)  \AAF(\alb, 3) |   ( 1 - r^{-1}  )^{\alb-3- k} z^{\alb-3-k} 
 , \\
\eal
\]
where we have used $\AAF(\alb, 3) =\alb (\alb-1)(\alb-2)$ and  $\int_0^1 \f{(1-s)^2}{2} d s =\tf16$.
Applying similar estimates to $K_{\alb, 2, J}(1, z)$ in \eqref{eq:K_intform}
and using $\int_0^1 ( 1-s) ds = \f12$ instead of $\int_0^1 \f{(1-s)^2}{2} d s =\tf16$, we get
\[
 |\pa_z^k K_{\alb, 2, J}(1, z)  | \leq | \AAF(\alb-3 , k) \AAF(\alb, 3)  |  ( 1 - r^{-1}  )^{\alb-3- k} z^{\alb-3-k} .
\]
Thus, combining the above scaling relation and estimates, and using $z = \f{y}{x}$, and rescaling $K_{\alb, \D}(x, y) = K_{\alb, 2}(x, y) - \tf{1}{x} K_{\alb, 1}(x, y)$, for $y > r x \geq 2x$, we prove 
\beq
\bal
 |  \pa_y^k K_{\alb, 1}(x, y) | 
 & =  |  \pa_y^k K_{\alb, 1,J}(x, y) | 
  \leq 
  | \tf13 \AAF(\alb-3 , k)  \AAF(\alb, 3) | ( 1 - r^{-1} )^{\alb-3- k} x^{\alb-k} z^{\alb-3-k} \\
 & \teq C_{ \eqref{eq:deri_K_12}, 1, r ,k}   x^{\alb-k} z^{\alb-3-k} =  C_{ \eqref{eq:deri_K_12}, 1, r ,k}  
 x^3 y^{\alb-3-k}
, \\
 |  \pa_y^k K_{\alb, 2}(x, y) | 
  & =  |  \pa_y^k K_{\alb, 2, J }(x, y) | 
 \leq 
 | \AAF(\alb-3 , k) \AAF(\alb, 3) | ( 1 - r^{-1} )^{\alb-3- k} x^{\alb-1-k} z^{\alb-3-k} \\
 &  \teq C_{ \eqref{eq:deri_K_12}, 2, r ,k} x^{\alb-1-k} z^{\alb-3-k} 
= C_{ \eqref{eq:deri_K_12}, 2, r ,k} x^2 y^{\alb-3-k} , \\
 |\pa_y^k K_{\alb, \D}(x, y) | & \leq  ( C_{ \eqref{eq:deri_K_12}, 1, r ,k} + 
 C_{ \eqref{eq:deri_K_12}, 2, r ,k} ) x^2 y^{\alb-3-k} \teq C_{ \eqref{eq:deri_K_12}, \D , r, k } x^2 y^{\alb-3-k}.
\eal
\label{eq:deri_K_12}
\eeq
Since $\alb = \tf13$ is fixed, we suppress this dependence.

\paragraph{\bf Case 2: $ r y < x, r > 2$}

In this case, we obtain $z = \f{y}{x} < r^{-1} \leq \tf12$, we estimate the kernel $K_{\alb, i, J}$.
Using the formula \eqref{eq:deri_K_J_12_0}, the chain rule, we prove 
\bseq\label{eq:deri_K_J_12}
\begin{align}
 |\pa_y^k K_{\alb, 1, J}(x, y)|
 & = x^{\alb-k} |\pa_z^k  ( (1 + z)^{\alb} - (1-z)^{\alb} )|
 \leq 2 | \AAF(\alb, k) | x^{\alb - k}   (1 - \tf1r )^{\alb- k} , \label{eq:deri_K_J_12:a} \\
  |\pa_y^k K_{\alb, 2, J}(x, y)|
  & =  \alb x^{\alb-1 - k} | \pa_z^k ( (1 + z)^{\alb-1}  - (1 - z)^{\alb-1} ) | \notag \\
 &  \leq 2 \alb | \AAF( \alb - 1, k ) | 
x^{\alb - 1- k}  (1 - r^{-1})^{\alb-k-1}   \label{eq:deri_K_J_12:b} .
\end{align} 
\eseq

When $k=0$, using mean value theorem, triangle inequality, and $z = \f{y}{x} < r^{-1}$, we get 
\beq
\bal
 | K_{\alb, 1, J}(x, y)| & \leq 2\alb x^{\alb}  z (1 - r^{-1})^{\alb-1}
 \teq C_{\eqref{eq:deri_K_J:y<x},1,r} x^{\alb } z,  \\
   | K_{\alb, 2, J}(x, y)| & \leq 2 \alb|\alb-1| x^{\alb-1} z (1 - r^{-1})^{\alb-2}
  \teq  C_{\eqref{eq:deri_K_J:y<x},2,r} x^{\alb -1} z, \\
   | K_{\alb, \D}(1, z)|  & \leq  | K_{\alb, 1, J}(1, z)|
   + | K_{\alb, 2, J}(2, z)|
   \leq ( C_{\eqref{eq:deri_K_J:y<x},1,r} + C_{\eqref{eq:deri_K_J:y<x},2,r} ) z
   \teq  C_{\eqref{eq:deri_K_J:y<x},\D,r} z.
 \eal 
 \label{eq:deri_K_J:y<x}
\eeq

\subsection{Basic integral estimates and sharp constants in Lemma \ref{lem:vel_est}}
\label{app:basic_int_est}

In this subsection, we discuss  basic integrals estimates 
and use them to derive the sharp constants in Lemma \ref{lem:vel_est}. 

\vs{-0.1in}
\subsubsection*{\bf Method B.1: Non-singular integrals}\label{int:Method1}

For $g \in C^{\g}[a, b]$ with some $\g > 0$, 
to estimate $\int_a^b g(y) d y$, we subdivide $[a, b]$ as $a = y_1 < y_2 <.. <y_n = b$ using suitable mesh $y_i$. Then we bound 
\beq\label{eq:basic_int_est:ul}
  \tts{\int}_a^b g(y) d y = \tts{\sum_{ 1 \leq i\leq n-1 }} \int_{y_i}^{y_{i+1}} 
  g(y) d y \in \tts{\sum_{ 1 \leq i\leq n-1 }}  g( [y_i, y_{i+1}]) (y_{i+1} - y_i),
\eeq
where $g( [y_i, y_{i+1}])$ is an enclosure of the range of $g$ on $[y_i, y_{i+1}]$. We also use the trivial bound 
\[
  \tts{\int}_a^b| g(y)| d y = \tts{\sum_{ 1 \leq i\leq n-1 }} \int_{y_i}^{y_{i+1}} 
  |g(y) | d y \leq  \tts{\sum_{ 1 \leq i\leq n-1 }} \| g\|_{ L^{\infty}[y_i, y_{i+1}] }  (y_{i+1} - y_i)
\]

\vs{-0.1in}
\subsubsection*{\bf Method B.2: Integrals near $0, \infty$}\label{int:Method2}
For $y \gg 1$ or $0\leq y \ll 1$, we derive the asymptotics of the integrand 
by $g(y) = y^{-k} |\log y|^q$ for $k> 1.1$.
For $y_a > 1$, we choose $y_b \gg 1$ and decompose 
\beq\label{eq:int_tail1}
 \tts{\int}_{y_a}^{\infty}  g(y) d y
  = \tts{ \int}_{y_a}^{y_b}  g(y) d y + \tts{\int}_{y_b}^{\infty} g(y) d y \teq I + II.
\eeq

We estimate $I$ using \hyr[int:Method1]{\itshape Method B.1} by subdividing the bounded interval $[y_a, y_b]$ for 
the form in \eqref{eq:int_tail1}. 

If $q \leq 0$ and $y_b > e$, we bound $ |II |\leq \int_{y_b}^{\infty} y^{-k} = \f{1}{k-1} y_b^{-k+1}  $. 

If $q > 0$,  we have the following useful inequality 
\beq\label{eq:poly_log_mono}
y^{-a} (\log y)^q \leq  y^{-a} (\log y)^q |_{ y = \exp( \f{q}{a}) }
= e^{-q} ( \tf{q}{a})^q , \  \forall \  y>  1,
\quad  \pa_y (y^{-a} (\log y)^q) < 0, \quad y > \max( \exp( \tf{q}{a} ), 1),
\eeq
for any $a > 0, q \geq 0$, which is proved easily by taking derivatives and finding the critical points
$y_* = \exp(q / a )$. 
For $y_b >1$, we use this inequality with $a=0.1$ to bound 
{\small
\[
II \leq  e^{-q}  (10 q)^q \int_{y_b}^{\infty} y^{-k + 0.1 } d y
\leq   e^{-q}  (10 q)^q  \cdot  \f{1}{k-1.1} \cdot y_b^{- (k - 1.1)}. 
\] 
}

For the integral near $y \in [0, y_b]$ with $y_b \ll 1$, we derive the asymptotics of the integrand 
by  $h(y)= y^k |\log y|^q$ for some $k>0$. Now, using a change of variable $ y \to 1/y$, we get 
\[
 \tts{\int_0^{m} } y^k |\log y|^q d y = \tts{ \int_{  m^{-1} }^{\infty} }  \, y^{-k-2} |\log y|^q d y.
\]
Since $m^{-1} > 1$, we apply the above method for the integral in the region $y \gg 1$.

\vs{-0.1in}

\subsubsection*{\bf Method B.3: Kernel with a sign}\label{int:Method3}
Given a singular kernel $k(y) = \pa_y K(y)$ and a non-singular weight $g(y)$ 
with an explicit formula, if $k(y)$ has a fix sign on $[a, b]$, we estimate the integral
{\small
\[
  \int_a^b k(y) g(y) d y
  \, \in g([a,b])  \cdot  \int_a^b k(y) d y
  = g([a, b])  \cdot ( K(b) - K(a) ),
\]
}
where $g([a, b])$ is an enclosure of the range of  $g$ on $ [a, b]$. Taking absolute value, we obtain
\[
  \tts{\int}_a^b | k(y) g(y)| d y
\leq \| g \|_{L^{\infty}[a,b]}  \cdot | K(b) - K(a) | ,
\]

From \eqref{eq:K_intform} 
\eqref{eq:KJ_sign} and \eqref{eq:ker}, for $z> 1$, we have a fix sign 
\bseq\label{eq:ker_sign}
\beq\label{eq:ker_sign:a}
K_{\alb, i}(z)  =  K_{\alb, i, J}(z) > 0,  \quad 
  K_{\alb, \D}(1, z) = K_{\alb, 2, J}(1, z) - K_{\alb, 1, J}(1, z) > 0,  \quad \forall \, z > 1.
\eeq

For $z \in [0,1 )$, we have 
\beq\label{eq:ker_sign:b}
\bal
  K_{\alb, 1, J}(1, z) 
 & = (1 + z)^{\alb} - (1 - z)^{\alb} \geq 0, 
\quad  K_{\alb,2, J}(1, z) = \alb ( ( 1 + z)^{\alb-1} -  (1 - z)^{\alb-1} )   \leq 0 ,  \\
 K_{\alb, \D}(1, z) &= K_{\alb, 2, J}(1, z) - K_{\alb, 1, J}(1, z) \leq  0 -0 = 0 ,  \\
K_{\alb, 2}(1, z) & =  \alb ( ( 1 + z)^{\alb-1} -  (1 - z)^{\alb-1}  - 2 z^{\alb-1} )  \leq 0 .
\eal 
\eeq
\eseq
By obtaining anti-derivative for $k(z): K(z)= \int k(z) $, we apply the above methods to these  kernels.

\vs{-0.1in}
\subsubsection*{\bf Method B.4. Nonlocal error}\label{int:Method4}
Consider $f = f_P + f_A$ with $f_P$ being a piecewise polynomial in the form \eqref{eq:Bspline_rep}
and $f_A$ taking the form
\beq\label{eq:nloc_err_fA}
f_A(x) = h_A(x), \mw{ \ or \ }  f_A(x) = \one_{ [ y_1, y_2]}(x) h_A(x),
 \quad \mw{for \ }
 h_A(x) = c_1 \bwf(x) + c_2 \bwf(x)  \cdot \chi_0  ( p x + q) ,
\eeq
with some $p>0, q\in \R, c_1, c_2 \in \R$,
where $\chi_0$ is defined in \eqref{def:chi_0}. We can estimate the piecewise derivative bounds 
for $f$ using the methods in Appendix \ref{app:piece_bound}.
For $x \geq \xed$, $f_P$ vanishes, and we bound $f = f_A$ using methods in Appendix \ref{app:profile_far_field_bound}. 
For $f_A =  \one_{ [ y_1, y_2]}(x) h_A(x)$, we apply the methods in Appendix \ref{app:piece_bound} to estimate \(\partial_x^i f(x)\) for $x \in [y_1, y_2]$. 
We choose the powers $\bb_{f,i} \in \R$ and coefficients $\mu_{f,i}$ with $1\leq i \leq n_f$,  
and then bound the weighted $L^{\infty}$ norm
\beq\label{eq:intro_wg_nonlocal_err}
  \nlinf{ \vp_f f }, \quad \vp_f(x) \teq \max\nolimits_{ i \leq n_f} \mu_{f, i}^{-1}   |x|^{ \bb_{f, i}}.
\eeq

We derive sharp constants in Lemma \ref{lem:vel_est} using the partition 
$\cI$ in \eqref{def:interval} and methods in Appendix \ref{app:sharp_constant}. Then we apply Lemma \ref{lem:vel_est} to bound the nonlocal errors 
associated with $\cK_{\bullet }$ \eqref{eq:ker}, e.g.
\[ 
 |\psio_x( f) + 2 \alb \cJ(f)| \leq C_{\psio_x, J}( \mu_f, \bb_f, \cI, x) \nlinf{ f \vp_f }.
\]

Since $C_{\bullet}( \mu_f, \bb_f, \cI, x)$ in \eqref{eq:vel_const} are globally defined  explcit functions, we bound the error at any point $x$ 
and any bounded interval $[x_l, x_u]$ using \hyr[bd:intval]{\its Interval Arithmetic} to 
enclose the range $f([x_l, x_u])$. 

We also obtain tight bound for $\cJ(f)(x)$ \eqref{eq:Jw} following
\hyr[bd:Method3]{\its Method B.3}
\[
   |\tts{\int_0^x} y^{\alb-1} f(y) dy| \leq \sum_{ 0\leq i\leq n } \alb^{-1} |y_{i+1}^{\alb} - y_{i}^{\alb}|\cdot  |f([y_i, y_{i+1}])|, \quad 0 = y_0 < ..< y_n =x.
\]

From estimates \eqref{eq:vel_est:b}, for $\tf{1}{x}\psio( f) + 2 \alb \cJ(f), \psio_x( f) + 2 \alb \cJ(f)$ and $\tf{1}{x}\psio( f)  -  \psio_x( f)$, we can choose the power with $\max b_i \in (\alb, 1)$ to obtain faster decay estimates in $x$.

\vs{-0.05in}
\subsubsection*{\bf Method B.5. Nonlocal error for compact support source}\label{int:Method5}
Suppose that $r \geq 2$ and $x \geq r y$ for any $y \in \supp(f) \cap \R_+$. Using 
\eqref{eq:deri_K_J:y<x} and the definitions of $\cK_{\alb,i, J}$ \eqref{eq:ker}, we obtain
\beq
\bal
 |\cK_{\alb, i, J}(f)(x)|
 & =   | \tts{\int}_0^{\infty}  K_{\alb, i, J}(x, y) f(y)  d y |
 \leq C_{\eqref{eq:deri_K_J:y<x}, i , r} x^{\alb+ 1-i} 
  \tts{\int}_0^{\infty} |  \f{y}{x} f(y) | d y  \\
 & \leq C_{\eqref{eq:deri_K_J:y<x}, i , r} x^{\alb-i}  \| y f(y) \|_{L^1(\R_+)}
 \label{eq:K_L1_est}
 \eal
\eeq
for $i=1,2$. Compared to Lemma \ref{lem:vel_est}, we use the weaker $L^1$ norm for $w$. When $w$ has a compact support, the above estimates allow us to obtain much faster decay estimates in $x$. Using 
the method in Appendix \ref{app:piece_bound}, we obtain the piecewise bounds 
for $f(y)$ and $y f(y)$ and derive the $L^1$ norm.

\subsubsection{Sharp constants in Lemma \ref{lem:vel_est} }\label{app:sharp_constant}

Recall the constants $C_{\bullet}(b, I)$ from \eqref{eq:vel_const} 
\beq\label{eq:non_C_recall}
\mCC_{\bullet}(b, I) = \tts{\int}_I |k(1, z)| \cdot |z|^{-b} d z 
\eeq
for $k$ being $K_{\alb, i, J}, K_{\alb, i}, K_{\alb, \D}$ in \eqref{eq:ker_sign}. 
We choose the intervals $I \in \cI$ in \eqref{def:interval}.
From \eqref{def:interval}, the intervals satisfy $I \subset [1, \infty)$ or $I \subset [0,1]$. 

For $I \subset [1,\infty]$,  if $I$ is bounded, since all five kernels $k(z)$ have a fixed sign by \eqref{eq:ker_sign:a}, we bound $\mCC_{\bullet}(b, I)$ using 
\hyr[int:Method3]{\its Method B.3}. For $I = [y_b, \infty)$ with $y_b \gg 1,y_b>2$, using estimates \eqref{eq:deri_K_12} with $k=0$, we bound 
\[
 \tts{\int}_{y_b}^{\infty} |k(1, z)| z^{-b} d z
  \leq C_{ \eqref{eq:deri_K_12}, \bullet, y_b, 0 } \int_{y_b}^{\infty} z^{\alb - 3 - b} d z
  = C_{ \eqref{eq:deri_K_12}, \bullet, y_b, 0 }  \cdot \tf{1}{ b + 2 -\alb} y_b^{\alb-2-b}.
\]
for $C_{ \eqref{eq:deri_K_12}, \bullet, y_b, 0 }$ being the constants corresponding to kernel $k(1,z)$ in \eqref{eq:deri_K_12}. 

For $I = [a, b] \subset [0, 1]$,
we choose a small $a_1 \in (0, \f14]$ and sub-partition 
\[I  = [a, \max( a, a_1) ] \cup  [ \max(a, a_1), b] = I_1 \cup I_2.
\]
We apply \hyr[int:Method3]{\its Method B.3} to the 
interval $I_2$ and to four kernels $K_{\alb, i, J}(1,z), K_{\alb, 2}(1, z)$, $K_{\alb, \D}$, which have a fixed sign on $I$ by \eqref{eq:ker_sign:b}. 
Since  $K_{\alb, 1}(1, z) \in C^{\alb} (I_2)$ by \eqref{eq:ker} and is non-singular, we apply 
\hyr[int:Method1]{\its Method B.1} to bound its integral in $I_2$. 
In region $I_1 = [0, a_1] $, using \eqref{eq:deri_K_J:y<x}, we bound 
\beq
\tts{\int}_0^{a_1} |k(1,z)| \cdot |z|^{-b} d z 
\leq \tts{\int}_0^{a_1} C_{ \eqref{eq:deri_K_J:y<x}, \bullet, a_1^{-1} }  |z|^{1-b} d z 
= \tf{1}{2-b}  C_{ \eqref{eq:deri_K_J:y<x}, \bullet, a_1^{-1} }   a_1^{2-b},
\label{eq:nonloc_C_est1}
\eeq
for kernel $k =K_{\alb,i, J}$ or $K_{\alb, \D}$, where $ C_{ \eqref{eq:deri_K_J:y<x}, \bullet, a_1^{-1} }$ is the corresponding constants. For $K_{\alb, i}(1, z) = K_{\alb, i, J}(1, z) - 2 \alb z^{\alb-1} $
 by \eqref{eq:ker}, using triangle inequality, for any $b \leq \alb$, we bound 
\[
 \tts{\int}_0^{a_1} |K_{\alb, i}(1, z) | \cdot |z|^{-b}
 \leq  \tts{\int}_0^{a_1} |K_{\alb, i,J}(1, z) | \cdot |z|^{-b}
 + 2 \alb \tts{\int}_0^{a_1}  z^{\alb-1-b} d z
 = II_1 + II_2, \quad II_2  = \f{2 \alb}{ \alb- b } a_1^{\alb - b}.
\]
Note that the above bound is $\infty$ when $b=\alb$ and we never use such an estimate.
The term $II_1$ is estimate by \eqref{eq:nonloc_C_est1}. 
Thus, we can derive sharp bounds for \eqref{eq:non_C_recall} rigorously.

\subsubsection{Bounding $\pa_x^3 V(f)$}\label{app:Vxxx}

To control $\psio(\bw), \psio_x(\bw)$ not just on grid points $\mmx, X$ \eqref{def:x_i}, \eqref{def:XX}, we apply \hyr[bd:interp]{\its Interpolation Estimates} and control higher order derivatives $\pa_x^3 \psio(\bw)$.  Since $|\bw| \les \ang x^{-1/3}$, 
we aim to obtain explicit constants in $|\pa_x^3 \psio(\bw) | \les \ang x^{-1/3-2}$.

For $f$ being odd, since $\pa_x^3 x = 0$, $\pa_x^2 f $ is odd, and $\psi(f)$ is a convolution \eqref{eq:1D_dyn0}, taking derivatives $\pa_x^3 V(f)$ and passing to $f$, we obtain
\beq\label{eq:Vxxx_est1}
\bal
 \pa_x^3 V(f) & = \cK(f) \teq  \tts{\int}_{\R_+} k(x, y) \pa_y^2 f(y) d y, 
 \ 
   k(x, y) = \alb ( | x+ y|^{\alb-1} -  \sgn(x-y) |x- y|^{\alb-1} )  .
 \eal
\eeq
The kernel $k(x, y)$ is the same as $K_{\alb, 2, J}$ without the $y^{\alb-1}$ part in \eqref{eq:ker:c}.

Using mean-value theorem,  for $z \geq 0$, we have the following simple properties for $k(x, y)$
\beq
\bal
k(x, y) & = \pa_y  ( |x+y|^{\alb} + |x-y|^{\alb} ),  \qquad
k(x, y) > 0,  \  \forall y > x , 
\quad k(x, y) <0 , \  \forall \ y \in [0, x) , \\
|k(x, y) | &  \leq 2 \alb |y - x|^{\alb-1} 
\leq 2 \alb   y^{\alb-1} (1 - 1 / r)^{\alb-1}, \qquad \forall \ y \geq r x  > x.
\eal
\label{eq:kernel_kz}
\eeq

We choose $x_{\mw{low}} \geq 1$ and  discuss the  estimates  for  $x \geq 0$ and an improved estimate  for small $x \leq x_{\mw{low}}$.

\vs{0.05in} 
\paragraph{\bf Estimates for $x\geq 0$.} We fix $b \in [0, \alb]$.  We decompose \eqref{eq:Vxxx_est1} as 
\beq
\cK(f) = ( \tts{\int}_{0\leq y \leq x/2 } + \tts{\int}_{y \geq x/2 } ) k(x, y) \pa_y^2 f(y) d y \teq 
I_{ \eqref{eq:Vxxx_est2},1 }  + I_{ \eqref{eq:Vxxx_est2}, 2}. 
\label{eq:Vxxx_est2}
\eeq

For $I_{ \eqref{eq:Vxxx_est2}, 2 }$,  using a change of variable $z = y/x$ and  
$k(x, y) = x^{\alb-1} k(1, z)$, we obtain
\[
 |I_{ \eqref{eq:Vxxx_est2}, 2}|
 \leq   \int_{y \geq x/2 } | k(x, y) | \cdot |y|^{-2- b} d y \cdot  \nlinf{ y^{2+ b} f_{yy}}
 =  |x|^{-2 + \alb -b} 
  \int_{z \geq 1/2 } | k(1, z) | \cdot |z|^{-2- b} d z \cdot  \nlinf{ y^{2+ b } f_{yy}}
\]

To bound the $z$-integral, we use \hyr[int:Method3]{\its Method B.3} and the sign properties of $k(1,z)$ in \eqref{eq:kernel_kz}. 
To estimate the far-field of the $z$-integral, we use \eqref{eq:kernel_kz}
and get 
\[ 
\tts{\int}_{z \geq M} |k(1,z) z^{-2- b} |  d z 
\leq 2 \alb (1 - 1/M)^{\alb-1}   \int_{z \geq M}   z^{ \alb -3- b}   d z 
\leq \f{2 \alb }{2 + b -\alb}  (1 - 1/M)^{\alb-1} M^{-2 +\alb - b}  .
\]

For $I_{ \eqref{eq:Vxxx_est2},1 }$, since the kernel is non-singular 
and $f$ is odd , we use integration by parts to obtain
\[
I_{ \eqref{eq:Vxxx_est2},1 }
= {\int}_{0\leq y \leq x/2 }  k(x, y) \pa_y^2 f(y) d y 
= k(x, \tf{x}{2} ) (\pa_y f)( \tf{x}{2} )
-  (\pa_y k)(x, \tf{x}{2}  )  f( \tf{x}{2} )
+  {\int}_{0\leq y \leq x/2 } \pa_y^2 k(x, y)  f(y) d y 
\]
where the boundary term at $y=0$ vanishes since $k(x, 0) = 0$ by \eqref{eq:kernel_kz} 
and $f(0) = 0$. 
Changing $y = x z$, using the scaling symmetry again
$(\pa_y^i k)(x, y)
= x^{ \alb -  1 -i} \pa_z^i k(1, z)$, 
and bounding $| \pa_y^i f(y)| 
\leq (x z)^{-i -b} \nlinf{ y^{i+b} \pa_y^i f} $, we bound 
\[
\bal
|I_{ \eqref{eq:Vxxx_est2},1 }|
 \leq  |x|^{\alb -b - 2} & \big(  | k(1, \tf12 ) | \cdot (\tf12 )^{-1-b}  \nlinf{ x^{b + 1} \pa_x f }
+ | \pa_z k(1, \tf12 ) | \cdot (\tf12)^{-b}  \nlinf{ x^{b }  f }  \\
& \qquad
 + \tts{ \int_0^{1/2} } | \pa_z^2 k(1, z)| z^{-b} d z  \nlinf{ x^{b }  f }  \big) .
\eal
\]
The $z$-integral is bounded by $  \sup_{z\in[0,1/2]} |\pa_z^2 k(1, z) | \cdot \f{1}{1-b} (1/2)^(1-b)$.
We use \hyr[bd:intval]{\its Interval Arithmetic}  and the formula of $k$ to bound supremum.

\vs{0.1in}
 \paragraph{\bf Improved estimate for $0\leq x \leq x_{\mw{low}}$}

 For small $x \leq  x_{\mw{low}}$, we do not need to perform decay estimates and obtain an improved estimate. We fix $y_1 >  x_{\mw{low}}$ with $y_1 \asymp 1$ and decompose 
\beq
\cK(f) = ( \tts{\int}_0^{y_1} + \tts{\int}_{ y_1 }^{\infty} ) k(x, y) \pa_y^2 f(y) d y \teq 
I_{ \eqref{eq:Vxxx_est3},1 }  + I_{ \eqref{eq:Vxxx_est3}, 2}. 
\label{eq:Vxxx_est3}
\eeq

For $ I_{ \eqref{eq:Vxxx_est3}, 2}$, since $\f{y}{x}\geq \f{y_1}{  x_{\mw{low}}} $,  we bound it using the decay \eqref{eq:kernel_kz} 
\[
|I_{ \eqref{eq:Vxxx_est3}, 2}| 
\leq \nlinf{ |y|^{2 + b} \pa_y^2 f  }
2 \alb    (1 - \f{ x_{\mw{low}}}{y_1} )^{\alb-1}
 \int_{y \geq y_1} y^{\alb-3 - b}
 = \nlinf{ y^{2 + b} \pa_y^2 f  }
   (1 - \f{ x_{\mw{low}}}{y_1} )^{\alb-1}  \f{2 \alb }{2 + b -\alb} y_1^{ -2 + \alb-b }.
\]

For $I_{ \eqref{eq:Vxxx_est3},1 }$, 
for $x \leq x_{\mw{low}} < y_1$, we use the sign 
and $k(x, y)  = \pa_y  ( |x+y|^{\alb} + |x-y|^{\alb} )$ \eqref{eq:kernel_kz} to get
\[
\bal
|I_{ \eqref{eq:Vxxx_est3},1 }|
& \leq \| f_{yy} \|_{L^{\infty} [0, y_1]} ( \int_0^{x} + \int_x^{y_1}) |k(x, y)| d y \\
& \leq  \| f_{yy} \|_{L^{\infty} [0, y_1]} \big| ( |x+y|^{\alb} +  |x-y|^{\alb} )   \big|_0^{x} \, \big|
+ \big| ( |x+y|^{\alb} +  |x-y|^{\alb} )   \big|_x^{y_1} \big| \\
& \leq  \| f_{yy} \|_{L^{\infty} [0, y_1]} \big( ( 2 - 2^{\alb} ) x_{\mw{low}}^{\alb}
+ \max( (2 x_{\mw{low}})^{\alb}, 2 (x_{\mw{low}} +y_1 )^{\alb}) \big).
\eal
\]

We optimize the above estimates for the two choices $b = 0.2, \alb$ and obtain
an upper bound for $V_{xxx}$ with explicit functions $F$
\beq
 |\pa_{xxx} V(\bw)(x)| \leq 
 \min ( F_{\eqref{eq:bound_vxxx_f}} ( x, 0.2),  |F_{\eqref{eq:bound_vxxx_f}} ( x, \alb) ) ,
 \quad  F_{\eqref{eq:bound_vxxx_f}} ( x, b)  \les \ang x^{\alb - b - 2}, \ b = 0.2, \alb
 \label{eq:bound_vxxx_f} .
\eeq

Using the above functions $F$ and applying \hyr[bd:intval]{\its Interval Arithmetic},
we obtain a rigorous bound for $|\pa_{xxx} V(\bw)(x) |$ in any interval $x \in [a, b]$ 
and decay estimates for large $x$.

\subsection{Nonlocal terms for Bspline with rigorous error bound}\label{app:nonlocal}

In this section, we present a method to compute the nonlocal terms $\cK_{\alb,i}(w)(x)$ \eqref{eq:ker}, $i=1,2$, for any $x \ge 0$ and $w$ in in the Bspline representation \eqref{eq:W_rep:WP}, \eqref{eq:Bspline_rep} or $w =\bwf$ defined in \eqref{eq:W_rep}, and derive  rigorous error bounds.

\subsubsection{Nonlocal terms for Bspline in \eqref{eq:Bspline_rep}}

By linearity and the definition of $B_i(x)$ in \eqref{eq:Bspline}, for any $w$ with the representation in \eqref{eq:Bspline_rep}, it suffices to derive 
\beq\label{eq:Bs_U_comp1}
 \cK_{\alb,i}( \mBs( \cdot, \ss))(x) = \int_0^{\infty} K_{\alb, i}(x, y ) \mBs(y ; \ss) d y,
 \quad   \  \ss = \{ \mmx_{i+j} \}_{-2\leq j\leq 2}, 
\eeq
for any $i=1,2.., \nnx$. We fix $i_* \in \{ 1,2,..,\nnx\}$ and $\ss  = \{ \mmx_{i_* +j} \}_{-2\leq j\leq 2}$  for the Bspline $\mBs(y; s)$. We \emph{only} evaluate the above integral for \emph{finite} many points $x$.

Let $\mmx_2$ be the mesh point in \eqref{def:x_i}.
We estimate the integral in three cases of $x \geq 0 $ differently 
\bseq\label{eq:Bs_U_comp2:case}
\beq
\mathbf{(1)}  \ \max( x, \mmx_2 ) \mmf < \max(0, s_1) , \quad \mathbf{(3)} \ x > \mmf s_5 ,
\quad  \mathbf{(2)}  \ \mw{other \ cases},  \quad \mmf = 4, 
\eeq
Note that $D = \supp(\mBs(\cdot ; \ss)) \cap \R_+ \subset [\max(0,s_1), s_5]$ by \eqref{eq:Bspline_prop1}. 
For any $y \in D$, in case (2), $y$ is comparable to $x$.
Since $\ss = \{ \mmx_{i_* +j} \}_{-2\leq j\leq 2} $, for $y \in D$, we obtain
\beq\label{eq:Bs_U_comp2:supp}
\bga
\mathbf{(1)}  \  \mmf \max( x, x_2 )  < \max(0, s_1) \leq y , \quad \mathbf{(3)} \ x > \mmf s_5 
\geq \mmf y , \\ 
\forall \, y \in \supp(\mBs(\cdot ; \ss)) \cap \R_+ \subset [\max(0,s_1), s_5]  = [\max(0, \mmx_{i_*-2} ), \mmx_{i_*+2} ].
\ega
\eeq
\eseq

\subsubsection{Case (1)}\label{app:nonlocal_x_leq_y}

In this case, we get $\mmf \cdot x < y$. From the definition of $\cK_{\alb, i}$ in \eqref{eq:ker}, the kernels are smooth away from $y=0$ and $y=x$.  To compute the non-singular integral with small round-off error, we use the 
5-th order Gaussian quadrature 
\footnote{
We choose $5$-th order since the abscissa and weights are given 
\emph{explicitly} and \emph{exactly} (see \eqref{eq:GQ5}), which simplify the rigorous error estimates.
Moreover, the fifth-order method provides sufficient accuracy for our purposes.
}
to approximate the integral $\int_a^b f(y) dy$:
\bseq\label{eq:GQ5}
{\small
\beq
\GQF(f, a, b) \teq \f{b-a}{2} \sum \nolimits_{ 1\leq i \leq 5} p_i f( \f{a+b}{2} + \f{b-a}{2} q_i ) 
=  \int_a^b f(y)  d y + \cE(f, a, b)
\eeq
}
where the nodes $q_i$ and weights $p_i, 1 \leq i \leq 5 $  are given by 
\footnote{
The nodes $q_i$ and weights $p_i$ given in \eqref{eq:GQ5}
satisfy $\sum_{1\leq i\leq 5} p_i q_i^k = 2 (k+1)^{-1}$ for even $k$ and $0$ for odd $k$, $k\leq 9$.
}
\beq
  \tts{ \mathbf{q} =  \B( 0, \, \pm \f13 \sqrt{ 5 - 2 \sqrt{ \f{10}{7} } }, \, \pm \f13 \sqrt{ 5 + 2 \sqrt{ \f{10}{7} } } \B) ,  
  \quad \mathbf{p} = ( \f{128}{225},  \f{322 + 13 \sqrt{70}}{900}, 
   \f{322 + 13 \sqrt{70}}{900},  \f{322 - 13 \sqrt{70}}{900},  \f{322 - 13 \sqrt{70}}{900} ). }
\eeq

The term $\cE(f, a, b)$ denotes the error, which satisfies the bound 
\cite[Section 5.3]{atkinson2008introduction}
\beq\label{eq:GQ5_error}
 |\cE(f, a, b) | = \f{(b-a)^{2n+1} (n!)^4}{ (2 n + 1) [( 2 n )!]^3 } |\pa^{2n} f(\xi)|, 
 \quad n = 5 , \quad 
  \mw{ \ for \ some  \ } \xi \in (a, b).
\eeq

\eseq

From \eqref{eq:Bs_U_comp2:supp}, we get $\supp (\mBs(y ; \ss)) \cap \R_+ 
=  [ s_{j_0} , s_5 ]$ for some $j_0 \in \{1,2, 3\}$. 
For the integral in each interval $[s_j, s_{j+1}]$ with $j_0 \leq j \leq 4$, we approximate the integral  as
\beq\label{eq:Bs_U_comp3}
  \int_{ s_j }^{  s_{j+1}  } K_{\alb,i}(x, y) \mBs(y ; \ss) d y  
=  \GQF\B( K_{\alb,i}(x , \cdot ) \mBs(\cdot ; \ss) , s_j, s_{j+1} \B)
 + \cE_j( \ss)(x)  .
\eeq

To bound the error, we apply the error estimates \eqref{eq:GQ5_error} on interval $y \in [s_j, s_{j+1}]$. Since $\mBs( y, \ss )$ is a cubic polynomial on $[s_j, s_{j+1}]$, 
using Leibniz rule, we obtain
\beq\label{eq:GQ5_err_est1}
|\pa^{2 n} \big( K_{\alb,i}(x , \xi ) \mBs(\xi ; \ss) \big) | \leq   \sum \nolimits_{ i \leq 3 } 
 \tts{ \binom{ 2n}{i} } |\pa_{\xi}^i \mBs(\xi, \ss)| \cdot | \pa_{\xi}^{2n-i} K_{\alb, i}(x, \xi) |,
 \quad \xi \in [s_j, s_{j+1}].
\eeq

Note that in Case (1), for $\xi \in [s_j, s_{j+1} ]$, we get $\xi \geq \mmf \max(x, x_2)$ \eqref{eq:Bs_U_comp2:case}. We apply  estimate \eqref{eq:deri_K_12} for $ 
 \pa_{\xi}^{2n-i} \cK_{\alb, i}(x, \xi)$ with $r = \mmf \geq 2$. 
 We apply \hyr[bd:Method1]{\its Method A.1} to bound $ \pa_{\xi}^i \mBs(\xi, \ss)|$ with 
 $\xi \in [s_j, s_{j+1} ]$. Then we obtain rigorous bound for 
 $|\pa^{2 n} K_{\alb,i}(x , \xi ) \mBs(\xi ; \ss)|$ and use \eqref{eq:GQ5_error} to control the error for the integral.

\subsubsection{Case (2): $ x \asymp y$}

In this case, the kernel can be singular as $y$ close to $x$. We derive the integral  \eqref{eq:Bs_U_comp1} exactly. To reduce the round-off error, 
 we expand the integral as in \eqref{eq:Bs_U_comp3}, and 
 use \eqref{eq:Bspline:b} for $ \mBs(y), y \leq s_3$ and  \eqref{eq:Bspline:a} for $\mBs(y), y \geq s_3$
which consists of least summands 
{\small
 \begin{align}\label{eq:Bs_U_comp4}
 & \cK_{\alb,i}( \mBs( \cdot, \ss))(x)  \\
 &  =
 \int_{s_{j_0}}^{s_3} 
 K_{\alb, i}(x, y) \cdot \B( - \sum_{ 1 \leq l \leq 2} \f{4}{m_l(\ss)}  (s_l - y)^3 \one_{s_l - y <0}  \B) d y + \int_{s_3}^{s_5} 
 K_{\alb, i}(x, y) \cdot \B(  \sum_{ 4 \leq l \leq 5} \f{4}{m_l(\ss)}  (s_l - y)^3 \one_{s_l - y > 0}  \B)  d y \notag \\
 &  = \sum_{l =1,2} \f{-4}{m_l(\ss)} \int_{\max(s_l, s_{j_0})}^{s_3} K_{\alb, i}(x, y) (s_l - y)^3 d y
 + \sum_{l=4, 5}  \f{4}{m_l(\ss)} \int_{\max(s_3, s_{j_0})}^{s_l} K_{\alb, i}(x, y) (s_l - y)^3 d y
 \notag
 \end{align}
 }
In the first term, we restrict $l\leq 2$ in the summation since when $l\geq 3$, we get $s_l \geq s_3 \geq y$. In the second term, we restrict $l \geq 4$ since when $l\leq 3$, we get $s_l \leq s_3 \leq y$. 

\vs{0.05in}
\paragraph{\bf 1. Exact integral formulas}
We first derive the integral
\beq\label{eq:Bs_U_comp4:cI}
 \cI(w, a,  l, \s, y_a, y_b) = \int_{y_a}^{y_b} (w + \s y)^{ a} ( z- y)^l d y, \quad \s \in \{ 1, -1\},
 \quad y_a \leq y_b,
\eeq
for $w+ \s y_a, w+ \s y_b \geq 0$ and $l \in \mathbb{Z}_{\geq 0}$. We omit the dependency on $w, y_a, y_b$ when no confusion arises. For $\s = 1$, since $l$ is integer, using a change of variable $\xi = w  +y$, 
we get 
\[
\bal
\cI( a, l, 1) & = \int_{y_a}^{y_b} (w +  y)^{ a} ( z- y)^l d y
= \int_{y_a + w}^{y_b + w} \xi^a ( z + w - \xi)^l d \xi  = \sum_{0\leq i\leq l} \binom{l}{i} (-1)^i (z+w)^{l-i} \int_{y_a + w}^{y_b + w}   \xi^{a+i} d \xi  .
\eal 
\]
We derive the $\xi$-integral using the general formula for $0\leq s \leq t$
{\small
\beq\label{eq:int_pow}
   \int_s^t \xi^p d \xi = \f{1}{p+1} ( t^{p+1} - s^{p+1}) , \quad 0\leq s \leq t, p > -1.
\eeq
}

When $\s =-1$, recall the assumption $ w - y_a , w - y_b \geq 0$. We introduce $\xi = w-y$ and obtain
\[
\cI( a, l, -1)  = \int_{y_a}^{y_b} (w - y)^{ a} ( z- y)^l d y
= \int_{w - y_b}^{w- y_a} \xi^a ( z - w +\xi)^l d \xi  = \sum_{0\leq i\leq l} \binom{l}{i}  (z-w)^{l-i} \int_{ w - y_b }^{w - y_a }   \xi^{a+i} d \xi  ,
\]
where the $\xi$ integral is obtained using \eqref{eq:int_pow}.

\paragraph{\bf 2. Taylor expansion for $\cI$ }

The above exact formula suffers from significant round-off error when $| w + \s z| \gg |y_a - z|, |y_b-z|$. 
\footnote{
The integral is $\cI$ is small due to $(y-z)^l$ in \eqref{eq:Bs_U_comp4:cI}.
For example, when $\s=1$, the exact formula evaluates  $(y_b + w)^{a+i+1} - (y_a + w)^{a+i+1}$ numerically 
and suffers from significant round-off error when $|y_a - y_b| \ll |y_a + w|$.
}
Thus, when 
\beq\label{eq:Bs_U_comp4:ass}
w + \s z > \mf{m_2} d_{yz},\quad  d_{yz} =\max(|y_a-z|, |y_b-z|) , \quad w + \s z> 0,
\eeq
we expand the kernel in \eqref{eq:Bs_U_comp4:cI} using Taylor expansion:
\[
\bal
\cI =  \int_{y_a}^{y_b} (w + \s z + \s(y-z) )^{ a} ( z- y)^l d y
= \sum_{ 0\leq i\leq k} \f{\AAF(a, i)}{i!} (w +\s z)^{a-i}
\s^{i} (-1)^l  \int_{y_a}^{y_b}  (y-z)^{i+l} d y
 + \cI_{k, \mw{err}},
\eal
\]
where $ \cI_{k, \mw{err}}$ denotes the remainder term in the Taylor expansion.  We evaluate the integral for $(y-z)$ using \eqref{eq:int_pow}. 
To control the remainder term in the Taylor expansion, for $y \in [y_a, y_b]$, 
and any $\xi$ between $0$ and $\s(y-z)$, condition \eqref{eq:Bs_U_comp4:ass} implies $|w+ \s z | > \mf{m_2} |z-y| \geq \mf{m_2} |\xi|$ and
 \[
  |\pa_\xi^{k+1} ( w + \s z + \xi)^{a}|
=| \AAF(a, k+1) (w + \s z + \xi)^{a-k-1}|
\leq \AAF(a, k+1) ( (w + \s z)(1-1/ \mf{m_2}) )^{a-k-1}
\]

Using $\int_{y_a}^{y_b} |y-z|^{l+k+1} d y \leq |y_b-y_a| d_{yz}^{l +k+1} $, we bound the error term in the Taylor expansion  as
\[
 |\cI_{k, \mw{err}}| \leq \f{\AAF(a, k+1)}{ (k+1)!} | (w + \s z)(1 - 1 / \mf{m_2} )|^{a -k-1} 
|y_b-y_a| \cdot d_{yz}^{l +k+1}, \quad  d_{yz} =\max(|y_a-z|, |y_b-z|) .
\]

\paragraph{\bf 3. Estimate the integral \eqref{eq:Bs_U_comp4} }

 Recall $K_{\alb, i}$ from 
 \eqref{eq:ker_recall_app}. 
 For $x, y \geq 0$, to derive \eqref{eq:Bs_U_comp4}, we use the above formulas for \eqref{eq:Bs_U_comp4:cI} 
  with $ \s = 1, w=x$ and $w=0$ to evaluate
    \[
\tts{\int}_{y_a}^{y_b} (x+y)^{\alb } (s_l - y)^3 d y , \quad   \tts{ \int}_{y_a}^{y_b} (x + y)^{\alb-1} (s_l - y)^3 d y,  \quad
\tts{ \int}_{y_a}^{y_b} y^{\alb-1} (s_l - y)^3 d y .
 \]
 For $|x-y|^{\alb-1} , \sgn(x-y)$, and $p \in \{0, 1\}$, since $x-y$ may change sign for $y \in [y_a, y_b]$, we decompose 
 {\small
  \[
 \bal
\int_{y_a}^{y_b}|x-y|^{\alb-1} (\sgn(x-y))^p f(y) d y
 & = (-1)^p \one_{x < y_a} \int_{y_a}^{y_b} (y-x)^{\alb-1}  f(y)  d y
 + \one_{x > y_b} \int_{y_a}^{y_b} (x-y)^{\alb} f(y) \\
 & \qquad  + \one_{ y_a \leq x \leq y_b} \big(  \int_{y_a}^x (x-y)^{\alb}
 f(y) d y
 +(-1)^p  \int_x^{y_b}(y-x)^{\alb}  f(y)  d y \big),
\eal
 \]
 }
where $f(y) =(s_l - y)^3$. For each case, the integrand $x-y$ or $y -x$ does not change sign. 
We apply \eqref{eq:Bs_U_comp4:cI} to evaluate each term. Note that we only evaluate the integrals for \emph{finite many points $x$}.

\subsubsection{Away from the singularity: Case (3)}\label{app:nonlocal_x_geq_y}

We recall $ x > \mmf y \geq 2 y$ for $y \in \supp(\mBs(\cdot ; \ss)) \cap \R_+ $ 
from \eqref{eq:Bs_U_comp2:supp}. Since $y^{\alb-1}$ in the kernel $K_{\alb, i}(x, y)$ \eqref{eq:ker_recall_app} are singular near $y=0$, we decompose
\[
 \cK_{\alb,i}( \mBs( \cdot, \ss))(x) =  \int_0^{\infty} K_{\alb, i, J}(x, y ) \mBs(y ; \ss) d y
 - 2 \alb x^{2-i} \int_0^{\infty} y^{\alb-1} \mBs(y ; \ss) d y \teq I + II.
\]

For $I$, since $x > y$, the kernels $K_{\alb, 1, J}(x, y) = (x+y)^{\alb} -(x-y)^{\alb}, 
K_{\alb, 2, J} = \alb ((x+y)^{\alb-1} -(x-y)^{\alb-1} )$ are not singular. As in Appendix \ref{app:nonlocal_x_leq_y}, we apply  the Gaussian quadrature \eqref{eq:GQ5} with error estimate \eqref{eq:GQ5_error} to estimate $I$. To bound $\pa_\xi^{2n} (K_{\alb, i, J}(x, \xi ) \mBs( \xi ; \ss)) $ for the error in \eqref{eq:GQ5_error} with $\xi \in  \supp(\mBs(\cdot ; \ss)) \cap \R_+ $, we use the Leibniz rule \eqref{eq:GQ5_err_est1}, \hyr[bd:Method1]{\its Method A.1} for  $ \pa_{\xi}^i \mBs(\xi, \ss)$,
and estimate 
\eqref{eq:deri_K_J_12} for $ 
 \pa_{\xi}^{2n-i} K_{\alb, J, i}(x, \xi)$ with $r = \mmf \geq 2$. 
 Recall from \eqref{eq:Bs_U_comp2:supp} that $ \mmf y, \mmf \xi  \leq x$ in this case. 

We evaluate $II$ exactly using formulas for \eqref{eq:Bs_U_comp4:cI} 
 with $\s = 1, w = 0,l=3, a = \alb -1$ and linearity. 

Using the methods in Section \ref{app:nonlocal_x_leq_y}-\ref{app:nonlocal_x_geq_y}, 
we derive  \eqref{eq:Bs_U_comp1} with rigorous error control.

\subsubsection{Nonlocal estimates for $\bwp$ with large $x$}\label{app:WP_very_far}

For $x$ very large, the direct computation method in the previous sections suffers from  round-off error. Since $\bwp$ decays, we bound $\cK_{\alb, i, J}(\bwp)$ perturbatively. We decompose
\beq\label{eq:bwp_lg_decomp1}
 \bwp = \bwp \one_{ x\leq \mmx_{i_*} } + \bwp \one_{ x> \mmx_{i_*} }  , \quad  \mmx_{i_*} \approx 10^4,
\eeq
where $\mmx_{i_*}$ is a mesh point \eqref{def:x_i}. For $x \gg \mmx_{i_*}$, since $ \bwp \one_{ x\leq \mmx_{i_*} } $ is not small pointwisely but it has a compact support, we use \hyr[int:Method5]{\its Method B.5} to  estimate $\cK_{\bullet}( \bwp \one_{ x\leq \mmx_{i_*} } ), \bullet \in \{ (\alb, i,J), (\alb, \D ) \}$. We bound $\cK_{\bullet}( \bwp \one_{ x > \mmx_{i_*} } ), \bullet \in \{ (\alb, i,J), (\alb, \D ) \}$ 
using \hyr[int:Method4]{\its Method B.4} with parameters for the weight \eqref{eq:intro_wg_nonlocal_err}:
\beq
    \bb_{\eqref{def:bb_WP}} = (-1, 0, 0.4),\quad  
  \mu_{\eqref{def:bb_WP}} = (1,1,1).
  \label{def:bb_WP}
\eeq
Using Lemma \ref{lem:vel_est} and taking only one power $b=0.4$, we obtain decay estimates for large $x$,
\bseq\label{eq:cK_WP_far}
\begin{align}\label{eq:cK_WP_far:a}
  |\cK_{\alb, 1 ,J}(G)| & \leq \mCC_{ \psio/x, J} 
( 0.4,   \R_+  )  x^{\alb +1 -0.4 } \nlinf{ G |y|^{0.4} } 
,\quad   |\cK_{\alb, 2 ,J}(G)| \leq \mCC_{\psio_x, J}(  0.4 ,   \R_+  )  x^{\alb - 0.4}
 \nlinf{  G |y|^{0.4} }, \notag \\
   |\cK_{\alb, \D}(G)| & \leq \mCC_{ \D, 0}(  0.4 ,   \R_+  )   x^{\alb - 0.4}  \nlinf{ G |y|^{0.4} } ,
\end{align}
where $\mCC_{ \psio/x, J} , \mCC_{\psio_x, J}$ are the constants defined in \eqref{eq:vel_const}. We apply \eqref{eq:cK_WP_far:a} to bound $G =   \bwp \one_{ x> \mmx_{i_*}}$.

We derive $\cJ(\bwp)(x) = \int_0^x y^{\alb-1} \bwp(y) d y $ on finite many points $x = X_{ \eqref{def:XX}, i}$  using  \hyr[bd:Method1]{\its Method A.1} for piecewise bounds of $\bwp(y)$ and  \hyr[int:Method2]{\its Method B.2} for rigorous integral. 
Using \eqref{eq:cK_WP_far}, treating  $\cJ(\bwp)(x)$ as the main terms and $ \cK_{\alb, i, J}$ as an error, and using  \eqref{eq:cK_cKJ}, we obtain rigorous estimates for  $\cK_{\alb , i}(x)$:
\beq
\bal
\cK_{\alb , i}(\bwp) & =  - 2 \alb x^{2-i} \cJ(\bwp) +  \cK_{\alb, i, J}(\bwp) , \\
 |\cK_{\alb, i, J}(\bwp)|  & \les \ang x^{\alb-0.4 + 2 -i}, 
\quad \supp( \cJ(\bwp)) \subset [ -\xed, \xed ] ,
\eal
\eeq
\eseq
where $\cJ(\bwp)$ has a compact support due to  \eqref{eq:Jw_alb} 
and $\supp(\bwp) \subset [-\xed, \xed]$.

\subsection{Nonlocal terms for $\bwf$ in bounded region}\label{app:bwf_near}

Recall $\bwf$ from \eqref{eq:W_rep}:
\[
   \bwf(x) \teq    \chi_1( x ) |x|^{- \f13 }  \big( \cca + \ccb (\log |x|)^{-\f13 } 
   + \ccc (\log |x| )^{- \f23 } \big), \quad 
    \chi_1(x) =  \f{  (x- z_0)^5}{ z_0^5 + (x- z_0 )^5} \one_{x\geq z_0} .
\]
In this section, we discuss methods to compute $\cK_{\alb,i}(\bwf)(x)$ with rigorous error estimates. 
We cannot derive $V(\bwf),V_x(\bwf)$ explicitly. 
For $x \leq z_{L_1}$ ($z_{L_1}$ chosen in \eqref{def:z_0}) not very large, we interpolate $\bwf$ using Bspline \eqref{eq:Bspline_rep} 
and use  methods in Appendix \ref{app:nonlocal} to derive $ \cK_{\alb, i}(\cwf)(x)$.
For large $x \geq z_{L_1}$, we expand the kernels to derive the asymptotics for 
$ \cK_{\alb, i}(\cwf)(x)$. We define 
\beq\label{def:chi_2}
  \chi_2(x) = 1 - \chi_0( \f{x}{z_L} -1)
= 
    1 - \f{ ( x / z_L -1)^5}{ 1 + (x / z_L - 1)^5 } \one_{x \geq z_L}
   =   \f{ z_L^5 }{ z_L^5 + (x  - z_L)^5 } \one_{x \geq z_L} , \eeq
where $z_L$ is chosen in \eqref{def:z_0} and $\chi_0$ is defined in \eqref{def:chi_0}. The function $\chi_2(x)$ plays a role as a cutoff function 
and $\chi_2 \ll 1$ for $x \gg z_L$.  We use the Bspline  \eqref{eq:Bspline_rep} 
to interpolate $ \chi_2(x)  \bwf$ and obtain coefficients $a_{F,i}$:
\bseq\label{eq:bwf_decomp1}
\beq
\bwf = \chi_2 \bwf +   (1 - \chi_2(x)) \bwf, \quad 
\cwf \teq \tts{\sum\nolimits_{1 \leq i \leq  \nnx}} a_{F,i } B_i(x) \approx \chi_2 \bwf.
\eeq
We treat $\cwf - \chi_2 \bwf$ as an error and derive the nonlocal terms for $\cK_{\bullet}$ defined in \eqref{eq:ker} using 
\beq
\cK_{\bullet}( \bwf)
= \cK_{\bullet}( \cwf ) + \cK_{\bullet}( (1 - \chi_2(x)) \bwf) + 
\cK_{\bullet}( \chi_2 \bwf - \cwf).
\eeq
\eseq

\subsubsection{ Nonlocal error}
We treat $\cK_{\bullet}( \cwf - \chi_2 \bwf)$ in \eqref{eq:bwf_decomp1} as an error. By definition
of $\chi_2$ in \eqref{def:chi_2}, for $x \geq \xed$, we obtain
\[
  |\cwf - \chi_2 \bwf| = |\chi_2 \bwf| \leq \tf{z_0^5}{ (\xed - z_L)^5 }  C_{ \eqref{eq:deri_WF_far}, z_0,  0,  \xed  } x^{-1/3}.
\]

Since $  \cwf$ is a piecewise polynomial \eqref{eq:bwf_decomp1}
and $\chi_2 \bwf $ has the form \eqref{eq:nloc_err_fA} by definition of $\chi_2$ \eqref{def:chi_2},
 we bound the nonlocal error 
using \hyr[int:Method4]{\its Method B.4} with parameters for the weight in \eqref{eq:intro_wg_nonlocal_err}:
\beq
 \bb_{ \eqref{def:bb_mu_u2} } = (-1, \, -0.5,\, 0, \, 0.2),
 \quad \mu_{ \eqref{def:bb_mu_u2} } = (0.4, \, 0.4, \, 1, \, 30).
 \label{def:bb_mu_u2}
\eeq
These parameters are not unique. The choice in \eqref{def:bb_mu_u2} yields relatively tight nonlocal bounds.

\subsubsection{Estimate for $x \leq z_{L_1}$}

Since $\cwf$ is representing using Bspline basis, we compute the nonlocal terms $\cK_{\alb, i}(\cwf)$  with rigorous error control using the methods in Appendix \ref{app:nonlocal}. 

Since $ (1 - \chi_2(x)) \bwf(x)$ is supported in $x \geq z_L$,  
for $x \in [0, z_{L_1}]$ and $ y\in \supp( 1 - \chi_2(x))$, 
we get $y /x \geq \f{z_L}{z_{L_1}} > 10 $ by  \eqref{def:z_0}. We expand the kernel 
using Case 1 in Appendix \ref{app:kernel_TL}  to derive 
{\small
\[
\bal
 \cK_{\alb, 1}(f)(x) & =  \sum_{ 3 \leq i \leq k, \ i \mw{ \ is \  odd} }
   \int_0^{\infty} \f{2 \AAF(\alb, i) }{ i!} \f{x^i}{y^i} \cdot y^{\alb}  f(y) d y
  + \cE_{1,k}(x),  
   \\
   |\cE_{1,k}(x)|  & \leq \int_0^{\infty} \B| \f{ 2\AAF(\alb, k+1) }{ (k+1)!} \cdot \f{x^{k+1}}{y^{k+1}} (1 - 
\tf1r )^{\alb-k-1} \cdot y^{\alb} f(y) \B| d y ,  
\quad 
 r = \f{z_L}{ z_{L_1}} > 10, 
 \\
 \cK_{\alb, 2}(f)(x)  & = \sum_{  2\leq i \leq k, \ i \mw{ \ is \  even} }\alb  \int_0^{\infty} \f{ 2 \AAF(\alb-1, i) }{ i!} 
  \f{x^i}{y^i} y^{\alb-1} f(y) d y  + \cE_{ 2, k}(x) ,  \\
  | \cE_{2 , k}(x)|  & \leq \alb\int_0^{\infty } \B| \f{ 2 \AAF(\alb-1, k+1) }{ (k+1) ! } \f{x^{k+1}}{y^{k+1}} (1 - \tf1r )^{ \alb-2-k} \cdot y^{\alb-1} f(y) \B| d y
 \eal 
\]
}
for $ \cK_{\alb, i}(f)(x), f = (1-\chi_2) \bwf$.  
From \eqref{def:bwfa} and \eqref{eq:bwfa_far}, for $y \geq 10$, we obtain 
\[
g(y) = f(y) y^{\alb}= (1 -\chi_2 ) y^{\alb} \bwf 
= (1-\chi_2) \chi_1 \bwfa ,
\quad  \bwfa \in [-6, -3].
\]
Thus, $g(y)$ does not change sign in its support. To derive the main term and control the error, we
integrate $\int_{z_L}^{\infty} g(y) y^{ - l}$ for $l=3, .., k+1$ with error control.
We choose $S_m \gg 1$ and decompose 
\[
 \int_{z_L}^{\infty} g(y) y^{-l} d y =
z_L^{ 1- l}  \int_1^{\infty} g( z_L s) s^{-l } d s 
 =  z_L^{ 1- l} \cdot ( \int_{1}^{S_m} g( z_L s) s^{-l } d s 
 + \int_{S_m}^{\infty} g( z_L s) s^{-l } d s )   \teq 
z_L^{ 1- l} (  I_1 + \cE_{g, l} ).
\]
For $I_1$, the integrand is not singular and has an explicit formula. We apply 
\eqref{eq:basic_int_est:ul}, 
\hyr[int:Method1]{\its Method B.1} to obtain its  rigorous enclosure. 
Since $S_m > 1$ and by definition of $|g(x) | \leq |\bwf(x) x^{\alb} |$, using \eqref{eq:deri_WF_far} 
with $k=0, L =z_L$ and $\chi_i \in [0,1]$, we estimate $\cE_{g, l}$ directly
\[
 |\cE_{g, l}|  \leq 
\tf{1}{l-1} S_m^{1-l}  \sup\nolimits_{s \geq S_m} |g(z_L s)| 
\leq \tf{1}{l-1} S_m^{1-l}  \sup\nolimits_{z \geq  z_L} |z^{\alb} \bwf(z)|
\leq C_{ \eqref{eq:deri_WF_far}, z_0,  0,  z_L }  \cdot \tf{1}{l-1} S_m^{1-l} .
\]

\subsection{Nonlocal estimate for $\bwf, \bw$ with large $ x \geq z_{L_1} $}
\label{app:bwf_far}

For $x$ sufficiently large, e.g. $x \geq z_{L_1}$, 
we do not decompose \eqref{eq:bwf_decomp1} since the Bspline-based method suffers from round-off error.
Instead, we expand the integral to derive the asymptotics. 
For $\bwf = \chi_1(y) y^{-1/3} \bwfa $ \eqref{def:bwfa}, 
which is supported in $|y| \geq z_0$ by \eqref{eq:W_rep},  using the kernel $K_{\alb, i, J}$ 
\eqref{eq:ker}, we first decompose
\begin{align}
 \cK_{\alb ,i} \bwf(x)  & = - 2 x^{2-i} \alb \cJ(\bwf)(x)
 + \cK_{\alb, i, J} ( \one_{y \geq z_0} y^{-1/3} \bwfa)(x)
+ \cK_{\alb, i, J}( (\one_{y \geq z_0} -\chi_1)  y^{-1/3} \bwfa )(x) \notag \\
&  \teq - 2 x^{2-i} \alb \cJ(\bwf)(x) + I_{ \eqref{eq:K_log_decomp1}, i} + \cE_{\eqref{eq:K_log_decomp1}, i} .
\label{eq:K_log_decomp1}
\end{align}
From Corollary \ref{cor:vel_est}, the main term is given by $ \cJ$-term.
 We treat $\cE_{\eqref{eq:K_log_decomp1}, i}$ as an error, since for $y$ near $x$, 
$|\chi_1-1|$ in the integrand is very small. 
We do not treat  \(  I_{ \eqref{eq:K_log_decomp1}, i} \) as an error, since it does not decay significantly faster than the $\cJ$-term. 
While the following estimates appear technical, it is quite straightforward to 
derive and track the bounds with computer-assistance. 

Recall $\cca, \ccb, \ccc$ from \eqref{def:WF_para}.  We define the main term 
\beq
\cJM(z)  \teq 
   \cca  \log z + \tf{3}{2} \ccb  (\log z)^{2/3} + 3 \ccc (\log z)^{1/3} .
   \label{def:J_main}
\eeq

\subsubsection{Estimates for  nonlocal error $\cE_{\eqref{eq:K_log_decomp1}, i}$}\label{app:bwf_err_lg}

By definition of $\chi_1$ \eqref{eq:W_rep},
$F(x)= (\one_{x \geq z_0} -\chi_1)  x^{-1/3} \bwfa$ is supported in $x \geq z_0$. 
Since $F(x)$ is not small for $x$ is not very large, we decompose $F$ 
\beq\label{eq:K_log_decomp1_err}
 F(x) = (\one_{x \geq z_0} -\chi_1(x))  x^{-1/3} \bwfa(x) ,
 \quad F= F \one_{ x\leq \mmx_{i_*} } + F \one_{ x> \mmx_{i_*} } \teq F_1 + F_2,
 \quad  \mmx_{i_*} \approx 10^4,
\eeq
and extend each part $F_1, F_2$ as an odd function in $\R$, where $\mmx_{i_*}$ is a mesh point \eqref{def:x_i}. Since we consider very large $x \geq z_{L_1} \gg   \mmx_{i_*} $ (see $z_{L_1}$ \eqref{def:z_0}) and since $F_1$ is not small pointwisely but it has a compact support, we use \hyr[int:Method5]{\its Method B.5} to  estimate $\cK_{\alb, i, J}(F_1)$ and obtain its fast decay estimate in $x$. We bound $ \cK_{\alb, i, J}(F_2)$ using \hyr[int:Method4]{\its Method B.4} with 
\beq
 \bb_{ \eqref{def:bb_mu_u4} } = 0.4,  \quad  \mu_{  \eqref{def:bb_mu_u4} } = 1.
 \label{def:bb_mu_u4}
\eeq
 for the weight in \eqref{eq:intro_wg_nonlocal_err}.  We choose one power since these error estimates 
apply only for very large $x \geq z_{L_1}$. 
 Since $  \bb_{ \eqref{def:bb_mu_u4} } = 0.4,$, 
by Lemma \ref{lem:vel_est}, we obtain decay estimates for large $x$,
\begin{align}
  |\cK_{\alb, 1 ,J}(F)| & \leq \mCC_{ \psio/x, J} 
( 0.4,   \R_+  )  x^{\alb +1 - 0.4 } \nlinf{ F |y|^{0.4} } ,
\quad 
   |\cK_{\alb, 2 ,J}(F)| \leq \mCC_{\psio_x, J}(  0.4 ,   \R_+  )  x^{\alb - 0.4 } \nlinf{ F y^{0.4}} , \notag \\
  |\cK_{\alb, \D}(F) |  & \leq \mCC_{\D, 0}(  0.4 ,   \R_+  )x^{\alb - 0.4 } \nlinf{ F y^{0.4}} ,
  \label{eq:K_log_err:far}
\end{align}
where $\mCC_{ \psio/x, J} , \mCC_{\psio_x, J}$ are the constants defined in \eqref{eq:vel_const}.

\subsubsection{Estimates for $I_{\eqref{eq:K_log_decomp1}, i}$}

Recall $\cJ(f)$ from \eqref{eq:Jw_alb}. Since we consider $x \geq z_{L_1}$, we decompose 
{\small
\beq
  \cJ(\bwf)(x) = \int_0^{ \mmx_{ \eqref{def:x_i}, i_*} }  y^{\alb-1} \bwf(y) d y + \int_{ \mmx_{ \eqref{def:x_i}, i_*} }^x y^{\alb - 1} \bwf(y) d y \teq 
  I_{ \eqref{eq:bwf_J_decomp1}, 1} +   I_{ \eqref{eq:bwf_J_decomp1}, 2},
  \label{eq:bwf_J_decomp1}
\eeq
}
where  $\mmx_{ \eqref{def:x_i}, i_*} \approx 10^4$ is the mesh chosen in \eqref{def:x_i} 
with index $i_*$. 
\footnote{
The actual value of $\mmx_{ \eqref{def:x_i}, i_*}$ is not important. We choose   
$\mmx_{ \eqref{def:x_i}, i_*}$ relatively large, and it is on the mesh $\mmx_{\eqref{def:x_i} }$ 
to simplify the error estimate. 
}
To derive $  I_{ \eqref{eq:bwf_J_decomp1}, 1}$ with high accuracy, we refine the mesh
$[\mmx_i, \mmx_{i+1}]$ by $\mmx_i = y_{i, 1} < ..< y_{i, N+1} = \mmx_{i+1}$ 
with $y_{i, j+1} - y_{i, j} =  \f{1}{N} (  \mmx_{i+1} - \mmx_i)$, and use Simpson's rule 
to approximate the integral for $F = y^{\alb-1} \bwf(y)$:
\beq
  \int_{\mmx_i}^{\mmx_{i+1} } F(y)d y
  = \sum\nolim_{ 1\leq j \leq N}  
  \f{ y_{i,j+1}- y_{i, j}}{6} \cdot ( F( y_{i, j} ) + F(y_{i, j+1}) + 4 F( \f{1}{2} (y_{ i, j} + y_{i,j+1}) )  )
  + \cE_{\eqref{eq:bwf_J_simp}}
  \label{eq:bwf_J_simp}
\eeq
 Using error estimates for Simpson's rule \cite[Section 5.1]{atkinson2008introduction}, and $y_{i, j+1} - y_{i, j} =  \f{1}{N} (  \mmx_{i+1} - \mmx_i)$, we get
\beq
|\cE_{\eqref{eq:bwf_J_simp}}| \leq \tts{ \sum}_{1\leq j\leq N} \f{(y_{i, j+1} - y_{i, j})^5}{ 2880} 
  \| \pa_x^4 F \|_{L^{\infty}(\mmx_i, \mmx_{i+1})}
  = \f{(\mmx_{i+1} - \mmx_i )^5}{  2880 N^4}  
  \| \pa_x^4 F \|_{L^{\infty}(\mmx_i, \mmx_{i+1})} .
\eeq

We bound $\pa_x^i \bwf$ using \hyr[bd:Method2]{\its Method A.2} and Leibniz rule to further bound $\pa_x^4 F =  \pa_x^4( \bwf y^{\alb-1})$. 
By partitioning $[0, \mmx_{ \eqref{def:x_i},* }]$ by $[\mmx_i, \mmx_{i+1}], 1\leq i \leq i_*-1$ and using linearity, we bound  $ I_{ \eqref{eq:bwf_J_decomp1}, 1}$.

Recall $\bwf(y) = \chi_1(y) y^{-1/3} \bwfa(y)$ from \eqref{def:bwfa}. For  $ I_{ \eqref{eq:bwf_J_decomp1}, 2}$, since for $y > 
\mmx_{ \eqref{def:x_i}, i_*}$, $\chi_1$ is close to $1$ and $1-\chi_1 \geq 0$ is decreasing, we bound 
\beq
I_{ \eqref{eq:bwf_J_decomp1}, 2}
= \tts{\int}_{ \mmx_{ \eqref{def:x_i}, i_*} }^x y^{- 1} \bwfa 
+ \tts{\int}_{ \mmx_{ \eqref{def:x_i}, i_*} }^x y^{ - 1} ( \chi_1 -1) \bwfa 
\teq  II_{ \eqref{eq:bwf_J_decomp2}} +  \cE_{ \eqref{eq:bwf_J_decomp2}}
\label{eq:bwf_J_decomp2}
\eeq

By definition of $\chi_1$ \eqref{def:chi_0} and using \eqref{eq:bwfa_far} for $\bwfa$,
for $x \geq \mmx_{i_*} > 10$,
we obtain 
\beq
\bal
 (1 - \chi_1 ) |\bwfa| & =  \tf{ z_0^5}{ (x-z_0)^5} |\bwfa|
 \leq x^{-5} \tf{  \mmx_{ i_*}^5 z_0^5}{ (\mmx_{ i_*}-z_0)^5} 
\max( |\cca|, | \bwfa( \mmx_{i_*})  | )
 \teq C_{ \eqref{eq:bwf_J_decomp2_err} }(\mmx_{i_*}) x^{-5}, \\
 \Longrightarrow \qquad \quad  \,  |\cE_{ \eqref{eq:bwf_J_decomp2}}| & \leq 
 C_{ \eqref{eq:bwf_J_decomp2_err} }(\mmx_{i_*}) \tts{\int}_{ \mmx_{ \eqref{def:x_i}, i_*} }^x y^{ - 6}  d y
 \leq 
 \tf15 C_{ \eqref{eq:bwf_J_decomp2_err} }(\mmx_{i_*}) \mmx_{ \eqref{def:x_i}, i_*}^{-5}.
 \label{eq:bwf_J_decomp2_err}
 \eal
 \eeq

Using \eqref{def:bwfa} for $\bwf$, \eqref{def:J_main} for $\cJM$, 
and $\int y^{-1}  (\log y)^a d y
= \f{1}{1+a} (\log y)^{a+1} + C$,   we derive he main term:
\beq
\bal
II_{ \eqref{eq:bwf_J_decomp2}} 
& = \tts{\int}_{ \mmx_{ \eqref{def:x_i}, i_*} }^x y^{ - 1} \bwfa(y)  dy
=  \cJM(x) -
\cJM( \mmx_{ \eqref{def:x_i}, i_*} ) . \\
  \eal
\label{eq:J_main_bw}
\eeq

We apply the above derivations to obtain rigorous bound for $\cJ(\bwf)(x)$ 
with $x= X_{\eqref{def:XX},i}$. For $x \geq \xed$ 
beyond  $\supp(\bwp)$ \eqref{eq:W_rep}, since $\bw = \bwp$, using the above argument, we obtain
\begin{align}
 \cJ(\bwf)(x) &= \cJ(\bwf)( \xed )
+ \tts{\int}_{ \xed }^x y^{\alb - 1} \bwf(y) d y   \label{eq:J_main_bw2} \\
&  =
 \cJM(x)  + \cJ(\bwf)( \xed ) - \cJM( \xed ) + \cE_{ \eqref{eq:J_main_bw2}}
 , \quad | \cE_{ \eqref{eq:J_main_bw2} } |
\leq   \tf15 C_{ \eqref{eq:bwf_J_decomp2_err} }( \xed ) \xed^{-5} , \notag 
\end{align}
where $\cJM$ is defined in \eqref{def:J_main}. Note that in the error bound, we have replaced $\mmx_{i_*}$ in  \eqref{eq:bwf_J_decomp2_err}
by $\xed$. Since the proof is essentially the same, we omit the details.

\subsubsection{ Estimates for $I_{ \eqref{eq:K_log_decomp1}, i}$}\label{app:WF_log}

We estimate the most technical term $I_{ \eqref{eq:K_log_decomp1} , i}$. Using a change of variable $z = \f{y}{x}$ and the scaling \eqref{eq:ker_scale}:
 $K_{\alb, i, J}(x, y) y^{-1/3} d y= x^{2-i} K_{\alb,i,J}(1, z) z^{-1/3} d z$, 
 we obtain
 \beq
I_{ \eqref{eq:K_log_decomp1}, i}
= \int_{z_0}^{\infty} K_{\alb, i, J}(x, y) y^{-\f13 } \bwfa(y) d y
 =   x^{ 2 -i} ( \int_{xz \geq z_0}^{1} + \int_1^{\infty}) K_{\alb, i, J}(1, z)
z^{- \f13}  \bwfa(x z) d z .
\label{eq:K_log_decomp1:IM}
\eeq
From \eqref{def:z_0}, $\f{z_0}{x}$ is much smaller than $1$ for $x \geq z_{L_1}$. The difficult term 
is $|\log (x z)|^{-\g}, \g = \f13,\f23$ in $\bwfa$. We expand them below so that $\log x, \log z$ are 
decoupled. 

\vs{0.1in}
\paragraph{\bf Expand $ (\log(x z))^{ - \g} $} 

We fix $x$. For $xz \geq z_0$, we obtain $\log x + \log z \geq \log z_0 > 0$. 
Denote $F(w) = w^{-\g}$. Using Taylor expansion 
of $F(\log x + s)$ around $s=0$ 
and $(\pa^i F)(\log  x) 
= \AAF( -\g, i) (\log x )^{-\g - i} $ (recall $\AAF$ from \eqref{def:factor}), we obtain
\beq
 \f{1}{ (\log x + \log z )^{\g}} = \f{1}{ (\log x  )^{\g}} 
+  \sum\nolim_{ 1\leq i \leq k} \f{\AAF(-\g, i) \cdot (\log z)^i  }{i!  \cdot  ( \log x )^{\g +i}  } 
+   \f{\AAF(-\g, k+1)}{ (k +1)!}
 \f{(\log z)^{k+1}}{  (\log \xi_z)^{\g + k +1} } ,
 \label{eq:K_log_decomp2}
\eeq
for some $\xi_z$ between $ x,  x z$. 
Applying the above expansion to \eqref{eq:K_log_decomp1:IM}, 
and using $\bwfa(x)= \cca + \ccb (\log x)^{-1/3} + \ccc (\log x)^{-2/3}$ \eqref{def:bwfa}, we derive 
\bseq\label{eq:K_log_decomp2:full}
{\small
\begin{align}
 &  \f{ I_{ \eqref{eq:K_log_decomp1}, \s}}{  x^{ 2 - \s} }
  = \bwfa(x) \cdot (S_{\s, L, 0}( \f{z_0}{x})  + S_{\s, R,0})
 + 
 \sum_{ 1\leq i \leq k} \B( \f{ \ccb \AAF(- \f13, i)  }{i!  \cdot  ( \log x )^{ \f13 +i}  }  
 +  \f{ \ccc \AAF(-\f23, i)  }{i!  \cdot  ( \log x )^{ \f23 +i}  }   \B)
 \cdot (S_{\s, L, i}( \f{z_0}{x})  + S_{\s, R, i})\notag \\
 &  \ \quad +      \int_{ z_0/x}^{\infty} 
K_{\alb, \s, J}(1, z) z^{-1/3} (\log z)^{k+1}
( 
   \f{ \ccb \cdot \AAF(-1/3, k+1)}{ (k +1)!  \cdot  (\log \xi_z)^{ 1/3 + k +1} }
 +   \f{\ccc \cdot \AAF(-2/3, k+1)}{ (k +1)! \cdot (\log \xi_z)^{ 2/3 + k +1} }
) ,
 \end{align}
 }
where $S_{\s, L, i}(h)$ denotes the integral:
\beq
\bal
      S_{\s, L, i}(h) \teq  \tts{ \int_h^1} K_{\alb, \s, J}(1, z) \cdot z^{-1/3} (\log z)^{i}  d z , 
      \quad S_{\s, R, i} \teq  \tts{ \int_1^{\infty}} K_{\alb, \s, J}(1, z) \cdot z^{-1/3} (\log z)^{i}  d z,
\eal
\eeq
for $   \s  \in \{1, 2 \}$, where $L, R$ are short for \emph{left} and \emph{right}. Moreover, using the sign properties of $K_{\alb, i, J}$ in \eqref{eq:ker_sign}, we obtain
\beq
 \tts{\int_h^{\infty} } |K_{\alb, \s, J}(1, z) \cdot z^{-1/3} (\log z)^{i}| d z 
= |S_{\s, L, i}(h) |+ |S_{\s, R,i}|.
\eeq
 \eseq

Below,  we integrate $S_{\s, \bullet}(\cdot)$ show that 
the  error terms  are much smaller than $|\log x |^{-1/3 - k}$.

\subsubsection*{\bf Compute integrals}\label{int:log_expand}

Using the expansion \eqref{eq:K_log_decomp2}, we decouple $x, z$ in the integral \eqref{eq:K_log_decomp1:IM}. Since $x \geq z_{L_1}$ and  the domain of integral in \eqref{eq:K_log_decomp1:IM} depends on $x$, to obtain upper and lower bounds for the integral uniformly, we choose $s_1$  and define $f_p$ as
\[
s_1 > \max(  z_{L_1}^{-1/2}, \, z_0 / z_{L_1})
\Rightarrow s_1 \geq \max( x^{-1/2} , z_0 x^{-1} ) , 
\quad 
f_p(z) = z^{-1/3} (\log z)^p ,\quad p \geq 0 \in \mathbb{Z} .
\]
Then we  obtain
\beq
  \tts{\int}_{ z \geq z_{\mw{st}}  }^1 K_{\alb, i, J} f_p(z) d z 
  = \int_{s_1}^1  K_{\alb, i, J} f_p(z) d z + \cE_{\eqref{eq:K_log_decomp3} ,s_1},
  \  |\cE_{\eqref{eq:K_log_decomp3} ,s_1} | \leq \int_0^{s_1}  | K_{\alb, i, J} f_p(z)| d z , \
   \label{eq:K_log_decomp3}
\eeq
for $z_{\mw{st}} \in \{ x^{-1/2}, z_0 / x \}$,  where 
By computing the first integral (discussed below) and bounding the error, we obtain 
an enclosure for $\mw{LHS} \in [I_l, I_u]$ \emph{uniformly} in $x$.

We estimate the integral for $z \geq 1$ and $z \leq 1$ separately. We design mesh $1 = y_1 < y_2 < ..< y_m$ for $z \geq 1$ and $0 <s_1 < s_2< .. < s_n =1$ for $z \leq 1$. In each domain $[a, b] \subset [1, \infty)$ or $[a, b]  \subset [0,1]$, 
we use linear interpolation for $f_p(y) = y^{-1/3} (\log y)^p$ \eqref{eq:lin_interp} and decompose 
\[
 \int_a^b K_{\alb, i, J}(1, z) f_p(z) d z
  =  \int_a^b K_{\alb, i, J}(1, z) 
  ( f_p(a) - \f{f_p(b) - f_p(a)}{b-a} ( a- z)) d z
+\cE_{i, l, a, b} , \quad f_p =  y^{-1/3} (\log y)^p,
\] 
where $\cE_{i, l , a, b}$ denotes the error part and satisfies 
\[
  |\cE_{ i, l, a, b}| \leq \tf{(b-a)^2}{8}\| \pa_x^2 f_p \|_{L^{\infty}[a,b]}  | \textstyle{\int}_a^b K_{\alb,i,J} (1, z ) d z |  .
\]

In the above absolute integral, we do not take absolute value on $K_{\alb, i, J}(1, z)$ since it 
does not change sign on $z \in [1, \infty]$ and $[0, 1]$ 
by \eqref{eq:ker_sign}.  We use the explicit formula and \hyr[bd:intval]{\its Interval Arithmetic} to bound  $\pa_x^2 f_p$. 
The above $z$-integral in the main term is linear combinations of the form \eqref{eq:Bs_U_comp4:cI} 
with $l =0, 1$. We use the formulas for $\cI$ in \eqref{eq:Bs_U_comp4:cI}  to bound these integrals rigorously.

Using  $f_p(z) = z^{-1/3} (\log z)^p$ and \eqref{eq:deri_K_12} with $x=1, y=z, k=0$, 
and \eqref{eq:deri_K_J:y<x} with $ x= 1 , y = z  $, 
for any $a < 1/2 < 2 < b$, we bound 
\[
\bal
  \tts{\int}_{b}^{\infty} |K_{\alb , 1 J}(1, z) f_p(z)  |
 d z
& \leq C_{ \eqref{eq:deri_K_12}, i, b, 0}  \tts{\int}_b^{\infty} z^{\alb  - 3} f_p(z) 
=  C_{ \eqref{eq:deri_K_12}, i, b, 0}  \tts{\int}_b^{\infty} z^{  - 3} (\log z )^p d z ,  \\
 \tts{\int}_{0}^{a} |K_{\alb , i, J}(1, z) f_p(z)  |
& \leq 
C_{ \eqref{eq:deri_K_J:y<x}, i, a } \tts{\int}_0^a z |f_p(z)| d z 
  \leq  C_{ \eqref{eq:deri_K_J:y<x}, i, a } \tts{\int}_0^a  z^{2/3} |\log z|^p d z  
  =  C_{ \eqref{eq:deri_K_J:y<x}, i, a } \tts{\int}_{1/a}^{\infty}  z^{-8/3} |\log z|^p d z 
\eal
\]
where in the last identity, we apply a change of variable $z \to z^{-1}$. 
We use \hyr[int:Method2]{\its Method B.2} to bound these integrals and the error 
$\cE_{\eqref{eq:K_log_decomp3}, s_1}$ in \eqref{eq:K_log_decomp3}.

\vs{0.05in}
\paragraph{\bf Bounding Taylor remainder }

Recall the expansion 
and the  remainder from \eqref{eq:K_log_decomp2:full}. 
Denote 
\beq\label{def:log_f_1}
h_{k+ \g}(\xi) = (\log \xi)^{-\g - k-1}
\eeq
which is decreasing in $\xi$ for $k + \g \geq 0$. Since $\xi_z$ is between $x, xz$, to obtain a 
sharper estimate, we partition the range of $z$ according to
\[
(y_5(x), y_4(x), y_3(x), y_2(x), y_1(x), y_0) = ( z_0 x^{-1}, x^{-15/16},x^{-7/8}  , x^{-3/4}, 
x^{-1/2}, \infty ) 
\]
with $y_{i+1}(x) < y_i(x)$ for $x > z_{L_1}$. Then we obtain
\[
\bal
 h_{k+ \g}( \xi_z)  & \leq h_{k+ \g}(\min(x, xz)) \leq \tts{ \sum_{ 0 \leq i \leq 4}} \one_{z \in [y_{i+1}, y_i ] } h_{k+ \g}(  x y_{i+1}(x) ), \\
 \quad h_{k+ \g}( x y_1(x)) &= h_{k+ \g}(x^{1/2}) = (\tf{1}{2} \log x )^{-\g - k-1}.
 \eal
\]

We choose $y_i = x^{-a}$  with some $ a \in (0, 1)$ in the above partition because
 $\log (x y_i) = (1-a) \log x \asymp \log x$ 
 and the upper bound $h_{k+ \g}( x y_i)$ is not very large. 
Since $h_{k+ \g}(x y_i(x)), y_i(x)$ are decreasing in $x \geq z_{L_1}$, using the above partition
and \eqref{eq:deri_K_J:y<x} for $z \leq x^{-1/2} \leq z_{L_1}^{-1/2}$, we bound the integral of remainder as
\beq
\bal
 & |\int_{xz\geq z_0}  K_{\alb, i, J}(1, z)   (\log z)^{k+1} h_{k+ \g}(\xi_z)|
 \leq 
 (\tf{1}{2} \log x )^{-\g - k-1}
 \int_{z  \geq x^{-1/2}} | K_{\alb, i, J}(1, z)  (\log z)^{k+1} | d z  \\
& \quad + \sum_{ 1\leq i \leq 4}
h_{k+ \g}( z_{L_1 } y_{i+1}( z_{L_1})  ) 
 C_{ \eqref{eq:deri_K_J:y<x}, i, z_{L_1}^{1/2} } 
\int_0^{ y_i( x ) } z^{-1/3 +1}  |\log z|^{k+1} d z 
\teq \sum_{0\leq i\leq 4} II_{ \eqref{eq:TL_remain1} ,i },
\label{eq:TL_remain1}
\eal 
\eeq
where we recall $z_{L_1}$ from \eqref{def:z_0}. 
Since $x^{-1/2} \leq z_{L_1}^{-1/2}$,  we bound the first $z$-integral 
using the decomposition \eqref{eq:K_log_decomp3} and 
the methods in \hyr[int:log_expand]{\its Compute the integrals}.

For the $z$-integral in the summation, since $k$ is a positive integer, 
for any $b>0$ and integer $n \geq 0$, we bound the 
indefinite integral analytically using integration by parts and $z^b = \f{1}{b+1} \cdot \f{d}{dz} z^{b+1}$
\[
|\tts{\int_0^a} z^{b} (\log z)^n d z  |
\leq |\f{1}{1+b} a^{b+1} (\log a)^n  | + \f{n}{1 + b} | \int_0^a z^b (\log z)^{n-1} d z | ,
\]
which leads to an upper bound 
\beq\label{eq:TL_rem}
|\tts{\int}_0^{ y_i( x ) } z^{-1/3 +1}  |\log z|^{k+1} d z |
\leq y_i(x)^{5/3} ( \sum_{j \leq k+1} c_{i, j}  |\log y_i(x)|^{j} ) ,
\quad y_i(x) \leq x^{-1/2}, \quad y_i(x)^{5/3} \leq  x^{-5/6}, 
\eeq
for $ 1\leq i \leq 4$. Thus, the $z$-integral has a decay rate at least $x^{-5/6} |\log x |^{k+1}$,
which is much faster than the main term in \eqref{eq:TL_remain1} and 
\eqref{eq:K_log_decomp2:full}.

\subsubsection{Nonlocal estimates for $\bwf, \bw$ with $x \geq \xed$ }\label{app:nloc_very_far}

Recall $V, \cK_{\alb,i}$ from \eqref{eq:ker}. For $x \geq \xed$ out of $\supp(\bwp)$, we refine the estimate for $V_x(\bw) - \tf{V}{x}(\bw)-4$ to cancel the next order error, as motivated in Section \ref{sec:ansatz}. 
Using $\bw = \bwp + \bwf$ \eqref{eq:W_rep}, $ \bwf = \chi_1 x^{-1/3} \bwfa $ \eqref{def:bwfa}, 
we decompose 
\[
\bal
V_x(\bw) - \tf{V}{x}(\bw)-4 & = \cK_{\alb, \D}(\bw)-4=   \cK_{\alb, \D}(\bwf)-4 + \cK_{\alb,\D}(\bwp)  \\
 &  = \cK_{\alb , \D}( \one_{y \geq z_0} y^{-1/3} \bwfa) - 4 +   
\cK_{\alb, \D}( ( \chi_1- \one_{y \geq z_0} )  y^{-1/3} \bwfa ) + \cK_{\alb, \D}(\bwp) .
\eal
\]

We estimate  $\cK_{\alb, \D}(\bwp) , \cK_{\alb, \D}( ( \chi_1- \one_{y \geq z_0} )  y^{-1/3} \bwfa ) $ using estimate \eqref{eq:cK_WP_far} and \eqref{eq:K_log_err:far}, 
which has a decay rate $x^{\alb-0.4}$.  For the main term $ \cK_{\alb , \D}( \one_{y \geq z_0} y^{-1/3} \bwfa) - 4 $, since $\cK_{\alb, \D}(x) = \cK_{\alb, 2, J}(x) - \tf{1}{x} \cK_{\alb, 1, J}(x) $ by \eqref{eq:ker:c}, using the expansion 
\eqref{eq:K_log_decomp2:full} for $I_{ \eqref{eq:K_log_decomp1}, i}$, and linearity, we obtain
\beq
  \cK_{\alb , \D}( \one_{y \geq z_0} y^{-1/3} \bwfa)-4 
  = I_{ \eqref{eq:K_log_decomp1}, 2} - \tf{1}{x} I_{ \eqref{eq:K_log_decomp1}, 1} - 4
=  \bwfa \int_{ z_0 / x}^{\infty}  K_{\alb, \D}(1 ,z) z^{-1/3} d z -4 +
\cE_{\eqref{eq:K_ux_min_u:expand}} ,
\label{eq:K_ux_min_u:expand}
\eeq
with the error term $\cE_{\eqref{eq:K_ux_min_u:expand}}$ 
bounded by
\bseq\label{eq:K_ux_min_u:err}
{\small
\begin{align}
 & |\cE_{\eqref{eq:K_ux_min_u:expand}} |
\leq  \sum_{ 1\leq i \leq k} \B( \f{ | \ccb \AAF(- \f13, i) |  }{i!  \cdot  ( \log x )^{ \f13 +i}  }  
 +  \f{ | \ccc \AAF(-\f23, i) |  }{i!  \cdot  ( \log x )^{ \f23 +i}  }   \B)
 \cdot  | S_{2,L, i}( \f{z_0}{x}) + S_{2, R, i}  - S_{1, L,i}( \f{z_0}{x} ) - S_{1, R,i} |
 \notag 
\\
    &+      \int_{ z_0/x}^{\infty} 
\B| K_{\alb, \D}(1, z) z^{-1/3} (\log z)^{k+1}\B| \cdot
\B|  \f{ \ccb \cdot \AAF(-1/3, k+1)}{ (k +1)!  \cdot  (\log \xi_z)^{ 1/3 + k +1} }
 +   \f{\ccc \cdot \AAF(-2/3, k+1)}{ (k +1)! \cdot (\log \xi_z)^{ 2/3 + k +1} } \B| dz ,
 \label{eq:K_ux_min_u:err:a} 
\end{align}
}
where $S_{\s,\bullet, i}(\cdot)$ is defined in \eqref{eq:K_log_decomp2:full}. We have estimated all the $z$-integrals in Appendix \ref{app:WF_log}. Since all the bounds decay faster than $|\log x|^{-4/3}$, using monotonicity
\footnote{
The integrals for the error terms in \eqref{eq:K_ux_min_u:err:a} are bounded by a function $F_{\cE}(x)$, where $F_{\cE}(x)$ is a finite linear combination of the term $(\log x)^{-\g-k-1}$ appearing in \eqref{eq:TL_remain1} and terms of the form $x^{-a}(\log x)^i$ arising from \eqref{eq:TL_rem}, with $0\leq i\leq k+1$ and $a\geq \frac{5}{6}$. Moreover, all the coefficients in the upper bounds are positive. We apply the expansion \eqref{eq:K_log_decomp2:full} with $k\leq 20$. 
From \eqref{eq:poly_log_mono}, since $x^{-a} (\log x)^{ k+1 + 4/3 }$ 
is decreasing in $x$ for $a \geq 0.01$ and $x \geq 10^{10} \geq \max( \f{k+3}{ a}, 1)$, 
we easily obtain $F_{\cE}(x) \leq F_{\cE}( L) (\log L)^{4/3}  (\log x)^{-4/3} $
for any $x \geq L \geq 10^{10}$.
}
we further bound 
\beq
|\cE_{\eqref{eq:K_ux_min_u:expand}} | \leq C_{\eqref{eq:K_ux_min_u:err}} |\log x|^{-4/3} .
\eeq
\eseq

For the main term in \eqref{eq:K_ux_min_u:expand}, using the key identity \eqref{eq:asym_int} and \eqref{eq:deri_K_J:y<x}, 
for $x \geq \xed$, we obtain
\bseq\label{eq:K_ux_min_u:main1}
\beq
\tts{\int}_{ z_0 / x}^{\infty}  K_{\alb, \D}(1 ,z) z^{-1/3} d z
 = -\f{2}{3} -  \tts{\int}_{ 0}^{ z_0 / x }  K_{\alb, \D}(1 ,z) z^{-1/3} d z
 \teq -\f{2}{3}+ \cE_{ \eqref{eq:K_ux_min_u:main1} } ,
\eeq
with 
\beq
|\cE_{ \eqref{eq:K_ux_min_u:main1} } | \leq C_{\eqref{eq:deri_K_J:y<x}, \D, \f{z_0}{  \xed } } \tts{\int}_0^{z_0 / x} z^{2/3} d z
 = C_{\eqref{eq:deri_K_J:y<x}, \D, \f{x_0} { \xed }} \f{3}{5}  z_0^{5/3} x^{-5/3} .
\eeq
\eseq

Recall $\bwfa = \cca + \ccb (\log x)^{-1/3} + \ccc (\log x)^{-2/3}$ from \eqref{def:bwfa}
and $\cca <0, \ccb, \ccc >0$. Since $ \cca \cdot (-\f23) - 4 = 0$, combining the above estimates, we prove 
\bseq\label{eq:K_ux_min_u:main}
\beq
V_x(\bw) - \tf{V}{x}(\bw)-4 
= \cK_{\alb, \D}(\bwf)-4 + \cK_{\alb,\D}(\bwp)   = -\tf23 \ccb (\log x)^{-1/3} - \tf23 \ccc (\log x)^{-2/3} + \cE_{
\eqref{eq:K_ux_min_u:main}, V_x, V}(x)
\eeq
with the error part satisfying 
\beq
\bal
  |\cE_{\eqref{eq:K_ux_min_u:main}, V_x, V}(x)| 
  & \leq \max( |\cca|, | \bwfa( \xed )|) 
  \cdot  C_{\eqref{eq:deri_K_J:y<x}, \D, \f{x_0} { \xed }} \tf{3}{5}  z_0^{5/3} x^{-5/3}
  + C_{\eqref{eq:K_ux_min_u:err}} |\log x|^{-4/3}  \\
 & \quad   + 
| \cK_{\alb, \D}( (\one_{y \geq z_0} -\chi_1)  y^{-1/3} \bwfa )| 
 + | \cK_{\alb, \D}(\bwp)| \les  (\lgp  x )^{-4/3}
\eal
\eeq
for any $x \geq \xed$. We prove the first estimate in \eqref{eq:lin_nloc:far}.
\eseq

Similarly, for $x \geq \xed$, by bounding $\cK_{\alb, 1, J](\bw)}, \cK_{\alb, \D}(\bw)$ 
using \eqref{eq:vel_est:b} in Lemma \ref{lem:vel_est} with $ \bb = \alb, \mu = 1$, we bound 
\beq
\bal
  |\bv_x - \tf{1}{x} \bv|  & =  |\cK_{\alb, \D}(\bw)| \leq \mCC_{\D,0 }( \alb, \R_+) \nlinf{\bw y^{1/3}}  
 \teq C_{ \eqref{eq:V_mix:far}, \bv_{\D} } , \\
 |x + \cK_{\alb, 1, J(\bw)}| & \leq  (1 + \mCC_{\psio/x, J}( \alb, \R_+) \nlinf{\bw y^{1/3}} )  x
 \teq C_{ \eqref{eq:V_mix:far} , \bv }  \cdot x,   
 \eal
 \label{eq:V_mix:far}
\eeq
where $\mCC_{\bullet}$ is the constant defined in \eqref{eq:vel_const}. Using \eqref{eq:cK_WP_far}
and \eqref{eq:J_main_bw2}, for $x \geq \xed$, we prove 
\bseq \label{eq:J_main_far}
\beq
\bal
 \cJ(\bw)(x) & = \cJ(\bw)( \xed )
 +  \cJ(\bwf)( x) - \cJ(\bwf)( \xed )  
  = \cJM(x) 
  + \cE_{ \eqref{eq:J_main_far}, \cJM} ,  \\
 \cE_{ \eqref{eq:J_main_far}, \cJM} &=  \cJ(\bw)( \xed ) - \cJM( \xed )
 +  \cE_{ \eqref{eq:J_main_bw2}}
 , \quad | \cE_{ \eqref{eq:J_main_bw2} } |
\leq   \tf15 C_{ \eqref{eq:bwf_J_decomp2_err} }( \xed ) \xed^{-5} ,
 \eal
\eeq
where $ \cJM$ is the main term in \eqref{def:J_main} 
\beq
\cJM(x)  = 
   \cca  \log x + \tf{3}{2} \ccb  (\log x)^{2/3} + 3 \ccc (\log x)^{1/3}.
\eeq
\eseq
Since $\bv =- 2 \alb x \cJ(\bw)  + x + \cK_{\alb, 1, J}(\bw)$, 
$\cca = -6$ \eqref{def:WF_para}, combining the above estimates, 
and extracting the main term $-2\alb \cca \log x = 4 \log x$ from $-2\alb \cJM$ \eqref{def:J_main}, we prove 
\begin{align}
 & \bv    =  ( - 2 \alb \cJM + \cE_{ \eqref{eq:V_main:far}, V/x } ) x , 
 \ 
 |\cE_{ \eqref{eq:V_main:far}, V}| \leq  C_{ \eqref{eq:V_mix:far} , \bv/x } + 2 \alb | \cE_{ \eqref{eq:J_main_far}, \cJM} | \leq C_{  \eqref{eq:V_main:far} , V/x} ,   \notag  \\
 & \bv    =  ( 4 + \cE_{ \eqref{eq:V_main:far}, \log} ) x   \log x  , \label{eq:V_main:far} \\
 & |\cE_{ \eqref{eq:V_main:far}, \log }| 
  \leq   \tf32 |\ccb|  (\log \xed )^{-\tf13} + 3 |\ccc| (\log \xed)^{-\tf23} 
 +  C_{  \eqref{eq:V_main:far} ,V/x} (\log \xed )^{-1} 
 \teq C_{  \eqref{eq:V_main:far} ,\log} \notag
 \end{align}

We verify $ 4 + C_{  \eqref{eq:V_main:far}, \log} > 0$ so that $\bv \gtr x \log x$.

\section{Bounding the explicit integrals for the fixed point argument}\label{app:proof}

In this Appendix, we combine the estimates in Appendix \ref{app:numerics} to estimate the 
$B_{\bullet}$ functions 
\begin{align}
   \uds{ \eqref{eq:est_I} } {B_I }  & = 
    \f{ \MCC_{\psio/x} ( x\we x_1 ) }{3  |\bar g ( x\we x_1 )| }  + 
 \int_0^{x_1 \we x}  \MCC_{\psio/x}(y) \B|   \big( \tf13 \Vmix + \Wmix \big)  \f{1}{\bv} \B| d y , \notag \\
\uds{ \eqref{eq:est_II1}} { B_{II, 1} } & = | \f{\rho  \MCC_{\psio/x, J} }{3 \bar g }(x) | 
  +
| \f{\rho  \MCC_{\psio/x, J} }{3 \bar g}(x_1) | 
+ \f13 \int_{x_1}^x \MCC_{\psio/x, J}(y) \B| \f{\rho_y \bar g - \rho \bar g_y}{  \bar g^2}  \B| d y , \notag  \\
 \uds{ \eqref{eq:est_II3} }{ B_{II,3 } }  
  & =  \int_{x_1}^x \B| \f{ \MCC_{\psio/x}(y) \cdot \Wmix(y)  }{ \bar V(y)} \B|  \rho(y) d y ,  \quad \uds{ \eqref{eq:est_II_III} } { B_{II,2, III} } 
  = 
 \int_{x_1}^x \B| 
     \f{\pa_y \rho}{\rho} + \f{2 \alb y^{\alb} \bw}{3 \bv} \B| 
     \cdot \min_i( \mu_i |y|^{-b_i}) \f{1}{|\bw|} d y  ,  \notag  
\end{align}
\begin{align}
\uds{ \eqref{eq:est_err} }{ B_{\cE}(x, \dfix ) } 
&= \int_0^x \f{|\cR(\bw)|  \cdot  (\rho + \vp^{-1} |\bw|^{-1} \dfix ) }{ 2 (\bv(y) - \MCC_{\psio/x}(y) y  \dfix )_+ } , \quad 
\uds{ \eqref{eq:est_non}}{ B_{\cN}(x, \dfix ) }     =  
 \int_0^x  B_{\cR} 
\cdot 
\f{  \vp^{-1} |\bw|^{-1}   + \bar g^{-1} \MCC_{\psio/x}(y) \rho  }{ 2 ( \bv(y) - \MCC_{\psio/x}(y) y 
\dfix )_+ }   , 
 \notag \\
  \uds{ \eqref{eq:est_cR:b} }{ B_{\cR}(x) }
 & = \min\B( \f{2}{3} \MCC_{\psio_x}( x) 
 + 2 \B| x \f{\bar W_x}{ \bar W } \B| \MCC_{\psio/x}(  x ),
 2 |  \Wmix(x)  | \MCC_{\psio_x}(  x) 
 + 2 \B| x \f{\bar W_x}{ \bar W } \B| \MCC_{\D}( x ) \B) ,
  \label{eq:cB_pf1}
\end{align}
in Lemma \ref{lem:close}, where we have used the following quantities 
to rewrite $B_I, B_{II,3}, B_{\cR}$ equivalently :
\beq\label{def:mix}
\Wmix(y) \teq  \tf{ y \pa_y \bw(y)}{\bw(y)} + \tf13,
\quad \Vmix(y) \teq \tf{ y \pa_y \bv(y)}{\bv(y) } - 1,
\quad  \tf13 \Vmix + \Wmix  = \tf{ y \pa_y \bv(y)}{3 \bv(y) } + \tf{ y \pa_y \bw(y) }{\bw(y)}.
\eeq
We use  $\Wmix , \Vmix$ to exploit cancellation and reduce round-off error. All the functions in the integral evaluate at $y$,  $\bar g(y) = \f{\bar V(y)}{y}$ \eqref{eq:nota1}, and $\MCC_{\bullet}(x)$ denotes $\MCC_{\bullet}(\vmu,\bb,\cI,x)$ for 
$\bullet \in \{ \psio/x,\psio_x,\D, (\psio/x, J), (\psio_x, J) \}$ in Lemma~\ref{lem:vel_est} for nonlocal estimates.

Note that we have substituted $\dfix$ for $\| w \|_{\cX}$ in the functions $B_{\cE}(x, \| w \|_{\cX} ), B_{\cN}(x, \| w \|_{\cX}  )$.

We recall the profiles $\bv, \bw$ from \eqref{eq:W_rep}, residual error $\cR(\bw)$ \eqref{eq:lin_nloc}, weights $\vp$ \eqref{norm:X}, $\rho$ \eqref{eq:rho1}.

\subsection{Estimates for $x \leq \xed$ }\label{app:pf_bound}

Recall $\bw = \bwf + \bwp$ from \eqref{eq:W_rep}, $V(\bw) = x + \psio(\bw) = x + \cK_{\alb,1}(\bw)$ from \eqref{eq:ker}
and the refined mesh $X_i$ \eqref{def:XX}. 
We summarize the evaluations and estimates of these functions in bounded domain $[0, \xed ]$.

\begin{itemize}[leftmargin=1em]
  \item 
  Evaluate $\pa_x^i \bwp, \pa_x^i \bwf, \pa_x^i \bw, i \leq 5$   on $x = X_i$ using \eqref{eq:W_rep}; bound $\pa_x^i \bwp, \pa_x^i \bwf, \pa_x^i \bw, i \leq 5$ using \hyr[bd:Method1]{\its Method A.1}, \hyr[bd:Method2]{\its A.2}, \hyr[bd:Method3]{\its A.3} and \hyr[bd:intval]{\its Interval Arithmetic}.

  \item  $\bv$:  Evaluate $\cK_{\alb, i}(\bwp)(x)$ following Appendix \ref{app:nonlocal} and $\cK_{\alb, i}(\bwf)(x)$ 
following Appendices \ref{app:bwf_near}, \ref{app:bwf_far}  on the mesh $x= X_i$ \eqref{def:XX},  with rigorous error estimates. 
Obtain explicit  bounds for $V_{xxx}(\bw) = \pa_x^3 \cK_{\alb, 1}(\bw)(x)$ 
for all $x \in \R_+$ following Appendix \ref{app:Vxxx}. 
Then using \hyr[bd:interp]{\its Interpolation Estimates}, we obtain rigorous piecewise bounds 
for $\bv, \pa_x \bv$. For $\pa_{xx} V(\bw) = \pa_{xx} \psio$, we use 
\[
  |\pa_{xx} V(x) - \tf{1}{b-a} ( \pa_y V(b) - \pa_y V(a) ) |
  =  |\pa_{xx} V(x) - \tf{1}{b-a} \tts{\int}_a^b \pa_y^2 V(y) d y|
\leq \| \pa_x^3 V \|_{L^{\infty}[a, b]} \f{b-a}{2}
\]
and the values of $\pa_x \bv$ at $X_i, X_{i+1}$,  
to bound $\pa_{xx} \bv(x)$ on $[X_i, X_{i+1}]$. 
Using grid point values and \hyr[bd:interp]{\its Interpolation Estimates}, 
we derive sharp piecewise bounds for $\pa_x \bar g= \f{\bv_x -  \bv / x }{x},
\f{x \pa_x \bw}{\bw} +  \f{1}{3} , \f{x \pa_x \bv}{\bv} - 1$.

Recall
\beq\label{eq:cR_recall}
\cR(\bw) =  3 - \alb - (1-\alb) \bv_x - 2 \bv \tf{\pa_x \bw}{ \bw},
\quad \bw \cR(\bw) = \bw(   3 - \alb - (1-\alb) \bv_x ) -2 \bv \pa_x \bw .
\eeq

\item For $\cR(\bw)$ in \eqref{eq:cR_recall}, we evaluate $\cR(\bw), \bw \cR(\bw)$ on mesh $x = X_i$ using values for $V, V_x, W$ on $x=X_i$ and Leibniz rule \eqref{eq:Leibiz}.  Obtain piecewise bound for $\pa_x^2 (\bw \cR(\bw))$ using Leibniz rule \eqref{eq:Leibiz},  \hyr[bd:Method1]{\its Method A.3}, and bounds for $\pa_x^i \bv, \pa_x^i \bw$ with $i\leq 3$. 
Using 
\hyr[bd:interp]{\its Interpolation Estimates}, \hyr[bd:intval]{\its Interval Arithmetic} and the above bounds for $\bw$, we obtain tight bound for $\bw \bar \cR(\bw)$ and $\bar \cR(\bw)$.

\item Obtain piecewise tight bounds for explicit weights $\vp, \rho$ using \hyr[bd:intval]{\its Interval Arithmetic}. 

\item Bound constants $\mCC_{\bullet}( \bb, \cI)$ \eqref{eq:vel_const} with parameters $\bb$ in \eqref{def:mu_b} and partition $\cI$ in \eqref{def:interval}  following Appendix \ref{app:sharp_constant}. Assemble functions $\MCC_{\bullet}(\vmu, \bb, \cI, x)$ as in \eqref{eq:vel_const}. 
Evaluate piecewise tight bounds for $\MCC_{\bullet}(\vmu, \bb, \cI, x)$ using \hyr[bd:intval]{\its Interval Arithmetic}.
\end{itemize}

\paragraph{\bf Removable singularity at $x=0$} 
The integrands in $ B_I$, $ B_{\cN}(x, \dfix)$, $B_{\cE}(x, \dfix )$ \eqref{eq:cB_pf1}  have a removable singularity 
$\f{\MCC_{\psio/x}}{\bv}$ at $x=0$. Since  $\bb_1<-1$ \eqref{def:mu_b}, 
by definition \eqref{eq:J_est} and using \eqref{eq:vel_const}, 
\[
\MCC_{\bullet}(\vmu, \bb, \cI, x) \leq \MCC_{\bullet, 0}(\vmu, \bb, \cI, x) 
\leq  \mu_1 \mCC_{ \bullet, 0}( \bb_1, \R_+ ) |x|^{\alb - \bb_1} ,
\quad \bullet \in \{ \psio/x, \psio_x \}
\]

Thus $\f{\MCC_{\bullet}}{x}$ is bounded near $x=0, \bullet \in \{ \psio/x, \psio_x \}$. 
Similarly, using $ \cR(\bw)(0) = 0$ due to $\pa_x \bv(0)=1$ (see Remark \ref{rem:wx0})
and $|\pa_x^2 \cR(\bw)| \les 1$,  we bound $ |\tf{1}{x} \cR(\bw)| \les 1$ near $x=0$. 
As a result, we  bound the functions $B_{\bullet}  \les C_{1, \bullet} |x|^{\alb +1}$
for $\bullet \in \{ I, \cN, \cE \}$ with some explicit constant $C_{1, \bullet}$.  
From \eqref{norm:X} and \eqref{def:mu_b}, since the coefficient in \eqref{eq:profile:b} satisfies 
$\vp |\bw| \les |x|^{1 + \bb_1} \les |x|^{-1}$, we obtain bounded functions 
$ \vp |\bw| B_{\bullet}, \bullet \in \{ I, \cN, \cE \}$, and hence bounded $B_{\mf{tot}}$ \eqref{eq:profile:b}  near $x=0$.

Using the above piecewise bounds, we obtain piecewise upper/lower bounds for integrands and boundary terms in 
\eqref{eq:cB_pf1} on  $[X_i, X_{i+1}] ,\forall 1 \leq i \leq \nnx-1$. Applying \hyr[int:Method1]{\its Method B.1}, 
we bound the integrals associated with \eqref{eq:cB_pf1} and the functions in \eqref{eq:cB_pf1}
for $x \in [X_i, X_{i+1}], 1\leq i \leq \nnX-1$.

\subsection{Far-field estimates for $x \geq \xed$ }\label{app:cR_far}

For $x \geq \xed$, we decompose the integral in \eqref{eq:cB_pf1} as
\[
\tts{  \int}_a^x f(y) d y = \int_a^{ \xed } f(y) d y+ \int_{ \xed }^x  f(y) dy = I +II.
\]
We estimate $I$ following Appendix \ref{app:pf_bound}. For $II$, we derive explciit bounds 
\beq\label{eq:pf_far_asym1}
 |f(y)| \leq  \tts{\sum}_{ 1\leq i \leq l } c_i x^{-1}  (\log x)^{-1- i/3 } ,
\quad \int_{\xed }^{x} x^{-1}  (\log x)^{-1- p} \leq \tf{1}{p}  (\log \xed )^{-p}, \ \forall \ p >0,
\eeq
with finite many terms and explicit constant $c_i$.

It remains to estimate the asymptotics for the functions in \eqref{eq:cB_pf1}.
Recall $\Wmix$ from \eqref{def:mix}. For $ x \geq \xed$, from \eqref{eq:W_rep}, 
we get 
$\bw = \bwf = \chi_1 x^{-1/3} \bwfa$ 
\eqref{eq:W_rep}, \eqref{def:bwfa}, we bound 
\beq\label{eq:W_mix_asym}
 \Wmix = \f13 + \f{x \pa_x \bw}{\bw}
 =   \f{ x \pa_x \bwfa   }{\bwfa  }+ \f{x\pa_x \chi_1}{\chi_1} 
= \f{- \tf{1}{3} \ccb (\log x)^{-4/3} - \f{2}{3} \ccc (\log x)^{-5/3}  }{ \bwfa}  + \f{x\pa_x \chi_1}{\chi_1}  .
\eeq

We estimate the first term on RHS using \eqref{eq:bwfa_far}. 
Since $x^{-5} |\log x|^{4/3}$ is decreasing for $x \geq \xed > 10^{20}$ by \eqref{eq:poly_log_mono}, \eqref{def:XX}, using \eqref{eq:bound_chi_tail}, \eqref{eq:deri_WF_far_k0}, 
and the formula \eqref{eq:W_rep} for $\chi_0$, we bound 
\beq
\bal
| \f{x\pa_x \chi_1}{\chi_1}|
\leq 
 \f{  C_{ \eqref{eq:bound_chi_tail},  \chi_0, k, z_0, \xed }  x^{-5} }{ \chi_1(\xed) }
 & \leq  \f{  C_{ \eqref{eq:bound_chi_tail},  \chi_0, k, z_0, \xed }  \xed^{-5} |\log \xed |^{4/3}  }{ \chi_1(\xed) } |\lgp x |^{-4/3}, \\
 (1 - \chi_1)(x) & \leq \tf{ \xed^5 z_0^5}{ ( \xed -z_0)^5 } x^{-5} .
  \label{eq:chi_asym_rat}
 \eal
\eeq

Using these estimates and \eqref{eq:W_mix_asym}, for $x \geq \xed$, we obtain
\beq
  |\Wmix| = |\tf13 + \tf{x \pa_x \bw}{\bw}| 
= | \tf{ x \pa_x ( x^{1/3} \bw)}{ x^{1/3} \bw } |
  \leq C_{ \eqref{eq:WW_mix_asym}}  (\log x )^{-4/3},
  \quad   | \tf{x \pa_x \bw}{\bw}| \leq \tf13 + C_{ \eqref{eq:WW_mix_asym}}  (\log \xed )^{-4/3},
  \label{eq:WW_mix_asym}
\eeq
with an explicit constant $C_{ \eqref{eq:WW_mix_asym}}$. The asymptotics of other terms are estimated similarly to \eqref{eq:chi_asym_rat}.

Below, we summarize the asymptotics estimates for functions in the formulas \eqref{eq:cB_pf1} for $x \geq \xed$.

\begin{itemize}[leftmargin=1.5em]

\item $\bw=\bwf $:
Since 
$ \bwf = \chi_1 y^{-1/3} \bwfa $, 
using 
\eqref{eq:deri_WF_far} and following Appendix \ref{app:profile_far_field_bound}, we estimate $\bwf = C x^{-1/3}$ with a tight bound for $C \in [C_l ,C_u]$. 
We also use  \eqref{eq:WW_mix_asym}.

\item $\bv$: Recall $\bar g = \tf{1}{x} \bv$. We use \eqref{eq:V_main:far} and \eqref{eq:V_mix:far} to get 
\[
\bv \in [4 -  C_{  \eqref{eq:V_main:far} ,\log}, 4 +  C_{  \eqref{eq:V_main:far} ,\log}] \cdot x \log x, \
| \pa_x \bar g |= |\tf{1}{x} ( \pa_x \bv - \tf{1}{x} \bv)|
\leq C_{ \eqref{eq:V_mix:far},\bv_{\D} } x^{-1}.
\] 

\item Weight $\vp$ \eqref{norm:X}: Since $\vpa |x|^{-\alb}$ is decreasing in $x$, 
we get $\vpa(x) \in [ \vpa( \xed )| \xed |^{-\alb}, \mu_\nmu^{-1} ] x^{\alb}$.
For $\rho$ \eqref{eq:rho1}, since $ \xed > x_1$ by \eqref{def:x1}, we get
\beq\label{eq:expand_rho}
 \rho(x) = \log(x_1)^{-1/3} (\log x)^{1/3}, \quad \rho_x = \tf13 x^{-1} \log(x_1)^{-1/3} (\log x)^{-2/3}.
\eeq

\item Nonlocal estimates:
We get
$\MCC_{ \bullet }( \vmu, \bb, \cI, x) \leq \mu_\nmu \mCC_{\bullet}( \alb ,\R_+) ,  \bullet \in \{ (\psio/x, J), ( \psio_x, J),
(\D, 0).$  by \eqref{eq:vel_const} \\
For $\MCC_J( \vmu, \bb, x)$ in \eqref{eq:J_est} and $x \geq \xed $, 
since $\bb_5 =\alb$ by \eqref{def:mu_b}, we have 
\beq
\bal
 \MCC_J(\vmu, \bb, x) & \leq  \MCC_J(\vmu, \bb, \xed ) + (\log x_1)^{1/3} \tts{\int}_{ \xed }^x 
 \vmu_{\nmu} y^{-1} 
 (\log y)^{-1/3} d y \\
  = &   \MCC_J(\vmu, \bb, \xed ) + 
   \vmu_{\nmu}  \tf32 (\log x_1)^{1/3}   ( (\log x)^{2/3}- (\log \xed )^{2/3} )
  \leq  C_{ \eqref{eq:J_e_asym} } (\log x)^{2/3}.
 \eal
 \label{eq:J_e_asym}
\eeq
The above function is linear in $(\log x)^{2/3}$. We derive $ C_{ \eqref{eq:J_e_asym} }$ 
using monotonicity. Combining the above two bounds,
and using definition \eqref{eq:J_est}, we derive 
\beq
 \MCC_{ \bullet }(\vmu, \bb, \cI, x) \leq  (2 \alb C_{ \eqref{eq:J_e_asym}} 
+ \MCC_{ \bullet, J} ( \log \xed )^{-2/3} )  (\log x)^{2/3} 
\teq C_{ \eqref{eq:vel_const_asym_psi}, \bullet }  (\log x)^{2/3} ,
\label{eq:vel_const_asym_psi}
\eeq
for  $\bullet \in \{ \psio/x, \psio_x\}$. Combining \eqref{eq:V_main:far} and the above estimate, we bound 
\beq
  (\bv(y) - \MCC_{\psio/x}(y) y \dfix)_+
  \geq ( 4 -  C_{  \eqref{eq:V_main:far} } - 
\dfix    C_{ \eqref{eq:vel_const_asym_psi}, \psio/x }  \log( \xed )^{-1/3} ) x \log x 
\teq C_{ \eqref{eq:V_l_asym} }  x \log x .
\label{eq:V_l_asym}
\eeq

\end{itemize}

Using the above estimates, we bound the asymptotics for integrands and boundary terms in 
\eqref{eq:cB_pf1} using triangle inequality except for the following two terms. Note that $B_I(x) = 0$
\eqref{eq:cB_pf1}
 for $x \geq \xed$. 

\subsubsection{Cancellations in integrands} 

We need to exploit the crucial cancellations for  $\f{\pa_x \rho}{\rho} + \f{2 \alb y^{\alb} \bw}{3 \bv}$ in the integrands for $B_{II,2,III}$, $B_{\cE}$, and $\cR(\bw)$ in \eqref{eq:cB_pf1}. See more discussions in Remark \ref{rem:int_wg} and Section \ref{sec:log_cancel}. Using a direct 
calculations and \eqref{eq:V_main:far}, we get 
\[
\bal
 I = \f{\pa_x \rho}{\rho} + \f{ 2 \alb x^{\alb} \bw}{3 \bv}
  = \f{1}{3 x \log x} + \f{ 2\alb \chi_1 \bwfa  }{3 \bv}
 = \f{1}{3 \bv} ( \f{ 1 }{x \log x} \bv + 2 \alb \bwfa + 2 \alb (\chi_1-1) \bwfa).
 \eal 
\]
Using $ \bv =( - \f23 \cJM + \cE_{ \eqref{eq:V_main:far}, V/x } ) x$  \eqref{eq:V_main:far} 
and the forms for $\cJM$ \eqref{eq:J_main_far} and $\bwfa$ \eqref{def:bwfa}, we derive
\beq
\bal
I & =\tf{1}{3 \bv} ( \tf23( - \tf{1}{\log x} \cJM + \bwfa) +  \cE )
= \tf{1}{3 \bv} ( 
   -\tf13 \ccb (\log x )^{-1/3}
  -  \tf43 \ccc (\log x)^{-2/3}
 +  \cE )  , \\
 |\cE | & \leq 2\alb (1-\chi_1) |\bwfa| 
+  |\log x |^{-1} C_{ \eqref{eq:V_main:far}, V/x}
\label{eq:expand_mix2}
\eal 
\eeq
We use \eqref{eq:chi_asym_rat} to bound asymptotics of $1-\chi_1$.

\subsubsection{Expand $ \cR(\bw)$ }\label{app:cR_expand}

We exploit the cancellation $\bv_x - \tf{1}{x} \bv - 4$ \eqref{eq:K_ux_min_u:main} 
and $ \tf13 + \f{x \pa_x \bw}{\bw}$ in \eqref{eq:WW_mix_asym} for a faster decay.  Using $\alb = \tf13$
and \eqref{eq:W_mix_asym}, we rewrite $\cR(\bw)$ in \eqref{eq:cR_recall} as 
\[
 \cR(\bw) = - \tf23 ( \bv_x - \tf1x \bv -4 ) - 2 \tf{\bv}{x} ( \tf{\pa_x \bw}{ \bw} + \tf13)
= - \tf23 ( \bv_x - \tf1x \bv -4 ) + 4 \alb \cJM \cdot \tf{x \pa_x (\bw x^{1/3}) }{ \bw x^{1/3}} - 2 (\tf{\bv}{x} + 2 \alb \cJM) \cdot \tf{x \pa_x (\bw x^{1/3}) }{ \bw x^{1/3}}
\]
where $\cJM$ is given in \eqref{eq:J_main_far}. Using \eqref{eq:K_ux_min_u:main},
\eqref{eq:V_main:far}, and \eqref{eq:W_mix_asym}, we derive 
\beq
\bal
 \cR(\bw) & = \f49 ( \ccb  (\log x )^{-1/3} + \ccc ( \log x)^{-2/3} )
 + \f43 \cJM  \cdot  \f{x \pa_x \bwfa }{ \bwfa  }  \\
 & \quad 
- \f23 \cE_{ \eqref{eq:K_ux_min_u:main},V_x, V} 
+ \f43 \cJM  \cdot \f{x \pa_x \chi_1}{\chi_1} 
 - 2 \cE_{ \eqref{eq:V_main:far}, V/x } \f{x \pa_x (\bw x^{1/3}) }{ \bw x^{1/3}}
\teq \cM_{\eqref{eq:cR_decomp} } 
+ \cE_{\eqref{eq:cR_decomp} } .
\label{eq:cR_decomp}
\eal
\eeq
where the main term $\cM_{\eqref{eq:cR_decomp} } $ denotes the first row on the $\mw{RHS}$,
and error part $\cE_{\eqref{eq:cR_decomp} }$ denotes the second row. 
Denote $s = (\log x )^{-1/3}$. 
Recall $\cJM$ from \eqref{def:J_main} and $\bwfa$ from \eqref{def:bwfa}.
We get
\[
\cJM = \log x \cdot ( \cca + \tf32 \ccb s + 3 \ccc s^2), 
\  \bwfa = \cca + \ccb s + \ccc s^2, 
\quad
 \pa_x s(x) = - \f{s(x)}{3x \log x},  
 \   \pa_x \bwfa = - \f{\ccb s + 2 \ccc s^2}{3 x \log x}. 
\]
Thus,  we estimate the main term $\cM$ as 
\[
\bal
\cM_{\eqref{eq:cR_decomp} } 
& = \tf49 ( \ccb s  + \ccc s^2 ) - \tf49( \cca + \tf32 \ccb s + 3 \ccc s^2) 
\cdot  \f{ \ccb s + 2 \ccc s^2}{  \cca + \ccb s + \ccc s^2 }  =  -\frac{2s^2\left(10 \ccc^2 s^2+8 \ccb \ccc s+\ccb^2 + 2 \cca \ccc \right)}{9\left(\ccc s^2+ \ccb s+ \cca \right)}  .
\eal
\]
Recall $\cca = -6$. Using \hyr[bd:intval]{\its Interval Arithmetic} 
to the $s$-function on $s \in [  0,  (\log \xed )^{-1/3}  ]$, we evaluate 
\beq
|\cM_{\eqref{eq:cR_decomp} } | \leq C_{ \eqref{eq:cR_main:far} } (\log x )^{-2/3}, \
\forall x \geq \xed .
\label{eq:cR_main:far}
\eeq

The estimates of $\cE_{\eqref{eq:cR_decomp}}$  follows from the bound for $\chi_1$ and \eqref{eq:WW_mix_asym}
\[
|\cE_{\eqref{eq:cR_decomp}}|
\leq 
\tf23 |\cE_{ \eqref{eq:K_ux_min_u:main},V_x, V} | 
+ \tf43 |\cJM|  \cdot  |\tf{x \pa_x \chi_1}{\chi_1} | 
 + 2 | C_{  \eqref{eq:V_main:far} ,V/x}| 
 \cdot 
 C_{ \eqref{eq:WW_mix_asym}}  (\log x )^{-4/3}.
\]

Since all constants in the above estimates are bounded, combining the above two estimates,
we prove the second estimate for $\cR(\bw)$ in \eqref{eq:lin_nloc:far}.

\subsection{Estimate for $B_\sst$ in Lemma \ref{lem:profile2} }\label{app:proof_near_stab}

The estimate for $B_\sst(x)$ in Lemma \ref{lem:profile2} follow the same strategy. 
For $x \in [0, \xed ]$,  we estimate it following Appendix \ref{app:pf_bound}. 
For $x \geq \xed$, instead of deriving the asymptotics of integrand 
in the form \eqref{eq:pf_far_asym1}, we derive 
asymptotics in the form of \eqref{eq:near_field_stab_idea}. 
Recall $\vmu_\sst,\bb_\sst \in \R^3$ from \eqref{def:mu_b_st}. Using 
\eqref{eq:V_l_asym} and \eqref{eq:vel_const}, we bound the asymptotics 
\[
  (\bv(y) - C_{\psio/x}(y) y \dfix)_+ \geq 
   C_{ \eqref{eq:V_l_asym} } \log \xed \cdot y ,
   \quad  C_{\psio/x}( \vmu_\sst, \bb_\sst, \cI, y)
   \leq \vmu_{\sst, \nst } \mCC_{ \psio/x, 0 }( \bb_{\sst, \nst }  , \R_+) 
   |y|^{\alb-\bb_{\sst, \nst }  },
\]
where $\nst$ denotes the last index in vectors $\vmu_{\sst}, \bb_\sst$. The asymptotics for $\vp, \rho$ are derived using monotonicity. With these asymptotics, 
we derive the far-field estimates following Appendix \ref{app:cR_far}.

\subsection{Implementation of the computer-assisted proof}\label{app:code}

Below, we discuss the implementation of the computer-assisted proof and the organization of the codes, which consist of five main folders. 
The complete computer-assisted codes are available at the Dropbox link.
\footnote{
\url{https://www.dropbox.com/scl/fo/2zhepdazxil8ocab2ezv9/AHNj-IlPrVqGaZLPGqHIawY?rlkey=ea3t5cdiz5rd5stxvucyo70z1&st=0dizs7oz&dl=0}
}

\vs{0.05in}
\paragraph{\bf Construct approximate profile: \texttt{Solve\_profile}.} 

This folder contains codes for constructing the approximate profile without tracking errors, 
implementing the numerical scheme in Section~\ref{sec:appr_profile}. 

\begin{itemize}[leftmargin=1.5em]

\item Run \texttt{Main\_setup.m} to prepare mesh, B-splines, and matrices for computing the nonlocal velocity.

\item Run \texttt{run\_Bspline.m} to solve the PDE~\eqref{eq:1D_dyn_solve} and construct the approximate profile.

\end{itemize}

Run \texttt{Plot\_solution.m} to generate Figure \ref{fig:solu_profile}. 
We have already constructed the profile and saved it in the file 
\texttt{solu\_C13\_pertb\_2000.mat}. Since these codes are \emph{not} part of the rigorous proof, one does \emph{not} need to run them and may skip this folder when verifying the proof.

\vs{0.05in}
\paragraph{\bf Main steps for the proof:  \texttt{Rigorous\_proof} }

This is the main folder containing codes for the proof.  One should implement the following codes in order.

\begin{itemize}[leftmargin=1em]

\item \textbf{Proof of Lemma \ref{lem:bw_basic_sign}}: \texttt{Main\_profile\_est.m}. Evaluate the approximate profile and error $\bw, \bv, \bar \cR$ on mesh $X$ \eqref{def:XX}, 
derive their piecewise bounds and asymptotics estimates.  
This codes verify  basic profile properties
in Lemma \ref{lem:bw_basic_sign}.

The main time-consuming step is the construction of the matrix $H_X$, which rigorously  evaluates the nonlocal terms with error controlled on the mesh $X$. We provide two options:
(1) set \texttt{Use\_HX = 1} to load the precomputed matrix;
(2) set \texttt{Use\_HX = 0} to generate the matrix $H_X$, which takes about $30$ minutes. \footnote{
Benchmark environment: All computations were performed on a MacBook Air (13-inch, 2024) with an Apple M3 chip, 16\,GB memory, running macOS Sonoma 14.6.
}
Other than computing this matrix, other steps take less than 3 minutes.

The following two programs each run in less than $5$ seconds.

\item \textbf{Proof of Lemma \ref{lem:close}}:  \texttt{Main\_fix\_point.m}.  It rigorously encloses the $B_{\bullet}$-functions in \eqref{eq:cB_pf1} following Appendix \ref{app:pf_bound}, \ref{app:cR_far}, and verify the inequalities in Lemma \ref{lem:close}.

\item 
\textbf{Proof of Lemma \ref{lem:profile2}}: \texttt{Main\_contraction.m} It rigorously encloses the function $B_{\sst}$ \eqref{eq:linW_est} in the proof of Theorem \ref{thm:near_field_stab} and  verify the inequalities in Lemma \ref{lem:profile2}.

\end{itemize}

This folder also contains several estimates of piecewise bounds following Appendix \ref{app:numerics}.

\vs{0.05in}
\paragraph{\bf \texttt{Rigorous\_evaluation}.}

This folder contains codes for evaluating B-spline representation \eqref{eq:Bspline_rep}, \(\bwf\) \eqref{eq:W_rep}, and the nonlocal terms \(\cK_{\bullet}(F)(x)\) \eqref{eq:ker}, 
following  Appendix \ref{app:numerics}.
The codes support both

\texttt{itl = 0}: computations without tracking errors, used to compute the approximate profiles;

\texttt{itl = 1} or \texttt{itl\_opt = 1}: rigorous interval-arithmetic computations tracking truncation and round-off errors, used in the rigorous proof with certified error control.

This folder does not contain the main executable scripts. 

\begin{itemize}[leftmargin=1.5em]

\item \texttt{Build\_solu\_WF.m}: Constructs functions and nonlocal terms associated with $\bwf$.

\item \texttt{BS4\_1D\_interp.m}: Constructs matrices mapping the coefficients $a_i$ in the Bspline representation \eqref{eq:Bspline_rep} to the values of its derivatives on the mesh $Z$.

\item \texttt{Calpha\_K\_interp.m}: Constructs matrices mapping the coefficients $a_i$ in the B-spline representation to the nonlocal terms $\psio(\ww)$ and $\psio_x(\ww)$ on the mesh $Z$.
\end{itemize}

\vs{0.05in}
\paragraph{\bf INTLAB toolbox and compatible functions: \texttt{Code\_itl}, \texttt{Intlab\_V11}} \
The basic MATLAB functions, \texttt{prod.m} for products, \texttt{cumsum.m} for cumulative sum, 
\texttt{sign.m} for sign, are not compatible with INTLAB. We therefore provide 
their compatible versions in 
 \texttt{Code\_itl}. 
 \texttt{Intlab\_V11} contains the MATLAB toolbox INTLAB (version 11 \cite{Ru99a}) for interval-arithmetic computations.

\bibliographystyle{plain}
\bibliography{selfsimilar}

\end{document}